\newif\ifsiam
\siamtrue
\siamfalse

\newif\ifthesis
\thesisfalse

\newif\ifarxiv
\arxivfalse
\arxivtrue

\def\ignore#1{}

\ifthesis
    \let\process=\use
\else
    \let\process=\ignore
\fi

\ifsiam
    \documentclass[review,onefignum,onetabnum]{siamonline220329}
    \newcommand\imagedir{./images}
\else
    \documentclass[english]{myarticle}
    \newcommand\imagedir{./images/arXiv/}
\fi

\newif\ifcompileimages
\compileimagestrue
\compileimagesfalse

\newcommand\tikzdir{./tikz}

\usepackage{lipsum}
\usepackage{amsmath,amsfonts,amssymb}
\usepackage[mathletters]{ucs}
\usepackage[normalem]{ulem}
\usepackage{graphicx}
\usepackage{epstopdf}
\usepackage{algorithmic}
\usepackage{multirow}
\usepackage{enumitem}
\usepackage{physics}
\usepackage{xspace}
\usepackage[nocompress]{cite}
\usepackage{breqn}
\usepackage[newfloat,frozencache=true,cachedir=minted-cache]{minted}
\usepackage{listofitems} 

\usepackage{environ}
\usepackage{caption}
\newenvironment{code}{\captionsetup{type=listing}}{}
\ifpdf%
  \DeclareGraphicsExtensions{.eps,.pdf,.png,.jpg}
\else
  \DeclareGraphicsExtensions{.eps}
\fi

\usepackage{enumitem}
\setlist[enumerate]{leftmargin=.5in}
\setlist[itemize]{leftmargin=.5in}

\ifsiam
\newsiamremark{remark}{Remark}
\crefname{remark}{Remark}{Remarks}
\newsiamremark{hypothesis}{Hypothesis}
\crefname{hypothesis}{Hypothesis}{Hypotheses}
\newsiamthm{claim}{Claim}
\fi

\newcommand{\TheTitle}{Bifurcation analysis of Bogdanov-Takens bifurcations in delay differential equations} 
\newcommand{\TheAuthors}{M.M. Bosschaert and Yu.A. Kuznetsov}

\ifsiam
\headers{\TheTitle}{\TheAuthors}
\fi

\title{{\TheTitle}\thanks{Submitted to the editors DATE.}}

\author{ 
M.M. Bosschaert\thanks{Department of Mathematics, Hasselt University, 
	Diepenbeek Campus, Agoralaan Gebouw D, 3590 Diepenbeek, Belgium
	(\email{maikel.bosschaert@uhasselt.be}).}
  \and
  Yu.A. Kuznetsov\thanks{Department of Mathematics, Utrecht University, 
	Budapestlaan 6, 3508 TA Utrecht, The Netherlands and 
	Department of Applied Mathematics, University of Twente, Zilverling Building, 
	7500AE Enschede, The Netherlands (\email{I.A.Kouznetsov@uu.nl}).}
}

\usepackage{amsopn}

\usepackage{color}
\definecolor{var}{rgb}{0,0.25,0.25}
\definecolor{comment}{rgb}{0,0.5,0}
\definecolor{kw}{rgb}{0,0,0.5}
\definecolor{str}{rgb}{0.5,0,0}
\definecolor{darkblue}{cmyk}{1,0,0,0.8}
\definecolor{darkred}{cmyk}{0,1,0,0.7}
\definecolor{orange}{cmyk}{0,0.5,1,0}
\definecolor{royalblue}{rgb}{0.00000,0.44700,0.74100}
\definecolor{royalorange}{rgb}{0.85000,0.32500,0.09800}%
\definecolor{royalyellow}{rgb}{0.92900,0.69400,0.12500}%
\definecolor{purple}{rgb}{0.5804, 0.0, 0.82745098}%
\definecolor{applegreen}{rgb}{0.55, 0.71, 0.0}
\definecolor{bittersweet}{rgb}{1.0, 0.44, 0.37}

\usepackage{pgfplots}
\usepackage{tikzscale}
\usepgfplotslibrary{groupplots}
\usetikzlibrary{pgfplots.groupplots}
\usetikzlibrary{plotmarks}
\usetikzlibrary{spy}
\usepgfplotslibrary{colorbrewer}
\pgfplotsset{compat=newest, ylabsh/.style={every axis y label/.style={at={(0,0.5)}, xshift=#1, rotate=90}}} 
            
\usetikzlibrary{external} 
\tikzsetexternalprefix{\imagedir} 
\tikzset{external/system call={lualatex -shell-escape -halt-on-error -interaction=batchmode -jobname "\image" "\texsource"}}
\usetikzlibrary{shapes,arrows,calc}
\usepgfplotslibrary{groupplots}

\newcommand{\includetikzscaled}[2][1.0]{%
    \tikzsetnextfilename{#2}%
    \ifcompileimages%
        \tikzsetnextfilename{#2}%
        \includegraphics[width=#1\linewidth]{\tikzdir/#2.tikz}
    \else
        \includegraphics{\imagedir/#2}
    \fi
}

\usepackage{mathtools}
\usepackage{textcomp} 
\definecolor{matlabYellow}{HTML}{FCFCDC}
\definecolor{matlabrulecolor}{HTML}{A8A8A8}
\definecolor{matlabGreen}{HTML}{379634}
\definecolor{matlabLilas}{RGB}{170,55,241}

\DeclareMathOperator{\atanh}{atanh}

\crefname{section}{Section}{Sections}
\crefname{subsection}{Section}{Sections}

\DeclareMathAlphabet\mathbfcal{OMS}{cmsy}{b}{n}

\newcommand{\rss}{\SUNSTAR{r}}

\newcommand{\PAIR}[3][]{\ensuremath{\langle #2,#3 \rangle_{#1}}}
\newcommand{\STAR}[1]{\ensuremath{#1^{\star}}}
\newcommand{\SUN}[1]{\ensuremath{#1^{\odot}}}
\newcommand{\SUNSTAR}[1]{\ensuremath{#1^{\odot\star}}}
\newcommand{\SUNSUN}[1]{\ensuremath{#1^{\odot\odot}}}

\newcommand{\WSTAR}{\ensuremath{\text{weak}^{\star}}}

\DeclareMathOperator{\LIP}{Lip}

\newcommand{\DEF}{\ensuremath{\coloneqq}}

\newcommand{\RR}{\ensuremath{\mathbb{R}}}
\newcommand{\RRR}[1]{\ensuremath{\RR^{#1\star}}} 
\newcommand{\CC}{\ensuremath{\mathbb{C}}}

\newcommand{\DET}{\operatorname{det}}

\newcommand{\DOM}{\ensuremath{\mathcal{D}}}
\newcommand{\INV}[1]{\ensuremath{#1^{\mathrm{INV}}}}
\newcommand{\BINV}[1]{\ensuremath{B^{\mathrm{INV}}_{#1}}}

\newcommand{\CM}{\ensuremath{\mathcal{W}^c_{\text{loc}}}}

\newcommand{\MATLAB}{\texttt{MATLAB}\xspace}
\newcommand{\OCTAVE}{\texttt{GNU Octave}\xspace}
\newcommand{\DDEBIFTOOL}{\texttt{DDE-BifTool}\xspace}

\setminted[julia]{linenos,breaklines,escapeinside=||, mathescape=true,
                  numbersep=3pt, gobble=2, frame=lines, framesep=2mm}
\setminted[MATLAB]{linenos,breaklines,escapeinside=||, mathescape=true,
                   numbersep=3pt, gobble=2, frame=lines, framesep=2mm}
\setminted[shell-session]{numbersep=3pt, frame=lines, framesep=2mm}
\usepackage{xr}
\makeatletter
\newcommand*{\addFileDependency}[1]{
  \typeout{(#1)}
  \@addtofilelist{#1}
  \IfFileExists{#1}{}{\typeout{No file #1.}}
}
\makeatother

\newcommand{\paper}{paper}

\ifpdf
\hypersetup{
  pdftitle = {\TheTitle},
  pdfauthor = {\TheAuthors}
}
\fi

\begin{document}

\ifarxiv \pdfbookmark[0]{Main Text}{maintext} \fi

\maketitle

\begin{abstract}
In this \paper{}, we will perform the parameter-dependent center manifold reduction
near the generic and transcritical codimension two Bogdanov--Takens bifurcation
in classical delay differential equations (DDEs). Using a generalization of the
Lindstedt-Poincar\'e method to approximate the homoclinic solution allows us to
initialize the continuation of the homoclinic bifurcation curves emanating from
these points. The normal form transformation is derived in the functional
analytic perturbation framework for dual semigroups (sun-star calculus) using a
normalization technique based on the Fredholm alternative. The obtained
expressions give explicit formulas, which have been implemented in the freely
available bifurcation software package \DDEBIFTOOL. The effectiveness is
demonstrated on various models\ifthesis\footnote{Submitted as}\fi.

\end{abstract}



\section{Introduction} 
The {\it Bogdanov--Takens bifurcation} caused by the presence of an equilibrium with a double zero eigenvalue is a well-studied singularity in dynamical systems.  It implies existence of saddle homoclinic orbits nearby, which is a global phenomenon. In particular, the codimension two Bogdanov--Takens bifurcation in finite-dimensional ordinary differential equations (ODEs) has been studied theoretically and applied in numerous reserach publications. The same is true for the infinite-dimensional dynamical systems generated by delay differential equations (DDEs). In the simplest case, often encountered in applications, such DDEs have the form
\begin{equation}
    \label{btdde:eq:discreteDDEs} 
    \dot{x}(t) = f(x(t),x(t-\tau_1),\ldots,x(t-\tau_m),\alpha),
    \qquad t \geq 0,
\end{equation}
where $x(t) \in \RR^n,\ \alpha \in \RR^p$, while $0 < \tau_1 < \tau_2 < \cdots <\tau_m$ are constant delays, and $f : \RR^{n \times (m + 1)} \times \RR^p \to \RR^n$ is a smooth mapping. These are known as \emph{discrete} DDEs.

Up to date, the standard available parameter-dependent center manifold theorem for DDEs in \cite{diekmann1995delay} assumed that the equilibrium persists for all nearby parameter values. This was a serious limitation, since in generic
unfoldings of the codimension two Bogdanov--Takens singularity it is not the case.

However, recently, in \cite{Switching2019}, this obstruction has been removed and the existence of finite-dimensional smooth parameter-dependent local center manifolds has been rigorously established in the functional analytic perturbation framework for dual semigroups (sun-star calculus) developed in \cite{Clement1987, Clement1988, Clement1989, Clement1989b}. Once the existence of these invariant manifolds is proved, the normalization technique for local bifurcations of ODEs developed in \cite{Kuznetsov1999} can be lifted rather easily to the infinite dimensional setting of DDEs. The advantages of this normalization technique are that the center manifold reduction and the calculation of the normal form coefficients are performed simultaneously by solving the so-called \emph{homological equation}. This method gives explicit expressions for the coefficients rather than a procedure as developed in \cite{Faria1995201, Faria1995}. The explicit expressions make them particularly suitable for both symbolic and numerical evaluation.

Indeed, utilizing the normalization method, the authors in \cite{Switching2019} obtained asymptotics to initialize the continuation of codimension one bifurcation curves of nonhyperbolic equilibria and cycles emanating from the codimension two \emph{generalized Hopf, fold-Hopf, Hopf-Hopf} and \emph{transcritical-Hopf} bifurcation points in DDEs of the form \cref{btdde:eq:discreteDDEs}. These asymptotics have been implemented into the fully \OCTAVE compatible \MATLAB package \DDEBIFTOOL \cite{DDEBIFTOOL,2014arXiv1406.7144S}.

Another recent developent is the rigorous derivation of higher-order asymptotics for the codimension one homoclinic bifurcation curve emanating from the generic codimension two Bogdanov--Takens bifurcation point in in two-parameter ODEs \cite{Bosschaert@Interplay}. Thus, by combining the results of the parameter-dependent center manifolds in DDEs from \cite{Switching2019} and the homoclinic asymptotics in ODEs from \cite{Bosschaert@Interplay}, we are in the position to perform the parameter-dependent center manifold reduction and normalization for the \emph{generic} and \emph{transcritical codimension two Bogdanov--Takens} bifurcations. This will allow us to initialize the continuation of codimension one bifurcation curves of nonhyperbolic equilibria and homoclinic solutions emanating from these codimension two points. Hopefully, our results and their software implementation will make the numerical analysis of 
Bogdanov--Takens bifurcations in DDEs from applications rather routine.

This \paper{} is organized as follows. We begin in \cref{btdde:sec:sunstar} with a short summary from \cite{Switching2019} on parameter-dependent center manifolds for classical DDEs and we state various results needed for the normalization technique.
In \cref{btdde:sec:Center_manifold_reduction} we describe the general technique that we use to derive the transformation from the orbital normal form on the parameter-dependent center manifold in the infinite-dimensional setting of DDEs.
In \cref{btdde:sec:parameter-dependent-center-manifold-reduction} the method is then applied to the generic and transcritical codimension two Bogdanov--Takens bifurcations. We provide explicit transformations necessary for the predictors of codimension one bifurcation curves. We do this in a form suitable for classical DDEs, covering cases that are more general than \cref{btdde:eq:discreteDDEs}. It is here where we see the true benefit of allowing for orbital normal forms on the center manifold. Indeed, we do not need to derive homoclinic asymptotics for the transcritical codimension two Bogdanov--Takens bifurcation separately. Instead, we only need to derive the center manifold transformation for the transcritical codimension two Bogdanov--Takens bifurcation. Then, using the blow-up transformations \cref{btdde:eq:blowup}, we obtain the same perturbed Hamiltonian system (up to order three) as in the generic Bogdanov--Takens bifurcation.

We employ our implementation in \DDEBIFTOOL to illustrate the accuracy of the codimension one bifurcation curve predictors through various example models, displaying the generic and transcritical codimension two Bogdanov--Takens bifurcations in \cref{btdde:sec:Examples}.  An in-depth treatment of the examples, including the \MATLAB and Julia source code to reproduce the obtained results, as well as a more in-depth analysis of the examples, are provided in the 
\ifthesis%
    \cref{chapter:BT_DDE_supplement}%
\fi%
\ifsiam%
    \hyperref[mysupplement]{online Supplement}%
\fi%
\ifarxiv%
    \hyperlink{mysupplement}{Supplement}%
\fi. \vskip 1ex

\section{Parameter-dependent center manifolds for DDEs}
\label{btdde:sec:sunstar}

Here we summarize those results from \cite{Switching2019} on parameter-dependent center manifold for classical DDEs, which are required for the normalization technique in \cref{btdde:sec:parameter-dependent-center-manifold-reduction}. For a general introduction on perturbation theory for dual semigroups (also known as sun-star calculus) we refer to \cite{diekmann1995delay}.

Consider the classical parameter-dependent DDE
\begin{equation}
    \tag{DDE}
    \label{btdde:eq:classicalDDE}
    \dot{x}(t)= F(x_t, \alpha), \qquad t \ge 0,
\end{equation}
where $F: X \times \RR^p \to \RR^n$ is $C^k$-smooth for some $k \ge 1$ with
$F(0,0) = 0$ and $X \DEF C([-h,0],\RR^n)$. Here for each $t \ge 0$, the
\emph{history function} $x_t : [-h,0] \to \RR^n$ defined by
\[
  x_t(\theta) \DEF x(t + \theta), \qquad \forall\,\theta \in [-h,0].
\]

It is convenient to split the right hand-side into its linear and nonlinear parts and write
\begin{equation}
  \label{btdde:eq:classicalDDE12}
      F(\phi,\alpha)= \PAIR{\zeta}{\phi} + D_2F(0,0)\alpha + G(x_t, \alpha).
\end{equation}
Here $\zeta : [0,h] \to \RR^{n \times n}$ is a matrix-valued function of bounded variation, normalized by the requirement that $\zeta(0) = 0$ and is right-continuous on the open interval $(0,h)$, and $G : X \to \RR^n$ is a $C^k$-smooth nonlinear operator with $G(0,0) = G_1(0,0) = G_2(0,0) = 0$. The pairing is defined by
\begin{equation}
  \label{btdde:eq:lindde_shorthand}
  \PAIR{\zeta}{\phi} \DEF \int_0^h{d\zeta(\theta)\phi(-\theta}),
\end{equation}
where the integral is of the Riemann--Stieltjes type.

Let $T$ be the $\mathcal{C}_0$-semigroup on $X$ corresponding to the linearization of \cref{btdde:eq:classicalDDE} at $0 \in X$ for the critical parameter value $\alpha = 0$. Suppose that the generator 
\[
    \DOM{(A)}  = \{ \phi \mid \dot \phi \in X, \dot \phi(0) = \PAIR{\zeta}{\phi}\}, 
    \qquad 
    A\phi = \dot \phi,
\]
of $T$ has $1 \le n_0 < \infty$ purely imaginary eigenvalues with corresponding $n_0$-dimensional real center eigenspace $X_0$. Then by \cite[Corollary 20]{Switching2019} there exists a $C^k$-smooth map $\mathcal{C} : U \times U_p \to X$ defined in a neighborhood of the origin in $X_0 \times \RR^p$ and such that for every sufficiently small $\alpha \in \RR^p$ the manifold $\CM(\alpha) \DEF \mathcal{C}(U,\alpha)$ is locally positively invariant for the semiflow generated by \cref{btdde:eq:classicalDDE} at parameter value $\alpha$.

Since $X$ is not reflexive, i.e. does not isomorphic to its dual space $\STAR{X}$, the adjoint semigroup $\STAR{T}$ is  only $\WSTAR$ continuous on $\STAR{X}$ and $\STAR{A}$ generates $\STAR{T}$ only in the $\WSTAR$ sense. The maximal subspace of strong continuity
\[
\SUN{X} \DEF \left\{\STAR{x} \in \STAR{X} \,:\, 
    t \mapsto \STAR{T}(t)\STAR{x} \text{ is norm-continuous on } \RR_+\right\}
\]
is invariant under $\STAR{T}$, and we have the representation
\begin{equation}
  \label{btdde:eq:xsun_dde}
  \SUN{X} = \RRR{n} \times L^1([0,h],\RRR{n}).
\end{equation}
The duality pairing between $\SUN{\phi} = (c,g) \in \SUN{X}$ and $\phi \in X$ is
\begin{equation}
  \label{btdde:eq:pairing_X_sun_X}
  \PAIR{\SUN{\phi}}{\phi} = c\phi(0)+ \int_{0}^{h}g(\theta)\,\phi(-\theta)\,d\theta.
\end{equation}

At this stage, we again have a $\mathcal{C}_0$-semigroup $\SUN{T}$ with generator $\SUN{A}$ on a Banach space $\SUN{X}$ so we can iterate the above construction once more. On the dual space 
  \[
    \SUNSTAR{X} = \RR^n \times L^\infty([-h,0], \RR^n),
  \]
we obtain the adjoint semigroup $\SUNSTAR{T}$ with the $\WSTAR$ generator  
\begin{equation}
    \label{btdde:eq:A_sunstar}
    \DOM{(\SUNSTAR A)}  = \{ (\alpha,\phi) \in \SUNSTAR{X} \mid \phi \in \LIP(\alpha) \}, 
    \qquad 
    \SUNSTAR A(\alpha,\phi) = ( \PAIR{\zeta}{\phi}, \dot \phi ).
\end{equation}
The duality pairing between $\SUNSTAR{\phi} = (a,\psi) \in \SUNSTAR{X}$ and
$\SUN{\phi} = (c,g) \in \SUN{X}$ is
\begin{equation}
  \label{btdde:eq:pairing_X_sun_star_X_sun}
  \PAIR{\SUNSTAR{\phi}}{\SUN{\phi}} = 
    ca + \int_{0}^{h}g(\theta)\,\psi(-\theta)\,d\theta.
\end{equation}

By restriction to the maximal subspace of strong continuity $\SUNSUN{X} =
\overline{\DOM(\SUNSTAR{A})}$, we end up with the $\mathcal{C}$-semigroup
$\SUNSUN{T}$. Its generator $\SUNSUN{A}$ is the part of $\SUNSTAR{A}$ in
$\SUNSUN{X}$.
The canonical injection $j : X \to \SUNSTAR{X}$ is given by
\begin{equation}
    \label{btdde:eq:j}
    j(\phi) = \left(\phi(0), \phi \right),
\end{equation}
mapping $X$ \emph{onto} $\SUNSUN{X}$. Therefore, $X$ is sun-reflexive with
respect to the shift semigroup $T$.

We are now in the position to state the second part of \cite[Corollary 20]{Switching2019}. That is, if the history $x_t$ associated with a solution of \cref{btdde:eq:classicalDDE} exists on some nondegenerate interval $I$ and $x_t \in \CM(\alpha)$ for all $t \in I$, then $u : I \to X$ defined by $u(t) \DEF x_t$ is differentiable and satisfies
\[
  j\dot{u}(t) = \SUNSTAR{A}ju(t) + (D_2F(0,0)\alpha)\rss 
                    + G(u(t),\alpha)\rss, \qquad \forall\,t \in I.
\]
Here, for $i = 1,\ldots,n$, we denote $\rss_i \DEF (e_i, 0)$, where $e_i$ is the
$i$th standard basis vector of $\RR^n$ and
\[
  w \rss \DEF \sum_{i=1}^n{w_i\rss_i}, \qquad 
    \forall\,w = (w_1,\ldots,w_n) \in \RR^n.
\]

\subsection{Spectral computations for classical DDEs in case of multiple eigenvalues}
It is well known that for classical DDEs all spectral information about the
generator $A$ is contained in a holomorphic \emph{characteristic matrix
function} $\Delta : \CC
\to \CC^{n \times n}$ defined by
\begin{equation}
\label{btdde:eq:CharMatrix}
  \Delta(z) \DEF zI - \hat{\zeta}(z) 
  \qquad \text{with} 
  \qquad \hat{\zeta}(z) \DEF \int_0^h{e^{-z\theta}\,d\zeta(\theta)},
\end{equation}
where $\zeta$ is the real kernel from \cref{btdde:eq:lindde_shorthand}, see
\cite[Sections IV.4 and IV.5]{diekmann1995delay}. In particular, the
eigenvalues of $A$ are the roots of the \emph{characteristic equation}
\begin{equation}
  \label{btdde:eq:main:det_delta}
\DET{\Delta(z)} = 0,
\end{equation}
and the algebraic multiplicity of an eigenvalue equals its order as a root of
\cref{btdde:eq:main:det_delta}.

For the normalization technique in \cref{btdde:sec:parameter-dependent-center-manifold-reduction},
we will need normalized representations for the (generalized) eigenfunctions and adjoint
(generalized) eigenfunctions of the generator $A$ and $A^\star$, respectively.
In this section, we will consider the more general case where $\lambda$ is an eigenvalue of
algebraic multiplicity $k \in \mathbb N$ and geometric multiplicity 1.  
Although in \cref{btdde:sec:parameter-dependent-center-manifold-reduction} we will
only need the special case where $\lambda=0$ is a double eigenvalue of $A$, the
expressions are useful when considering for example the 1:1 resonant Hopf
bifurcation and the triple zero bifurcation.

\begin{proposition}
\label{btdde:proposition:eigenvalues_multiplicity_k}
Let $\lambda$ be an eigenvalue of the generator $A$ with algebraic multiplicity
$k\in\mathbb N$ and geometric multiplicity one, then there are $($generalized\,$)$
eigenfunctions $\phi_i$ such that
\begin{equation}
\label{btdde:eq:eigenspaces_eigenfunctions}
A\phi_0 = \lambda\phi_0,\qquad A\phi_i = \lambda\phi_i + \phi_{i-1}, \qquad i \in \{1,\dots,k-1\},
\end{equation}
and adjoint $($generalized\,$)$ eigenfunctions $\psi_i$ such
that
\begin{equation}
\label{btdde:eq:eigenspaces_ad_eigenfunctions}
A^{\star}\psi_{k-1} = \lambda\psi_{k-1}, \qquad 
    A^{\star}\psi_{k-i} = \lambda\psi_{k-i} + \psi_{k-i + 1}, \qquad i \in \{2,\dots,k\}.
\end{equation}
Let the ordered set $(q_0,\dots,q_{k-1})$ of vectors be a Jordan chain for $\Delta(\lambda)$, i.e.
$q_0 \neq 0$ and 
\[
    \Delta(z)[q_0 + (z-\lambda) q_1 + \cdots + (z-\lambda)^k q_{k-1}] = \mathcal O((z-\lambda)^k).
\]
Similarly, let $(p_{k-1},\dots,p_0)$ be a Jordan chain for $\Delta^T(\lambda)$.
Then the $($generalized\,$)$) eigenfunctions and adjoint $($generalized\,$)$ eigenfunctions are given by 
\begin{equation}
\label{btdde:eq:eigenfunction_and_adjoint_eigenfunctions}
\begin{aligned}
\phi_i \colon [-h,0] \rightarrow \mathbb R^n &\colon \theta \mapsto e^{\lambda\theta} \sum_{l=0}^i q_{i-l} \frac{\theta^l}{l!}, \\
\psi_i \colon [0,h]  \rightarrow \mathbb R^n &\colon \theta \mapsto p_i + \sum_{l=0}^{k-1-i} p_{i + l} \int_0^\theta \int_{\sigma}^h 
    e^{\lambda(\sigma-s)}\frac{(\sigma-s)^{l}}{l!} d\zeta(s) d\sigma,
\end{aligned}
\end{equation}
for $i \in \{0,\dots, k-1\}$, respectively.
Furthermore, the following identities hold
\begin{equation}
\label{btdde:eq:eigenfunction_identities}
\begin{aligned}
\left<\psi_i,\phi_j\right> & = \left<\psi_{i + 1},\phi_{j + 1}\right>, & i,j\in\{0,\dots,k-2\},\\
\left<\psi_{k-1},\phi_{k-1}\right> & = p_{k-1} \sum_{l=0}^{k-1} \frac{\Delta^{(l + 1)}(\lambda)}{(l + 1)!}q_{k-1-l}, \\
\left<\psi_i,\phi_j\right> & = 0, & i>j, \\
\left<\psi_0,\phi_j\right> &= \sum_{l=0}^{k-1} \sum_{m=0}^j p_l \frac{\Delta^{(l + m+1)}(\lambda)}{(l + m+1)!} q_{j-m}, 
                                 & j>0,
\end{aligned}
\end{equation}
which can be normalized to satisfy 
\begin{equation}
    \label{btdde:eq:normalization_identity}
    \langle\psi_i,\phi_j\rangle = \delta_{ij}.
\end{equation}
\end{proposition}

\begin{proof}
The (generalized) eigenspace at an eigenvalue $\lambda$ of $A$ of algebraic multiplicity $k$ and geometric multiplicity 1 is given by 
\[
\mathcal{N}((A-\lambda)^k),
\]
which leads to the expressions in \cref{btdde:eq:eigenspaces_eigenfunctions} and similarly for \cref{btdde:eq:eigenspaces_ad_eigenfunctions}. The representations of the (generalized) eigenfunctions and adjoint (generalized) eigenfunctions can be found in Theorem IV.5.5 and IV.5.9 in \cite{diekmann1995delay}, respectively.

The first identity in \cref{btdde:eq:eigenfunction_identities} follows directly from
\[
    \left<\psi_{i + 1},\phi_{j + 1}\right> = \left<(\lambda-A^\star)\psi_i,\phi_{j + 1}\right>
    = \left<\psi_i,(\lambda-A)\phi_{j + 1}\right>
    = \left<\psi_i,\phi_j\right>,
\]
where $i,j\in\{0,\dots,k-2\}$.
For the second identity in \cref{btdde:eq:eigenfunction_identities} we notice that
\begin{align*}
    \left<\psi_{k-1},\phi_{k-1}\right> 
        &= \int_0^h d\psi_{k-1}(\theta) \phi_{k-1}(-\theta) 
         = p_{k-1}q_0 + \int_0^h \psi'_{k-1}(\theta) \phi_{k-1}(-\theta) d\theta \\
        &=  p_{k-1}q_0 + \int_0^h p_{k-i} \int_{\theta}^h e^{\lambda(\theta-s)} d\zeta(s) 
                e^{-\lambda\theta} \sum_{l=0}^{k-1} q_{i-l} \frac{(-1)^l\theta^l}{l!}d\theta \\
        &=  p_{k-1}q_0 + p_{k-i} \sum_{l=0}^{k-1} \int_0^h \int_{\theta}^h e^{-\lambda s}  
                \frac{(-1)^l\theta^l}{l!} d\zeta(s) d\theta  q_{i-l}\\
        &=  p_{k-1}q_0 + p_{k-i} \sum_{l=0}^{k-1} \int_0^h \int_0^s e^{-\lambda s}  
                \frac{(-1)^l\theta^l}{l!}d\theta  d\zeta(s) q_{i-l}\\
        &=  p_{k-1}q_0 + p_{k-i} \sum_{l=0}^{k-1} \int_0^h  e^{-\lambda s}  
            \frac{(-1)^l s^{l + 1}}{(l + 1)!} d\zeta(s) q_{i-l}
        = p_{k-1} \sum_{l=0}^{k-1} \frac{\Delta^{(l})(\lambda)}{(l + 1)!} q_{i-l},
\end{align*}
where we used Fubini\textquoteright s theorem to change the order of
integration. The last equality holds since
\[
\Delta'(z) =  zI + \int_0^h{\theta e^{-z\theta}\,d\zeta(\theta)} 
\]
and
\[
\Delta^{(n)}(z) =  (-1)^{n + 1}\int_0^h{\theta^n e^{-z\theta}\,d\zeta(\theta)}, \quad n > 1.
\]

Using the first identity in \cref{btdde:eq:eigenfunction_identities} and that for $j>0$ we have
\begin{align*}
    \left< \psi_j,\phi_0\right> = \left< (\lambda-A^\star)\psi_{j-1},\phi_0\right> 
                                = \left< \psi_{j-1},(\lambda-A)\phi_0\right> = 0,
\end{align*}
the third identity in \cref{btdde:eq:eigenfunction_identities} follows.

For the last identity in \cref{btdde:eq:eigenfunction_identities}, we have that
\begin{align*}
    \left< \psi_0, \phi_j \right> 
    &= \int_0^h d\psi_0(\theta) \phi_j(-\theta) 
         = p_0q_j + \int_0^h \psi'_0(\theta) \phi_j(-\theta) d\theta \\
    &= p_0 q_j + \int_0^h \left( \sum_{l=0}^{k-1} p_l 
        \int_\theta^h e^{\lambda(\theta-s)}\frac{(\theta-s)^l}{l!} d\zeta(s)  
        e^{-\lambda\theta} \sum_{m=0}^j q_{j-m} (-1)^m\frac{\theta^m}{m!} \right) d\theta \\
    &= p_0 q_j + \sum_{l=0}^{k-1} \sum_{m=0}^j p_l \int_0^h  
        \int_\theta^h e^{-s\lambda}\frac{(\theta-s)^l}{l!} d\zeta(s)  
        (-1)^m\frac{\theta^m}{m!} d\theta q_{j-m}   \\
    &= p_0 q_j + \sum_{l=0}^{k-1} \sum_{m=0}^j p_l \int_0^h  
        \int_0^s e^{-s\lambda}\frac{(\theta-s)^l}{l!}   
        (-1)^m\frac{\theta^m}{m!} d\theta d\zeta(s) q_{j-m} \\
    &= p_0 q_j + \sum_{l=0}^{k-1} \sum_{m=0}^j p_l \int_0^h  
    e^{-s\lambda} \frac{(-1)^{l + m} s^{l + m+1}}{(l + m+1)!} d\zeta(s) q_{j-m}
    \process{%
        \\ &
    }
    = \sum_{l=0}^{k-1} \sum_{m=0}^j p_l \frac{\Delta^{(l + m+1)}(\lambda)}{(l + m+1)!} q_{j-m}.
\end{align*}
Here we again used Fubini's theorem to reverse the order of integration and the beta function of Euler to integrate the term
\[
\int_0^s (\theta-s)^l\theta^m d\theta.
\]

To prove the normalization condition $\langle\psi_0,\phi_0\rangle = 1$,
we start by showing that $\langle\psi_0,\phi_0\rangle$
is non-vanishing. Consider the direct sum decomposition 
\begin{align*}
X & = \mathcal{N}((\lambda-A)^k)\oplus\overline{\mathcal{R}((\lambda-A)^k)}
    = \mathcal{N}((\lambda-A)^k)\oplus{}^{\bot}\mathcal{N}((\lambda-A^{\star})^k),
\end{align*}
see Theorem IV.2.5 in \cite{diekmann1995delay}. Since $\phi_0\in\mathcal{N}((\lambda-A)^k)$
and $\mathcal{N}((\lambda-A^{\star})^k)$ is spanned by $\{\psi_0, \cdots \psi_k\}$
it follows that $\langle\psi_0,\phi_0\rangle\neq0$.

In order to achieve the normalization in \cref{btdde:eq:normalization_identity}, we
observe that the eigenfunctions $\phi_j\, (j\in\{0,\dots k-1\})$, are invariant
under the transformations
\begin{align}
    \phi_0 &\rightarrow \alpha \phi_0, \\
    \phi_j &\rightarrow \alpha (\phi_j + \delta_j \phi_0), \quad j\in\{1,\dots,n\},
\end{align}
where $\alpha, \delta_j \in \mathbb R$ for $j\in\{0,\dots,k-1\}$ and $\alpha \neq 0$.
Using the first identity in \cref{btdde:eq:eigenfunction_identities}, we see that it is sufficient
to normalize $\left< \psi_0, \phi_0 \right>$ to 1 and $\left<\psi_0, \phi_j \right>$ to 0.
Thus, using the invariance of the eigenfunctions, we obtain the solutions
\begin{align*}
    \alpha = \frac{1}{\left< \psi_0, \phi_0 \right>}, \quad
    \delta_j = -\left< \psi_0, \phi_j \right>, \quad j\in\{1,\dots,n\}.
\end{align*}
We finish with the remark that the invariance of the eigenfunctions is equivalent to the invariance
for the Jordan chains $(q_0,q_1, \dots, q_{k-1}) \rightarrow \alpha (q_0,q_1 + \delta_1 q_0, \dots, q_{k-1} + \delta_{k-1} q_0)$.
\end{proof}

\begin{remark}
The Jordan chains for $(q_0,\dots,q_{k-1})$ and $(p_{k-1},\dots,p_0)$ in
the previous \cref{btdde:proposition:eigenvalues_multiplicity_k} can be computed
as follows. Set for $j\in\mathbb N$
\[
    P_j = P_j(z) = \frac{\Delta^{(j-1)}(z)}{(j-1)!},
\]
and define $A_k = A_k(z)$ to be the $(nk)\times(nk)$-matrix
\[
    A_k = \begin{pmatrix}
    P_1 & 0 & \cdots & 0 \\
    P_1 & P_2 & \cdots & 0 \\
    \cdots &  & \ddots & \vdots \\
    P_k & P_{k-1} & \cdots &  P_1,
    \end{pmatrix}
    \qquad k \in \mathbb N.
\]
Then the Jordan chains for $(q_0,\dots,q_{k-1})$ and  $(p_{k-1},\dots,p_0)$ can
be constructed via 
\[
A_k(\lambda)
\begin{pmatrix}
    q_0 \\
    q_1 \\
    \vdots \\
    q_{k-1}
\end{pmatrix} = 0, \qquad
\begin{pmatrix}
    p_0 & p_1 & \cdots & p_{k-1}
\end{pmatrix}
A_k(\lambda) = 0,
\]
see \cite[Chapter IV Exercise 5.11]{diekmann1995delay}.
\end{remark}

\subsection{Solvability of linear operator equations with double eigenvalues}
\label{btdde:sec:solvability}

When computing the normal form coefficients for the Bogdanov--Takens
bifurcation in \cref{btdde:sec:parameter-dependent-center-manifold-reduction}, using
the normalization technique described in \cref{btdde:sec:Center_manifold_reduction},
we will encounter linear operator equations of the following form
\begin{equation}
  \label{btdde:eq:general_system_sunstar}
  \left( \lambda - \SUNSTAR{A} \right)j v = (w_0,w),
\end{equation}
where $\lambda$ is a double eigenvalue, while no other eigenvalues are present
on the imaginary axis, $(w_0,w) \in \SUNSTAR{X}$ is given and $v \in D(A)$ is
the unknown. Note that in general $w_0 \neq w(0)$. Although we will in
\cref{btdde:sec:parameter-dependent-center-manifold-reduction} only need the special
case where $\lambda = 0$, the results below hold for non-zero (complex)
eigenvalues as well.

Firstly, since $\lambda = 0$ is an eigenvalue, we need to assure that
\cref{btdde:eq:general_system_sunstar} is solvable. Therefore, let $\psi_1$, as given
in \cref{btdde:proposition:eigenvalues_multiplicity_k} with $k=2$ and $\lambda=0$, be
the adjoint eigenfunction of $A^\star$. Then
\cref{btdde:eq:general_system_sunstar} has a solution
$(\phi_{0},\phi)\in\mathcal{D}(A^{\odot\star})$ if and only if $(w_0,w)$
annihilates the null space $\mathcal N(\lambda I-\STAR A)$, i.e., if and only
if 
\begin{equation}\tag{FSC}
  \label{btdde:eq:FSC}
  \PAIR{(w_0,w)}{\psi_1} = 0.
\end{equation}
A proof can be found in \cite[Lemma 3.2]{Janssens:Thesis}. This condition is
often referred to as the \emph{Fredholm solvability condition}. 

\begin{proposition}
\label{btdde:prop:solution_double_eig}
Suppose $\lambda$ is a double eigenvalue of A and assume that
\cref{btdde:eq:general_system_sunstar} is consistent for a given $(w_0,w)\in
X^{\odot\star}$. Let $(q_0,q_1)$ and $(p_1,p_0)$ be the Jordan chains for
$\Delta(0)$ and $\Delta^T(0)$, respectively. Let $\phi_0$ be an eigenfunction of
$A$ as in \cref{btdde:proposition:eigenvalues_multiplicity_k}. Then the solution to
\cref{btdde:eq:general_system_sunstar} is given by
\[
v(\theta) = e^{\lambda\theta}\xi +
            \int_{\theta}^{0}e^{\lambda(\theta-s)}w(s)\,ds + 
            \gamma \phi_0(\theta),
            \qquad\theta\in[-h,0],
\]
with
\[
\xi = \INV{\Delta(\lambda)}\left[
            w_0 + 
            \int_{0}^{h}d\zeta(\theta) 
                \int_{-\theta}^{0}e^{-\lambda (\theta + s)}w(s) \,ds
       \right].
\]
and $\gamma$ some constant. 
\end{proposition}

\begin{proof}
From formula \cref{btdde:eq:A_sunstar} we see that \cref{btdde:eq:general_system_sunstar} is nothing more than solving the
first order ordinary differential equation 
\begin{equation}
    \label{btdde:eq:inODE}
    \lambda v-\dot{v} = w,
\end{equation}
which must satisfy
\begin{equation}
    \label{btdde:eq:inODE_condition}
    \lambda v(0)-\int_{0}^{h}d\zeta(\theta)v(-\theta) = w_0.
\end{equation}

Solving \cref{btdde:eq:inODE} for $v \in \mathcal D(A)$ we obtain
\[
v(\theta) = e^{\lambda\theta}v(0) + \int_{\theta}^{0}e^{\lambda(\theta-s)}w(s)\,ds\qquad\left(\theta\in[-h,0]\right).
\]
It then follows from \cref{btdde:eq:inODE_condition} that
\[
    \Delta(\lambda)v(0) = w_0 +\int_{0}^{h}d\zeta(\theta)\int_{-\theta}^0 e^{-\lambda (\theta + s)}w(s) \,ds,
\]
where we used the definition of the characteristic matrix in \cref{btdde:eq:CharMatrix}.
Solving for $v(0)$ yields
\[
v(0) = \INV{\Delta(\lambda)}
        \left\{ w_0 +\int_{0}^{h}d\zeta(\theta)\int_{-\theta}^0 e^{-\lambda (\theta + s)}w(s) \,ds \right\} +\gamma q_{0},
\]
where $\gamma$ is some constant. Now define
\[
\tilde{v}(\theta) = e^{\lambda\theta}\xi + \int_{\theta}^{0}e^{\lambda(\theta-s)}w(s)\,ds,
\]
so that
\[
v(\theta) = \tilde{v}(\theta) + \gamma \phi_{0}(\theta).
\]
\end{proof}

\begin{remark}
We observe that the expression for $v(0)$ itself involves a bordered {\em matrix} inverse,
\[
\INV{\Delta}(\lambda) : \mathcal{R}(\Delta(\lambda)) \to \CC^n,
\]
which assigns the unique solution of the extended linear system
\[
  \Delta(\lambda)x = y, \qquad q_0 \cdot x = 0,
\]
to every $y \in \CC^n$ for which the system $\Delta(\lambda)x = y$ is consistent. In practice, $x=\INV{\Delta}(\lambda)y$ can be obtained by solving the nonsingular bordered \emph{matrix} system
\[
\begin{pmatrix}
\Delta(\lambda) & p_1^T \\
q_0^T & 0
\end{pmatrix}\begin{pmatrix}
x\\
s
\end{pmatrix}=\begin{pmatrix}
y\\
0
\end{pmatrix}
\]
for the unknown $(x,s) \in \CC^{n+1}$ that, by Cramer's rule, necessarily
satisfies $s = 0$. The properties of (finite dimensional) bordered linear
systems and their role in numerical bifurcation analysis are discussed more
extensively in \cite{Keller1987Numerical} and \cite[Chapter
3]{govaerts2000numerical}.
\end{remark}

In \cref{btdde:sec:parameter-dependent-center-manifold-reduction}, we will encounter
solely systems in which $\lambda=0$ and $w$ is of polynomial type. For this
situation, we have the following results.

\begin{corollary}
\label{btdde:corollary:sol_double_zero_eig_polynomial} 
Suppose that in addition to the assumptions in \cref{btdde:prop:solution_double_eig} that $\lambda = 0$ and the right-hand side in \cref{btdde:eq:general_system_sunstar} is given by
\[
\begin{pmatrix}
w_0 \\ w
\end{pmatrix}
=
\kappa \rss - j\left(\theta \mapsto c_0 + c_1 \theta + \dots + c_n \theta^n\right),
\]
then
\[
    v(\theta) = \xi + c_0 \theta + \frac{c_1}2 \theta^2 + \dots + \frac{c_n}{n + 1} \theta^{n + 1} + \gamma \phi_0(\theta), 
                \qquad\left(\theta\in[-h,0]\right),
\]
with
\begin{align*}
    \xi & = \INV{\Delta}(0)\left[\kappa - \Delta'(0) c_0  - \frac{\Delta''(0)}2 c_1 - \dots - \frac{\Delta^{(n + 1)}(0)}{n + 1} c_n \right],
\end{align*}
is the solution to the system
\[
-\SUNSTAR A j v  = 
\begin{pmatrix}
w_0 \\ w
\end{pmatrix}.
\]
\end{corollary}

\begin{remark}
In \cref{btdde:sec:parameter-dependent-center-manifold-reduction} we will use the shorthand notation
\[
    v = \BINV 0 (\kappa - w)
\]
for the solutions in \cref{btdde:corollary:sol_double_zero_eig_polynomial}.
\end{remark}

\begin{lemma}
\label{btdde:lemma:pairing_with_adjoint_eigenfunctions}
Suppose that in addition to the assumptions in \cref{btdde:prop:solution_double_eig} that $\lambda = 0$ and
let $w\in X$ be given by
\[
    w \colon [-h,0] \rightarrow \mathbb R^n 
      \colon \theta \mapsto c_0 + c_1 \theta + \dots + c_n \theta^n.
\]
Then the pairing with the adjoint (generalized) eigenfunctions $\psi_0$ and
$\psi_1$ in \cref{btdde:proposition:eigenvalues_multiplicity_k} with $k=0$ and
$\lambda=0$ are given by
\begin{align*}
    \left< \psi_0, w \right> &=
    p_0 \left( \Delta'(0) c_0 + \frac{\Delta''(0)}2 c_1 + \dots + \frac{\Delta^{(n + 1)}(0)}{n + 1} c_n \right) \\
    & \quad + p_1 \left( \frac{\Delta''(0)}{2} c_0 + 
        \frac{\Delta'''(0)}{3\times 2} c_1 + \dots + \frac{\Delta^{(n + 2)}(0)}{(n + 2)(n + 1)} c_n \right)
\end{align*}
and
\begin{align*}
    \left< \psi_1, w \right> &=
    p_1 \left( \Delta'(0) c_0 + \frac{\Delta''(0)}2 c_1 + \dots + \frac{\Delta^{(n + 1)}(0)}{n + 1} c_n \right),
\end{align*}
respectively.
\end{lemma}

\section{Parameter-dependent center manifold reduction combined with normalization and time-reparametrization}
\label{btdde:sec:Center_manifold_reduction}

In this section, we combine the results from
\cite{Bosschaert@Interplay} and \cite{Switching2019}. That is, we lift the normalization technique
for local bifurcations of ODEs to the infinite dimensional settings of DDEs
\cite{Switching2019}, while simultaneously incorporating a time-reparametrization 
to allow for a further simplification of the normal form \cite{Bosschaert@Interplay}.
Thus, suppose that $0 \in X$ is a stationary state of \cref{btdde:eq:classicalDDE} at the
critical parameter value $0 \in \RR^p$ and assume there are $n_0 \ge 1$
eigenvalues on the imaginary axis, counting algebraic multiplicities. Let $P_0$
be the corresponding \emph{real} spectral projector on $X$, so the range $X_0$
of $P_0$ is the \emph{real} $n_0$-dimensional center eigenspace.
\cite[Corollary 20]{Switching2019}  applies to give a parameter-dependent local
center manifold $\CM(\alpha)$ for \cref{btdde:eq:classicalDDE}.

We allow for the introduction of a new parameter $\beta$ defined in a
neighborhood of $0 \in \RR^p$ such that $\alpha = K(\beta)$ for some locally
defined $C^k$-diffeomorphism $K : \RR^p \to \RR^p$ that is to be determined
below. If $u : I \to X$ with $u(t) \DEF x_t \in \CM(\alpha)$ is as in
\cite[Corollary 20]{Switching2019}, then $u$ is differentiable on $I$ and
satisfies
\begin{equation}
  \label{btdde:eq:ODEonCM}
  j\dot{u}(t) = \SUNSTAR{A}ju(t) + (D_2F(0,0)K(\beta))\rss 
    + R(u(t),K(\beta)), \qquad \forall\,t \in I,
\end{equation}
where $R$ encodes the nonlinear part of $F$ as given by $G(u(t),K(\beta))\rss$.
Choose a basis $\Phi$ of $X_0$ and let $\mathcal{H} : \RR^{n_0} \times \RR^p
\to X$ be a locally defined $C^k$-smooth parametrization of $\CM(\alpha)$ with
respect to $\Phi$ and in terms of the new parameter $\beta$. For every $t \in
I$ we define $v(t) \in \RR^{n_0}$ as the coordinate vector of $P_0 u(t)$ with
respect to $\Phi$. Then $v : I \to \RR^{n_0}$ satisfies a parameter-dependent
ordinary differential equation of the form
\begin{equation}
  \label{btdde:eq:ODEexpansion}
  \dot{v}(\eta) = \sum_{|\nu| + |\mu| \geq 1}\frac{1}{\nu!\mu!}g_{\nu\mu}v^{\nu}(\eta)\beta^{\mu}.
\end{equation}
The multi-indices $\nu$ and $\mu$ have lengths $n_0$ and $p$, respectively. We
assume that \cref{btdde:eq:ODEexpansion} is known. Since $\mathcal{H}$ parametrizes
$\CM(\alpha)$,
\begin{equation}
    \label{btdde:eq:u_H_relation}
    u(t(\eta)) = \mathcal{H}(v(\eta), \beta), \qquad t \in I,
\end{equation}
with both $u$ and $v$ depending on the parameter, although this is left
implicit in the notation. Next, let the time $t$ in \cref{btdde:eq:ODEonCM} and the
time $\eta$ in the normal form \cref{btdde:eq:ODEexpansion} be related through the
parameter-dependent time-rescaling
\begin{equation}
\label{btdde:eq:vartheta}
    \frac{dt}{d\eta} = \vartheta(v, \beta), \qquad 
        \vartheta\colon \mathbb R^{n_c} \times \mathbb R^2 \to \mathbb R.
\end{equation}
Substituting the above relation \cref{btdde:eq:u_H_relation} into \cref{btdde:eq:ODEonCM}
and taking into account \cref{btdde:eq:vartheta} produces the \emph{homological equation}
\begin{equation}
  \label{btdde:eq:homological_equation}
  \tag{HOM}
  \left ( \SUNSTAR{A}j\mathcal{H}(v,\beta) + (D_2F(0,0)K(\beta))\rss 
          + R(\mathcal{H}(v,\beta),K(\beta)) \right) \vartheta(v,\beta)
   = jD_1\mathcal{H}(v,\beta)\dot{v},
\end{equation}
with $\dot{v}$ given by the parameter-dependent normal form
\cref{btdde:eq:ODEexpansion}. The unknowns in \cref{btdde:eq:homological_equation} are
$\mathcal{H}$, $K$, $\vartheta$, and the coefficients $g_{\nu\mu}$ from
\cref{btdde:eq:ODEexpansion}. For $r, s \ge 0$ with $r + s \ge 1$ we denote by
$D_1^rD_2^sF(0,0) : X^r \times [\RR^p]^s \rightarrow \RR^n$ the mixed Fr\'echet
derivative of order $r + s$, evaluated at $(0,0) \in X \times \RR^p$, with the
understanding that at most one of the factor spaces $X^r$ or $[\RR^p]^s$ is
absent if either $r = 0$ or $s = 0$. We expand the nonlinearity $R$ as
\begin{equation}
  \label{btdde:eq:R}
  R(\phi,\alpha) = \sum_{r + s > 1}\frac{1}{r!s!}D_1^rD_2^sF(0,0)(\phi^{(r)},\alpha^{(s)})\rss,
\end{equation}
where $\phi^{(r)} \DEF (\phi,\dots,\phi)\in X^r$ and $\alpha^{(s)} \DEF
(\alpha,\dots,\alpha)\in [\RR^p]^s$. The mappings $\mathcal{H}$, $K$, and
$\theta$ can be expanded as
\begin{equation}
\label{btdde:eq:hKthetaexpansion}
\begin{aligned}
    \mathcal{H}(v,\beta)\ ={}& \sum_{|\nu| + |\mu| \geq 1}\dfrac{1}{\nu!\mu!}h_{\nu\mu}v^{\nu}\beta^{\mu}, \\
    K(\beta) ={}& \sum_{|\mu| \geq 1}\dfrac{1}{\mu!}K_{\mu}\beta^{\mu}, \\
    \vartheta(w, \beta) ={}& \sum_{|\nu|+|\mu| \geq 0} \frac1{\nu!\mu!} \theta_{\nu\mu} w^\nu \beta^\mu.
\end{aligned}
\end{equation}
Substituting \cref{btdde:eq:ODEexpansion,btdde:eq:R,btdde:eq:hKthetaexpansion} into
\cref{btdde:eq:homological_equation}, collecting coefficients of terms
$v^{\nu}\beta^{\mu}$ from lower to higher order and solving the resulting
linear systems, one can solve recursively for the unknown coefficients
$g_{\nu\mu}$, $h_{\nu\mu}$, and $K_{\mu}$ by applying the Fredholm alternative
and taking bordered inverses, as explained in \cref{btdde:sec:solvability}.

\begin{remark}
To determine the coefficients $g_{\nu\mu}$, $h_{\nu\mu}$, $K_{\mu}$ and
$\theta_{\nu\mu}$ that are needed to include in order to translate asymptotics
to solutions in the normal form \cref{btdde:eq:ODEexpansion} in such a way that
the approximation order to the asymptotics are also obtained in the original
system under consideration, is a non-trivial task. Indeed, one needs to
understand precisely which coefficients $g_{\nu\mu}$ in
\cref{btdde:eq:ODEexpansion} affect the obtained asymptotic up to a certain
order. Then, using \cite[Proposition 1]{Bosschaert@Interplay}, we can derive
which coefficients $h_{\nu\mu}$, $K_{\mu}$ and $\theta_{\nu\mu}$ are sufficient
to include into the transformations.
\end{remark}

\ifarxiv
\section{Parameter-dependent center manifold reduction near Bog\-danov--Takens points}
\else
\section{Parameter-dependent center manifold reduction near Bogdanov--Takens points}
\fi
\label{btdde:sec:parameter-dependent-center-manifold-reduction} 
Using the method as outlined in \cref{btdde:sec:Center_manifold_reduction}, we derive
here the coefficients needed to translate the homoclinic third-order asymptotics
emanating from generic and transcritical codimension two Bogdanov--Takens
bifurcations to the parameter-dependent center manifold. Additionally, we derive
asymptotics for the codimension one equilibria bifurcations emanating from these
Bogdanov--Takens bifurcation points.

Thus, suppose that \cref{btdde:eq:classicalDDE} has an equilibrium $x_0 \equiv 0$ at
the critical parameter value $\alpha_0 = (0,0) \in \mathbb R^2$ with a double
(but not semisimple) zero eigenvalue
\begin{equation}
    \lambda_{1,2} = 0,
\end{equation}
which are the only eigenvalues on the imaginary axis. Furthermore, we assume
that the critical normal coefficients $a$ and $b$, to be defined below, are
non-zero. We then consider two different cases, depending on whether the
equilibrium remains fixed under parameter variation or not. For both cases there
are additional transversality conditions which need to be met, guaranteeing the
local invertibility of the parameter mapping $K$, see \cref{btdde:eq:K_expansion-BT}.

\subsection{Generic Bogdanov--Takens bifurcation}
\label{btdde:sec:generic_bogdanov-takens}

The $C^\infty$-equivalent normal form on the parameter-dependent
center manifold is
\begin{equation}
\label{btdde:eq:normal_form_orbital}
\begin{cases}
\begin{aligned}
	\dot w_0 & =  w_1, \\
	\dot w_1 & =  \beta_1 + \beta_2 w_1 + aw_0^2 + b w_0 w_1 + w_0^2 w_1
								h(w_0,\beta) + w_1^2 Q(w_0,w_1,\beta),
\end{aligned}
\end{cases}
\end{equation}
where $h$ is $C^\infty$ and $Q$ is $N$-flat for an a priori given $N$,
see~\cite{Broer1991}. Here, the dot represents the derivative with respect to
the new time $\eta$ of $w_i(\eta)(i = 0,1)$.  Furthermore, it is shown in
\cite{Bosschaert@Interplay} that we can assume $h(0,0) = 0$. 
By \cite[Proposition 1]{Bosschaert@Interplay}, it is both necessary and sufficient
in order to translate the third-order homoclinic predictor to the system
\cref{btdde:eq:classicalDDE} to expand the functions $R \colon X \times \mathbb R^2 \rightarrow \SUNSTAR{X}$, 
$\mathcal{H} \colon \mathbb R^2 \times \mathbb R^2 \rightarrow X$,
$K \colon \mathbb R^2 \rightarrow \mathbb R^2$, and $\vartheta \colon \mathbb R^2 \rightarrow \mathbb R$ 
 defined in \cref{btdde:sec:Center_manifold_reduction}
as follows
\begin{align}
R(u,\alpha) = & \label{btdde:eq:R_expansion_BT}
    \left(\frac{1}{2} B(u,u) + A_1(u,\alpha) + \frac{1}{2} J_2(\alpha,\alpha) + \frac{1}{6} C(u,u,u) 
    + \frac{1}{2} B_1(u,u,\alpha) \right. \\
    &\phantom{\Bigl(}+ \left. \frac{1}{2} A_2(u,\alpha,\alpha) + \frac{1}{6} J_3(\alpha,\alpha,\alpha) 
    + \mathcal O(\|(u\|,\|\alpha)\|^4) \right) \rss , \nonumber \\
\mathcal H(w,\beta) = {}& \label{btdde:eq:h_expansion-BT}
    h_{0010}\beta_1 + h_{0001} \beta_2 
    + \frac12 h_{2000}w_0^2 + h_{1100}w_0w_1 + \frac12 h_{0200}w_1^2 \\
    & + h_{1010}w_0\beta_1 + h_{1001}w_0\beta_2 + h_{0110}w_1\beta_1 
    + h_{0101}w_1\beta_2 + \frac12 h_{0002}\beta_2^2\nonumber \\
    & + h_{0011}\beta_1\beta_2 + \frac16 h_{3000}w_0^3 + \frac12 h_{2100}w_0^2w_1 
    + h_{1101}w_0w_1\beta_2 + \frac12 h_{2001}w_0^2\beta_2\nonumber \\
    & + \frac{1}{6}h_{0003}\beta_2^3 + \frac12 h_{1002}w_0\beta_2^2 
    + \frac12 h_{0102}w_1\beta_2^2 \nonumber \\
    & + \mathcal{O}(|w_1|^3 + |w_0w_1^2| + |\beta_2w_1^2| + |\beta_1|\|w\|^2
    +|\beta_1^2|\|w\| + |\beta_1^2| + \|(w,\beta)\|^4), \nonumber \\
K(\beta) = {}& \label{btdde:eq:K_expansion-BT}
    K_{10}\beta_1 + K_{01}\beta_2 + \frac{1}{2}K_{02}\beta_2^{2} 
	+ K_{11}\beta_1\beta_2 + \frac16 K_{03} \beta_2^3 \\
    &+ \mathcal{O}(|\beta_1|^2 + |\beta_1||\beta_2|^2 + |\beta_1|^2|\beta_2|  
	+ |\|\beta\|^4), \nonumber \\
\vartheta(w,\beta) = {}& \label{btdde:eq:theta_expansion_bt}
    1 + \vartheta_{1000}w_0 + \vartheta_{0001} \beta_2 
    + \mathcal O\left(|w_1| + |\beta_2| + \|(w,\beta)\|^2\right).
\end{align}
Here $B$, $A_1$, $J_2$, $C$, $B_1$, $A_2$, and $J_3$ are the standard
multilinear forms arising from the expansion of $F(u,\alpha)$. For example,
\[
B(u,u) = D^2_1F(0,0)(u,u),~J_2(\alpha,\alpha) = D_2^2F(0,0)(\alpha,\alpha),~B_1(u,u,\alpha) = D^1_2D^2_1F(0,0)(u,u,\alpha),
\]
etc. Explicit formulas for the multilinear forms for the simplest DDE
\cref{btdde:eq:discreteDDEs} are given in \cite[Section 6]{Switching2019}.

We insert the expansions
\cref{btdde:eq:R_expansion_BT,btdde:eq:h_expansion-BT,btdde:eq:K_expansion-BT,btdde:eq:theta_expansion_bt}
into the homological equation \cref{btdde:eq:homological_equation}. 

\subsubsection{(Generalized) eigenfunctions}
By collecting the coefficients of the linear terms in $w$ in the
homological equation, we obtain precisely the systems
defining the (generalized) eigenfunctions. By
\cref{btdde:proposition:eigenvalues_multiplicity_k} with $k=2$ and $\lambda=0$ we
obtain
\begin{align*}
    \phi_0  & =  \vartheta \mapsto q_0, \\
    \phi_1  & =  \vartheta \mapsto \vartheta q_0 + q_1, \\
    \psi_1 & = \left( p_1, \vartheta \mapsto p_1 \int_\vartheta^h \, d\zeta(s) \right), \nonumber \\
    \psi_0 & = \left( p_0, \vartheta \mapsto p_0 \int_\vartheta^h \, d\zeta(s)
                            + p_1 \int_{\vartheta}^h (\vartheta-s)\,d\zeta(s) \right),\nonumber 
\end{align*}
where $(q_0,q_1)$ and $(p_1,p_0)$ are the Jordan chain for $\Delta(0)$ and $\Delta^T(0)$, respectively.
Furthermore, we assume the eigenfunctions to be normalized such that
\begin{equation}
    \label{btdde:eq:bt_normalization_condition_eigenfunctions}
    \langle\psi_i,\phi_j\rangle = \delta_{ij}, \qquad 0\leq i,j \leq 1.
\end{equation}

\begin{remark}
Note that the normalization condition above does not uniquely define the eigenfunction. Indeed,
the transformation
\begin{equation}
    \psi_1 \mapsto \frac1c \psi_1, \qquad
    \psi_0 \mapsto \frac1c \psi_0, \qquad
    \phi_0 \mapsto c \phi_0, \qquad
    \phi_1 \mapsto c \phi_1,
\end{equation}
for some non-zero constant $c$, leaves \cref{btdde:eq:bt_normalization_condition_eigenfunctions} invariant.
The same holds true for the transformation
\begin{equation}
    \psi_0 \mapsto \psi_0 + \tilde c \psi_1, \qquad
    \phi_1 \mapsto \phi_1 - \tilde c \phi_0, \qquad
\end{equation}
for some constant $\tilde c$. In~\cite{Kuznetsov2005practical}, the condition
\begin{equation}
\label{btdde:eq:q0q0} 
q_0^T q_0 = 1,\qquad q_1^T q_0 = 0,
\end{equation}
is imposed to uniquely define the vectors $\{q_0,q_1,p_1,p_0\}$ up to a plus or
minus sign. However, since the non-uniqueness in the eigenfunctions does not alter the
order of accuracy of the homoclinic predictors, we will not explicitly impose \cref{btdde:eq:q0q0}.
\end{remark}

\subsubsection{Critical coefficients}
\label{btdde:sec:critical_coefficients}
Collecting the $w_0^2$, $w_0w_1$ and $w_1^2$ terms in the homological equation \cref{btdde:eq:homological_equation},
lead to the systems
\begin{align}
-\SUNSTAR A jh_{2000} & =  B(\phi_0,\phi_0)\rss - 2aj\phi_1 , \label{btdde:eq:Assh2000}\\
-\SUNSTAR A jh_{1100} & =  B(\phi_0,\phi_1)\rss - j\left(b\phi_1 - \vartheta_{1000} \phi_0 + h_{2000}\right), \label{btdde:eq:Assh1100} \\
-\SUNSTAR A jh_{0200} & =  B(\phi_1,\phi_1)\rss - 2j h_{1100}. \label{btdde:eq:Assh0200}
\end{align}

Pairing equations \cref{btdde:eq:Assh2000,btdde:eq:Assh1100} with the adjoint
eigenfunctions $\psi_1$ and $\psi_0$ yields critical normal form coefficients
\begin{align*}
a & = \dfrac{1}{2}p_1B(\phi_0,\phi_0),\\
b & = p_0B(\phi_0,\phi_0) + p_1B(\phi_0,\phi_1).
\end{align*}

Now that \cref{btdde:eq:Assh2000,btdde:eq:Assh1100} are solvable, we use
\cref{btdde:corollary:sol_double_zero_eig_polynomial} to define the functions
\begin{align*}
\hat h_{2000} & = \BINV 0 \left( B(\phi_0,\phi_0) - 2aj\phi_1 \right), \\
\hat h_{1100} & = \BINV 0 \left( B(\phi_0,\phi_1) - (b\phi_1 + \hat h_{2000}) \right).
\end{align*}
It follows that the general solutions of the systems \cref{btdde:eq:Assh2000,btdde:eq:Assh1100}
are given by
\begin{align*}
h_{2000} &= \hat h_{2000} + \gamma_1 \phi_0, \\
h_{1100} &= \hat h_{1100} + \gamma_1 \phi_1 - \vartheta_{1000} \phi_1 + \gamma_2 \phi_0.
\end{align*}
The constant $\gamma_1$ is determined by the solvability condition from \cref{btdde:eq:Assh0200}, which gives
\begin{equation*}
\gamma_1 = p_0  B(q_0,q_1) - \langle \psi_0, \hat h_{2000} \rangle 
								+ \frac12 p_1 B(q_1,q_1) + \vartheta_{1000},
\end{equation*}
where $\langle \psi_0,\hat h_{2000}\rangle$ is given through \cref{btdde:lemma:pairing_with_adjoint_eigenfunctions}.

To determine the constant $\gamma_2$ and the coefficient $\vartheta_{1000}$ we
consider the $w_0^3$ and $w_0^2w_1$ terms in the homological equation \cref{btdde:eq:homological_equation}. After
some simplification, we obtain the systems
\begin{align}
\label{btdde:eq:Assh3000}
-\SUNSTAR{A} jh_{3000} ={}& \left[3 B(h_{2000},\phi_0) + C(\phi_0,\phi_0,\phi_0)\right]\rss 
                            - 6a j\left( h_{1100} - \vartheta_{1000} \phi_1 \right), \\
\label{btdde:eq:Assh2100}
-\SUNSTAR{A} jh_{2100} ={}& \left[2 B(h_{1100},\phi_0) + B(h_{2000},\phi_1) + C(\phi_0,\phi_0,\phi_1)\right]\rss \\
                         & - j \left( 2 a h_{0200} + 2 b h_{1100} + h_{3000} 
                            - 2 \vartheta_{1000} (b\phi_1 - \vartheta_{1000}\phi_0 +h_{2000})\right).
\end{align}

The solvability condition of the first equation determines $\vartheta_{1000}$ as
\begin{equation}
\label{btdde:eq:theta1000}
\vartheta_{1000} = -\frac1{12a} p_1 \left\{ 
			3B(\hat h_{2000},\phi_0) + C(\phi_0,\phi_0,\phi_0)
            \right\} + \frac12 \langle \psi_1, \hat h_{1100} \rangle .
\end{equation}

The solvability condition for the system in \cref{btdde:eq:Assh2100} yields, after a
rather lengthy calculation, that $\gamma_2$ is determined by 
\begin{align}
\label{btdde:eq:gamma_2}
\gamma_2 &= \frac1{6a} 
\bigg[ p_1\left\{ 2 B(\hat h_{1100},q_0) + B(\hat h_{2000},q_1) + C(q_0,q_0,q_1) \right\} 
        +2a p_0 B(q_1,q_1)  \nonumber \\
	& \qquad - b p_1 B(\phi_1,\phi_1) + p_0 \left( 3B(\hat h_{2000},q_0) + C(q_0,q_0,q_0) \right) 
    + 3 \gamma_1 b - 10 a \langle \hat h_{1100}, \psi_0 \rangle \bigg]. \nonumber 
\end{align}

Since the systems in \cref{btdde:eq:Assh0200,btdde:eq:Assh2100,btdde:eq:Assh3000} are now all consistent,
we are allowed to take the bordered inverses to obtain
\begin{align*}
h_{0200} ={}& \BINV 0 \left(B(\phi_1,\phi_1) - h_{1100}\right), \\
h_{3000} ={}& \BINV 0 \left(3 B(h_{2000},\phi_0) + C(\phi_0,\phi_0,\phi_0)
                 - 6a \left( h_{1100} - \vartheta_{1000} \phi_1 \right)\right) \\
h_{2100} ={}& \BINV 0 \left(2 B(h_{1100},\phi_0) + B(h_{2000},\phi_1) + C(\phi_0,\phi_0,\phi_1) \right. \\
                         & - \left.\left( 2 a h_{0200} + 2 b h_{1100} + h_{3000} 
                            - 2 \vartheta_{1000} (b\phi_1 - \vartheta_{1000}\phi_0 +h_{2000})\right)\right).
\end{align*}

\subsubsection{Parameter-dependent linear coefficients}
The coefficients of the linear terms in $\beta$ give the
systems
\begin{equation}
\label{btdde:eq:Assh0010}
\begin{aligned}
-\SUNSTAR Ajh_{0001} &= J_1K_{01}, \\
-\SUNSTAR Ajh_{0010} &= J_1K_{10} - \phi_1.
\end{aligned}
\end{equation}
Since $p_1$ and $J_1$ are known, we can calculate 
\begin{equation*}
    \nu := (p_1 J_1)^T.
\end{equation*}
By the transversality condition, the vector $\nu$ is nonzero. It then follows
from the Fredholm alternative that
\begin{equation*}
\begin{aligned}
K_{01}   &= \delta_1\hat K_{01}, \\
h_{0001} &= \delta_1 \left( \hat h_{0001} + \gamma_3 \phi_0 \right), \\
K_{10}   &= \hat K_{10} + \delta_2 K_{01}, \\
h_{0010} &= \hat h_{0010} + \delta_2 h_{0001} + \gamma_4 \phi_0,
\end{aligned}
\end{equation*}
where
\begin{equation*}
\begin{aligned}
\hat K_{10}   &= \frac1{\|\nu\|^2}\nu, \\
\hat h_{0010} &= \BINV 0 \left( J_1K_{10} - \phi_1 \right) \\
\hat K_{01}   &=
\begin{pmatrix}
    0 & -1 \\ 1 & 0
\end{pmatrix} \hat K_{10}, \\
\hat h_{0001} &= \BINV 0 \left( J_1\hat K_{01} \right),
\end{aligned}
\end{equation*}
and $\delta_{1,2}$, $\gamma_{3,4}$ are real constants determined by the
solvability condition of the $w\beta$ terms in the homological equation.
Collecting the corresponding systems in the homological equation yields 
\begin{align}
    -\SUNSTAR Ajh_{1001} &= \left(B(h_{0001},\phi_0) + A_1(\phi_0,K_{01})\right)\rss, \label{btdde:eq:Assh1001} \\
    -\SUNSTAR Ajh_{0101} &= \left(B(h_{0001},\phi_1) + A_1(\phi_1,K_{01})\right)\rss-j(h_{1001} + \phi_1 - \vartheta_{0001}\phi_0), \label{btdde:eq:Assh0101}\\
    -\SUNSTAR Ajh_{1010} &= \left(B(h_{0010},\phi_0) + A_1(\phi_0,K_{10})\right)\rss-j(h_{1100} - \vartheta_{1000}\phi_1), \nonumber \\
    -\SUNSTAR Ajh_{0110} &= \left(B(h_{0010},\phi_1) + A_1(\phi_1,K_{10})\right)\rss-j(h_{0200} + h_{1010}) \nonumber.
\end{align}
The solvability condition for the first two systems yields
\begin{equation*}
\begin{aligned}
\gamma_3 &= -\frac{p_1 \left( 
				B(\hat h_{0001},\phi_0) + A_1(\phi_0,\hat K_{01}) \right)}{2a}, \\
\delta_1 &= \frac{1}{p_1 \left(
			B(\hat h_{0001},\phi_1) + A_1(\phi_1,\hat K_{01}) \right) + p_0
      \left( B(\hat h_{0001},\phi_0) + A_1(\phi_0,\hat K_{01}) \right) + \gamma_3 b},
\end{aligned}
\end{equation*}
while the solvability condition for the latter two systems yields
\begin{equation*}
\begin{aligned}
    \gamma_4 &= \frac{\left< \psi_1, h_{1100}  \right> - \vartheta_{1000} - p_1 \left( B(\hat
				h_{0010},\phi_0) + A_1(\phi_0,\hat K_{10})\right)}{2a} ,\\
\delta_2 &= -p_1 \left( B(\hat h_{0010},\phi_1) + A_1(\phi_1,\hat K_{10}) \right)
    - \gamma_4 b + \left< \psi_1, h_{0200} \right> \\
	& \quad - p_0 \left( B(\hat h_{0010},\phi_0) + A_1(\phi_0,\hat K_{10}) \right) + \left< \psi_1, h_{1100} \right>.
\end{aligned}
\end{equation*}
Note that the denominator in $\delta_1$ is nonzero by the transversality
condition.

\subsubsection{Coefficients \texorpdfstring{$h_{1010} \text{ and } h_{0110}$}{h1010 and h0110}}

Since we do not need to use the non-uniqueness in the systems for the
coefficients $h_{1010}$ and $h_{0110}$ to simplify higher-order systems,
it is sufficient to let
\begin{equation*}
\begin{aligned}
    h_{1010} &= \BINV 0 \left(B(h_{0010},\phi_0) + A_1(\phi_0,K_{10})-(h_{1100} - \vartheta_{1000}\phi_1)\right), \\
    h_{0110} &= \BINV 0 \left(B(h_{0010},\phi_1) + A_1(\phi_1,K_{10})-(h_{0200} + h_{1010})\right).
\end{aligned}
\end{equation*}

\subsubsection{Coefficients \texorpdfstring{$(\vartheta_{0001},\gamma_5),h_{1001},
h_{0101}, h_{2001}, h_{1101}$}{(theta0001,gamma5),h1001,h0101,h2001,h1101}}

Define
\begin{equation*}
\begin{aligned}
    \hat h_{1001} &= \BINV 0 \left(B(h_{0001},\phi_0) + A_1(\phi_0,K_{01})\right), \\
    \hat h_{0101} &= \BINV 0 \left(\left(B(h_{0001},\phi_1) + A_1(\phi_1,K_{01})\right)-\left(\hat h_{1001} + \phi_1\right)\right). \\
\end{aligned}
\end{equation*}
Then the general solutions to the systems in \cref{btdde:eq:Assh1001,btdde:eq:Assh0101} are given by
\begin{equation*}
\begin{aligned}
h_{1001} &= \hat h_{1001} + \gamma_5 \phi_0, \\
h_{0101} &= \hat h_{0101} + \gamma_5 \phi_1 - \vartheta_{0001} \phi_1. \\
\end{aligned}
\end{equation*}

In order to determine $\gamma_5$ and $\vartheta_{0001}$, we consider the systems
corresponding to the $w_0^2\beta_2$ and $w_0w_1\beta_2$ terms in the homological
equation. These are given by
\begin{equation}
\label{btdde:eq:Assh2001_Assh1101}
\begin{aligned}
-\SUNSTAR Ajh_{2001} &= A_1(h_{2000},K_{01}) +B(h_{0001},h_{2000}) + 2 B(h_{1001},\phi_0) \\
				& \quad + B_1(\phi_0,\phi_0,K_{01}) + C(h_{0001},\phi_0,\phi_0) - 2 a j( h_{0101} - \vartheta_{0001} \phi_1), \\
-\SUNSTAR Ajh_{1101} &= A_1(h_{1100},K_{01}) + B(h_{0001},h_{1100}) + B(h_{0101},\phi_0) + B(h_{1001},\phi_1)  \\
				& \quad  + B_1(\phi_0,\phi_1,K_{01}) + C(h_{0001},\phi_0,\phi_1) - j\left[b h_{0101} + h_{1100} + h_{2001} \right. \\
				& \quad \left. - \vartheta_{1000} (h_{1001} + \phi_1 - \vartheta_{0001} \phi_0) 
                               - \vartheta_{0001} (h_{2000} + b \phi_1 - \vartheta_{1000} \phi_0)\right]  .
\end{aligned}
\end{equation}
The Fredholm solvability condition leads to the following system to be solved
\begin{equation}
\label{btdde:eq:gamma_5_theta0001}
\begin{pmatrix}
				 2a &  4a \\
				  b &   b 
\end{pmatrix}
\begin{pmatrix}
				\gamma_5 \\
				\vartheta_{0001}
\end{pmatrix}
=
\begin{pmatrix}
				\zeta_1 \\
				\zeta_2 
\end{pmatrix}.
\end{equation}
Here the right-hand side is given by
\begin{equation}
\label{btdde:eq:zeta1_zeta2}
\begin{aligned}
    \zeta_1 &=  2a\langle \psi_1, \hat h_{0101} \rangle - p_1 \left[   
			  A_1(h_{2000},K_{01}) + B(h_{0001},h_{2000}) \right.  \\
			& \left. \qquad + 2 B(\hat h_{1001},\phi_0)
			+ B_1(\phi_0,\phi_0,K_{01}) + C(h_{0001},\phi_0,\phi_0) \right], \\
    \zeta_2 &= \langle \psi_1, b\hat h_{0101} + h_{1100} - \vartheta_{1000}\hat h_{1001} \rangle - \vartheta_{1000}  \\
				& \qquad  - p_1 \left[ A_1(h_{1100},K_{01}) + 
				B(h_{0001},h_{1100}) + B(\hat h_{0101},\phi_0) \right. \\
				& \qquad \left. + B(\hat h_{1001},\phi_1) + B_1(\phi_0,\phi_1,K_{01})
				+ C(h_{0001},\phi_0,\phi_1) \right]  \\
                & \quad + 2a\langle \psi_0, \hat h_{0101} \rangle 
                - p_0 \left[ A_1(h_{2000},K_{01}) +
				B(h_{0001},h_{2000}) \right. \\
				& \qquad + \left. 2 B( \hat h_{1001},\phi_0) +
				B_1(\phi_0,\phi_0,K_{01}) + C(h_{0001},\phi_0,\phi_0) \right].
\end{aligned}
\end{equation}	
Notice that the matrix in \cref{btdde:eq:gamma_5_theta0001} is invertible by the
non-degeneracy condition. 

Now that the systems in \cref{btdde:eq:Assh2001_Assh1101} are
solvable, we obtain
\begin{equation}
\label{btdde:eq:h2001_h1101}
\begin{aligned}
h_{2001} &= \BINV 0 \left( \left(A_1(h_{2000},K_{01}) +B(h_{0001},h_{2000}) + 2 B(h_{1001},\phi_0) \right.\right. \\
				& \quad + \left.\left. B_1(\phi_0,\phi_0,K_{01}) + C(h_{0001},\phi_0,\phi_0)\right) - 2 a \left( h_{0101} - \vartheta_{0001} \phi_1\right)\right), \\
h_{1101} &= \BINV 0 \left( \left( A_1(h_{1100},K_{01}) + B(h_{0001},h_{1100}) + B(h_{0101},\phi_0) + B(h_{1001},\phi_1) \right.\right.  \\
				& \quad  \left. + B_1(\phi_0,\phi_1,K_{01}) + C(h_{0001},\phi_0,\phi_1) \right) - \left[b h_{0101} + h_{1100} + h_{2001} \right. \\
				& \quad \left.\left. - \vartheta_{1000} (h_{1001} + \phi_1 - \vartheta_{0001} \phi_0) 
                               - \vartheta_{0001} (h_{2000} + b \phi_1 - \vartheta_{1000} \phi_0)\right]\right).
\end{aligned}
\end{equation}

\subsubsection{Coefficients \texorpdfstring{$K_{11} \text{ and } h_{0011}$}{K11
				and h0011}}

Collecting the systems corresponding to the $\beta_1 \beta_2$ term in the
homological equation yields
\begin{equation}
\label{btdde:eq:Assh0011}
\begin{aligned}
-\SUNSTAR A jh_{0011} &= \left(J_1 K_{11} + A_1(h_{0001},K_{10}) + A_1(h_{0010},K_{01}) \right. \\
						& \quad +
                        \left. B(h_{0001},h_{0010}) + J_2(K_{01},K_{10})\right)\rss - j\left(h_{0101}-\vartheta_{0001}\phi_1\right).
\end{aligned}
\end{equation}
Using the identity 
\begin{equation*}
   p_1J_1K_{10}=1 
\end{equation*}
from the second system in \cref{btdde:eq:Assh0010}, combined with the solvability
condition, yields
\begin{equation}
\begin{aligned}
    K_{11}={}& \left[\langle \psi_1, \hat h_{0101} \rangle + \gamma_5 - 2\vartheta_{0001} 
                 - p_1\left(A_1(h_{0001},K_{10}) + A_1(h_{0010},K_{01}) \right. \right. \nonumber\\
             &\left.\left. + B(h_{0010},h_{0001}) + J_2(K_{10},K_{01}) \right)\vphantom{\left< \psi_1, \hat h_{0101} \right>} \right]K_{10}.
\end{aligned}
\end{equation}
It follows that 
\begin{equation}
\begin{aligned}
h_{0011}(\vartheta) &= \BINV 0 \left( J_1 K_{11} + A_1(h_{0001},K_{10}) + A_1(h_{0010},K_{01} \right.) \\
						& \left. \quad + B(h_{0001},h_{0010}) + J_2(K_{01},K_{10})-\left(h_{0101}-\vartheta_{0001}\phi_1\right)\right).
\end{aligned}
\end{equation}

\subsubsection{Coefficients
				\texorpdfstring{$K_{02},h_{0002},h_{1002},h_{0102}$}
				{h0002,K02,h1002,h0102}}

The systems corresponding to the $\beta_2^2, w_0\beta_2^2$, and $w_1\beta_2^2$,
terms in the homological equation yields
\begin{equation}
\label{btdde:eq:Ah0002_Ah1002_Ah0102}
\begin{aligned}
-\SUNSTAR A jh_{0002}={}&  \left(J_1K_{02} + 2A_1(h_{0001},K_{01})
							 + B(h_{0001},h_{0001}) + J_2(K_{01},K_{01})\right)\rss, \\
-\SUNSTAR A jh_{1002}={}& \left( 2A_1(h_{1001},K_{01}) + A_1(\phi_0,K_{02}) + A_2(\phi_0,K_{01},K_{01}) \right.
					 \\ & + B(\phi_0,h_{0002}) + 2B(h_{0001},h_{1001}) +
					 2B_1(\phi_0,h_{0001},K_{01}) \\
  & \left. + C(\phi_0,h_{0001},h_{0001}) \right)\rss, \\
-\SUNSTAR A jh_{0102}={}& \left( 2A_1(h_{0101},K_{01}) + A_1(\phi_1,K_{02}) + A_2(\phi_1,K_{01},K_{01}) \right. \\
  & + B(\phi_1,h_{0002}) + 2B(h_{0001},h_{0101}) + 2B_1(\phi_1,h_{0001},K_{01}) \\
  & \left. + C(\phi_1,h_{0001},h_{0001})\right)\rss -j \left[ 2h_{0101} + h_{1002} \right. \\
  & \left. -2\vartheta_{0001} (h_{1001} + \phi_1 - \vartheta_{0001}\phi_0) \right].
\end{aligned}
\end{equation}
The first system is solved similarly as \cref{btdde:eq:Assh0011}. In order to make the second and third systems
consistent, define
\begin{equation*}
\begin{aligned}
\hat K_{02}={}&-p_1\left[2A_1(h_{0001},K_{01}) + B(h_{0001},h_{0001})
				+J_2(K_{01},K_{01})\right]K_{10}, \\
                \hat h_{0002} ={}& \BINV 0 \left( J_1 \hat K_{02} + 2A_1(h_{0001},K_{01})
                + B(h_{0001},h_{0001}) + J_2(K_{01},K_{01})\right).
\end{aligned}
\end{equation*}
Then the general solutions to the first system in \cref{btdde:eq:Ah0002_Ah1002_Ah0102}
can be written as
\begin{equation*}
\begin{aligned}
				K_{02}={}& \hat K_{02} + \delta_3 K_{01}, \\
				h_{0002}={}& \hat h_{0002} +  \delta_3 h_{0001} + \gamma_6 \phi_0.
\end{aligned}
\end{equation*}
Substituting these two expressions into the last two systems of \cref{btdde:eq:Ah0002_Ah1002_Ah0102} and
using the solvability condition yields
\begin{equation*}
\begin{aligned}
\gamma_6 &= -\frac1{2a} p_1 \left[ 2A_1(h_{1001},K_{01}) +
				A_1(\phi_0,\hat K_{02}) + A_2(\phi_0,K_{01},K_{01}) \right. \\
	        & \quad + B(\phi_0,\hat h_{0002}) + 2B(h_{0001},h_{1001}) + 2B_1(\phi_0,h_{0001},K_{01}) \\
            & \left. \quad + C(\phi_0,h_{0001},h_{0001}) \right],  \\
\delta_3 &= 2 \langle \psi_1, \hat h_{0101} \rangle + 2 \gamma_5 - 2\vartheta_{0001} \langle \psi_1, \hat h_{1001} \rangle 
                - 4\vartheta_{0001} -p_1 \left[ 2A_1(h_{0101},K_{01}) \right. \\ 
            & \quad + A_1(\phi_1,\hat K_{02}) + A_2(\phi_1,K_{01},K_{01}) + B(\phi_1,\hat h_{0002}) \\
            & \quad \left. + 2B(h_{0001},h_{0101}) + 2B_1(\phi_1,h_{0001},K_{01}) +  C(\phi_1,h_{0001},h_{0001}) \right] \\
            & \quad - p_0 \left[ 2A_1(h_{1001},K_{01}) + A_1(\phi_0,\hat K_{02}) +
            A_2(\phi_0,K_{01},K_{01}) \right. \\ 
            & \quad + B(\phi_0,\hat h_{0002}) + 2B(h_{0001},h_{1001}) + 2B_1(\phi_0,h_{0001},K_{01}) \\
            & \left. \quad + C(\phi_0,h_{0001},h_{0001}) \right] - \gamma_6 b,
\end{aligned}
\end{equation*}
where $\langle \psi_1, \hat h_{1001} \rangle = p_0 \left( B(h_{0001},\phi_0) + A_1(\phi_0,K_{01}) \right)$.

Now that the last two systems in \cref{btdde:eq:Ah0002_Ah1002_Ah0102} are consistent,
we obtain
\begin{equation*}
\begin{aligned}
h_{1002} ={}& \BINV 0 \left( 2A_1(h_{1001},K_{01}) + A_1(\phi_0,K_{02}) + A_2(\phi_0,K_{01},K_{01}) \right.
					 \\ & + B(\phi_0,h_{0002}) + 2B(h_{0001},h_{1001}) +
					 2B_1(\phi_0,h_{0001},K_{01}) \\
  & \left. + C(\phi_0,h_{0001},h_{0001}) \right), \\
h_{0102} ={}& \BINV 0 \left(\left( 2A_1(h_{0101},K_{01}) + A_1(\phi_1,K_{02}) + A_2(\phi_1,K_{01},K_{01}) \right.\right. \\
  & + B(\phi_1,h_{0002}) + 2B(h_{0001},h_{0101}) + 2B_1(\phi_1,h_{0001},K_{01}) \\
  & \left. + C(\phi_1,h_{0001},h_{0001})\right) - \left[ 2h_{0101} + h_{1002} \right. \\
  & \left.\left. -2\vartheta_{0001} (h_{1001} + \phi_1 - \vartheta_{0001}\phi_0) \right]\right).
\end{aligned}
\end{equation*}

\subsubsection{Coefficients \texorpdfstring{$K_{03} \text{ and } h_{0003}$}{K03
				and h0003}}
\label{btdde:subsubsection:K03_h0003}

Collecting the systems corresponding to the $\beta_2^3$ term in the
homological equation yields
\begin{equation*}
\begin{aligned}
-\SUNSTAR A jh_{0003} ={}& \left( J_1 K_{03} + A_1(h_{0001},K_{02}) + A_1(h_{0002},K_{01})
				+ 2 (A_1(h_{0001},K_{02}) \right. \\
				& + A_1(h_{0002},K_{01}) + 3 B(h_{0001},h_{0002}) + 3 J_2(K_{01},K_{02}) \\
				& + 3 A_2(h_{0001},K_{01},K_{01}) + 3 B_1(h_{0001},h_{0001},K_{01}) \\
				& \left. + C(h_{0001},h_{0001},h_{0001}) + J_3(K_{01},K_{01},K_{01}) \right)\rss.
\end{aligned}
\end{equation*}
This equation is solved similarly as equation \cref{btdde:eq:Assh0011}. We obtain
\begin{equation*}
\begin{aligned}
K_{03}={}& -p_1 \left[ A_1(h_{0001},K_{02}) + A_1(h_{0002},K_{01})
				+ 2 A_1(h_{0001},K_{02}) \right. \\
				& + 2 A_1(h_{0002},K_{01}) + 3 B(h_{0001},h_{0002}) + 3 J_2(K_{01},K_{02}) \\
				& + 3 A_2(h_{0001},K_{01},K_{01}) + 3 B_1(h_{0001},h_{0001},K_{01}) \\
				& \left. + C(h_{0001},h_{0001},h_{0001}) +
				J_3(K_{01},K_{01},K_{01}) \right] K_{10}, \\
h_{0003} ={}& \BINV 0 \left( J_1 K_{03} + A_1(h_{0001},K_{02}) +
				A_1(h_{0002},K_{01}) + 2 A_1(h_{0001},K_{02}) \right. \\
				& + 2 A_1(h_{0002},K_{01}) + 3 B(h_{0001},h_{0002}) + 3 J_2(K_{01},K_{02}) \\
				& + 3 A_2(h_{0001},K_{01},K_{01}) + 3 B_1(h_{0001},h_{0001},K_{01}) \\
				& \left. + C(h_{0001},h_{0001},h_{0001}) +
				J_3(K_{01},K_{01},K_{01})\right).
\end{aligned}
\end{equation*}

\subsubsection{Homoclinic asymptotics}
\label{btdde:sec:generic_bt_homoclinic_asymptotics}
The third-order homoclinic asymptotics for the homoclinic orbits emanating from \cref{btdde:eq:normal_form_orbital}
have been derived in \cite{Bosschaert@Interplay} and are given by
\begin{equation}
\label{btdde:eq:third_order_predictor_LP_tau}
\begin{cases}
\begin{aligned}
w_0(\eta)  &= \frac{a}{b^2} 
\tilde {u}\left(\tanh\left(\xi\left(\frac{a}{b}\epsilon\eta\right)\right)\right) \epsilon^2, \\
w_1(\eta)  &= \frac{a^2}{b^3}
\tilde {v}\left(\tanh\left(\xi\left(\frac{a}{b}\epsilon\eta\right)\right)\right) \epsilon^3, \\
\beta_1    &= -4 \frac{a^3}{b^4}\epsilon^4, \\
\beta_2    &= \frac{a}{b}\epsilon^2\tau,
\end{aligned}
\end{cases}
\end{equation}
where
\begin{equation}
\begin{aligned}
\label{btdde:eq:tau_utilde_vtilde_xi}
\tau ={}& \frac{10}{7} + \frac{288}{2401} \epsilon^2 + \mathcal{O}(\epsilon^4), \\
\tilde {u}(\zeta) ={}& 2 - \left(1-\zeta^2\right) \left(6 + \frac{18}{49}\epsilon^2 \right) + \mathcal{O}(\epsilon^4), \\
\tilde  v(\zeta) ={}&  = -\left[ -12 + \frac{72}{7} \zeta \epsilon - \left( \frac{90}{49} + \frac{162 }{49} \zeta^2 \right)\epsilon^2 \right.
  	 + \left. \left( \frac{3888}{2401} \zeta - \frac{216}{343}\zeta^3 \right) \epsilon^3 \right]
(1-\zeta^2) \zeta + \mathcal{O}(\epsilon^4), \\
\xi(s) ={}& s - \frac67\log(\cosh(s))\epsilon + \left( -\frac{18 s}{49}+\frac{45 \tanh (s)}{98}+\frac{36}{49} \tanh (s) \log (\cosh(s))\right)\epsilon^2 \\
          &+ \left( -\frac{117}{343} + \frac{3 \sech^2(s)}{4802}\left(-504 \log^2(\cosh(s))-276 \cosh (2 s) \log (\cosh (s)) \right. \right. \\
        {}& \left. \left. \phantom{\frac12} + 102 \log (\cosh (s)) + 252 s \sinh (2 s) \right) \right) \epsilon^3 + \mathcal{O}(\epsilon^4).
\end{aligned}
\end{equation}
To translate the asymptotics to a generic codimension two Bogdanov--Takens point
in \cref{btdde:eq:classicalDDE}, one should substitute the expressions in
\cref{btdde:eq:third_order_predictor_LP_tau} into \cref{btdde:eq:h_expansion-BT,btdde:eq:K_expansion-BT}.
Furthermore, to accurately approximate the profiles of the homoclinic orbit, one needs to numerically solve
the equation 
\begin{equation*}
    t(\eta) - t = 0,
\end{equation*}
for $\eta$, where
\begin{equation*}
\label{btdde:eq:tOfTau}
\begin{aligned}
    t(\eta) ={}& \eta \left( 1 + \vartheta_{0001}\frac{a}{b}\epsilon^2\left(\frac{10}{7}+ \frac{288}{2401} \epsilon^2 + \mathcal{O}(\epsilon^4) \right) \right) + 
				 \vartheta_{1000} \frac1{b} \epsilon \left( \xi(\frac{a}{b}\eta\epsilon) -6 \tanh (\xi(\frac{a}{b}\eta\epsilon) ) + \right. \\
               &\left( \frac{18 \sech^2(\xi(\frac{a}{b}\eta\epsilon) )}{7}+\frac{12}{7} \log (\cosh (\xi(\frac{a}{b}\eta\epsilon) )) \right) \epsilon \\ 
               & +\frac{9}{49} \left(4 \xi(\frac{a}{b}\eta\epsilon) -9 \tanh (\xi(\frac{a}{b}\eta\epsilon) )+5 \tanh (\xi(\frac{a}{b}\eta\epsilon) ) \sech^2(\xi(\frac{a}{b}\eta\epsilon))\right) \epsilon^2 \\ 
               & \left. +\frac{18 \left(-21 \sech^4(\xi(\frac{a}{b}\eta\epsilon) )+47 \sech^2(\xi(\frac{a}{b}\eta\epsilon) )+8 \log (\cosh (\xi(\frac{a}{b}\eta\epsilon) ))\right)}{2401} \epsilon^3 \right) + \mathcal{O}(\epsilon^4).
\end{aligned}
\end{equation*}

\subsubsection{Codimension-one equilibria bifurcations}
To approximate the fold and Hopf curves and their corresponding equilibria in
\cref{btdde:eq:normal_form_orbital}, one should substitute the expressions for $\beta$
and the equilibrium coordinates into the expansions \cref{btdde:eq:h_expansion-BT} and
\cref{btdde:eq:K_expansion-BT}. It follows that the fold curve is approximated by 
\begin{equation*}
\begin{cases}
\begin{aligned}
    x &= \mathcal H\left(0,0,0,0\right), \\
    \alpha &= K(0,\epsilon),
\end{aligned}
\end{cases}
\end{equation*}
for $|\epsilon|$ small, and the Hopf curve by
\begin{equation*}
\begin{cases}
\begin{aligned}
    \omega &= \epsilon, \\
    x &= \mathcal H\left(-\frac{\epsilon^2}{2a},0,-\frac{\epsilon^4}{4a)},\frac{b\epsilon^2}{2a}\right), \\
    \alpha &= K\left(-\frac{\epsilon^4}{4a},\frac{b\epsilon^2}{2a}\right),
\end{aligned}
\end{cases}
\end{equation*}
for $\epsilon > 0$ small.

\subsection{Transcritical Bogdanov--Takens bifurcation}
\label{btdde:sec:transcritical-Bogdanov-Takens}

\begin{table}
\begin{center}
\begingroup
\renewcommand*{\arraystretch}{1.4}
\begin{tabular}{rl}
\hline
order in $\epsilon$ & affected terms \\
\hline%
\(\epsilon^{-2}\) & $w_0$ \\
\(\epsilon^{-1}\) & $w_1$ \\
\(\epsilon^0\)    & $w_0^2$, $w_0 \beta_1$, $w_0 \beta_2$ \\
\(\epsilon^1\)    & $w_1\beta_1$, $w_1\beta_2$, $w_0w_1$ \\
\(\epsilon^2\)    & $w_0\beta_1^{2}$, $w_0\beta_1\beta_2$, $w_0\beta_2^{2}$, $w_0^{2}\beta_1$, $w_0^{2}\beta_2$, $w_0^{3}$, $w_1^{2}$ \\
\(\epsilon^3\)    & $w_1\beta_1^{2}$, $w_1\beta_1\beta_2$, $w_1\beta_2^{2}$, $w_0 w_1\beta_1$, $w_0w_1\beta_2$,  $w_0^{2}w_1$

\\[0.2cm]
\hline
\end{tabular}
\endgroup
\caption{
Terms in the reduced system restricted to the parameter-dependent center manifold of the
codimension two Bogdanov--Takens point in \cref{btdde:eq:discreteDDEs} that affect the third-order homoclinic predictor.} 
\label{btdde:table:terms_affecting_predicor}
\end{center}
\end{table}
In this section, we consider the situation in which the origin remains an equilibrium
under variation of the parameters.

Following~\cite{Broer1991}, it is not difficult to show that the $C^\infty$-equivalent normal form on the parameter-dependent
center manifold takes the form
\begin{equation}
\label{btdde:eq:normal_form_orbital_tbt}
\begin{cases}
\begin{aligned}
    \dot w_0 & =  w_1, \\
    \dot w_1 & =  \beta_1 w_0 + \beta_2 w_1 + aw_0^2 + b w_0 w_1 + w_0^2 w_1 h(w_0,\beta) + w_1^2 Q(w_0,w_1,\beta),
\end{aligned}
\end{cases}
\end{equation}
where $h$ is $C^\infty$ and $Q$ is $N$-flat for an a priori given $N$. Here,
the dot represents the derivative with respect to the new time $\eta$ of
$w_i(\eta)(i = 0,1)$. 

In \cref{btdde:table:terms_affecting_predicor}, we have listed the terms in the
parameter-dependent orbital normal form
\cref{btdde:eq:normal_form_orbital_tbt} affecting the third-order homoclinic
asymptotics in \cref{btdde:eq:second_order_nonlinear_oscillator}. Therefore, by
\cite[Theorem 1]{Bosschaert@Interplay}, it is both necessary and sufficient in
order to translate the third-order homoclinic predictor to the system
\cref{btdde:eq:classicalDDE} to expand the functions $R \colon X \times \mathbb R^2 \rightarrow \SUNSTAR{X}$, 
$\mathcal{H} \colon \mathbb R^2 \times \mathbb R^2 \rightarrow X$, $K \colon
\mathbb R^2 \rightarrow \mathbb R^2$, and $\vartheta \colon \mathbb R^2
\rightarrow \mathbb R$ defined in \cref{btdde:sec:Center_manifold_reduction} as
follows
\begin{align}
R(u,\alpha) =& \label{btdde:eq:R_expansion_TBT}
    \left(\frac{1}{2} B(u,u) + A_1(u,\alpha) + \frac{1}{6} C(u,u,u) + \frac{1}{2} B_1(u,u,\alpha) \right. \\
    &\phantom{\Bigl(}+ \left. \frac{1}{2} A_2(u,\alpha,\alpha) 
    + \mathcal O(\|(u\|,\|\alpha)\|^4) \right) \rss , \nonumber \\
        \mathcal H(w,\beta) = {}& \label{btdde:eq:h_expansion-TBT}
        \phi_0 w_0 + \phi_1 w_1 + \frac{1}{2}h_{2000}w_0^2 + h_{1100}w_0w_1 + \frac{1}{2}h_{0200}w_1^2 + h_{1010}w_0\beta_1 \\
    & +h_{1001}w_0\beta_2 + h_{0110}w_1\beta_1 + h_{0101}w_1\beta_2 + \frac{1}{2}h_{0102}w_1\beta_2^2 \nonumber \\
    & +h_{0111}w_1\beta_1\beta_2 + \frac{1}{2}h_{0120}w_1\beta_1^2 + \frac{1}{2}h_{1002}w_0\beta_2^2 + h_{1011}w_0\beta_1\beta_2 \nonumber \\
    & +\frac{1}{2}h_{1020}w_0\beta_1^2 + h_{1101}w_0w_1\beta_2 + h_{1110}w_0w_1\beta_1\nonumber \\
    & +\frac{1}{2}h_{2001}w_0^2\beta_2 + \frac{1}{2}h_{2010}w_0^2\beta_1 + \frac{1}{2}h_{2100}w_0^2w_1 + \frac{1}{6}h_{3000}w_0^3 \nonumber \\
    & +\mathcal{O}(|\beta_2w_1^2| + |\beta_1w_1^2| + |w_1^3| + |w_0w_1^2|) + \mathcal{O}(\left\Vert (w,\beta)\right\Vert ^4), \nonumber\\
K(\beta) =& K_{10}\beta_1 + K_{01}\beta_2 + 
            \frac{1}{2}K_{20}\beta_1^2 + K_{11}\beta_1\beta_2 + \frac{1}{2}K_{02}\beta_2^2
             +\mathcal{O}(\left\Vert \beta\right\Vert ^3).\label{btdde:eq:K_expansion_TBT} \\
\vartheta(w,\beta) = {}& \label{btdde:eq:theta_expansion_TBT}
        1 + \vartheta_{1000}w_0 + \vartheta_{0010} \beta_1 + \vartheta_{0001} \beta_2 
        + \mathcal O\left(|w_1| + \|(w,\beta)\|^2\right).
\end{align}
The coefficients of the time-reparametrization $\vartheta$ have been determined
such that the linear systems resulting from the homological equations are
solvable, see \cite[Remark 2.2]{Bosschaert@Interplay}. Note that, since the
steady-state remains fixed under variations of parameters, we leave
out all coefficients in the expansion of $\mathcal{H}$ which solely depend on
the parameters.

Also note that the critical coefficients, i.e. coefficients not depending on any
parameters, have already been derived in \cref{btdde:sec:critical_coefficients}.
Therefore, we will start directly with the parameter-dependent systems.

\subsubsection{Quadratic terms}

Collecting the coefficients of the linear and quadratic terms in the
homological equation lead to the systems
\begin{equation}
\begin{aligned}
\label{btdde:eq:second_order_systems_tbt}
-\SUNSTAR A jh_{1010} & = A_1(\phi_0,K_{10})-j\phi_1, \\
-\SUNSTAR A jh_{1001} & = A_1(\phi_0,K_{01}), \\
-\SUNSTAR A jh_{0110} & = A_1(\phi_1,K_{10}) - j(h_{1010} - \vartheta_{0010}\phi_0), \\
-\SUNSTAR A jh_{0101} & = A_1(\phi_1,K_{01}) - j(h_{1001} - \vartheta_{0001}\phi_0)  - \phi_1.
\end{aligned}
\end{equation}
The solvability condition implies that
\begin{equation}
\begin{aligned}
\label{btdde:eq:second_order_solvability_condition_tbt}
0 & = 1-p_1 A_1(\phi_0,K_{10}), \\
0 & = -p_1 A(\phi_0,K_{01}), \\
0 & = -p_1 A_1(\phi_1,K_{10}) + \left< \psi_1,h_{1010}\right>, \\
0 & = -p_1 A_1(\phi_1,K_{01}) + \left< \psi_1,h_{1001}\right> + 1.
\end{aligned}
\end{equation}
Pairing the first two systems in \cref{btdde:eq:second_order_systems_tbt}
with $\psi_0$ yields
\begin{align*}
\left< \psi_1,h_{1010} \right>  & = -p_0A_1(\phi_0,K_{10}),\\
\left< \psi_1,h_{1001} \right>  & = -p_0A_1(q_0,K_{01}).
\end{align*}
Using these identities in the last two equations of
\cref{btdde:eq:second_order_solvability_condition_tbt} yields
\begin{equation}
\label{btdde:eq:tbt_k_system}
\begin{cases}
0 & = -p_1A_1(\phi_1,K_{10})-p_0A_1(\phi_0,K_{10}),\\
0 & = -p_1A_1(\phi_1,K_{01})-p_0A_1(\phi_0,K_{01}) + 1.
\end{cases}
\end{equation}

Now, let
\begin{align*}
K_{10} & = \delta_1e_1 + \delta_2e_2,\\
K_{01} & = \delta_3e_1 + \delta_4e_2,
\end{align*}
where $e_1$ and $e_2$ are the standard basis vectors in $\mathbb R^2$ and
$\delta_i(i = 1\dots4)$ are constants to be determined. Substituting the above
expressions for $K_{10}$ and $K_{01}$ into equations
\cref{btdde:eq:tbt_k_system}, and the first two equations of
\cref{btdde:eq:second_order_solvability_condition_tbt}, we obtain, by using the
linearity of the multilinear forms, that
\begin{equation}
\begin{cases}
0 & = \delta_1\left(p_1A_1(\phi_1,e_1) + p_0A_1(\phi_0,e_1)\right) + \delta_2\left(p_0A_1(\phi_0,e_2) + p_1A_1(\phi_1,e_2)\right),\\
1 & = \delta_3\left(p_1A_1(\phi_1,e_1) + p_0A_1(q_0,e_1)\right) + \delta_4\left(p_0A_1(\phi_1,e_2) + p_1A_1(q_0,e_2)\right).
\end{cases}
\end{equation}
and
\begin{equation}
\begin{cases}
1 & = \delta_1p_1A_1(\phi_0,e_1) + \delta_2p_1A_1(\phi_0,e_2),\\
0 & = \delta_3p_1(A(\phi_0,e_1) + \delta_4p_1(A(\phi_0,e_2),
\end{cases}
\end{equation}
respectively. Thus, the constants $\delta_i(i = 1\dots4)$ are obtained
by solving the system
\begin{multline*}
\left(\begin{array}{cc}
p_1A_1(\phi_0,e_1) & p_1A_1(\phi_0,e_2)\\
p_1A_1(\phi_1,e_1) + p_0A_1(\phi_0,e_1) & p_0A_1(\phi_0,e_2) + p_1A_1(\phi_1,e_2)
\end{array}\right)\\
\left(\left[\begin{array}{cc}
K_{10} & K_{01}\end{array}\right]\right) = \left(\begin{array}{cc}
1 & 0\\
0 & 1
\end{array}\right).
\end{multline*}

Using \cref{btdde:corollary:sol_double_zero_eig_polynomial} we now define the coefficients
\begin{align}
    \hat h_{1010} & = \BINV 0 \left(A_1(\phi_0,K_{10}\right),\phi_1)\\
    \hat h_{1001} & = \BINV 0 \left(A_1(\phi_0,K_{01}\right),\\
    \hat h_{0110} & = \BINV 0 \left(A_1(\phi_1,K_{10}),h_{1010}\right),\\
    \hat h_{0101} & = \BINV 0 \left(A_1(\phi_1,K_{01}),h_{1001} + \phi_1\right).
\end{align}
Then it follows that solutions to \cref{btdde:eq:second_order_systems_tbt} are given by
\begin{equation}
\label{btdde:eq:seconder_order_coefficients_phase_space_tbt}
\begin{aligned}
    h_{1010} &= \hat h_{1010} + \gamma_3 \phi_0, \\
    h_{1001} &= \hat h_{1001} + \gamma_4 \phi_0, \\
    h_{0110} &= \hat h_{0110} + \gamma_3 \phi_1 - \vartheta_{0010} \phi_1, \\
    h_{0101} &= \hat h_{0101} + \gamma_4 \phi_1 - \vartheta_{0001} \phi_1,
\end{aligned}
\end{equation}
where $\gamma_3$ and $\gamma_4$ are constants to be determined. Note that we
still have freedom in the coefficients $h_{0110}$ and $h_{0101}$ by an addition of
a multiple of the eigenfunction $\phi_0$. However, for our purposes, we will 
not need to include it here.

\subsubsection{Cubic terms}
By collecting the cubic terms in the homological equation, we obtain the following systems
\begin{equation}
\label{btdde:eq:third_order_systems_tbt}
\begin{aligned}
-\SUNSTAR A jh_{2010} ={}& \left[A_1(h_{2000},K_{10})+ 2 B(h_{1010},\phi_0) + B_1(\phi_0,\phi_0,K_{10})\right]\rss \\
                         & -2j\left(a h_{0110} + h_{1100} - \vartheta_{1000} \phi_1 - a \vartheta_{0010} \phi_1   \right), \\
-\SUNSTAR A jh_{1110} ={}& \left[A_1(h_{1100},K_{10}) + B(h_{0110},\phi_0) + B(h_{1010},\phi_1) + B_1(\phi_0,\phi_1,K_{10})\right]\rss \\
                         & -j\left(b h_{0110} + h_{0200} + h_{2010} - \vartheta_{1000} (h_{1010} - \vartheta_{0010} \phi_0) \right. \\
                         & \left. - \vartheta_{0010} (b \phi_1 - \vartheta_{1000} \phi_0 + h_{2000})  \right), \\
-\SUNSTAR A jh_{2001} ={}& \left[A_1(h_{2000},K_{01}) + 2 B(h_{1001},\phi_0) + B_1(\phi_0,\phi_0,K_{01})\right]\rss \\
                         & - 2 a j(h_{0101} - \vartheta_{0001} \phi_1), \\
-\SUNSTAR A jh_{1101} ={}& \left[A_1(h_{1100},K_{01}) + B(h_{0101},\phi_0) + B(h_{1001},\phi_1) + B_1(\phi_0,\phi_1,K_{01})\right]\rss \\
                        & -j\left(b h_{0101} + h_{1100} + h_{2001} - \vartheta_{1000} (h_{1001} + \phi_1 - \vartheta_{0001} \phi_0) \right. \\
                        & \left. - \vartheta_{0001} (b \phi_1 - \vartheta_{1000} \phi_0 + h_{2000})  \right), \\
-\SUNSTAR A jh_{1002} ={}& \left[2 A_1(h_{1001},K_{01}) + A_1(\phi_0,K_{02}) + A_2(\phi_0,K_{01},K_{01})\right]\rss, \\
-\SUNSTAR A jh_{0102} ={}& \left[2 A_1(h_{0101},K_{01}) + A_1(\phi_1,K_{02}) + A_2(\phi_1,K_{01},K_{01})\right]\rss \\
                         & -j\left(2 h_{0101} + h_{1002} - 2 \vartheta_{0001} (h_{1001} + \phi_1 - \vartheta_{0001} \phi_0) \right), \\
-\SUNSTAR A jh_{1011} ={}& \left[A_1(h_{1001},K_{10}) + A_1(h_{1010},K_{01}) + A_1(\phi_0,K_{11}) + A_2(\phi_0,K_{01},K_{10})\right]\rss \\
                         & -j\left(h_{0101} - \vartheta_{0001} \phi_1\right), \\
-\SUNSTAR A jh_{0111} ={}& \left[A_1(h_{0101},K_{10}) + A_1(h_{0110},K_{01}) + A_1(\phi_1,K_{11}) + A_2(\phi_1,K_{01},K_{10})\right]\rss, \\
                         & -j\left(h_{0110} + h_{1011} - \vartheta_{0010} (h_{1001} + \phi_1 - \vartheta_{0001} \phi_0) - \vartheta_{0001} (h_{1010} - \vartheta_{0010} \phi_0) \right), \\
-\SUNSTAR A jh_{1020} ={}& \left[2 A_1(h_{1010},K_{10}) + A_1(\phi_0,K_{20}) + A_2(\phi_0,K_{10},K_{10})\right]\rss \\
                         & - 2 j\left(h_{0110} - \vartheta_{0010} \phi_1\right), \\
-\SUNSTAR A jh_{0120} ={}& \left[2 A_1(h_{0110},K_{10}) + A_1(\phi_1,K_{20}) + A_2(\phi_1,K_{10},K_{10})\right]\rss \\
                         & -j\left(h_{1020} - 2 \vartheta_{0010} (h_{1010} - \vartheta_{0010} \phi_0)\right). \\
\end{aligned}
\end{equation}
To solve the above systems, we write the second order parameter coefficients in $K$ with respect to a basis in $K_{10}$ and $K_{01}$ as follows
\begin{equation}
\label{btdde:eq:second_order_parameter_coefficients_tbt}
\begin{aligned}
    K_{02} ={}& \gamma_5 K_{10} + \gamma_6 K_{01}, \\
    K_{11} ={}& \gamma_7 K_{10} + \gamma_8 K_{01}, \\
    K_{20} ={}& \gamma_9 K_{10} + \gamma_{10} K_{01}.
\end{aligned}
\end{equation}
By applying the Fredholm alternative to the systems in
\cref{btdde:eq:third_order_systems_tbt}, we obtain ten equations to be solved for (also the
ten) variables $\vartheta_{0010}$, $\vartheta_{0001}$, $\gamma_3$, $\gamma_4$,
$\gamma_5$, $\gamma_6$, $\gamma_7$, $\gamma_8$, $\gamma_9$, $\gamma_{10}$.

Pairing the first two equations in \cref{btdde:eq:third_order_systems_tbt} with $\psi_1$ and using 
\cref{btdde:eq:seconder_order_coefficients_phase_space_tbt} yields
\begin{equation}
\label{btdde:eq:pairing_first_two_equations}
\begin{aligned}
0 ={}& p_1\left[A_1(h_{2000},K_{10})+ 2 B(\hat h_{1010},\phi_0) + B_1(\phi_0,\phi_0,K_{10})\right] \\
                         & -2\langle\psi_1,a \hat h_{0110} + h_{1100}\rangle + 2\vartheta_{1000}
                            + 2a\gamma_3 + 4a \vartheta_{0010}, \\
0 ={}& p_1 \left[A_1(h_{1100},K_{10}) + B(h_{0110},\phi_0) + B(h_{1010},\phi_1) + B_1(\phi_0,\phi_1,K_{10})\right] \\
                         & -\left<\psi_1,ah_{0110} + h_{0200} + h_{2010} - \vartheta_{1000} (b h_{1010} - \vartheta_{0010} \phi_0) \right. \\
                         & \left. - \vartheta_{0010} (b \phi_1 - \vartheta_{1000} \phi_0 + h_{2000})  \right>. \\
\end{aligned}
\end{equation}
Since $h_{2010}$ is still to be determined we use that
\[
    \left< \psi_1, h_{2010} \right> 
    = \left< j h_{2010}, A^\odot \psi_0 \right> 
    = \left< \SUNSTAR A j h_{2010}, \psi_0 \right>,
\]
Furthermore, from \cref{btdde:eq:Assh1100} we obtain that
\begin{align}
    \label{btdde:eq:psi1_Assh1000}
    \left<\psi_1, b \phi_1 - \vartheta_{1000} \phi_0 + h_{2000}  \right> &{}= p_1 B(\phi_0,\phi_1).
\end{align}
Thus, the second equation in \cref{btdde:eq:pairing_first_two_equations} becomes
\begin{align*}
0 ={}& p_1 \left[A_1(h_{1100},K_{10}) + B(h_{0110},\phi_0) + B(h_{1010},\phi_1) + B_1(\phi_0,\phi_1,K_{10})\right] \\
     & + p_0 \left[A_1(h_{2000},K_{10})+ 2 B(h_{1010},\phi_0) + B_1(\phi_0,\phi_0,K_{10})\right] \\
     & -2\left<\psi_0, a h_{0110} + h_{1100}\right> -\left<\psi_1,bh_{0110} + h_{0200} - \vartheta_{1000} (h_{1010} - \vartheta_{0010} \phi_0)\right> \\
     & + \vartheta_{0010} p_1 B(\phi_0,\phi_1).
\end{align*}
Substituting \cref{btdde:eq:second_order_parameter_coefficients_tbt} into the above expression, and using the lineariry of the multilinear forms yields
\begin{align*}
0 ={}& p_1 \left[A_1(h_{1100},K_{10}) + B(\hat h_{0110},\phi_0) + B(\hat h_{1010},\phi_1) + B_1(\phi_0,\phi_1,K_{10})\right] \\
     & + p_0 \left[A_1(h_{2000},K_{10})+ 2 B(\hat h_{1010},\phi_0) + B_1(\phi_0,\phi_0,K_{10})\right] \\
     & -2\langle\psi_0, a \hat h_{0110} + h_{1100}\rangle -\langle\psi_1, b \hat h_{0110} + h_{0200} - \vartheta_{1000} \hat h_{1010}\rangle \\
     & + b \gamma_3 + b\vartheta_{0010}.
\end{align*}
It follows that the constants $\gamma_3$ and $\vartheta_{0010}$ are give by
\begin{equation}
\label{btdde:eq:gamma_3_theta0010_tbt}
\begin{pmatrix}
				 2a &  4a \\
				  b &   b 
\end{pmatrix}
\begin{pmatrix}
				\gamma_3 \\
				\vartheta_{0010}
\end{pmatrix}
=
\begin{pmatrix}
				\zeta_1 \\
				\zeta_2 
\end{pmatrix},
\end{equation}
where the right-hand side is given by
\begin{equation}
\label{btdde:eq:zeta1_zeta2_tbt}
\begin{aligned}
\zeta_1 ={}& -p_1\left[A_1(h_{2000},K_{10})+ 2 B(\hat h_{1010},\phi_0) + B_1(\phi_0,\phi_0,K_{10})\right] \\
           & +2\langle\psi_1,a \hat h_{0110} + h_{1100}\rangle - 2\vartheta_{1000}, \\
\zeta_2 ={}& -p_1 \left[A_1(h_{1100},K_{10}) + B(\hat h_{0110},\phi_0) + B(\hat h_{1010},\phi_1) + B_1(\phi_0,\phi_1,K_{10})\right] \\
           & - p_0 \left[A_1(h_{2000},K_{10})+ 2 B(\hat h_{1010},\phi_0) + B_1(\phi_0,\phi_0,K_{10})\right] \\
           & +2\langle\psi_0, a \hat h_{0110} + h_{1100}\rangle +\langle\psi_1, b \hat h_{0110} + h_{0200} - \vartheta_{1000} \hat h_{1010}\rangle.
\end{aligned}
\end{equation}
Notice here that the matrix given in the left-hand side in
\cref{btdde:eq:gamma_3_theta0010_tbt} is equivalent to the matrix in
\cref{btdde:eq:gamma_5_theta0001} when performing the center manifold reduction near
the generic Bogdanov--Takens bifurcation. 

The third and fourth systems in \cref{btdde:eq:third_order_systems_tbt} lead to a
similar system. Indeed, by pairing these systems with the adjoint eigenfunction
$\psi_1$ and using \cref{btdde:eq:seconder_order_coefficients_phase_space_tbt} leads to the equations
\begin{equation}
\label{btdde:eq:pairing_second_two_equations}
\begin{aligned}
    0 ={}& p_1\left[A_1(h_{2000},K_{01}) + 2 B(\hat h_{1001},\phi_0) + B_1(\phi_0,\phi_0,K_{01})\right] 
    \process{%
        \\ &
    }
 - 2a\langle \psi_1, \hat h_{0101} \rangle + 2a \gamma_4 + 4 a \vartheta_{0001}, \\
0 ={}& p_1\left[A_1(h_{1100},K_{01}) + B(h_{0101},\phi_0) + B(h_{1001},\phi_1) + B_1(\phi_0,\phi_1,K_{01})\right] \\
                        & -\left< \psi_1, b h_{0101} + h_{1100} + h_{2001} - \vartheta_{1000} (h_{1001} + \phi_1 - \vartheta_{0001} \phi_0) \right. \\
                        & \left. - \vartheta_{0001} (b \phi_1 - \vartheta_{1000} \phi_0 + h_{2000}) \right>. \\
\end{aligned}
\end{equation}
By using \cref{btdde:eq:psi1_Assh1000,btdde:eq:seconder_order_coefficients_phase_space_tbt} the second equation becomes
\begin{align*}
0 ={}& p_1\left[A_1(h_{1100},K_{01}) + B(\hat h_{0101},\phi_0) + B(\hat h_{1001},\phi_1) + B_1(\phi_0,\phi_1,K_{01})\right] + \\
     & p_0 \left[A_1(h_{2000},K_{01}) + 2 B(\hat h_{1001},\phi_0) + B_1(\phi_0,\phi_0,K_{01})\right] - 2 a \langle \psi_0, \hat h_{0101} \rangle \\
                        & -\langle \psi_1, b \hat h_{0101} + h_{1100} - \vartheta_{1000} \hat h_{1001} \rangle  + \vartheta_{1000} 
                          + b \gamma_4 + b \vartheta_{0001}.
\end{align*}
It follows that the constants $\gamma_4$ and $\vartheta_{0001}$ are give by
\begin{equation}
\label{btdde:eq:gamma_4_theta0001_tbt}
\begin{pmatrix}
    2a &  4a \\
     b &   b 
\end{pmatrix}
\begin{pmatrix}
    \gamma_4 \\
    \vartheta_{0001}
\end{pmatrix}
=
\begin{pmatrix}
    \zeta_3 \\
    \zeta_4 
\end{pmatrix},
\end{equation}
where the right-hand side is given by
\begin{equation}
    \label{btdde:eq:zeta3_zeta4_tbt}
    \begin{aligned}
        \zeta_3 ={}& -p_1\left[A_1(h_{2000},K_{01}) + 2 B(\hat h_{1001},\phi_0) + B_1(\phi_0,\phi_0,K_{01})\right] + 2a\langle \psi_1, \hat h_{0101} \rangle\\
        \zeta_4 ={}& -p_1\left[A_1(h_{1100},K_{01}) + B(\hat h_{0101},\phi_0) + B(\hat h_{1001},\phi_1) + B_1(\phi_0,\phi_1,K_{01})\right] + \\
                   &- p_0 \left[A_1(h_{2000},K_{01}) + 2 B(\hat h_{1001},\phi_0) + B_1(\phi_0,\phi_0,K_{01})\right] + 2 a \langle \psi_0, \hat h_{0101} \rangle \\
                                & + \langle \psi_1, b \hat h_{0101} + h_{1100} - \vartheta_{1000} \hat h_{1001} \rangle + \vartheta_{1000}. 
    \end{aligned}
\end{equation}

Now that $\vartheta_{0010}$, $\vartheta_{0001}$, $\gamma_3$, and $\gamma_4$ are
known, we can apply the Fredholm alternative to the remaining six systems in
\cref{btdde:eq:third_order_systems_tbt}. This yields equations which can be solved
separately, i.e. without a matrix as in
\cref{btdde:eq:zeta1_zeta2_tbt,btdde:eq:zeta3_zeta4_tbt}. To keep the derivation readable,
we divide six equations into three groups. We start by pairing the fifth and
sixth equations in \cref{btdde:eq:third_order_systems_tbt} with $\psi_1$. We obtain that 
\begin{align*}
\gamma_5 ={}& -p_1\left[2 A_1(\hat h_{1001}, K_{01}) + A_2(\phi_0,K_{01},K_{01})\right]. \\
0 ={}& p_1 \left[2 A_1(h_{0101},K_{01}) + A_1(\phi_1,K_{02}) + A_2(\phi_1,K_{01},K_{01})\right] \\
     & -\left< \psi_1, 2 h_{0101} + h_{1002} - 2 \vartheta_{0001} (h_{1001} + \phi_1 - \vartheta_{0001} \phi_0) \right>, \\
  ={}& p_1 \left[2 A_1(\hat h_{0101},K_{01}) + A_2(\phi_1,K_{01},K_{01})\right]  
     + p_0 \left[2 A_1(\hat h_{1001},K_{01})+ A_2(\phi_0,K_{01},K_{01})\right] \\
     & -2 \langle \psi_1, \hat h_{0101} \rangle + \gamma_6 + 2 \vartheta_{0001}
\end{align*}
Here we used
\cref{btdde:eq:seconder_order_coefficients_phase_space_tbt,btdde:eq:second_order_parameter_coefficients_tbt},
and the (by now) equalities given \cref{btdde:eq:second_order_solvability_condition_tbt} and 
\cref{btdde:eq:tbt_k_system}. It follows that 
\begin{equation}
\begin{aligned}
\gamma_6 ={}& -p_1 \left[2 A_1(\hat h_{0101},K_{01}) + A_2(\phi_1,K_{01},K_{01})\right]  
              - p_0 \left[2 A_1(\hat h_{1001},K_{01})+ A_2(\phi_0,K_{01},K_{01})\right] \\
            & +2 \langle \psi_1, \hat h_{0101} \rangle - 2 \vartheta_{0001}. \nonumber
\end{aligned}
\end{equation}
By applying the Fredholm alternative to the last four systems in
\cref{btdde:eq:third_order_systems_tbt}, we obtain equations which can be solved 
similarly as the last two equations. We obtain that
\begin{equation}
\begin{aligned}
\gamma_7  ={}& - p_1 \left[A_1(\hat h_{1001},K_{10}) + A_1(\hat h_{1010},K_{01}) +  A_2(\phi_0,K_{01},K_{10})\right]
              + \langle\psi_1, \hat h_{0101} \rangle - 2 \vartheta_{0001}, \\
\gamma_8 ={}& -p_1 \left[A_1(\hat h_{0101},K_{10}) + A_1(\hat h_{0110},K_{01}) + A_2(\phi_1,K_{01},K_{10})\right] + \langle\psi_1, \hat h_{0110} \rangle\\
     & -p_0 \left[A_1(\hat h_{1001},K_{10}) + A_1(\hat h_{1010},K_{01}) + A_2(\phi_0,K_{01},K_{10})\right] + \langle\psi_0, \hat h_{0101}\rangle - \vartheta_{0010},\\
\gamma_9 ={}& -p_1 \left[2 A_1(\hat h_{1010},K_{10}) + A_2(\phi_0,K_{10},K_{10})\right] + 2 \langle\psi_1, \hat h_{0110} \rangle - 4 \vartheta_{0010}, \\
\gamma_{10} ={}& -p_1 \left[2 A_1(\hat h_{0110},K_{10}) + A_2(\phi_1,K_{10},K_{10})\right] \\
     & - p_0 \left[2 A_1(\hat h_{1010},K_{10}) + A_2(\phi_0,K_{10},K_{10})\right] + 2 \langle\psi_0, \hat h_{0110}\rangle.
\end{aligned}
\end{equation}
\begin{align*}
\end{align*}

Now when all systems in \cref{btdde:eq:third_order_systems_tbt} are consistent,
we use \cref{btdde:corollary:sol_double_zero_eig_polynomial} to obtain the solutions
\begin{equation}
\label{btdde:eq:third_order_systems_solutions_tbt}
\begin{aligned}
h_{2010} ={}& \BINV 0 \left(A_1(h_{2000},K_{10})+ 2 B(h_{1010},\phi_0) + B_1(\phi_0,\phi_0,K_{10}) \right., \\
            & \left. 2 \left(a h_{0110} + h_{1100} - \vartheta_{1000} \phi_1 - a \vartheta_{0010} \phi_1   \right)\right), \\
h_{1110} ={}& \BINV 0 \left(A_1(h_{1100},K_{10}) + B(h_{0110},\phi_0) + B(h_{1010},\phi_1) + B_1(\phi_0,\phi_1,K_{10})\right., \\
            & b h_{0110} + h_{0200} + h_{2010} - \vartheta_{1000} (h_{1010} - \vartheta_{0010} \phi_0) \\
            & \left. - \vartheta_{0010} (b \phi_1 - \vartheta_{1000} \phi_0 + h_{2000})  \right), \\
h_{2001} ={}& \BINV 0 \left( A_1(h_{2000},K_{01}) + 2 B(h_{1001},\phi_0) + B_1(\phi_0,\phi_0,K_{01}), 2 a (h_{0101} - \vartheta_{0001} \phi_1) \right), \\
h_{1101} ={}& \BINV 0 \left(A_1(h_{1100},K_{01}) + B(h_{0101},\phi_0) + B(h_{1001},\phi_1) + B_1(\phi_0,\phi_1,K_{01})\right., \\
            & b h_{0101} + h_{1100} + h_{2001} - \vartheta_{1000} (h_{1001} + \phi_1 - \vartheta_{0001} \phi_0) \\
            & \left. - \vartheta_{0001} (b \phi_1 - \vartheta_{1000} \phi_0 + h_{2000}) \right), \\
h_{1002} ={}& \BINV 0 \left(2 A_1(h_{1001},K_{01}) + A_1(\phi_0,K_{02}) + A_2(\phi_0,K_{01},K_{01})\right), \\
h_{0102} ={}& \BINV 0 \left(2 A_1(h_{0101},K_{01}) + A_1(\phi_1,K_{02}) + A_2(\phi_1,K_{01},K_{01})\right., \\
            & \left. 2 h_{0101} + h_{1002} - 2 \vartheta_{0001} (h_{1001} + \phi_1 - \vartheta_{0001} \phi_0) \right), \\
h_{1011} ={}& \BINV 0 \left(A_1(h_{1001},K_{10}) + A_1(h_{1010},K_{01}) + A_1(\phi_0,K_{11}) + A_2(\phi_0,K_{01},K_{10})\right., \\
            & \left. h_{0101} - \vartheta_{0001} \phi_1\right), \\
h_{0111} ={}& \BINV 0 \left(A_1(h_{0101},K_{10}) + A_1(h_{0110},K_{01}) + A_1(\phi_1,K_{11}) + A_2(\phi_1,K_{01},K_{10})\right., \\
            & \left. h_{0110} + h_{1011} - \vartheta_{0010} (h_{1001} + \phi_1 - \vartheta_{0001} \phi_0) - \vartheta_{0001} (h_{1010} - \vartheta_{0010} \phi_0) \right), \\
h_{1020} ={}& \BINV 0 \left(2 A_1(h_{1010},K_{10}) + A_1(\phi_0,K_{20}) + A_2(\phi_0,K_{10},K_{10}), 2(h_{0110} - \vartheta_{0010} \phi_1) \right), \\
h_{0120} ={}& \BINV 0 \left(2 A_1(h_{0110},K_{10}) + A_1(\phi_1,K_{20}) + A_2(\phi_1,K_{10},K_{10})\right., \\
                         & \left.h_{1020} - 2 \vartheta_{0010} (h_{1010} - \vartheta_{0010} \phi_0)\right). \\
\end{aligned}
\end{equation}

\subsubsection{Homolinic asymptotics}
\label{btdde:sec:transcritical_bt_homoclinic_asymptotics}
To approximate the homoclinic solutions emanating from the transcritical Bogdanov--Takens point,
we apply the singular rescaling
\begin{equation}
\label{btdde:eq:blowup}				
\beta_1 = \pm 4 \frac{a^2}{b^2} \epsilon^2, \quad 
\beta_2 = \frac a b \left(\tau \pm 2\right) \epsilon^2, \quad 
w_0 = \frac a{b^2} (u \mp 2) \epsilon^2, \quad
w_1 = \frac{a^2}{b^3} v \epsilon^3, \quad
s = \frac ab \epsilon \eta, \quad (\epsilon \neq 0),
\end{equation}
to \cref{btdde:eq:normal_form_orbital_tbt} with $h(0,0)=0$ to obtain the second-order
nonlinear oscillator
\begin{equation}
    \label{btdde:eq:second_order_nonlinear_oscillator}
    \ddot u = -4 + u^2 + \dot u \left( u + \tau \right)\epsilon + \mathcal O(\epsilon^4).
\end{equation}
Note that this is precisely the second-order nonlinear oscillator obtained in
\cite{Bosschaert@Interplay} when deriving asymptotics for the homoclinic
solutions emanating from the generic Bogdanov--Takens bifurcation.
Therefore, the third-order homoclinic asymptotics for the homoclinic orbits emanating from \cref{btdde:eq:normal_form_orbital_tbt}
are given by
\begin{equation}
\label{btdde:eq:third_order_predictor_LP_tau_tbt}
\begin{cases}
\begin{aligned}
w_0(\eta)  &= \frac{a}{b^2} \left[ \tilde {u}\left(\tanh\left(\xi\left(\frac{a}{b}\epsilon\eta\right)\right)\right) \mp 2 \right] \epsilon^2, \\
w_1(\eta)  &= \frac{a^2}{b^3} \tilde {v}\left(\tanh\left(\xi\left(\frac{a}{b}\epsilon\eta\right)\right)\right) \epsilon^3, \\
\beta_1 &= \pm 4 \frac{a^2}{b^2} \epsilon^2, \\
\beta_2 &= \frac a b \left(\tau \pm 2\right) \epsilon^2, \quad 
\end{aligned}
\end{cases}
\end{equation}
where $\tau, \tilde u, \tilde v$, and $\xi$ are again given by \cref{btdde:eq:tau_utilde_vtilde_xi}.
Next, we need to obtain $\eta(t)$ to approximate the profiles of the homoclinic solution. For
this we integrate \cref{btdde:eq:theta_expansion_TBT} to obtain
\[
t(\eta) = \left(1  + \vartheta_{0010} \beta_1 + \vartheta_{0001} \beta_2 \right)\eta  + \vartheta_{1000} \frac a{b^2} \int \left[ \tilde {u}\left(\tanh\left(\xi\left(\frac{a}{b}\epsilon\eta\right)\right)\right) \mp 2 \right] \epsilon^2 \, d\eta.
\]
Here, $\beta_1$ and $\beta_2$ are given by \cref{btdde:eq:third_order_predictor_LP_tau_tbt} and from \cite{Bosschaert@Interplay} we obtain that
\begin{equation}
\label{btdde:eq:int_u_s_regular_perturbation}
\begin{aligned}
\epsilon \frac a b \int
\tilde {u}\left(\tanh\left(\xi\left(\frac{a}{b}\epsilon\eta\right)\right)\right) \, d\eta
={}& 2 \tilde\xi -6 \tanh (\tilde\xi ) + 
     \left( \frac{18 \sech^2(\tilde\xi )}{7}+\frac{12}{7} \log (\cosh (\tilde\xi ))
		 \right) \epsilon \\ &+
		 \frac{9}{49} \left(4 \tilde\xi -9 \tanh (\tilde\xi )+5 \tanh (\tilde\xi ) \sech^2(\tilde\xi
		 )\right) \epsilon^2 \\ &+
		 \frac{18 \left(-21 \sech^4(\tilde\xi )+47 \sech^2(\tilde\xi )+8 \log (\cosh
		 (\tilde\xi ))\right)}{2401} \epsilon^3  \\
                                &+ \mathcal{O}(\epsilon^4),
\end{aligned}
\end{equation}
where $\tilde\xi = \xi(\frac{a}{b}\eta\epsilon)$. 

It remains to numerically solve the equation 
\begin{equation*}
    t(\eta) - t = 0,
\end{equation*}
for $\eta$.

\subsubsection{Codimension-one equilibrium bifurcations}
To approximate the transcritical and Hopf curves and their corresponding equilibria in
\cref{btdde:eq:normal_form_orbital_tbt}, one should substitute the expressions for $\beta$
and the equilibrium coordinates into the expansions \cref{btdde:eq:h_expansion-TBT} and
\cref{btdde:eq:K_expansion_TBT}. It follows that the transcritical curve is approximated by 
\begin{equation*}
\begin{cases}
\begin{aligned}
    x &= \mathcal H\left(0,0,0,0\right), \\
    \alpha &= K(0,\epsilon),
\end{aligned}
\end{cases}
\end{equation*}
for $|\epsilon|$ small. The Hopf curves are approximated by
\begin{equation*}
\begin{cases}
\begin{aligned}
    \omega &= \epsilon, \\
    x &= \mathcal H\left(0,0,-\epsilon^2,0\right), \\
    \alpha &= K\left(-\epsilon^2,0\right),
\end{aligned}
\end{cases}
\end{equation*}
and
\begin{equation*}
\begin{cases}
\begin{aligned}
    \omega &= \epsilon, \\
    x &= \mathcal H\left(-\frac{\epsilon^2}{a},0,\epsilon^2,\frac ba\epsilon^2\right), \\
    \alpha &= K\left(\epsilon^2,\frac ba\epsilon^2\right),
\end{aligned}
\end{cases}
\end{equation*}
for $\epsilon > 0$ small.

\section{Examples}
\label{btdde:sec:Examples}

In this section, we will demonstrate the correctness of the center manifold
reduction and the homoclinic asymptotics from
\cref{btdde:sec:parameter-dependent-center-manifold-reduction} on four different
models from
\cite{Jiao2021,giannakopoulos2001bifurcations,jiang2007bogdanov,dong2013bogdanov}
in which Bogdanov--Takens bifurcation points are present. Using our
implementation in \DDEBIFTOOL, we will compute the local unfolding of the
Bogdanov--Takens singularities and compare the emanating codimension one curves
with the derived asymptotics.  We will show a clear improvement to the
homoclinic solutions by using the third-order instead of the first order asymptotics. 

Additionally, following \cite{Bosschaert@Interplay}, we will also show that the
approximation order of the homoclinic asymptotic lifts correctly to the
parameter-dependent center manifold, see in particular
\cref{btdde:fig:convergencePlotsGeneric,btdde:fig:convergencePlotsTranscritical}. 

Lastly, we will perform numerical simulations near the bifurcation point under
consideration. The reason for this extra step is twofold. Firstly, it will
provide additional verification of the given results. Secondly, it seems to be
a standard part in papers studying Bogdanov--Takens bifurcation points in
DDEs. We will integrate the DDE directly, not the reduced
system on the center manifold, which is often the case in literature. The
simulation is performed using the Julia package {\tt DifferentialEquations.jl}
\cite{rackauckas2017differentialequations}. However, since in this section only
the main results are given, we provide full details (including simulation
results) in the %
\ifthesis%
    \cref{chapter:BT_DDE_supplement}%
\fi%
\ifsiam%
    \hyperref[mysupplement]{online Supplement}%
\fi%
\ifarxiv%
    \hyperlink{mysupplement}{Supplement}%
\fi. Furthermore, the source
code of the examples has been included into the \DDEBIFTOOL software package.
Alltogether, this will hopefully provide a good starting point when considering other
models.

\subsection{Bogdanov--Takens bifurcation in a predator-prey system with double Allee effect}
\label{btdde:sec:example:predator_prey}
\begin{figure}[ht!]
    \centering
    \includegraphics{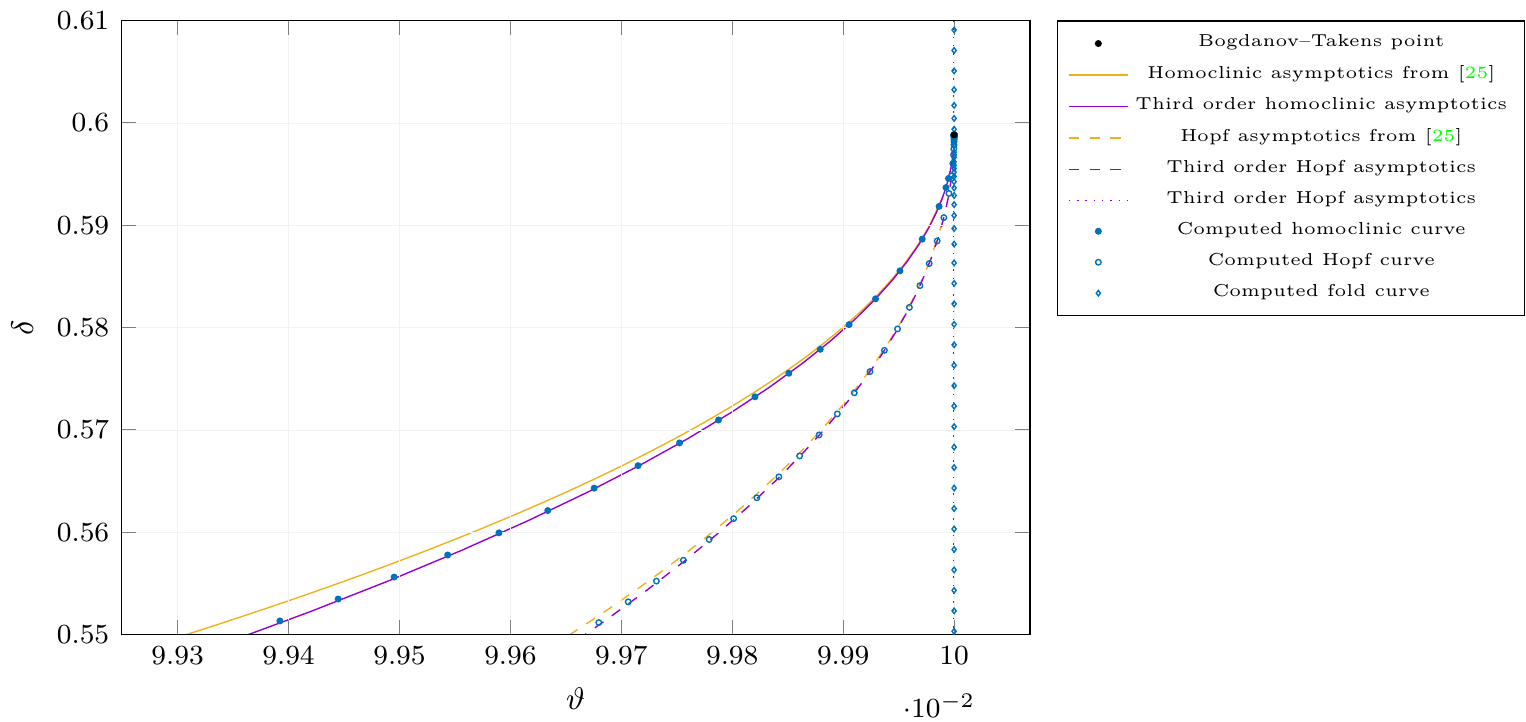}
    \caption{Bifurcation diagram near the derived generic Bogdanov-Takens point in
        \cref{btdde:eq:double_alle_effect_rescaled} comparing computed codimension one
        using \DDEBIFTOOL with the asymptotics obtained in this \paper{} and in \cite{Jiao2021}.}
    \label{btdde:fig:DoubleAlleeEffectCompareParameters}
\end{figure}
\begin{figure}[ht!]
    \centering
    \includegraphics{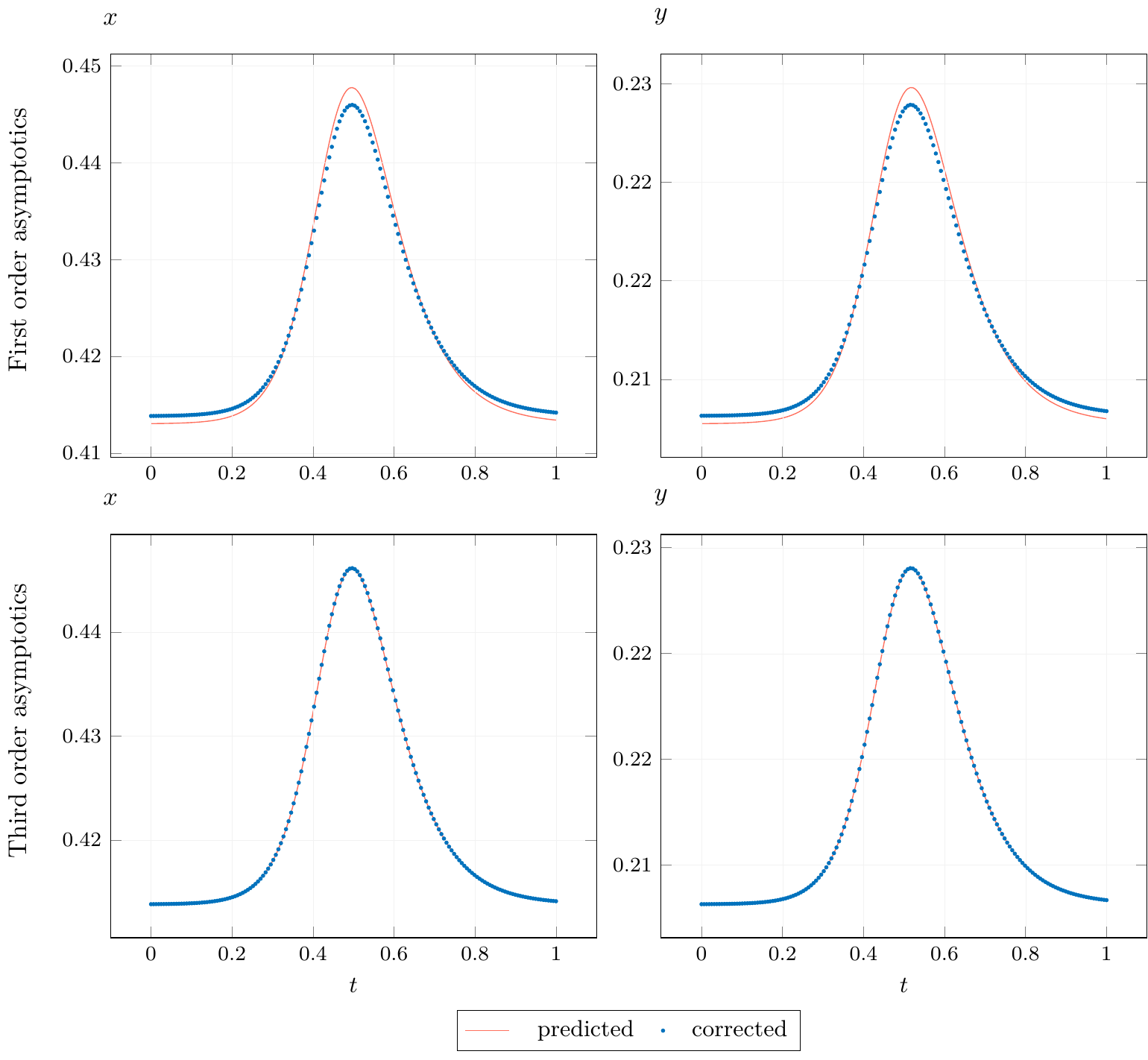}
    \caption{Comparison between the profiles of the first and third-order asymptotics from
    \cref{btdde:sec:generic_bt_homoclinic_asymptotics} with the Newton correct solutions near the generic
        Bogdanov--Takens bifurcation in \cref{btdde:eq:double_alle_effect_rescaled} with the
        perturbation parameter set to $\epsilon=0.3$.}
    \label{btdde:fig:DoubleAlleeEffectCompareProfiles}
\end{figure}
\begin{figure}[ht]
    \includegraphics{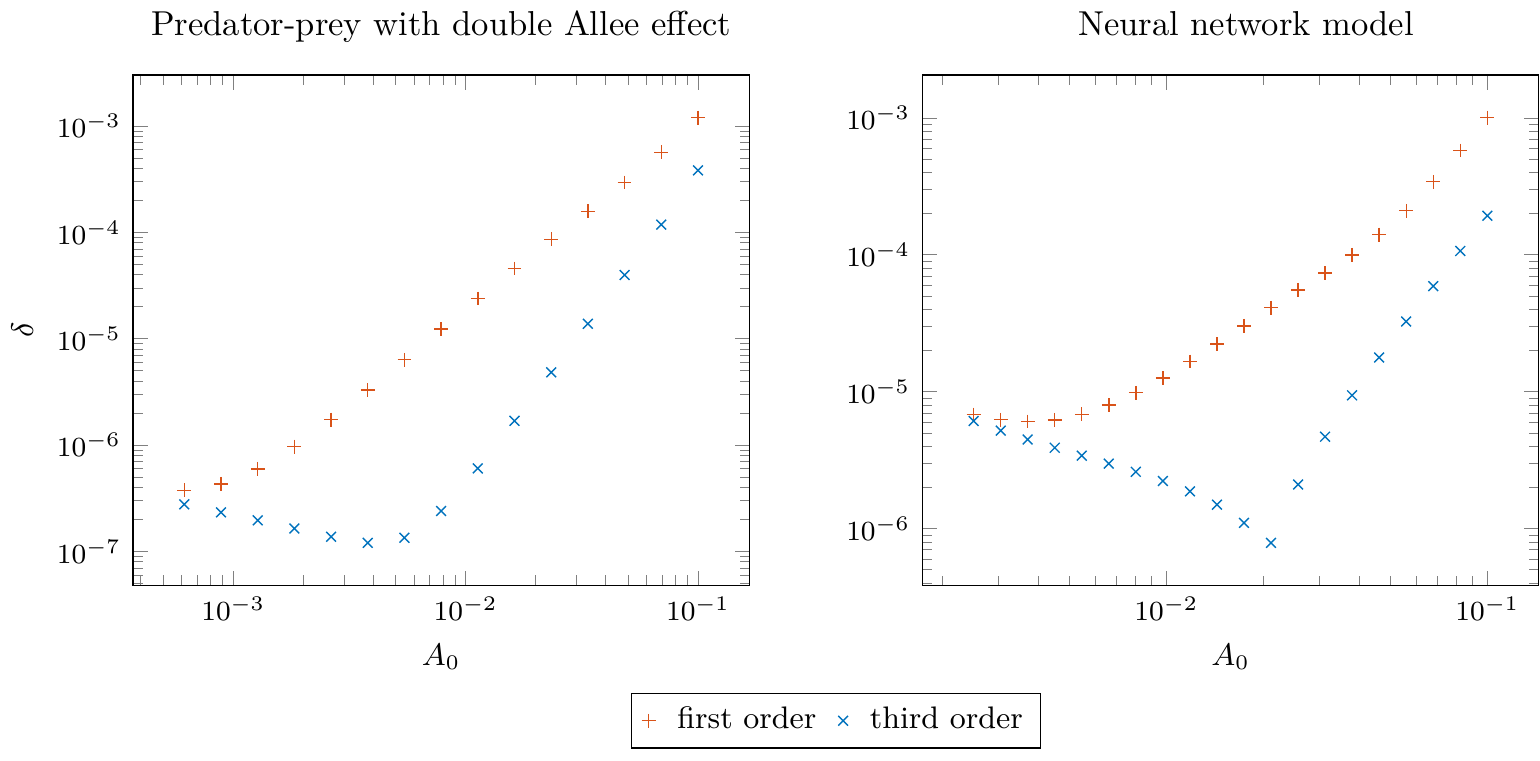}
    \caption{On the abscissa is the approximation to the amplitude $A_0$ and on
        the ordinate the relative error $\delta$ between the constructed
        solution to the defining system
        \cite{connecting@2002} for the homoclinic orbit and the Newton
        corrected solution.}
    \label{btdde:fig:convergencePlotsGeneric}
\end{figure}
In \cite{Jiao2021} the following predator-prey model with double Allee effect and delay is considered
\begin{equation}
\label{btdde:eq:double_alle_effect}
\begin{cases}
    \dot x(t) = \dfrac{rx}{x+n_0}\left(1-\dfrac1 K\right)\left(x - m_0\right) - \dfrac{cxy}{x+\varrho y},\\
    \dot y(t) = -dy + \dfrac{c_1 x(t-\tau)y}{x(t-\tau) + \varrho y(t-\tau)}.
\end{cases}
\end{equation}
Here, the time delay $\tau \geq 0$ is introduced due to the fact that the
reproduction of predator after consuming the prey is not instantaneous, but is
mediated by some time lag required for gestation.
The variables and parameters occurring in \cref{btdde:eq:double_alle_effect}
have the following meaning:
\begin{itemize}
\item $x,y\colon \mathbb R \rightarrow \mathbb R$ denote the prey and predator population densities, respectively.
\item $r$ denotes the maximum prey population growth in absence of the Allee effect.
\item $K$ is the carrying capacity of the environment.
\item $m_0$ is the Allee threshold.
\item $n_0$ is the auxiliary parameter in order to quantify the strength of the Allee effect.
\item $c$ denotes the capturing rate of the predator.
\item $\varrho$ is the half capturing saturation constant.
\item $c_1$ is the conversion rate of prey into predators biomass.
\item $d$ is the per capita predator mortality rate.
\end{itemize}
Following \cite{Jiao2021}, let $(x,y,t) = \left(K\bar x, \frac K \varrho \bar y, \frac{\bar t}r\right)$, and immediately 
dropping the bars again for readability, then \cref{btdde:eq:double_alle_effect} becomes
\begin{equation}
\label{btdde:eq:double_alle_effect_rescaled}
\begin{cases}
    \dot x(t) = x \left( \dfrac{(1-x)(x-\gamma)}{x+\vartheta} - \dfrac{\alpha y}{x+y} \right), \\
    \dot y(t) = \delta y \left( -1 + \dfrac{ m x(t-\tau) }{ x(t-\tau) + y(t-\tau) }\right),
\end{cases}
\end{equation}
where $\vartheta = \frac{n_0}K, \alpha=\frac c{r\varrho}, \gamma = \frac{m_0}K, \delta = \frac dr$ and $m=\frac{c_1}d$.
In \cite{Jiao2021} the phase-space is (although possible) unnecessarily
enlarged to include a parameter. The resulting extended system is used to
derive conditions for a Bogdanov--Takens bifurcation to occur and obtain the
normal form on the center manifold. In our opinion, one should only use the
extended system to prove the existence of a parameter-dependent center manifold,
from which it then follows that we can perform the normalization as
described in \cref{btdde:sec:Center_manifold_reduction}. Thus, solving \cref{btdde:eq:double_alle_effect_rescaled}
for a double equilibrium with respect to $(x,y)$ and $\vartheta$ yields
\begin{equation}
    \begin{cases}
        x_0 = \dfrac{m+\alpha-m\alpha+m\gamma}{2m}, \\
        y_0 = (m-1)x_0, \\
        \vartheta_0 = \dfrac{(\alpha+m(1-\alpha+\gamma))^2 - 4m^2\gamma}{4(m-1)m\alpha},
    \end{cases}
\end{equation}
for $m\neq 1$. The parameters should of course be chosen such that $x$ and $y$
are non-negative. The characteristic matrix at
$(x,y,\vartheta)=(x_0,y_0,\vartheta_0)$ becomes
\[
\begingroup
\renewcommand*{\arraystretch}{2.5}
    \Delta(z) = \begin{pmatrix}
        z + \dfrac{1-m}{m^2}\alpha & \dfrac \alpha {m^2} \\
        -\dfrac{(m-1)^2\delta}m e^{-z\tau} & z + \dfrac{(m-1)\delta}m e^{-z\tau},
    \end{pmatrix}
\endgroup
\]
for $m\neq 0$. A simple calculation confirms indeed that $\det\Delta(0)=0$.
Furthermore, $\det\Delta'(0)=\frac{(m-1)(m\delta-\alpha)}{m^2}$ and
$\det\Delta''(0)=2 - \frac{2 (m-1) \delta\tau}{m}$. Thus, for
$(x,y,\vartheta,\delta)=(x_0,y_0,\vartheta_0,\delta_0)$, with $\delta_0 =
\frac\alpha m$, we have a double zero eigenvalue, provided that
$\det\Delta''(0) \neq 0$. To confirm their analytical findings in \cite{Jiao2021}
numerically, the parameters $\gamma=0.15,\alpha=0.9$ and $m=1.50298303$ are
fixed and $(\vartheta,\delta)$ are taken as unfolding parameters. For these
parameters, we indeed have that $\det\Delta''(0) \neq 0$. Calculating the normal
form coefficients $a$ and $b$ reveal that,
\[
a \approx 0.1479, \qquad b \approx 1.4457.
\]
Thus, the codimension two Bogdanov--Takens bifurcation is non-degenerate.
Furthermore, since the equilibrium $(x_0,y_0)$ depends on the unfolding
parameters, we are in the generic case. Using our implementation of the
homoclinic predictor \cref{btdde:sec:generic_bt_homoclinic_asymptotics} in
\DDEBIFTOOL, we start the continuation of the homoclinic branch emanating from
the Bogdanov--Takens point.

In \cref{btdde:fig:DoubleAlleeEffectCompareParameters}, we compare the parameters of
the computed homoclinic branch with our third-order parameter asymptotics and
the parameters for the homoclinic branch given in \cite{Jiao2021}. We observe
that although the derivation of the asymptotics for the parameters is
non-rigorous, it provides a good predictor. Nonetheless, our third-order
asymptotics for the parameter values provides clearly a better approximation to
the parameter values of the computed homoclinic branch.

Of course, most work lies in the derivation of the asymptotics for the
homoclinic solutions in phase-space itself, which is not available in
\cite{Jiao2021}. In \cref{btdde:fig:DoubleAlleeEffectCompareProfiles} we compare the
first and third-order asymptotics for the profiles of the homoclinic solution.
Here, we set the perturbation parameter to $\epsilon=0.3$ such that both
approximations converge, and we observe a visual difference.

However, a better way to compare the first and third-order
asymptotics for the homoclinic solution numerically is through the convergence
plots, similarly as in \cite{Bosschaert@Interplay}. In
\cref{btdde:fig:convergencePlotsGeneric} we clearly see an improvement of using the
third-order asymptotics over the first order asymptotics. We furthermore
observe the standard V-shaped graphs due to round-off error.
\begin{figure}[ht]
    \includegraphics{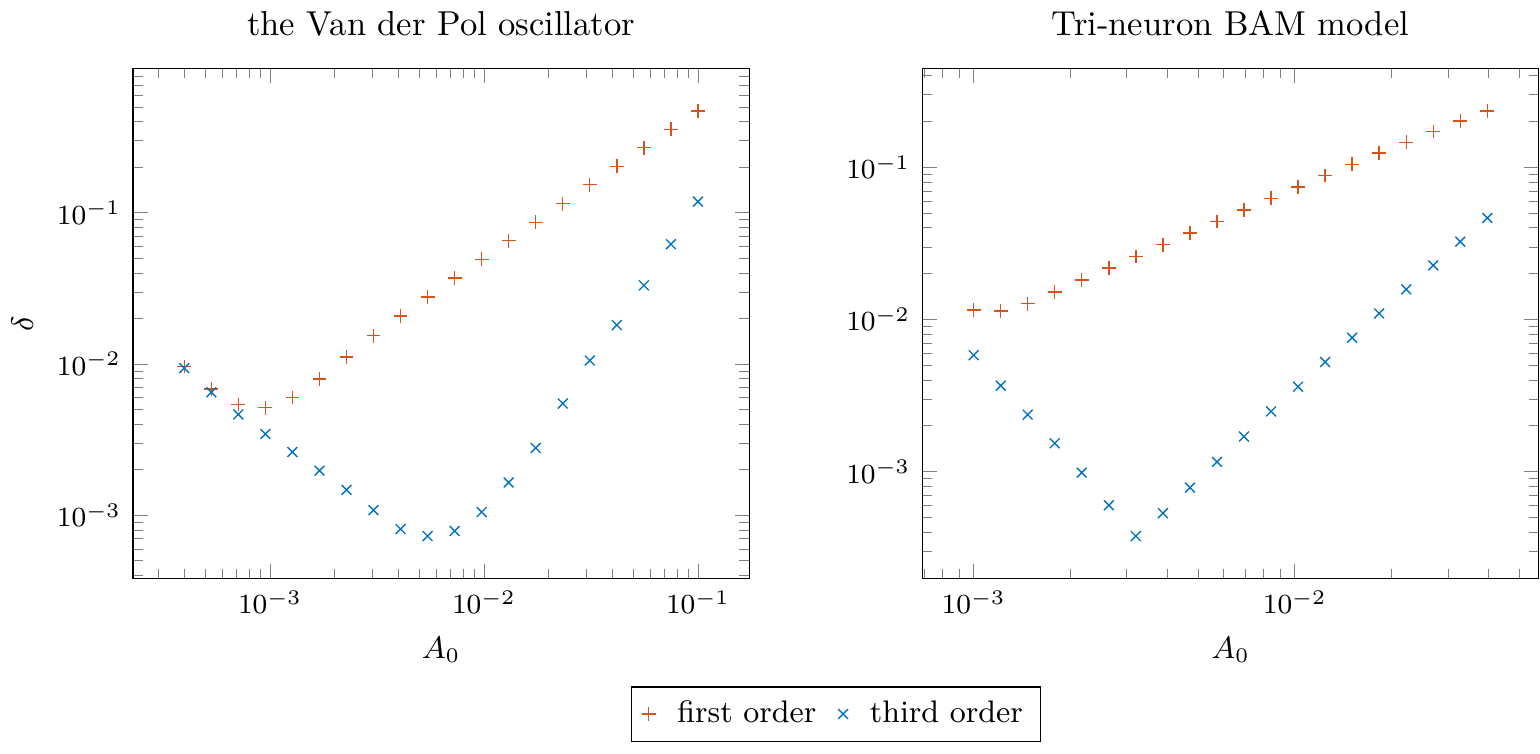}
    \caption{On the abscissa is the approximation to the amplitude $A_0$ and on
        the ordinate the relative error $\delta$ between the constructed solution
        to the defining system for the homoclinic orbit
        and the Newton corrected solution.}
    \label{btdde:fig:convergencePlotsTranscritical}
\end{figure}

\subsection{\ifthesis \phantom{ } \fi Generic Bogdanov--Takens bifurcation in a neural network model}
In this example, we will consider the model 
\begin{equation}
\label{btdde:eq:neural_network}
\begin{cases}
\mu\dot{u}_1(t) = -u_1(t) + q_{11}\alpha(u_1(t - T))-q_{12}u_2(t - T) + e_1,\\
\mu\dot{u}_2(t) = -u_2(t) + q_{21}\alpha(u_1(t - T))-q_{22}u_2(t - T) + e_2,
\end{cases}
\end{equation}
which describes the dynamics of a neural network consisting of an
excitatory and inhibitory neuron \cite{giannakopoulos2001bifurcations}.
The variables and parameters occurring in \cref{btdde:eq:neural_network}
have the following neurophysiological meaning:
\begin{itemize}
\item $u_1,u_2:\mathbb{R}\rightarrow\mathbb{R}$ denote the total post-synaptic
potentials of the excitatory and inhibitory neurons, respectively.
\item $\mu>0$ is a time constant characterizing the dynamical properties
of cell membrane.
\item $q_{ik}\geq0$ represent the strength of the connection line from
the $k$th neuron to the $i$th neuron.
\item $\alpha:\mathbb{R}\rightarrow\mathbb{R}$ is the transfer function
which describes the activity generation of the excitatory neuron as
a function of its total potential $u_1$. The function $\alpha$
is smooth, increasing and has a unique turning point at $u_1 = \theta$.
The transfer function corresponding to the inhibitory neuron is assumed
to be the identity.
\item $T\geq0$ is a time delay reflecting synaptic delay, axonal and dendritic
propagation time.
\item $e_1$ and $e_2$ are external stimuli acting on the excitatory
and inhibitory neuron, respectively.
\end{itemize}

Following \cite{giannakopoulos2001bifurcations} we consider the equation \cref{btdde:eq:neural_network} with
\begin{align*}
\alpha(u_1) & = \frac{1}{1 + e^{-4u_1}}-\frac{1}{2},\qquad q_{11} = 2.6,\qquad q_{21} = 1.0,\qquad q_{22} = 0.0,\\
\mu & = 1.0,\qquad T = 1.0,\qquad e_2 = 0.0,
\end{align*}
and $Q: = q_{12},\,E: = e_1$ as bifurcation parameters. Substituting
into \cref{btdde:eq:neural_network} yields
\begin{equation}
\label{btdde:eq:neural_network_subs}
\begin{cases}
\dot{u}_1(t) = -u_1(t) + 2.6\alpha(u_1(t - 1))-Qu_2(t - 1) + E,\\
\dot{u}_2(t) = -u_2(t) + \alpha(u_1(t - 1)).
\end{cases}
\end{equation}
Notice that for any steady-state we have the symmetry
$(u_1,u_2,E)\rightarrow(-u_1,-u_2,-E)$. It is easy to explicitly derive that the system has a double eigenvalue zero for
\begin{equation}
    \left\{
    \begin{aligned}
        u_1(t) &= \frac14 \log\left(\frac{8 - \sqrt{39}}5\right) \approx -0.2617, \\
        u_2(t) &= -\frac12 \sqrt{\frac{3}{13}} \approx -0.2402, \\
        Q &= \frac{13}{10}, \\
        E &= \frac{\sqrt{39} - 10\atanh \sqrt{\frac{3}{13}}}{20} \approx 0.0505.
    \end{aligned}\right.
\end{equation}

The dependence of the equilibria on the parameters $(Q,E)$ yields a generic
Bogdanov--Takens bifurcation, see \cite{giannakopoulos2001bifurcations}. Notice
that the normal form reduction in \cite{giannakopoulos2001bifurcations} is
incorrect, which leads to the normal form for a transcritical Bogdanov--Takens
bifurcation.

\begin{figure}[ht!]
    \centering
    \includetikzscaled{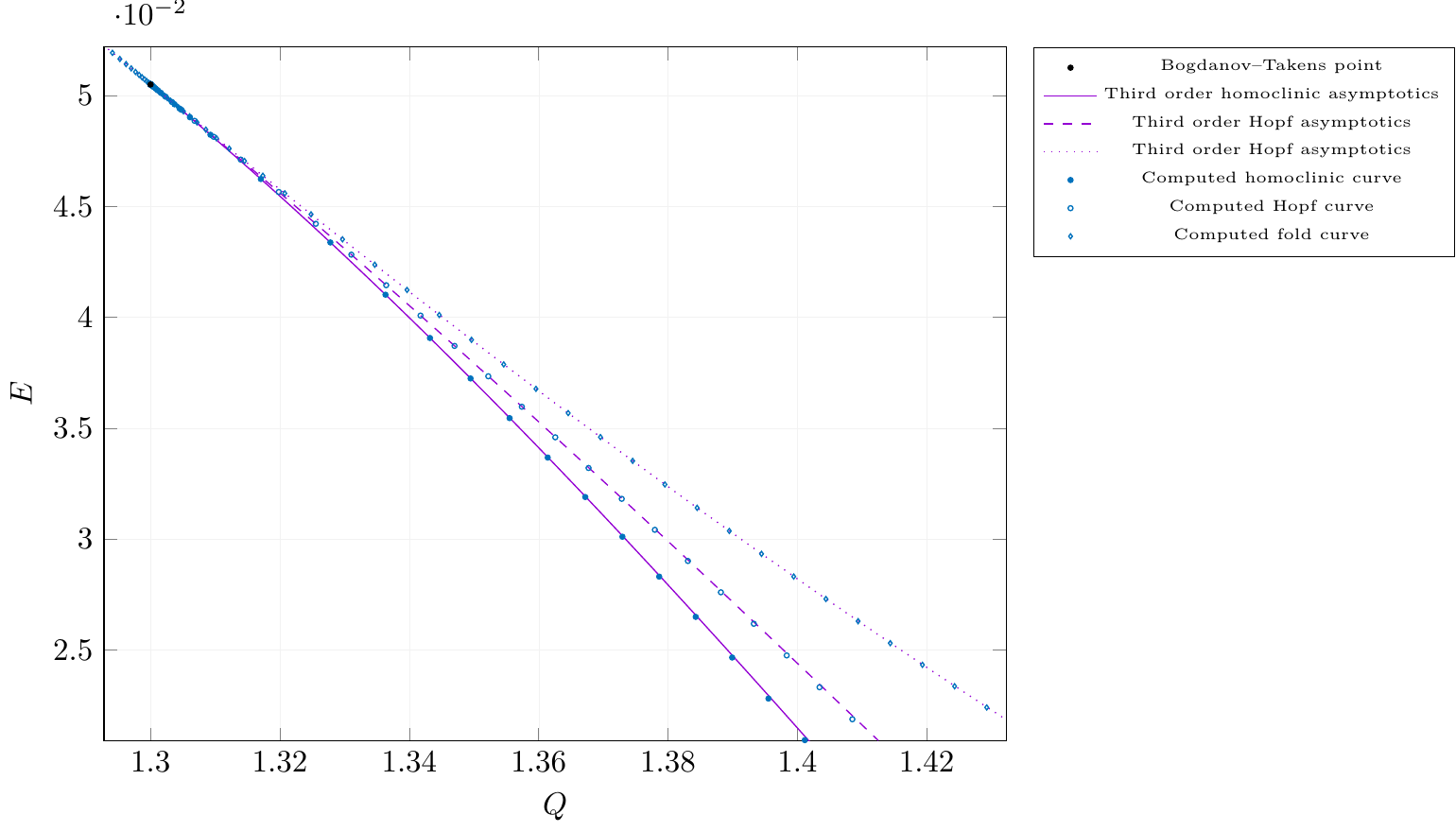}
    \caption{Bifurcation diagram near the derived generic Bogdanov-Takens point in
        \cref{btdde:eq:neural_network_subs} comparing computed codimension one curves using
        \DDEBIFTOOL with the third-order homoclinic parameter asymptotics obtained
        in \cref{btdde:sec:generic_bt_homoclinic_asymptotics}.}
    \label{btdde:fig:NeuralNetworkCompareParameters}
\end{figure}

In \cref{btdde:fig:NeuralNetworkCompareParameters} we compared the computed
codimension one curves emanating from a generic Bogdanov--Takens point in
\cref{btdde:eq:neural_network_subs} using our implementation in \DDEBIFTOOL with the
third-order homoclinic parameter asymptotics obtained in
\cref{btdde:sec:generic_bt_homoclinic_asymptotics}. We see that the predicted and
computed curves are almost indistinguishable. Furthermore, in
\cref{btdde:fig:convergencePlotsGeneric}, we see that using the third-order homoclinic
asymptotics over the first order is clearly superior.

\subsection{Transcritical Bogdanov--Takens bifurcation in the Van der Pol oscillator with delay feedback}
We consider the Van der Pol oscillator with delay feedback \cite{jiang2007bogdanov}
given by 
\begin{equation}
\label{btdde:eq:dde_vanderPol}
\ddot{x}(t) + \epsilon(x^2(t)-1)\dot{x}(t) + x(t) = \epsilon g(x(t-\tau)),
\end{equation}
where $\epsilon>0$ is a parameter, $\tau>0$ is a delay and $g:\mathbb{R}\rightarrow\mathbb{R}$
is a smooth function with $g(0) = 0$ and $g'(0)\neq0$. We rewrite
the Van der Pol equation \cref{btdde:eq:dde_vanderPol} as
\begin{equation}
\label{btdde:eq:vanderPolOscillator}
\begin{cases}
    \dot{x}_1 = x_2,\\
    \dot{x}_2 = \epsilon g(x_1(t-\tau))-\epsilon(x_1^2-1)x_2-x_1.
\end{cases}
\end{equation}
Rescaling time with $t\rightarrow\dfrac{t}{\tau}$ to normalize the
delay yields
\begin{equation}
\label{btdde:eq:vanderPolOscillatorRescaled}
\begin{cases}
\dot{x}_1 = \tau x_2,\\
\dot{x}_2 = \tau\left(\epsilon g(x_1(t-1))-\epsilon(x_1^2-1)x_2-x_1\right).
\end{cases}
\end{equation}
This allows us to treat $\tau$ as a bifurcation parameter.

Following \cite{jiang2007bogdanov}, we consider \cref{btdde:eq:dde_vanderPol} with
\[
g(x) = \frac{e^x-1}{c_1e^x + c_2},
\]
where $c_1 = \dfrac{1}{4}$ and $c_2 = \dfrac{1}{2}$. Then the trivial
equilibrium undergoes a transcritical Bogdanov--Takens bifurcation at parameter
values $(\epsilon,\tau) = (0.75,0.75)$, see \cite{jiang2007bogdanov} and the
supplement. 
\begin{figure}[ht]
    \centering
    \includetikzscaled{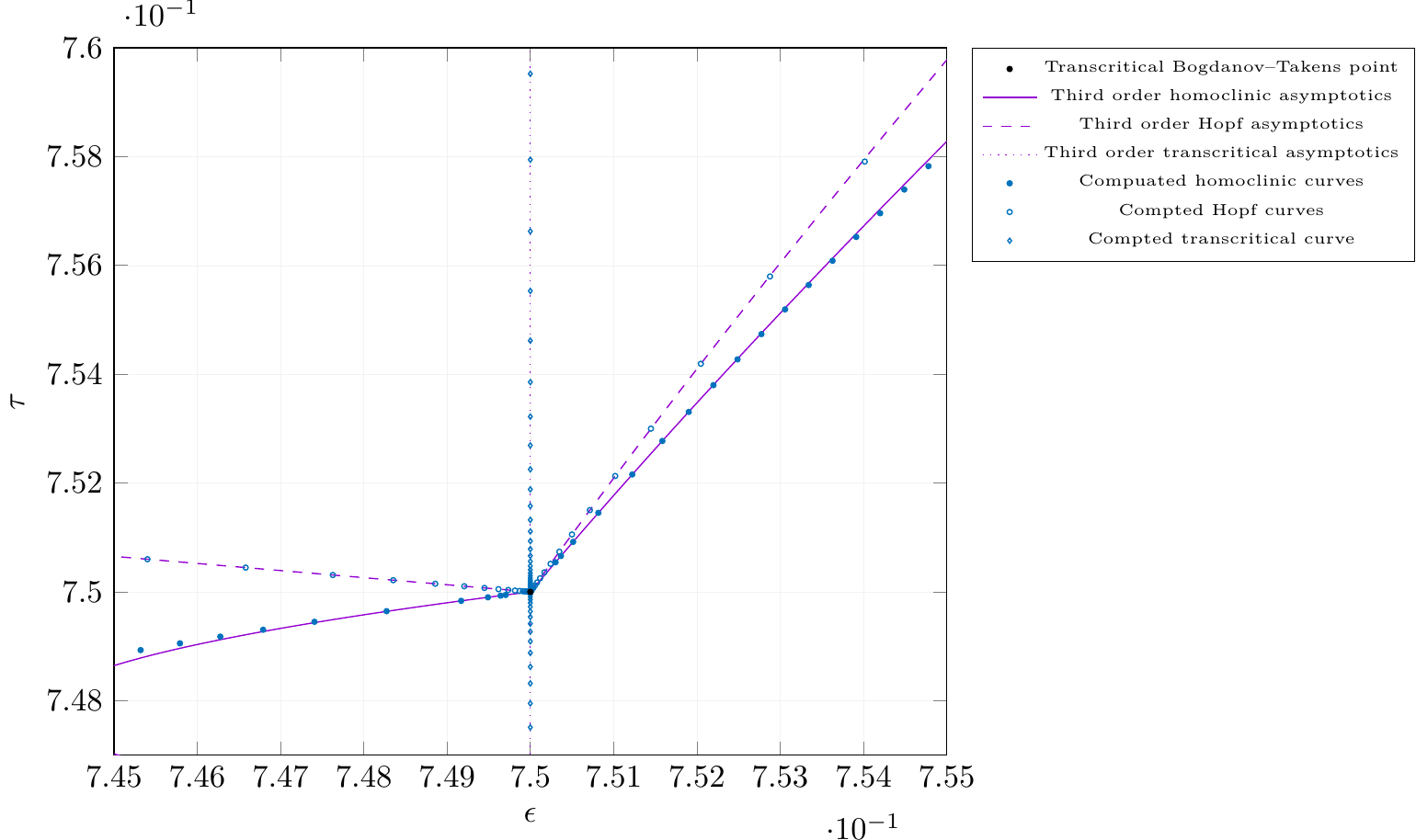}
    \caption{Bifurcation diagram near the derived generic Bogdanov-Takens point in
        \cref{btdde:eq:vanderPolOscillatorRescaled} comparing computed codimension one curves using
        \DDEBIFTOOL with the third-order homoclinic parameter asymptotics obtained
        in \cref{btdde:sec:transcritical_bt_homoclinic_asymptotics}.}
    \label{btdde:fig:vanderPolOscillatorCompareParameters}
\end{figure}

In \cref{btdde:fig:vanderPolOscillatorCompareParameters} we compared the computed
codimension one curves emanating from a transcritical Bogdanov--Takens point in
\cref{btdde:eq:vanderPolOscillatorRescaled} using our implementation in \DDEBIFTOOL
with the third-order homoclinic parameter asymptotics obtained in
\cref{btdde:sec:transcritical_bt_homoclinic_asymptotics}. We see that the predicted
and computed curves are almost indistinguishable. Furthermore, in
\cref{btdde:fig:convergencePlotsTranscritical} we see that using the third-order
homoclinic asymptotics over the first order is clearly superior.

\subsection{Transcritical Bogdanov--Takens bifurcation in a tri-neuron BAM neural network model}
\label{btdde:sec:Tri-neuron-BAM-neural}

We consider a three-component system of a tri-neuron bidirectional
associative memory (BAM) neural network model with multiple delays \cite{dong2013bogdanov}.
The architecture of this BAM model is illustrated in \cref{btdde:fig:BAM_architecture_graph}. 

\begin{figure}
\centering
\includetikzscaled[0.75]{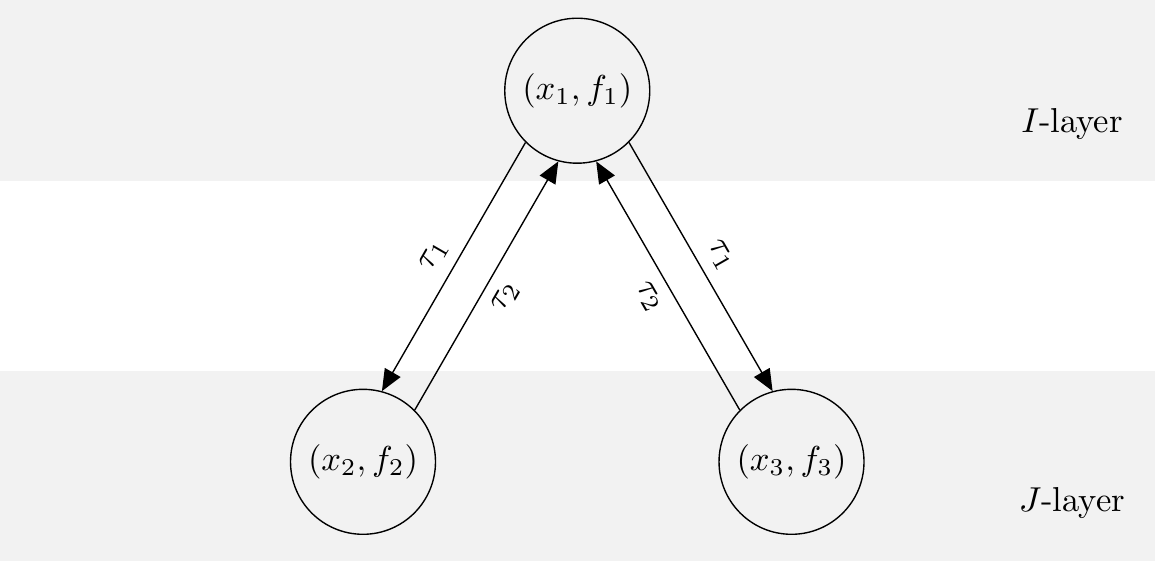}
\caption{The graph of architecture for model \cref{btdde:eq:tri_neuron_BAM}.}
\label{btdde:fig:BAM_architecture_graph}
\end{figure}

\begin{figure}[ht]
\centering
\includetikzscaled{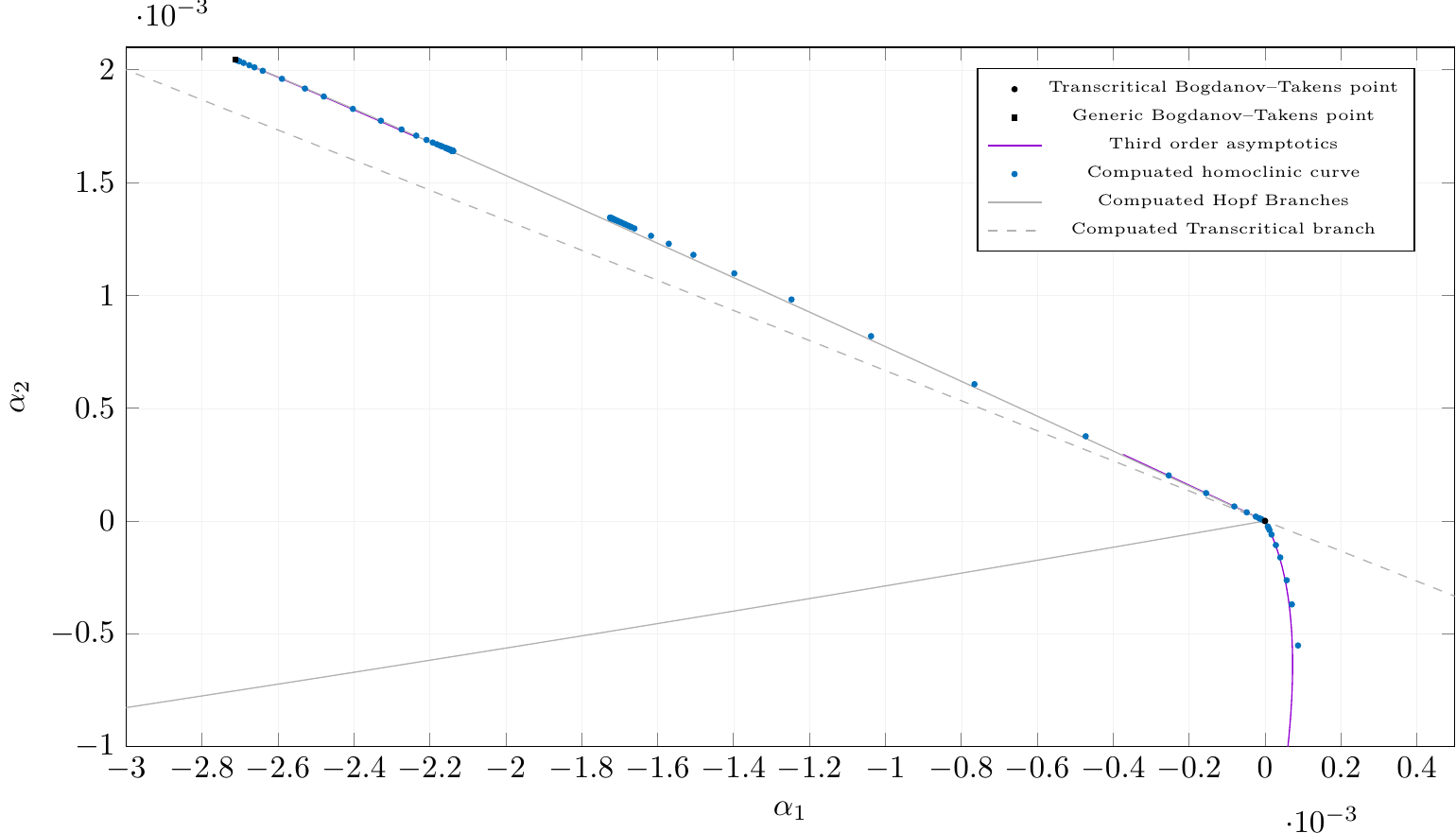}
\caption{
Bifurcation diagram near the transcritical Bogdanov--Takens bifurcation and
a generic Bogdanov--Takens in \cref{btdde:eq:tri_neuron_BAM-u} comparing computed
codimension one using \DDEBIFTOOL with the asymptotics obtained in this
\paper{}.}
\label{btdde:fig:triNeuronBAMNeuralNetworkModelCompareParameters}
\end{figure}

In this model, there is only one neuron with the activation function
$f_{1}$ on the $I$-layer and there are two neurons with respective
activation functions $f_{2}$ and $f_{3}$ on the $J$-layer. It is assumed 
that the time delay from the $I$-layer to the $J$-layer is $\tau_{1}$,
while the time delay from the $J$-layer to the $I$-layer is $\tau_{2}$.
Then the network can be described by the following delay differential equation
\begin{equation}
\label{btdde:eq:tri_neuron_BAM}
\begin{cases}
\dot{x}_{1}(t) & =-\mu_{1}x_{1}(t)+c_{21}f_{1}(x_{2}(t-\tau_{2}))+c_{31}f_{1}(x_{3}(t-\tau_{2})),\\
\dot{x}_{2}(t) & =-\mu_{2}x_{2}(t)+c_{12}f_{2}(x_{1}(t-\tau_{1})),\\
\dot{x}_{3}(t) & =-\mu_{3}x_{3}(t)+c_{13}f_{3}(x_{1}(t-\tau_{1})).
\end{cases}
\end{equation}
Here
\begin{itemize}
\item $x_{i}(t)\,(i=1,2,3)$ denote the state of the neuron at time $t$;
\item $\mu_{i}(i=1,2,3)$ describe the attenuation rate of internal neurons
processing on the $I$-layer and the $J$-layer and $\mu_{i}>0$;
\item the real constants $c_{i1}$and $c_{1i}\,(2,3)$ denote the neurons
in two layers: the $I$-layer and the $J$-layer.
\end{itemize}
Letting $u_{1}(t)=x_{1}(t-\tau_{1}),u_{2}(t)=x_{2}(t),u_{3}(t)=x_{3}(t)$
and $\tau=\tau_{1}+\tau_{2}$, then system \cref{btdde:eq:tri_neuron_BAM}
is equivalent to the following system

\begin{equation}
\label{btdde:eq:tri_neuron_BAM-u}
\begin{cases}
\dot{u}_{1}(t) & =-\mu_{1}u_{1}(t)+c_{21}f_{1}(u_{2}(t-\tau))+c_{31}f_{1}(u_{3}(t-\tau)),\\
\dot{u}_{2}(t) & =-\mu_{2}u_{2}(t)+c_{12}f_{2}(u_{1}(t)),\\
\dot{u}_{3}(t) & =-\mu_{3}u_{3}(t)+c_{13}f_{3}(u_{1}(t)).
\end{cases}
\end{equation}

In \cite{dong2013bogdanov} conditions for \cref{btdde:eq:tri_neuron_BAM-u} are derived
for the origin to have a double zero eigenvalue, while all other eigenvalues
have negative real parts. Since the provided expression did not give the correct results,
we corrected the given expression from \cite{dong2013bogdanov}.
We postpone the proof of the next two lemma's to the 
\ifthesis%
    \cref{chapter:BT_DDE_supplement}%
\fi%
\ifsiam%
    \hyperref[mysupplement]{online Supplement}%
\fi%
\ifarxiv%
    \hyperlink{mysupplement}{Supplement}%
\fi.

\begin{lemma}
\label{btdde:lem:BAM_double_eigenvalue}
Assume that $f_{i}(0)=0\,(i=1,2,3)$,
$f_{i}'(0)\neq0\,(i=1,2,3)$ and $\mu_{2}\neq\mu_{3}$, then the steady-state
$(u_{1},u_{2},u_{3})=(0,0,0)$ has a double zero eigenvalue at 
\begin{align*}
c_{21} & =c_{21}^{0}=\frac{\mu_{2}^{2}\left(\mu_{1}\left(\mu_{3}\tau+1\right)+\mu_{3}\right)}{c_{12}\left(\mu_{2}-\mu_{3}\right)f_{1}'(0)f_{2}'(0)},\\
c_{31} & =c_{31}^{0}=\frac{\mu_{3}^{2}\left(\mu_{1}\left(\mu_{2}\tau+1\right)+\mu_{2}\right)}{c_{13}\left(\mu_{3}-\mu_{2}\right)f_{1}'(0)f_{3}'(0)}.
\end{align*}
\end{lemma}

\begin{lemma}
\label{btdde:lemma:triNeuralBAMNetworkModelEigenvalues}
\textup{Correction to \cite[Lemma 3]{dong2013bogdanov}.}
Let $(c_{21},c_{31})=(c_{21}^{0},c_{31}^{0})$,
\begin{equation}
    \label{btdde:sm:eq:omega_0} 
    \omega_0 = \frac{\sqrt{-\mu_1^2 - \mu_2^2 - \mu_3^2 + \sqrt{\zeta_0}}}{\sqrt{2}}
\end{equation}
and $0<\tau<\tau_{0}$, where $\tau_0$ is the minimum positive solution to the nonlinear equation
\begin{equation}
    \label{btdde:sm:eq:tan} 
    \tan (\tau \omega_0) = \frac{b_0\zeta_1 - a_0\zeta_2}{a_0\zeta_1 + b_0\zeta_2},
\end{equation}
with
\begin{align*}
a_0 &= -\mu_1\mu_2\mu_3, \\ 
b_0 &= -\omega_0(\mu_2\mu_3 + \mu_1(\mu_2 + \mu_3 + \mu_2\mu_3\tau)), \\
\zeta_0 &= \mu_1^4 + (\mu_2^2 + \mu_3^2)^2 + 8\mu_1\mu_2\mu_3(\mu_2 + \mu_3 + \mu_2\mu_3\tau) + \\
        &\qquad 2\mu_1^2(\mu_3^2 + 4\mu_2\mu_3(1 + \mu_3\tau) + \mu_2^2(1 + 2\mu_3\tau(2 + \mu_3\tau))), \\
\zeta_1 &= \mu_1\mu_2\mu_3 - (\mu_1 + \mu_2 + \mu_3)\omega_0^2, \\
\zeta_2 &= \mu_2\mu_3\omega_0 + \mu_1(\mu_2 + \mu_3)\omega_0 - \omega_0^3.
\end{align*}
Then the center manifold near the Bogdanov--Takens point is locally attractive.
\end{lemma}

For the numerical verification we consider, as in the simulations in
\cite[Example 1]{dong2013bogdanov}, the system \cref{btdde:eq:tri_neuron_BAM-u} with
the activation functions
\[
f_{1}(x)=\tanh(x)+0.1x^{2},\quad f_{2}(x)=f_{3}(x)=\tanh(x),
\]
and parameter values
\[
\mu_{1}=0.1,\mu_{2}=0.3,\mu_{3}=0.2,c_{12}=c_{13}=1,\tau=5.
\]
Then, from \cref{btdde:lem:BAM_double_eigenvalue}, we obtain two critical
values 
\[
(c_{21}^{0},c_{31}^{0})=(0.36,-0.22),
\]
at which there is a transcritical Bogdanov--Takens point. Furthermore, since
$\tau < \tau_0 \approx 5.4320$ the center manifold is attractive. In fact, we
will show in the
\ifthesis%
    \cref{chapter:BT_DDE_supplement} %
\fi%
\ifsiam%
    \hyperref[mysupplement]{online Supplement} %
\fi%
\ifarxiv%
    \hyperlink{mysupplement}{Supplement} %
\fi
that the center manifold is attractive for
$0<\tau<13.2309348879375$. We write the system \cref{btdde:eq:tri_neuron_BAM-u}
as
\begin{equation}
\label{btdde:eq:tri_neuron_BAM-u-1}
\begin{cases}
\dot{u}_{1}(t) & =-\mu_{1}u_{1}(t)+\left(c_{21}^{0}+\alpha_{1}\right)f_{1}(u_{2}(t-\tau))+\left(c_{31}^{0}+\alpha_{2}\right)f_{1}(u_{3}(t-\tau)),\\
\dot{u}_{2}(t) & =-\mu_{2}u_{2}(t)+c_{12}f_{2}(u_{1}(t)),\\
\dot{u}_{3}(t) & =-\mu_{3}u_{3}(t)+c_{13}f_{3}(u_{1}(t)),
\end{cases}
\end{equation}
where $(\alpha_{1},\alpha_{2})$ are the new parameter values such
that at $(\alpha_{1},\alpha_{2})=(0,0)$ we have a Bogdanov-Takens
bifurcation. The critical normal form coefficients 
\[
(a,b)\approx(0.0012,-0.0135),
\]
indicate stable cycles. In
\cref{btdde:fig:triNeuronBAMNeuralNetworkModelCompareParameters}, we have plotted the
local unfolding of the transcritical Bogdanov-Takens bifurcation using our
implementation in DDE-BifTool to start the continuation of the transcritical,
Hopf, and homoclinic codimension one bifurcation curves. Furthermore, using the
ability to detect codimension two bifurcation points while continuing Hopf
points, we discovered another Bogdanov--Takens point. From the transversality
conditions, we see the secondary Bogdanov--Takens point is of the generic case.
Thus, we can start continuation of the homoclinic orbits emanating from this
point as well. We see that the numerically continued homoclinic curves in the
upper half plane only exists in a very small parameter region. Without our
predictor, this homoclinic curve would be extremely difficult to locate. Lastly,
in \cref{btdde:fig:convergencePlotsTranscritical} the first and third-order
homoclinic asymptotics are compared. Here we again see expected improvements of using the
third-order homoclinic asymptotics. We refer to the
\ifthesis%
    \cref{chapter:BT_DDE_supplement} %
\fi%
\ifsiam%
    \hyperref[mysupplement]{online Supplement} %
\fi%
\ifarxiv%
    \hyperlink{mysupplement}{Supplement} %
\fi to see the beautiful
homoclinic orbits along the homoclinic curve in the lower half plane, and for
an overall more detailed treatment of this example. In fact, we will show that
the transcritical Bogdanov--Takens and generic Bogdanov--Takens points are
connected, not only by a Hopf curve, but also through a homoclinic curve. 

\section{Concluding remarks}
We have provided explicit formulas needed to initialize the codimension one
equilibrium and homoclinic bifurcation curves emanating from the generic and transcritical codimension two Bogdanov--Takens bifurcation points in classical
DDEs. Applications to four different models from the literature are given, confirming the
correctness of the derivation of the time-reparametrization
parameter-dependent
center manifold transformation and the codimension one asymptotics.

By extending the normalization technique to include the time-reparametrization we
are allowed to use orbital normal forms, instead of only smooth normal forms.
One benefit of this approach is the reusability of the codimension one curves
emanating from the universal unfolding of the Bogdanov--Takens codimension two
bifurcation. Indeed, by a simple transformation \cref{btdde:eq:blowup}, we obtain the
homoclinic asymptotics for the transcritical Bogdanov--Takens bifurcation.

In this \paper{}, we have restricted to the class of classical DDEs, and for the
applications to the class of discrete DDEs. However, the proof in
\cite{Switching2019} of the existence of a smooth parameter-dependent center
manifold is given in the general context of perturbation theory for dual
semigroups (sun-star calculus). Therefore, the applicability of this result
extends beyond classical DDEs. For example, in \cite{VanGils2013} and
\cite{Dijkstra2015} the technique was used to calculate the critical normal
form coefficients for Hopf and Hopf-Hopf bifurcations occurring in neural field
models with propagation delays. For these models, sun-reflexivity is lost, which
is typical for delay equations in abstract spaces or with infinite delay.
However, it is often possible to overcome this functional analytic
complication, so dual perturbation theory can still be employed successfully
\cite{Diekmann2008,Diekmann2012blending,VanGils2013,Janssens2019}. It follows
that the derived coefficients in
\cref{btdde:sec:parameter-dependent-center-manifold-reduction} are valid in these
settings as well.

Similarly, in \cite{Sieber@2017} it is demonstrated that formally the
normalization method still works for state-dependent delay differential
equations. However, to employ our formulas in this situation, one first needs
to implement the continuation of homoclinic orbits for state-dependent DDEs. Of
course, the asymptotics are still useful to see where     the homoclinic orbit
should be located. Furthermore, the asymptotics can be used to numerically
approximate the homoclinic solutions by periodic orbits with a large period.

Returning to the setting of classical DDEs, the most obvious next challenge is
to derive normal forms for bifurcations of periodic orbits by generalizing
\cite{Kuznetsov2005,DeWitte2013,DeWitte2014}. The resulting formulas can then
be implemented in \DDEBIFTOOL to facilitate numerical bifurcation analysis of
periodic orbits in supported types of classical DDEs.


\ifarxiv
\clearpage
\hypertarget{mysupplement}{}
\begin{mytitle}
    \title{\textsc{supplementary materials for:}\\ \TheTitle}
    \pdfbookmark[0]{Supplement}{supplement}
    \maketitle
\end{mytitle}
\thispagestyle{plain}

\ResetCounters

In this supplement, we will provide a full walk-through of the examples given
in \cref{btdde:sec:Examples} with the open source bifurcation software package
\DDEBIFTOOL\footnote{\url{http://ddebiftool.sourceforge.net/}}
\cite{2014arXiv1406.7144S}. Additionally, the Julia code, used for the
numerical simulation, with the package {\tt
DifferentialEquations.jl}\footnote{\url{https://github.com/SciML/DifferentialEquations.jl}}
\cite{rackauckas2017differentialequations} is shown. This allows other researchers to replicate the findings in the main text fully. The given code can easily be modified to study other DDE models undergoing generic or transcritical Bogdanov-Takens bifurcations. While studying these models, some new results were
obtained.

The code in this supplement has also been included into the \DDEBIFTOOL package on the source-forge repository and can be executed without the need to copy and paste. In this supplement, we mainly focus on the initialization and continuation of the various codimension one equilibrium and homoclinic bifurcation curves emanating from the Bogdanov--Takens points and on numerical simulation near the bifurcation points. The online tutorials, as well as the manual and its references, provide a comprehensive overview of \DDEBIFTOOL's capabilities and functionality.

Note that reading this supplement may feel, at times, somewhat repetitive. We see
this as a positive sign. Indeed, our predictors need little to none adjustment
before starting continuation of the codimension one curves emanating from the
generic and transcritical Bogdanov--Takens bifurcation.

To follow the examples below, we recommend using the \DDEBIFTOOL package
supplied with the article. Also, note that the \MATLAB code has been tested on 
\MATLAB 2020b and \OCTAVE 6.4.0 using \OCTAVE symbolic
package version 2.9.0. Different results may occur on other versions of 
\MATLAB or \OCTAVE.

\section[Predator-prey system with double Allee effect]
        {Generic Bogdanov--Takens bifurcation in a predator-prey system with double Allee effect}
In \cite{Jiao2021} the following predator-prey model with double Allee effect
and delay is considered
\begin{equation}
\label{sm:eq:double_alle_effect}
\begin{cases}
    \dot x(t) = \dfrac{rx}{x+n_0}\left(1-\dfrac1 K\right)\left(x - m_0\right) - \dfrac{cxy}{x+\varrho y},\\
    \dot y(t) = -dy + \dfrac{c_1 x(t-\tau)y}{x(t-\tau) + \varrho y(t-\tau)}.
\end{cases}
\end{equation}
Here, the time delay $\tau \geq 0$ is introduced due to the fact that the
reproduction of predator after consuming the prey is not instantaneous, but is
mediated by some time lag required for gestation.
The variables and parameters occurring in \cref{sm:eq:double_alle_effect}
have the following meaning:
\begin{itemize}
\item $x,y\colon \mathbb R \rightarrow \mathbb R$ denote the prey and predator population densities, respectively.
\item $r$ denotes the maximum prey population growth in absence of the Allee effect.
\item $K$ is the carrying capacity of the environment.
\item $m_0$ is the Allee threshold.
\item $n_0$ is the auxiliary parameter in order to quantify the strength of the Allee effect.
\item $c$ denotes the capturing rate of the predator.
\item $\varrho$ is the half-capturing saturation constant.
\item $c_1$ is the conversion rate of prey into predators biomass.
\item $d$ is the per capita predator mortality rate.
\end{itemize}
Following \cite{Jiao2021} let $(x,y,t) = \left(K\bar x, \frac K \varrho \bar y, \frac{\bar t}r\right)$, and immediately 
dropping the bars again for readability, then \cref{sm:eq:double_alle_effect} becomes
\begin{equation}
\label{sm:eq:double_alle_effect_rescaled}
\begin{cases}
    \dot x(t) = x \left( \dfrac{(1-x)(x-\gamma)}{x+\vartheta} - \dfrac{\alpha y}{x+y} \right), \\
    \dot y(t) = \delta y \left( -1 + \dfrac{ m x(t-\tau) }{ x(t-\tau) + y(t-\tau) }\right),
\end{cases}
\end{equation}
where $\vartheta = \frac{n_0}K, \alpha=\frac c{r\varrho}, \gamma = \frac{m_0}K, \delta = \frac dr$ and $m=\frac{c_1}d$.

Thus, $(x,y,\vartheta,\delta)=(x_0,y_0,\vartheta_0,\delta_0)$, with $\delta_0 =
\frac\alpha m$, we have a double zero eigenvalue, provided that
$\det\Delta''(0) \neq 0$. To confirm their analytical findings in \cite{Jiao2021}
numerically, the parameters $\gamma=0.15,\alpha=0.9$ and $m=1.50298303$ are
fixed and $(\vartheta,\delta)$ are taken as unfolding parameters. For these
parameters, we indeed have that $\det\Delta''(0) \neq 0$.

\begin{remark}
    The \MATLAB files for this demonstration can be found in the directory
    \mintinline[breaklines,breakafter=/]{MATLAB}{demos/tutorial/VII/neural_network_model} relative to the main
    directory of the \DDEBIFTOOL package.
\end{remark}

\subsection{Generate system files}
Before we start to analyze the system with \DDEBIFTOOL, we first create a
\emph{system file}. This file contains the definition of the system
\cref{sm:eq:double_alle_effect_rescaled}, the standard derivatives needed for
calculation of the eigenvalues and eigenvectors, the continuation of
bifurcation points, cycles, and also the multilinear forms, see \cite[Section
6]{Switching2019}, used for the calculation of the coefficients of the critical
and parameter-dependent normal forms and center manifold transformation.
Alternatively, one can only supply the system itself, see
\cref{sm:lst:wo_system_file}. Then finite difference is used to approximate the
derivatives. However, this is less efficient and less accurate, and therefore not
recommended. A separate script \mintinline{MATLAB}{gen_sym_predator_prey.m} is used to
create a system file. The most important parts of this script are listed and
discussed below.

\newcommand\pathToDDEBifToolDemos{./ddebiftooldemofiles/VII}
\inputminted[firstline=18, lastline=50]{MATLAB}{\pathToDDEBifToolDemos/predator_prey/gen_sym_predator_prey.m}
The variable \mintinline{MATLAB}{ddebiftoolpath} is directed to the \DDEBIFTOOL main
folder, which should have been extracted somewhere on the computer. Here, a path
relative to the current working directory is used. Note that although we only
use the parameters $(\theta,\delta)$ as unfolding parameters, in the current
version of \DDEBIFTOOL, we also need to include the delay(s) in the list of
parameters. After running the script, the function \mintinline{MATLAB}{dde_sym2funcs}
creates two system files \mintinline{MATLAB}{sym_predator_prey_mf.m} and \mintinline{MATLAB}{sym_predator_prey.m}.
The first file \mintinline{MATLAB}{sym_predator_prey_mf.m} implements the higher order derivatives
as multilinear forms, as explained in \cite[Section 6]{Switching2019}, and therefore will be the only file used.
The second file \mintinline{MATLAB}{sym_predator_prey.m}
uses directional derivatives to implement the higher order derivatives. The
directional derivatives approach \emph{formally} allows the use of
state-dependent delays, see \cite{Sieber@2017}. Although both approaches yield
(up to rounding errors) identical normal form coefficients and center manifold
transformations, multilinear forms are more efficient to compute.

\subsection{Loading the \DDEBIFTOOL package}
\label{sm:sec:loading_DDE-BIFTool}
Now that a system file is created, we continue with \DDEBIFTOOL to analyze
\cref{sm:eq:double_alle_effect_rescaled} numerically. The code in the following
sections highlight the important parts of the file \mintinline{MATLAB}{predator_prey.m}.
The package \DDEBIFTOOL consists of a set of \MATLAB routines. Thus, in order to start
using \DDEBIFTOOL, we only need to add \DDEBIFTOOL (sub)directories to the search
path.
\begin{listing}[!ht]
\inputminted[firstline=17, lastline=25]{MATLAB}{\pathToDDEBifToolDemos/predator_prey/predator_prey.m}
\caption{Code to add \DDEBIFTOOL scripts to the search path.}
\label{sm:lst:searchpath}
\end{listing}
There are four subdirectories added to the search path:
\par
\medskip
\begin{description}
\item[ddebiftool] Containing the core files of \DDEBIFTOOL.
\item[ddebiftool\_extra\_psol] An extension for enabling continuation of periodic orbit bifurcations for delay-differential equations with constant or state-dependent delay.
\item[ddebiftool\_extra\_nmfm] An extension for normal form computation.
\item[ddebiftool\_utilities] Containing various utilities.
\end{description}

\subsection{Set parameter names}
The following code allows us to use
\mintinline{MATLAB}{ind.theta} instead of remembering the
index of the parameter $\beta$ in the parameter array, and similarly for the
other parameters.
\inputminted[firstline=27, lastline=30]{MATLAB}{\pathToDDEBifToolDemos/predator_prey/predator_prey.m}
In this way, fewer mistakes are likely to be made, and the code is easier to read.

\subsection{Initialization}
Next, we set up the \mintinline{MATLAB}{funcs} structure, containing information about
where the system and its derivatives are stored, a function pointing to which
parameters are delays, and various other settings.
\inputminted[firstline=32, lastline=35]{MATLAB}{\pathToDDEBifToolDemos/predator_prey/predator_prey.m}
Alternatively, when no system files have been generated, one could initialize
the system \cref{sm:eq:double_alle_effect_rescaled} as in \cref{sm:lst:wo_system_file}.

\begin{listing}[!ht]
\begin{minted}{MATLAB}
%% Define funcs structure without symbolic derivatives
m = 1.502983803;
alpha = 0.9;
gamma = 0.15;
dx_dt = @(x,y,theta) x.*((1-x).*(x-gamma)./(x+theta) ...
                     - alpha.*y./(x+y))
dy_dt = @(y,xt,yt,delta) delta.*y.*(-1 + m.*xt./(xt+yt))
predator_prey_sys = @(xx,par)  ...
   [dx_dt(xx(1,1,:),xx(2,1,:),par(1,ind.theta,:)); ...
    dy_dt(xx(2,1,:),xx(1,2,:),xx(2,2,:),par(1,ind.delta,:))];
% % Set funcs structure
funcs = set_funcs('sys_rhs', predator_prey_sys, ...
    'sys_tau', @() ind.tau,...
    'x_vectorized', true, 'p_vectorized', true);
\end{minted}
\caption{Code to define the system without a system file.}
\label{sm:lst:wo_system_file}
\end{listing}
Inspecting the output of the \mintinline{MATLAB}{funcs} handle gives.
\begin{minted}{shell-session}
>> funcs

funcs =

  struct with fields:

                 sys_rhs: @(x,p)dde_wrap_rhs(x,p,funcs.sys_rhs,funcs.x_vectorized,
                            funcs.p_vectorized)
                sys_ntau: @()0
                 sys_tau: @()ind.tau
                sys_cond: @dde_dummy_cond
                sys_deri: @(x,p,nx,np,v)dde_gen_deriv(funcs.sys_dirderi,x,p,nx,np,v,1)
                sys_dtau: []
              sys_mfderi: {}
             sys_dirderi: {@(x,p,dx,dp)dde_dirderiv(derivbase,x,p,dx,dp,
                            order-ldirderi, 'nf',size(x,1),'hjac',
                            funcs.hjac(order))  [function_handle]}
             sys_dirdtau: []
            x_vectorized: 1
            p_vectorized: 1
                    hjac: @(ord)eps^(1/(2+ord))
      sys_cond_reference: 0
              lhs_matrix: @(sz)lhs_matrix(sz,funcs.lhs_matrix)
                  tp_del: 0
       sys_deri_provided: 0
    sys_dirderi_provided: 0
\end{minted}
The output shows that no derivative file is supplied. In this case, the
derivatives are calculated using finite-difference approximations with the
function \mintinline{MATLAB}{dde_dirderiv}. Again, we do not recommend using the latter
approach. However, it can be useful for debugging purposes.

\subsection{Set parameter range}
Since we are only interested here in the local unfolding, we restrict the
allowed parameter range for the unfolding parameters. In practice, one may have
physical restrictions which must be satisfied. Additionally, we also limit the
maximum allowed step size during continuation. By doing so, we obtain more refined
data to compare against our predictors.
\inputminted[firstline=37, lastline=40]{MATLAB}{\pathToDDEBifToolDemos/predator_prey/predator_prey.m}

\subsection{Stability and coefficients of the generic Bogdanov--Takens point}
We manually construct a steady-state at the Bogdanov--Takens point derived in
\cref{btdde:sec:example:predator_prey}, see also \cite{Jiao2021}.
\inputminted[firstline=42, lastline=55]{MATLAB}{\pathToDDEBifToolDemos/predator_prey/predator_prey.m}
Inspecting the \mintinline{MATLAB}{stst.stability} structure yields
\begin{minted}{shell-session}
>> stst.stability.l1

ans =

   1.0e-06 *

   0.197590320692559
  -0.197590756777676

>> 
\end{minted}

The eigenvalues confirm that the point under consideration is indeed (an
approximation to) a Bogdanov--Takens point. Next, we convert the steady-state
point to a Bogdanov--Takens point and calculate the normal form coefficients
with the function \mintinline{MATLAB}{nmfm_bt_orbital}, which implements the coefficients
derived in \cref{btdde:sec:generic_bogdanov-takens}. For this, we need to set the argument
\mintinline{MATLAB}{free_pars} to the unfolding parameter $(\theta,\delta)$. These
coefficients will be used to start the continuation of the codimension one branches
emanating from the Bogdanov--Takens point.
\inputminted[firstline=57, lastline=60]{MATLAB}{\pathToDDEBifToolDemos/predator_prey/predator_prey.m}
The coefficients for the normal form, the time-reparametrization, and the center
manifold transformation coefficients are stored in the \mintinline{MATLAB}{bt.nmfm}
structure. Additionally, also the approximation to the parameters and center
manifold transformation is given, which is used for the predictors.
\begin{minted}{shell-session}
>> bt.nmfm

ans =

  struct with fields:

          a: -0.145177185481861
          b: -1.446000370122628
  theta1000: -0.200700969579673
  theta0001: 5.128950206684711
        K10: [2x1 double]
        K01: [2x1 double]
        K02: [2x1 double]
        K11: [2x1 double]
        K03: [2x1 double]
       phi0: [1x1 struct]
       phi1: [1x1 struct]
      h0010: [1x1 struct]
      h0001: [1x1 struct]
      h2000: [1x1 struct]
      h1100: [1x1 struct]
      h0200: [1x1 struct]
      h1010: [1x1 struct]
      h1001: [1x1 struct]
      h0110: [1x1 struct]
      h0101: [1x1 struct]
      h0002: [1x1 struct]
      h0011: [1x1 struct]
      h3000: [1x1 struct]
      h2100: [1x1 struct]
      h1101: [1x1 struct]
      h2001: [1x1 struct]
      h0003: [1x1 struct]
      h1002: [1x1 struct]
      h0102: [1x1 struct]
          K: @(beta1,beta2)K10*beta1+K01*beta2+1/2*K02*beta2.^2
                +K11*beta1.*beta2+1/6*K03*beta2^3
          H: [function_handle]
\end{minted}

Since the sign of $ab$ is positive, we expect to find unstable periodic orbits nearby the 
Bogdanov--Takens point.

\subsection{Comparing profiles of computed and predicted homoclinic orbits}
To test the homoclinic asymptotics from
\cref{btdde:sec:generic_bt_homoclinic_asymptotics} we compare the first and third
order asymptotics to the Newton corrected solution. For this, we use the 
function \mintinline[breaklines,breakafter=_]{MATLAB}{C1branch_from_C2point}. This function returns a branch, which
by default returns two initial corrected approximations in order to start continuation of the
codimension one curve under consideration. By setting the argument
\mintinline{MATLAB}{'predictor'} to \mintinline{MATLAB}{true} the approximations are left uncorrected.
To make the comparison visually clear, we set the perturbation parameter 
$\epsilon=0.3$ (\mintinline{MATLAB}{step = 0.3} in the code below).
The code below produces \cref{sm:fig:DoubleAlleeEffectCompareProfiles}.
The difference between the two approximations is clearly noticeable. While
the first order asymptotics is close to the Newton corrected solution, the third
order asymptotic is indistinguishable at this scale from the Newton corrected
solution.
\inputminted[firstline=62, lastline=80]{MATLAB}{\pathToDDEBifToolDemos/predator_prey/predator_prey.m}
\begin{figure}[ht!]
    \centering
    \includegraphics{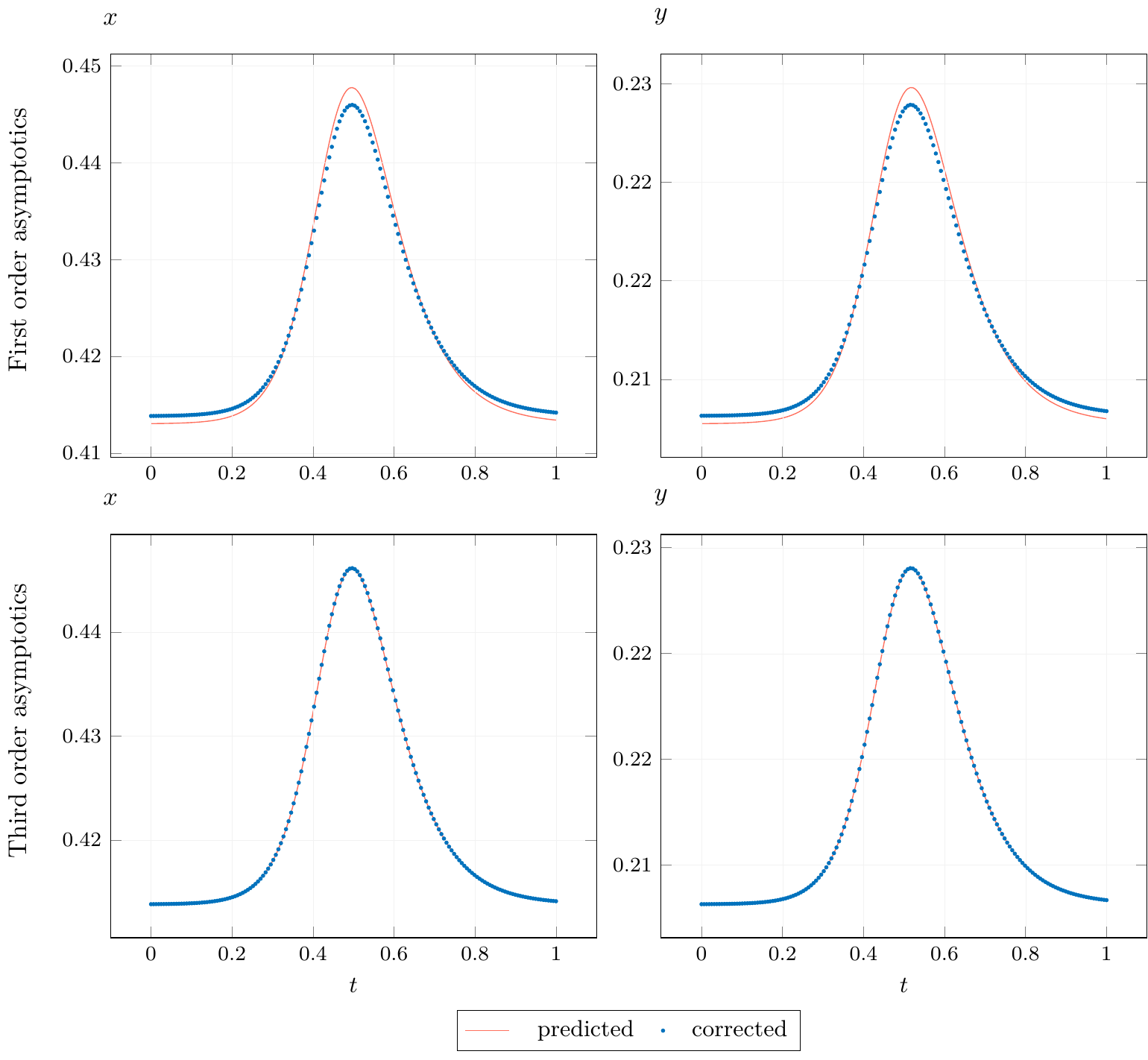}
    \caption{Comparison between the first and third-order asymptotics from
    \cref{btdde:sec:generic_bt_homoclinic_asymptotics} near the generic
        Bogdanov--Takens bifurcation in \cref{sm:eq:double_alle_effect_rescaled} with the
        perturbation parameter set to $\epsilon=0.3$.}
    \label{sm:fig:DoubleAlleeEffectCompareProfiles}
\end{figure}

\subsection{Continuation of the emanating codimension one curves}
To continue the three codimension one curves emanating from the generic
Bogdanov--Takens point, we can simply use the function
\mintinline[breaklines,breakafter=_]{MATLAB}{C1branch_from_C2point}, as shown in the code below. To monitor the
continuation process, the argument \mintinline{MATLAB}{plot} must be set to \mintinline{MATLAB}{1}.
The most important setting is the perturbation parameter (or multiple),
\mintinline{MATLAB}{step} in the code below. If left out, default step sizes are defined.
However, depending on the problem, no convergence may then be obtained.
\inputminted[firstline=82, lastline=100]{MATLAB}{\pathToDDEBifToolDemos/predator_prey/predator_prey.m}

\subsection{Predictors of the codimension one curves emanating from the Bogdanov--Takens point}
Before we provide the bifurcation diagram in the next section, we first obtain
the predictors for the codimension one curves. For this, we again use the
function \mintinline[breaklines,breakafter=_]{MATLAB}{C1branch_from_C2point}. We set the argument
\mintinline{MATLAB}{predictor} to \mintinline{MATLAB}{1} and provide a range of
perturbation parameters.
\inputminted[firstline=118, lastline=137]{MATLAB}{\pathToDDEBifToolDemos/predator_prey/predator_prey.m}
In the last part of the code above, we added the asymptotics obtained from \cite{Jiao2021}.

\subsection{Bifurcation diagram}
The code below produces a similar figure as \cref{sm:fig:DoubleAlleeEffectCompareParameters} in \MATLAB.
\inputminted[firstline=139, lastline=164]{MATLAB}{\pathToDDEBifToolDemos/predator_prey/predator_prey.m}
\begin{figure}[ht]
    \centering
    \includegraphics{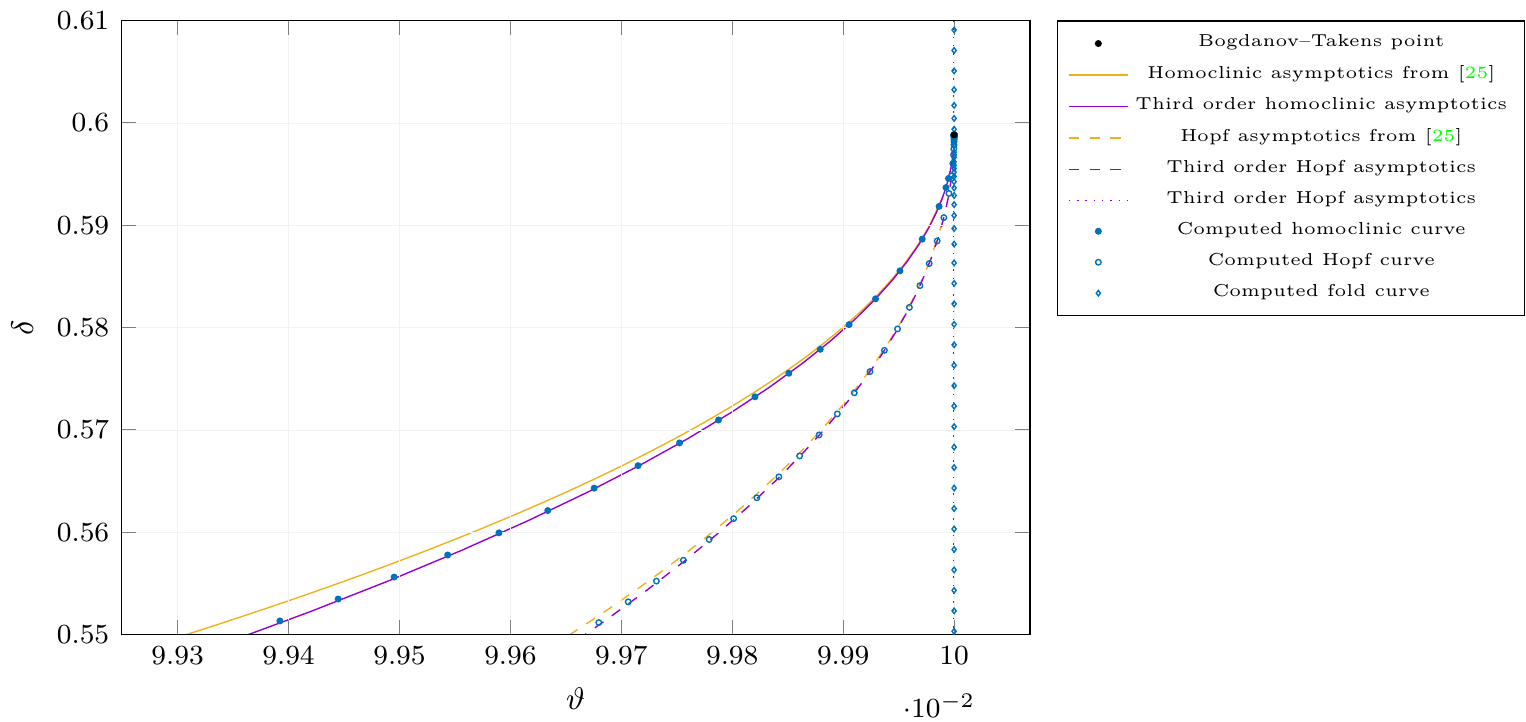}
    \caption{Bifurcation diagram near the analytically derived generic Bogdanov-Takens point in
        \cref{sm:eq:double_alle_effect_rescaled} comparing computed codimension one
    curves using \DDEBIFTOOL with the asymptotics obtained in this \paper{} and in \cite{Jiao2021}.}
    \label{sm:fig:DoubleAlleeEffectCompareParameters}
\end{figure}

\subsection{Compare homoclinic solutions in phase-space}
To get an impression of the third-order homoclinic asymptotics in
phase-space, we compare the corrected and uncorrected homoclinic solutions
with the perturbation parameter ranging from $0.1$ to $0.3$.
The code below results in \cref{sm:fig:DoubleAlleeEffectCompareOrbitsPhaseSpace}.
We see that the corrected and predicted homoclinic orbits are nearly identical.
\inputminted[firstline=166, lastline=182]{MATLAB}{\pathToDDEBifToolDemos/predator_prey/predator_prey.m}
The \MATLAB console shows the following output.
\begin{minted}{shell-session}
hcli from BT: branch 1 of  1 correction of point 1, success=1
hcli from BT: branch 1 of  1 correction of point 2, success=1
hcli from BT: branch 1 of  1 correction of point 3, success=1
hcli from BT: branch 1 of  1 correction of point 4, success=1
hcli from BT: branch 1 of  1 correction of point 5, success=1
hcli from BT: branch 1 of  1 correction of point 6, success=1
hcli from BT: branch 1 of  1 correction of point 7, success=1
hcli from BT: branch 1 of  1 correction of point 8, success=1
hcli from BT: branch 1 of  1 correction of point 9, success=1
hcli from BT: branch 1 of  1 correction of point 10, success=1
\end{minted}
That is, all predictions in this range are successfully corrected.
\begin{figure}[ht]
    \centering
    \includetikzscaled{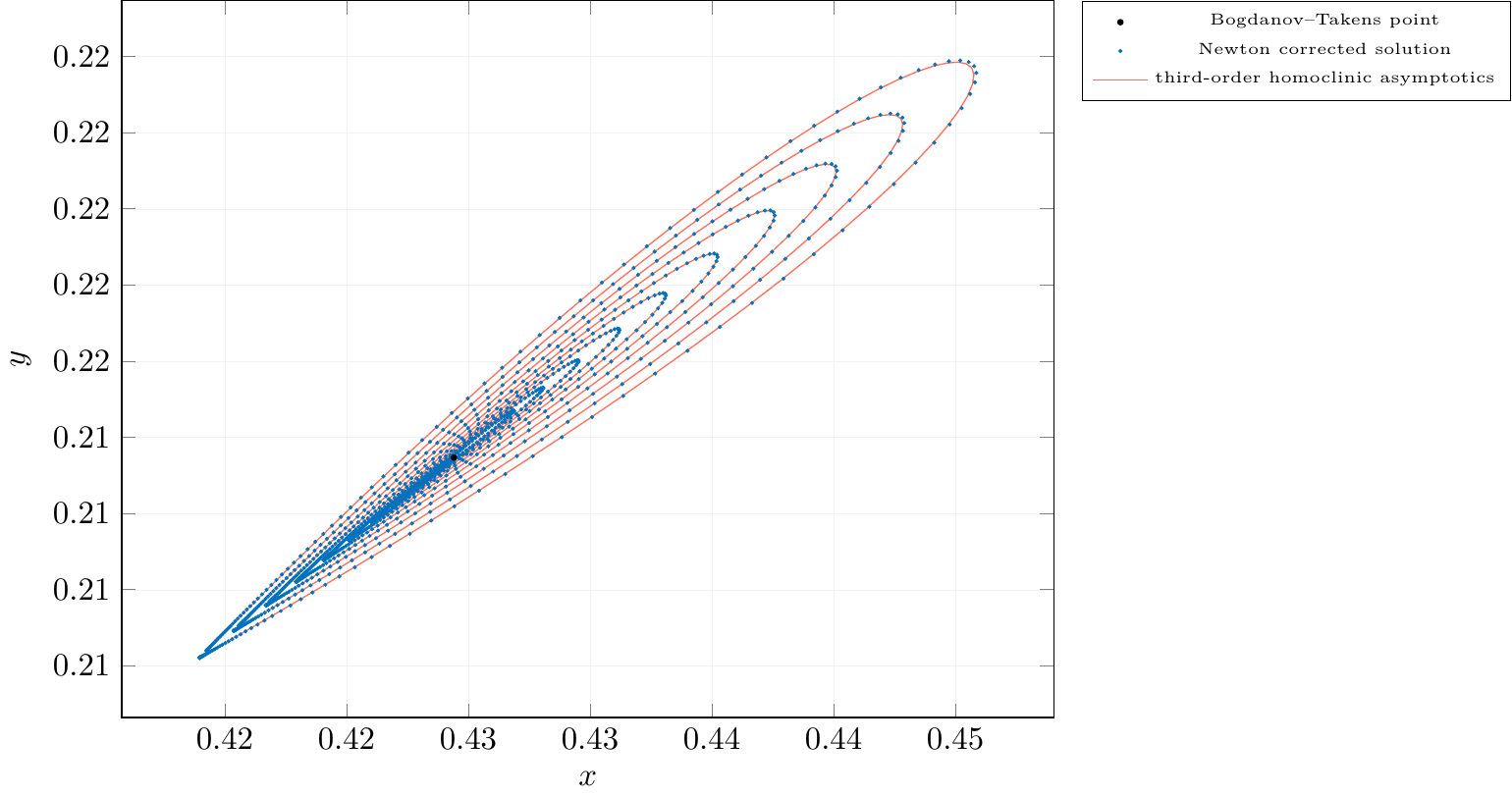}
    \caption{Plot comparing the third-order homoclinic asymptotics from
    \cref{btdde:sec:generic_bt_homoclinic_asymptotics} near the generic
    Bogdanov--Takens bifurcation in \cref{sm:eq:double_alle_effect_rescaled} with
    the Newton correct homoclinic solutions in $(x,y)$ phase-space.}
    \label{sm:fig:DoubleAlleeEffectCompareOrbitsPhaseSpace}
\end{figure}

\subsection{Convergence plot}
In \cref{sm:fig:DoubleAlleeEffectCompareProfiles} we compared the profiles of
the first and third-order homoclinic asymptotics. Although the improvement of
the third-order asymptotics is clearly visible, a better way to numerically
compare the different orders is by creating a log-log convergence plot. 
Since we create convergence plots for all examples treated in this supplement,
we created the function \mintinline{MATLAB}{convergence_plot}.
\begin{code}
\inputminted{MATLAB}{\pathToDDEBifToolDemos/convergence_plot.m}
\label{sm:lst:convergence_plot}
\caption{Auxiliary function for creating convergence plots.}
\end{code}
Using this function, the code below yields \cref{sm:fig:DoubleAlleeEffectConvergencePlot}.
\inputminted[firstline=184, lastline=195]{MATLAB}{\pathToDDEBifToolDemos/predator_prey/predator_prey.m}
\begin{figure}[ht]
    \centering
    \ifcompileimages%
        \tikzsetnextfilename{DoubleAlleeEffectConvergencePlot}%
        \input{tikz/DoubleAlleeEffectConvergencePlot}%
    \else
        \includegraphics{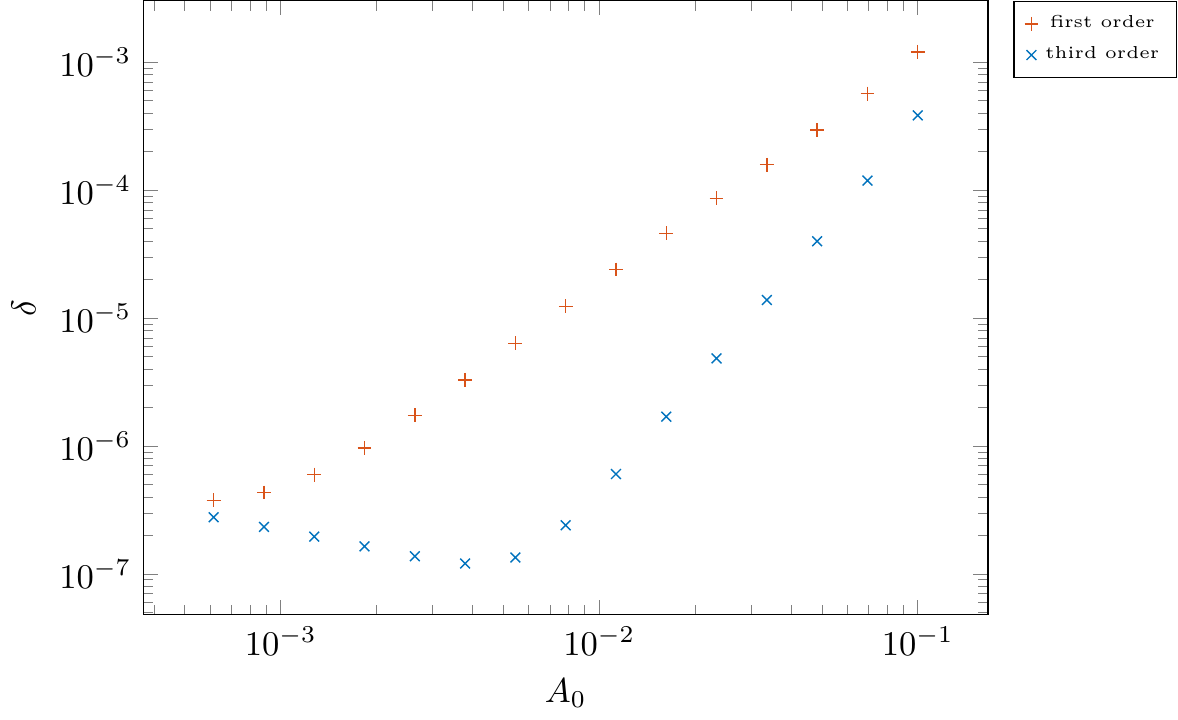}
    \fi

    \caption{On the abscissa is the approximation to the amplitude $A_0$ and on the ordinate the
        relative error $\delta$ between the constructed solution
        \mintinline{MATLAB}{hcli_pred} to the defining system for the homoclinic orbit and
        the Newton corrected solution \mintinline{MATLAB}{hcli_corrected}.}
    \label{sm:fig:DoubleAlleeEffectConvergencePlot}
\end{figure}

\subsection{Simulation with {\tt DifferentialEquations.jl}}
\begin{figure}[ht!]
    \centering
    \includegraphics{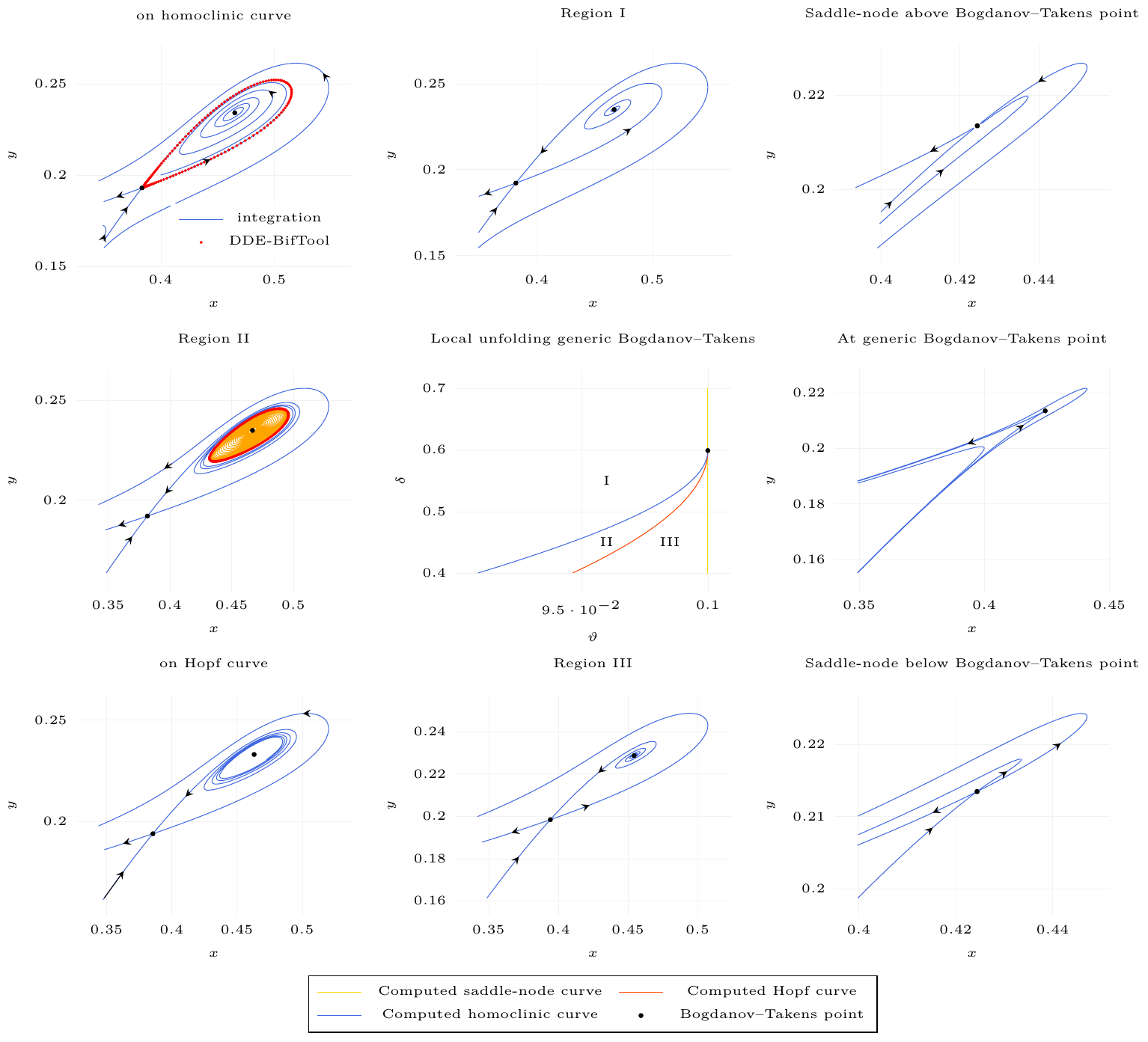}
    \caption{Bifurcation diagram near the derived generic Bogdanov-Takens
        point in \cref{sm:eq:neural_network}. In the center, we plotted the
        computed codimension one curves emanating from the Bogdanov--Takens
        point using \DDEBIFTOOL with the third-order homoclinic asymptotics
        obtained in \cref{btdde:sec:generic_bt_homoclinic_asymptotics}. The simulations
        surrounding the center plot have been performed in Julia.}
    \label{sm:fig:double_alle_effect-bifurcation-diagram}
\end{figure}
We finish this demonstration by simulating the dynamics near the generic
Bogdanov--Takens point. In \cref{sm:fig:double_alle_effect-bifurcation-diagram}
we created the full local unfolding of the singularity. Note
that, compared with \cite{Jiao2021}, the simulation is done with the original
delay differential equations \cref{sm:eq:double_alle_effect_rescaled} and not
with the ordinary differential equations of the reduced system on the center
manifold. We are able to do this since the parameter-dependent center manifold
is locally attractive. In order to integrate the system in the reverse direction, i.e.,
to obtain the orbits in the stable manifold of the equilibria, we multiplied 
the right-hand side of the system in \cref{sm:eq:double_alle_effect_rescaled} by $-1$.
Note that, in general, this will not provide an accurate approximation at all.
However, since the delay is relatively small, this approximation is accurate
enough for our application. Also, note that the Bogdanov--Takens point still
exists for the approximate system. Nonetheless, even without the backward solutions,
the bifurcation diagram shows that the numerical analysis obtained in
\DDEBIFTOOL is correct.

Since the code for creating the local unfolding diagram is
rather long, we show the code for reproducing the plot simulating the system
near the homoclinic orbit. The code for creating the bifurcation diagram in
\cref{sm:fig:double_alle_effect-bifurcation-diagram} can be found in the
GitHub repository.

\subsubsection{Loading necessary Julia packages}
We start by loading the necessary packages. These are
\begin{itemize}
    \item {\tt DifferentialEquations.jl} A suite for numerically solving differential equations written in Julia.
    \item {\tt GLMakie.jl} For high-level plotting on the GPU.
    \item {\tt NonlinearEigenproblems.jl} A nonlinear eigenvalue problem determine a scalar $\lambda$ and a vector $v$ such that $M(\lambda)v=0$. In our case, the matrix $M(\lambda)$ will be the characteristic matrix.
    \item {\tt DDEBifTool.jl} We created this very minimalistic package to have some functionality for normal form calculations of DDEs in Julia. Here we use it to calculate the derivatives of the system necessary for {\tt NonlinearEigenproblems.jl}.
    \item {\tt DelimitedFiles.jl} Reading and writing of CSV files.
    \item {\tt PGFPlotsX.jl} A Julia package for creating publication quality figures using the LaTeX library PGFPlots as the backend.
\end{itemize}
\newcommand\pathToJuliaFiles{simulation}
\inputminted[firstline=1, lastline=8]{julia}{\pathToJuliaFiles/predator_prey_simulation_article.jl}

\subsubsection{Define system}
Next we define the system to be integrated, a system to approximate the reverse
flow, and also an allocating version used for stability calculations.
\inputminted[firstline=10, lastline=38]{julia}{\pathToJuliaFiles/predator_prey_simulation_article.jl}

\subsubsection{Functions for plotting arrows} \label{sm:eq:arrow_functions}
We define a function to show in which direction the orbits flow, which is
useful when plotting in phase-space. We also define a function to show the
direction of the leading eigenvectors of the characteristic matrix.

\begin{code}
\inputminted[firstline=40, lastline=71]{julia}{\pathToJuliaFiles/predator_prey_simulation_article.jl}
\caption{Functions for plotting arrows.}
\label{sm:lst:arrow_fucntions}
\end{code}

\subsubsection{Define parameters, equilibria}
We define parameters located on the continued homoclinic branch with
\DDEBIFTOOL. Then define the non-trivial equilibria points in
\cref{sm:eq:double_alle_effect_rescaled}, which can be derived analytically.
\inputminted[firstline=73, lastline=80]{julia}{\pathToJuliaFiles/predator_prey_simulation_article.jl}

\subsubsection{Plot equilibria and homoclinic orbit}
By plotting the homoclinic orbit obtained with \DDEBIFTOOL, we can compare with the numerical simulations. 
\inputminted[firstline=82, lastline=91]{julia}{\pathToJuliaFiles/predator_prey_simulation_article.jl}

\subsubsection{Eigenvectors}
Next, we calculate and plot the leading eigenvectors of the characteristic matrix at the saddle-node bifurcation point.
\inputminted[firstline=93, lastline=105]{julia}{\pathToJuliaFiles/predator_prey_simulation_article.jl}

\subsubsection{Define callback}
Since we are only interested in the flow near the equilibria points, we create a
discrete callback to ensure the orbits do not become too large.
\inputminted[firstline=107, lastline=110]{julia}{\pathToJuliaFiles/predator_prey_simulation_article.jl}

\subsubsection{Integrate the system}
Now we define the problem to be integrated and choose the algorithm to be used.
The first of the five numerical simulations below starts near the unstable
eigendirection. We rotated the eigenvector slightly to follow the homoclinic
orbit very close. The rotation value of $\alpha_0$ was actually obtained by using
the bisection method. However, we did not include this code here. 
\inputminted[firstline=112, lastline=147]{julia}{\pathToJuliaFiles/predator_prey_simulation_article.jl}

\subsubsection{Finish the plot}
Lastly, we add arrows to the obtained solutions using the function
\mintinline{julia}{draw_arrow_on_solution} defined above. Also, we add the
legend and re-plot the equilibria, so that they appear on top.
\inputminted[firstline=149, lastline=161]{julia}{\pathToJuliaFiles/predator_prey_simulation_article.jl}
We should now obtain an interactive figure similar to the left figure in \cref{sm:fig:doubleAlleeEffectHomoclinicSimulation}.
\begin{figure}[!ht]
    \centering
    \includegraphics{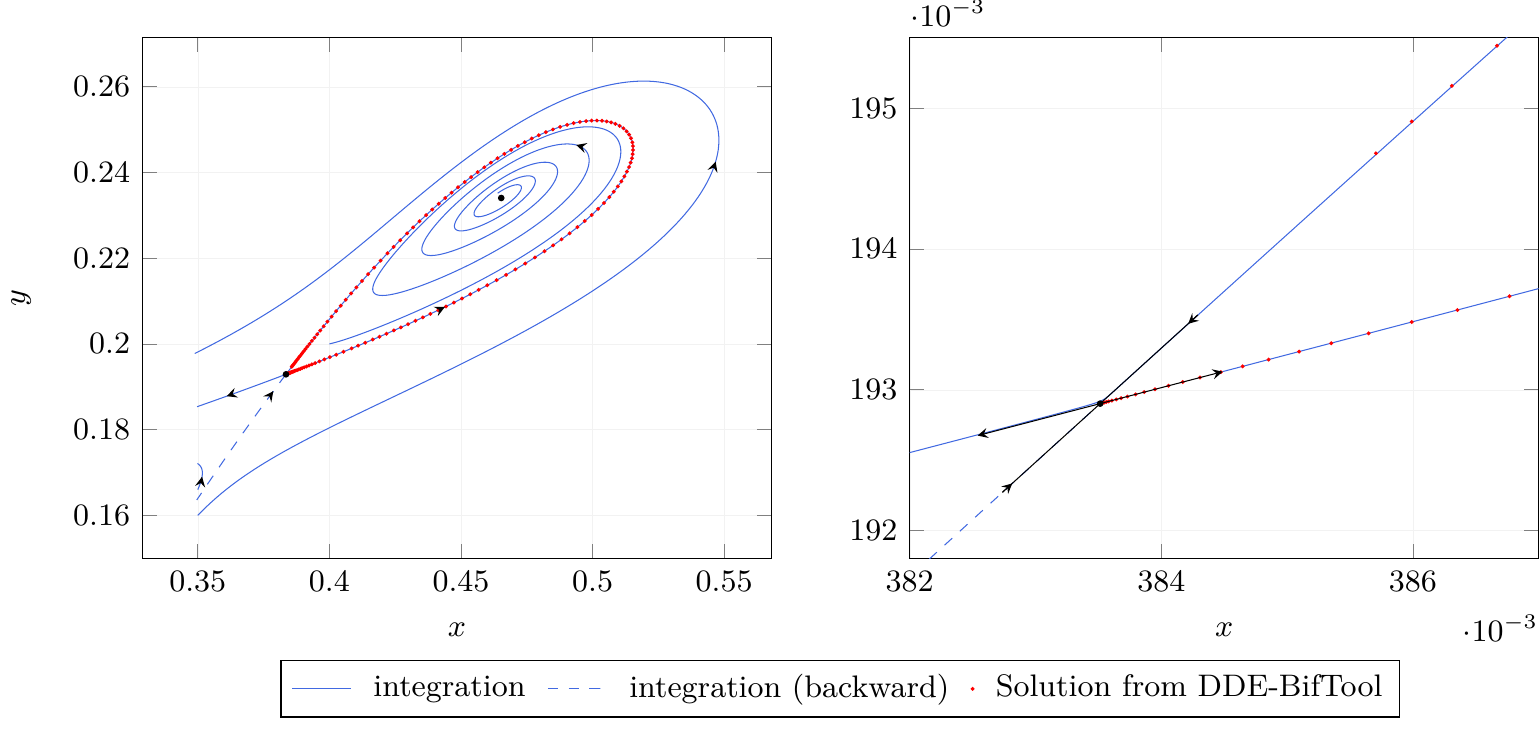}
    \caption{Integration of the delayed predator-prey model
        \cref{sm:eq:double_alle_effect_rescaled} at parameter values $(\theta,
        \delta) = (0.094448552842823, 0.447783343351055)$ obtained from
        continuation of the homoclinic curve emanating from the
        Bogdanov--Takens point using \DDEBIFTOOL. In the plot to the right
        a close-up near the saddle is given. Additionally, the 
        leading eigenvectors of the characteristic matrix
        are shown.}
    \label{sm:fig:doubleAlleeEffectHomoclinicSimulation}
\end{figure}

\section[Bogdanov--Takens bifurcation in a neural network model]
        {Generic Bogdanov--Takens bifurcation in a neural network model}

In this example, we will consider the model 
\begin{equation}
\label{sm:eq:neural_network}
\begin{cases}
\mu\dot{u}_1(t) = -u_1(t) + q_{11}\alpha(u_1(t\text{-}T))-q_{12}u_2(t\text{-}T) + e_1,\\
\mu\dot{u}_2(t) = -u_2(t) + q_{21}\alpha(u_1(t\text{-}T))-q_{22}u_2(t\text{-}T) + e_2,
\end{cases}
\end{equation}
which describes the dynamics of a neural network consisting of
excitatory and inhibitory neurons \cite{giannakopoulos2001bifurcations}.
The variables and parameters occurring in \cref{sm:eq:neural_network}
have the following neurophysiological meaning:
\begin{itemize}
\item $u_1,u_2:\mathbb{R}\rightarrow\mathbb{R}$ denote the total post-synaptic
potential of the excitatory and inhibitory neurons, respectively.
\item $\mu>0$ is a time constant characterizing the dynamical properties
of the cell membrane.
\item $q_{ik}\geq0$ represent the strength of the connection line from
the $k$th neuron to the $i$th neuron.
\item $\alpha:\mathbb{R}\rightarrow\mathbb{R}$ is the transfer function
which describes the activity generation of the excitatory neuron as
a function of its total potential $u_1$. The function $\alpha$
is smooth, increasing and has an unique turning point at $u_1 = \theta$.
The transfer function corresponding to the inhibitory neuron is assumed
to be the identity.
\item $T\geq0$ is a time delay reflecting synaptic delay, axonal and dendritic
propagation time.
\item $e_1$ and $e_2$ are external stimuli acting on the excitatory
and inhibitory neuron, respectively.
\end{itemize}

Following \cite{giannakopoulos2001bifurcations}, we consider equation \cref{sm:eq:neural_network} with
\begin{align*}
\alpha(u_1) & = \frac{1}{1 + e^{-4u_1}}-\frac{1}{2},\qquad q_{11} = 2.6,\qquad q_{21} = 1.0,\qquad q_{22} = 0.0,\\
\mu & = 1.0,\qquad T = 1.0,\qquad e_2 = 0.0,
\end{align*}
and $Q: = q_{12},\,E: = e_1$ as bifurcation parameters. Substituting
into \cref{sm:eq:neural_network} yields
\begin{equation}
\label{sm:eq:neural_network_subs}
\begin{cases}
\dot{u}_1(t) = -u_1(t) + 2.6\alpha(u_1(t - 1))-Qu_2(t - 1) + E,\\
\dot{u}_2(t) = -u_2(t) + \alpha(u_1(t - 1)).
\end{cases}
\end{equation}
Notice that for any steady-state we have the symmetry
\begin{equation}
\label{sm:eq:neuralNetworkSymmetry}
    (u_1,u_2,E)\rightarrow(-u_1,-u_2,-E).
\end{equation}
It is easy to explicitly derive that
the system has a double eigenvalue zero for
\begin{equation}
\left\{
\begin{aligned}
    u_1(t) &= \frac14 \log\left(\frac{8 - \sqrt{39}}5\right) \approx -0.2617, \\
    u_2(t) &= -\frac12 \sqrt{\frac{3}{13}} \approx -0.2402, \\
    Q &= \frac{13}{10}, \\
    E &= \frac{\sqrt{39} - 10\atanh \sqrt{\frac{3}{13}}}{20} \approx 0.0505.
\end{aligned}
\right.
\end{equation}

\begin{remark} 
    The \MATLAB files for this demonstration can be found in the directory
    \mintinline[breaklines,breakafter=/]{MATLAB}{demos/tutorial/VII/neural_network_model} relative to the main
    directory of the \DDEBIFTOOL package. Here, we omit the code to generate the
    system file. We assume that the system file
    \mintinline{MATLAB}{sym_neural_network_mf.m} has been
    generated with the script \mintinline{MATLAB}{sym_neural_network.m}. Also, we assume
    that the \DDEBIFTOOL package has been loaded as in
    \cref{sm:lst:searchpath}. The code in
    \crefrange{sm:sec:neural_network_model:pars_and_funcs}{sm:sec:neural_network_model:bifurcation_diagramII}
    highlights the important parts of the file
    \mintinline{MATLAB}{neural_network_model.m}. 
\end{remark}

\subsection{Set parameter names and funcs structure} 
\label{sm:sec:neural_network_model:pars_and_funcs}
As in the previous example, we set the parameter names and define the \mintinline{MATLAB}{funcs} structure.
\inputminted[firstline=28, lastline=39]{MATLAB}{\pathToDDEBifToolDemos/neural_network_model/neural_network_model.m}

\subsection{Set parameter range}
Since we are only interested here in the local unfolding, we restrict the
allowed parameter range for the unfolding parameters. In practice, one may have
physical restrictions which must be satisfied. Additionally, we also limit the
maximum allowed step size during continuation. By doing so, we obtain more refined
data to compare against our predictors.
\inputminted[firstline=41, lastline=44]{MATLAB}{\pathToDDEBifToolDemos/neural_network_model/neural_network_model.m}

\subsection{Stability and coefficients of the generic Bogdanov--Takens point}
We manually construct a steady-state at the generic Bogdanov--Takens point.
\inputminted[firstline=46, lastline=56]{MATLAB}{\pathToDDEBifToolDemos/neural_network_model/neural_network_model.m}
The \MATLAB console shows the following output.
\begin{minted}{shell-session}
ans =

   1.0e-07 *

   0.548156278544666
  -0.548156314155979
\end{minted}
The eigenvalues confirm that the point under consideration is indeed (an
approximation to) a Bogdanov--Takens point. Furthermore, the remaining eigenvalues have
negative real parts. Next, we calculate the normal form coefficients, the
time-reparametrization, and the transformation to the center manifold with the
function \mintinline{MATLAB}{nmfm_bt_orbital}, which implements the coefficients as derived in
\cref{btdde:sec:generic_bogdanov-takens}. For this, we need to set the argument
\mintinline{MATLAB}{free_pars} to the unfolding parameter $(Q,E)$. These
coefficients will be used to start the continuation of the codimension one branches
emanating from the Bogdanov--Takens point.
\inputminted[firstline=58, lastline=62]{MATLAB}{\pathToDDEBifToolDemos/neural_network_model/neural_network_model.m}
The \MATLAB console shows the following output.
\begin{minted}{shell-session}
ans =

  struct with fields:

          a: -0.190382124055415
          b: -0.951910620277072
  theta1000: -0.546026508575597
  theta0001: 1.473611111111115
        K10: [2x1 double]
        K01: [2x1 double]
        K02: [2x1 double]
        K11: [2x1 double]
        K03: [2x1 double]
       phi0: [1x1 struct]
       phi1: [1x1 struct]
      h0010: [1x1 struct]
      h0001: [1x1 struct]
      h2000: [1x1 struct]
      h1100: [1x1 struct]
      h0200: [1x1 struct]
      h1010: [1x1 struct]
      h1001: [1x1 struct]
      h0110: [1x1 struct]
      h0101: [1x1 struct]
      h0002: [1x1 struct]
      h0011: [1x1 struct]
      h3000: [1x1 struct]
      h2100: [1x1 struct]
      h1101: [1x1 struct]
      h2001: [1x1 struct]
      h0003: [1x1 struct]
      h1002: [1x1 struct]
      h0102: [1x1 struct]
          K: @(beta1,beta2)K10*beta1+K01*beta2+1/2*K02*beta2.^2
                +K11*beta1.*beta2+1/6*K03*beta2^3
          H: [function_handle]
\end{minted}
Since the sign of $ab$ is positive, we expect to find unstable periodic orbits nearby the 
Bogdanov--Takens point.

\subsection{Comparing profiles of computed and predicted homoclinic orbits}
To test the homoclinic asymptotics from
\cref{btdde:sec:generic_bt_homoclinic_asymptotics} we compare the first and third
order asymptotics to the Newton corrected solution. For this, we use the 
function \mintinline[breaklines,breakafter=_]{MATLAB}{C1branch_from_C2point}. This function returns a branch, which
by default returns two initial corrected approximations in order to start continuation of the
codimension one curve under consideration. By setting the argument
\mintinline{MATLAB}{'predictor'} to \mintinline{MATLAB}{true} the approximations are left uncorrected.
To make the comparison visually clear, we set the perturbation parameter 
$\epsilon=0.25$ (\mintinline{MATLAB}{step = 0.25} in the code below).
The code below produces \cref{sm:fig:NeuralNetworkCompareProfiles}.
The difference between the two approximations is clearly noticeable. While
the first order asymptotics is close to the Newton corrected solution, the third
order asymptotics is indistinguishable at this scale from the Newton corrected
solution.
\inputminted[firstline=64, lastline=82]{MATLAB}{\pathToDDEBifToolDemos/neural_network_model/neural_network_model.m}
\begin{figure}[ht!]
    \includegraphics{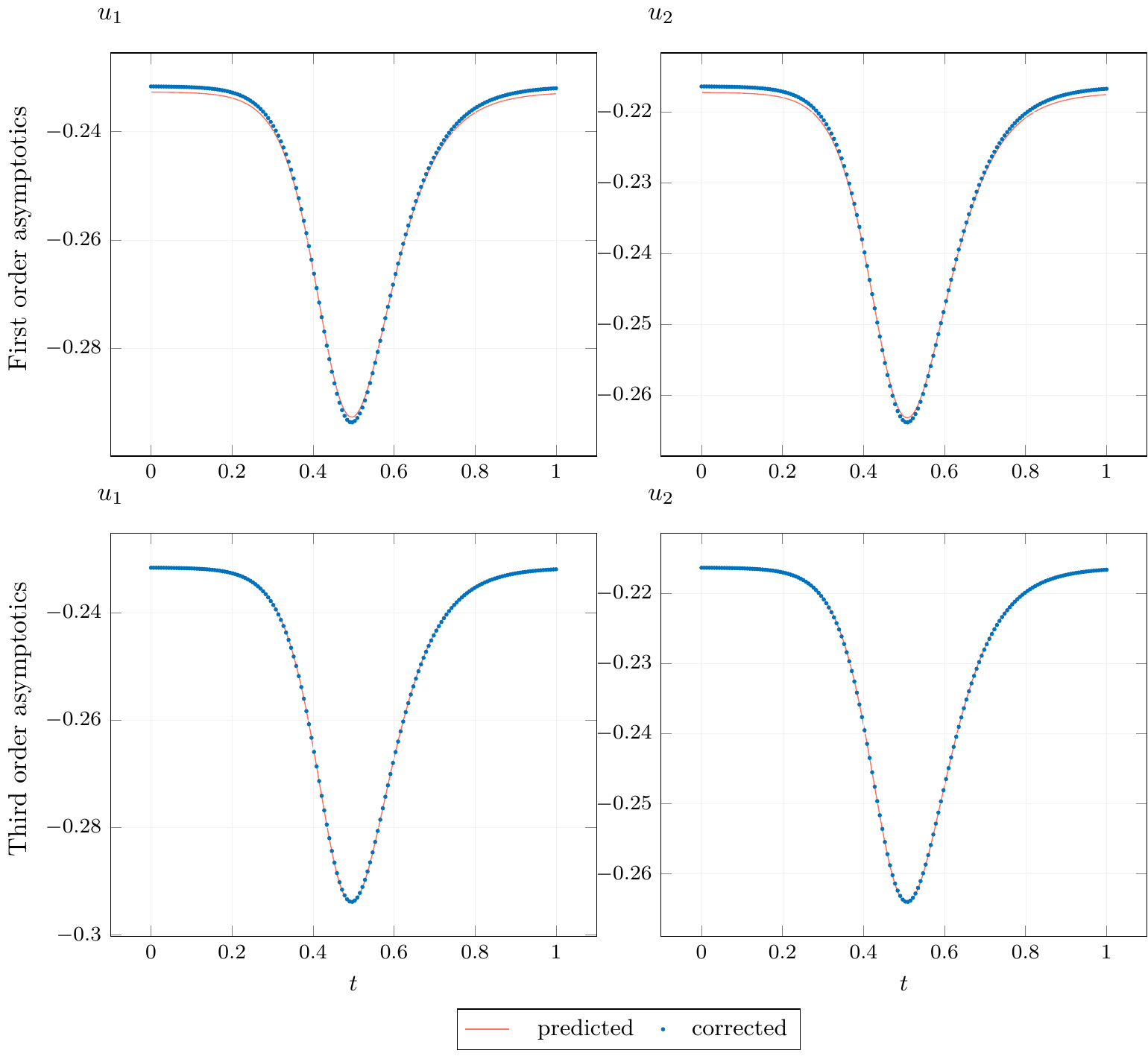}
    \caption{Comparison between the first and third-order asymptotics from
    \cref{btdde:sec:generic_bt_homoclinic_asymptotics} near the generic
        Bogdanov--Takens bifurcation in \cref{sm:eq:neural_network} with the
        perturbation parameter set to $\epsilon=0.25$.}
    \label{sm:fig:NeuralNetworkCompareProfiles}
\end{figure}

\label{sm:sec:neural_network_model:continuation}
\subsection{Continuation of the codimension one curves emanating}
To continue the three codimension one curves emanating from the generic
Bogdanov--Takens point, we can simply use the function
\mintinline[breaklines,breakafter=_]{MATLAB}{C1branch_from_C2point}, as shown in the code below. To monitor the
continuation process, the argument \mintinline{MATLAB}{plot} must be set to \mintinline{MATLAB}{1}.
The most important setting is the perturbation parameter (or multiple),
\mintinline{MATLAB}{step} in the code below. If left out, default step sizes are defined.
However, depending on the problem, no convergence may then be obtained.
\inputminted[firstline=84, lastline=102]{MATLAB}{\pathToDDEBifToolDemos/neural_network_model/neural_network_model.m}

\subsection{Predictors of the codimension one curves emanating from the Bogdanov--Takens point}
Before we provide the bifurcation diagram in the next section, we first obtain the predictors
for the codimension one curves. For this, we again use the function
\mintinline[breaklines,breakafter=_]{MATLAB}{C1branch_from_C2point}. We set the argument \mintinline{MATLAB}{predictor} to \mintinline{MATLAB}{1}
and provide a range of perturbation parameters.
\inputminted[firstline=121, lastline=133]{MATLAB}{\pathToDDEBifToolDemos/neural_network_model/neural_network_model.m}
In the last part of the code above, we added the asymptotics obtained from \cite{Jiao2021}.

\subsection{Bifurcation diagram}
The code below produces (a figure similar to) \cref{sm:fig:NeuralNetworkCompareParameters}.
\inputminted[firstline=135, lastline=156]{MATLAB}{\pathToDDEBifToolDemos/neural_network_model/neural_network_model.m}
\begin{figure}[ht]
    \centering
    \includetikzscaled{NeuralNetworkCompareParameters}
    \caption{Bifurcation diagram near the derived generic Bogdanov-Takens point in
        \cref{sm:eq:neural_network} comparing computed codimension one curves using
        \DDEBIFTOOL with the third-order homoclinic parameter asymptotics obtained
        in \cref{btdde:sec:generic_bt_homoclinic_asymptotics}.}
    \label{sm:fig:NeuralNetworkCompareParameters}
\end{figure}

\subsection{Compare homoclinic solutions in phase-space}
To obtain an impression of the third-order homoclinic asymptotics in
phase-space, we compare the corrected and uncorrected homoclinic solutions
with the perturbation parameter ranging from $0.1$ to $0.3$.
The code below results in \cref{sm:fig:NeuralNetworkCompareOrbitsPhaseSpace}.
We see that the corrected and predicted homoclinic orbits are nearly identical.
\inputminted[firstline=158, lastline=174]{MATLAB}{\pathToDDEBifToolDemos/neural_network_model/neural_network_model.m}
\begin{figure}[ht]
    \centering
    \includetikzscaled{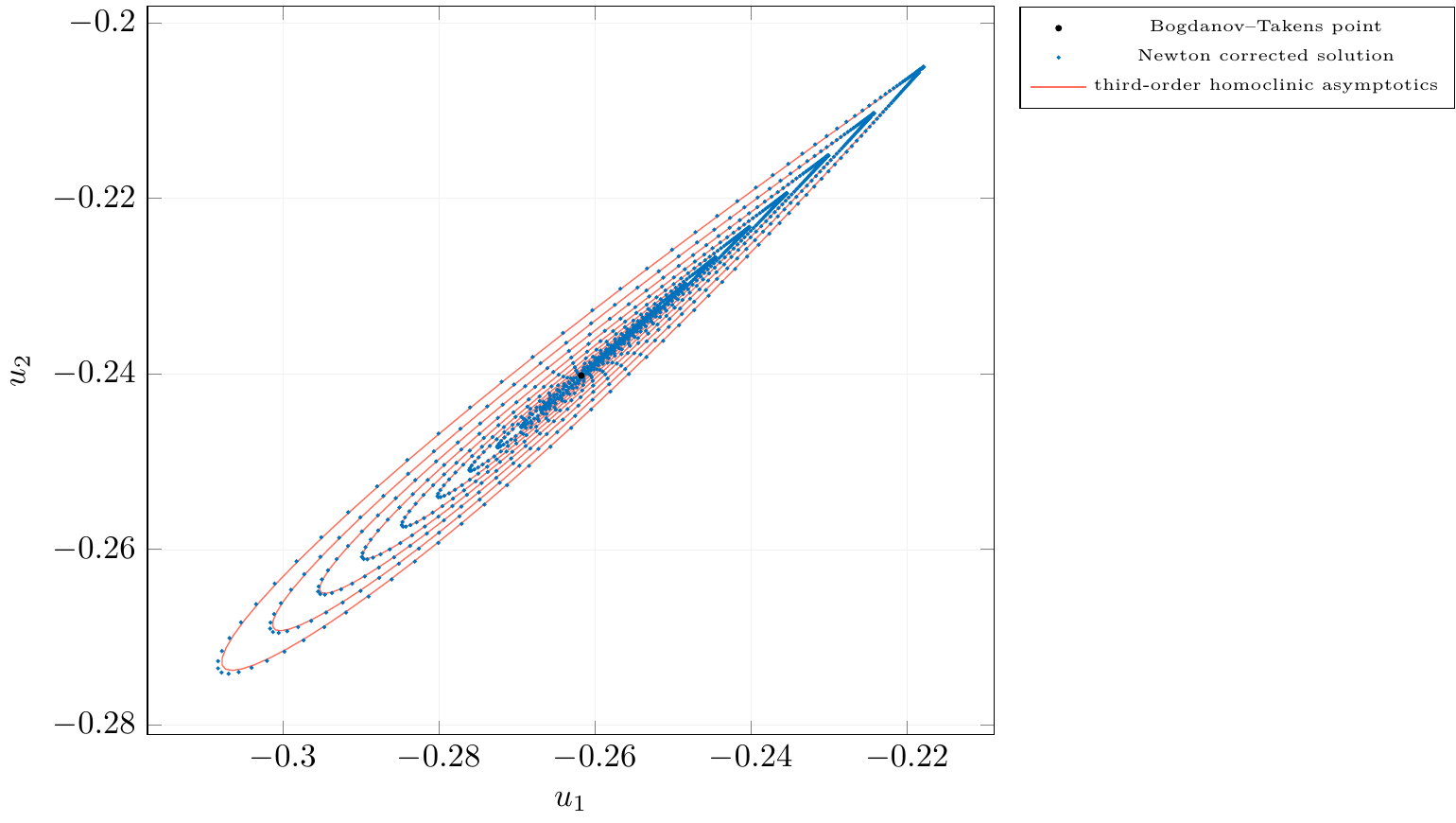}
    \caption{Plot comparing the third-order homoclinic asymptotics from
        \cref{btdde:sec:generic_bt_homoclinic_asymptotics} near the generic
        Bogdanov--Takens bifurcation in \cref{sm:eq:neural_network} with the
        Newton correct homoclinic solutions in $(u_1,u_2)$ phase-space.}
    \label{sm:fig:NeuralNetworkCompareOrbitsPhaseSpace}
\end{figure}

\subsection{Convergence plot}
Using the function from \cref{sm:lst:convergence_plot}, we create a log-log
convergence plot comparing the convergence order of the first and third order
homoclinic asymptotics from \cref{btdde:sec:generic_bt_homoclinic_asymptotics}.
The code below yields \cref{sm:fig:NeuralNetworkConvergencePlot}.
\inputminted[firstline=176, lastline=187]{MATLAB}{\pathToDDEBifToolDemos/neural_network_model/neural_network_model.m}
\begin{figure}[ht]
    \centering
    \ifcompileimages%
        \tikzsetnextfilename{NeuralNetworkConvergencePlot}%
        \input{tikz/NeuralNetworkConvergencePlot}%
    \else
        \includegraphics{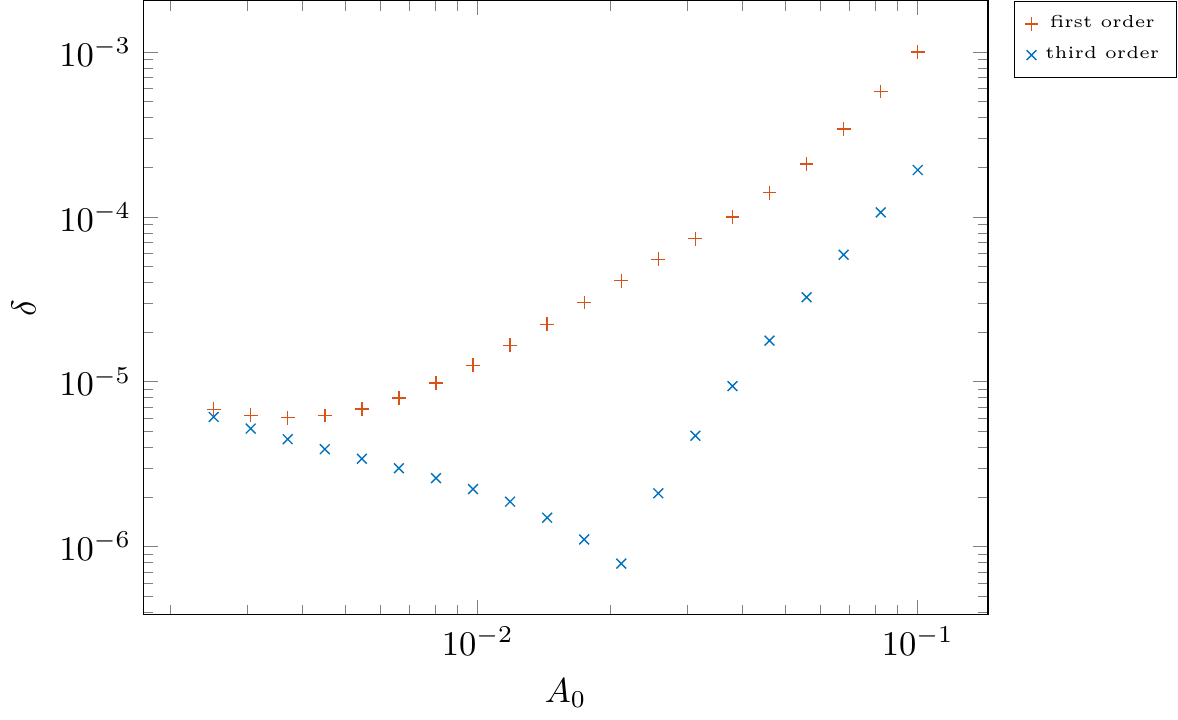}
    \fi

        \caption{On the abscissa is the approximation to the amplitude $A_0$ and on
        the ordinate the relative error $\delta$ between the constructed solution
        \mintinline{MATLAB}{hcli_pred} to the defining system for the homoclinic orbit
        and the Newton corrected solution \mintinline{MATLAB}{hcli_corrected}.}
    \label{sm:fig:NeuralNetworkConvergencePlot}
\end{figure}

\subsection{Continuation of the codimension one curves emanating from the second Bogdanov--Takens point}
For completeness, we also continue the codimension one curves emanating from
the second Bogdanov--Takens point, which exists due to the symmetry
\cref{sm:eq:neuralNetworkSymmetry}. Of course, we could just use the symmetry
instead of computing the curves numerically. However, we use it as an
additional verification of our derived asymptotics. The code below defines the
second Bogdanov--Takens point, calculates the stability, and continues the Hopf
and homoclinic bifurcation curves.
\inputminted[firstline=189, lastline=216]{MATLAB}{\pathToDDEBifToolDemos/neural_network_model/neural_network_model.m}

\subsection{Bifurcation diagram with two Bogdanov--Takens points}
\label{sm:sec:neural_network_model:bifurcation_diagramII}
Now that we continued the Hopf and homoclinic bifurcation curves emanating from
the second Bogdanov--Takens point, we can reconstruct the bifurcation diagram
given in \cite[Figure 7]{giannakopoulos2001bifurcations}. The code below
results into a similar figure as \cref{sm:fig:NeuralNetworkCompareParametersII} in \MATLAB.
\inputminted[firstline=218, lastline=245]{MATLAB}{\pathToDDEBifToolDemos/neural_network_model/neural_network_model.m}
\begin{figure}[ht]
    \centering
    \includetikzscaled{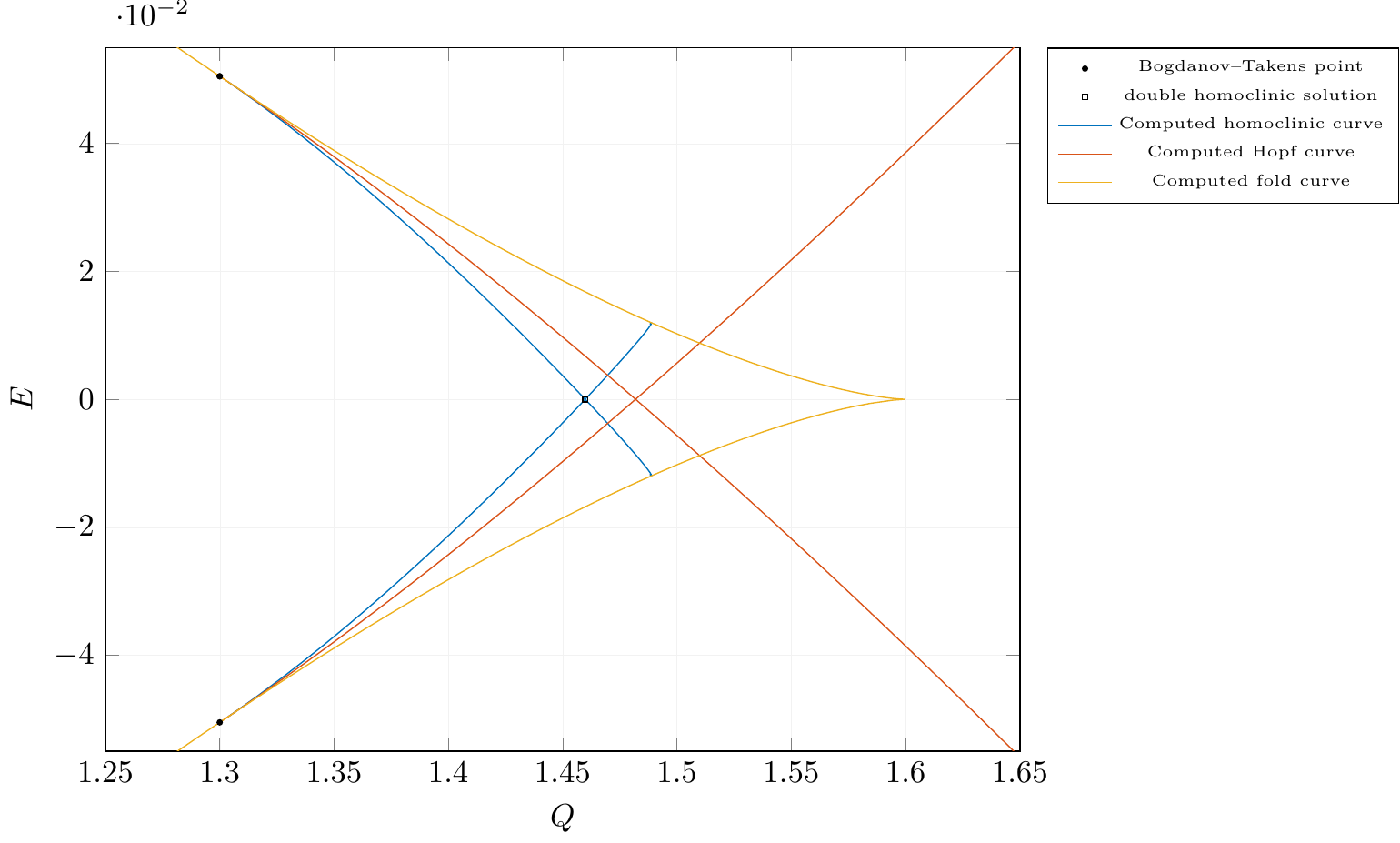}
    \caption{Reconstruction of the bifurcation diagram given in \cite[Figure
        7]{giannakopoulos2001bifurcations}. We could have used the symmetry
        \cref{sm:eq:neuralNetworkSymmetry} instead of computing the additional
        curves numerically. However, it provides us an additional verification of
        our derived asymptotics.}
    \label{sm:fig:NeuralNetworkCompareParametersII}
\end{figure}

\subsection{Simulation with {\tt DifferentialEquations.jl}}
Here we will perform two simulations. The first simulation will be at the
double homoclinic orbits, which will confirm the continuation of both
homoclinic orbits and is also ascetically pleasing, see \cref{sm:fig:NeuralNetworkSimulationHomoclinic}. The second simulation will
be in the region where there should be unstable periodic orbits, see \cref{sm:fig:NeuralNetworkPeriodicSimulation}.

\begin{figure}[ht]
    \includegraphics{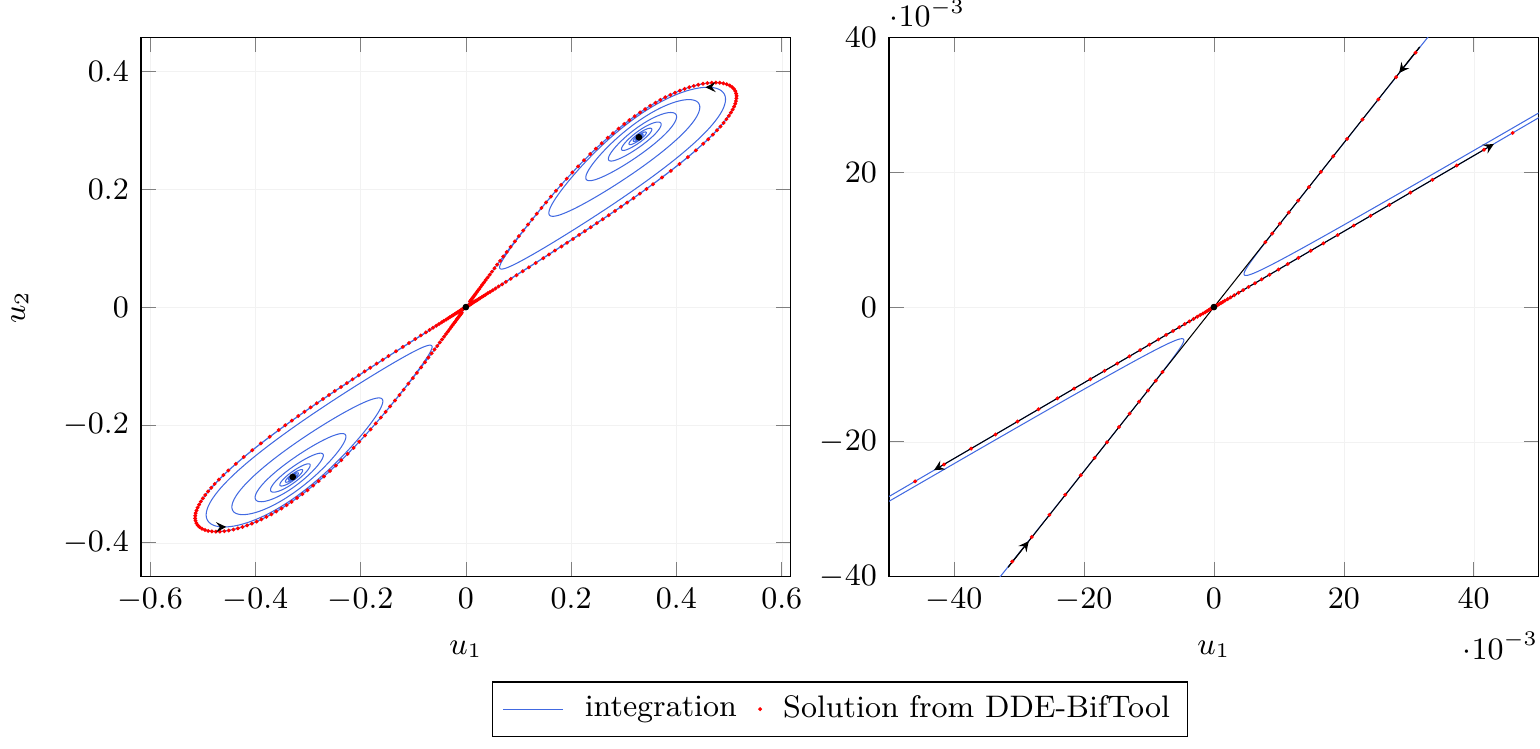}
    \caption{Comparing the computed double homoclinic orbit in \cref{sm:eq:neural_network}
    with \DDEBIFTOOL with the solutions obtained from numerical simulation with Julia.
    In the right plot is a close-up near the equilibrium at the origin. Also, the
    leading stable and unstable eigenvectors of the characteristic matrix are plotted. We see the numerical integrated solution
    intersects all the red points from the solution from \DDEBIFTOOL.}
    \label{sm:fig:NeuralNetworkSimulationHomoclinic}
\end{figure}

\begin{figure}[ht]
    \centering
    \includegraphics{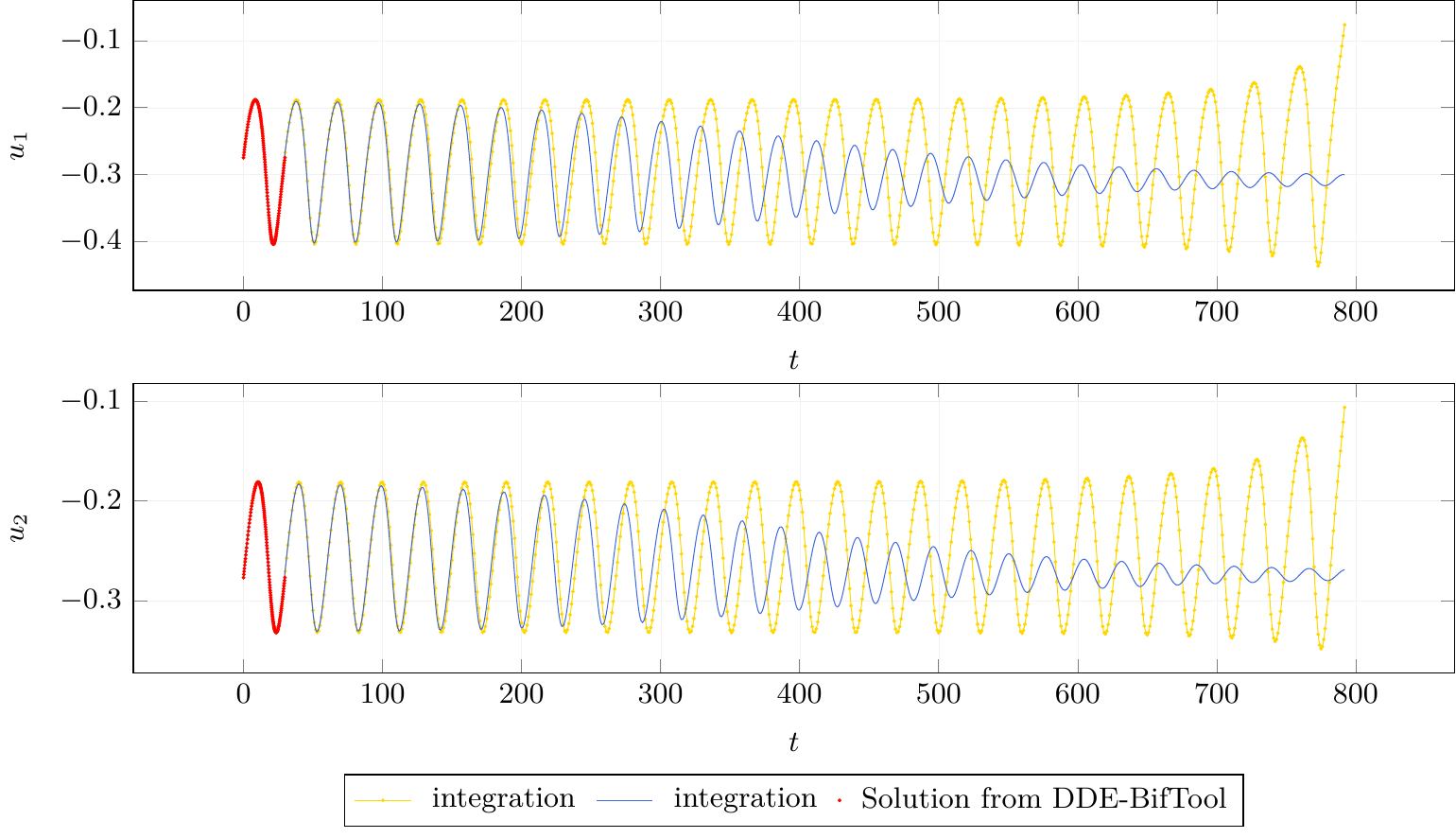}
    \caption{Comparing a computed periodic orbit in \cref{sm:eq:neural_network}
        with \DDEBIFTOOL at $(Q,E)=(1.476442865781454, 0.0)$ with the solution
        obtained from numerical simulation with Julia near the periodic orbit.
        The yellow dotted line has the constant history function $(u_1,u_2) =
        (-0.274863341578762, -0.27715979849863204)$ slightly below the periodic
        orbit.  The blue line has the constant history function $(u_1,u_2) =
        (-0.274863341578762, -0.276969798498632)$, a point on the periodic
        orbit (red dots) located with \DDEBIFTOOL.
    }
    \label{sm:fig:NeuralNetworkPeriodicSimulation}
\end{figure}

\subsubsection{Loading necessary Julia packages}
We start by loading the necessary packages.

\begin{listing}[!ht]
\inputminted[firstline=1, lastline=6]{julia}{\pathToJuliaFiles/neural_network_model_simulation_article.jl}
\caption{Loading Julia packages for simulation in \cref{sm:eq:neural_network}.}
\label{sm:lst:neuralNetworkLoadingPacakges}
\end{listing}

In the previous demonstration we were able to derive the equilibria
analytically.  Here we will solve for the equilibria numerically with the
packages {\tt IntervalArithmetic.jl} \cite{IntervalArithmetic} and  {\tt
IntervalRootFinding.jl} \cite{IntervalRootFinding}.

\subsubsection{Define system}
We define the system to be integrated, a system to approximate the reverse
flow, and also an allocating version used for stability calculations.
\inputminted[firstline=8, lastline=27]{julia}{\pathToJuliaFiles/neural_network_model_simulation_article.jl}

\subsubsection{Functions for plotting arrows}
We define a function to show in which direction the orbits flow, which is
useful when plotting in phase-space. We also define functions to show the
direction of the leading eigenvectors of the characteristic matrix.
The code is shown in \cref{sm:lst:arrow_fucntions}.

\subsubsection{Create figure with several axes}
We create a figure containing multiple axis in which we will plot the bifurcation diagram and
the homoclinic and periodic orbits.
\inputminted[firstline=63, lastline=73]{julia}{\pathToJuliaFiles/neural_network_model_simulation_article.jl}

\subsubsection{Plot bifurcation diagram in the middle}
Loading the continued bifurcation curves obtained with \DDEBIFTOOL and plot these in the middle axis.
This should give a similar bifurcation diagram as in \cref{sm:fig:NeuralNetworkCompareParametersII}.
\inputminted[firstline=76, lastline=102]{julia}{\pathToJuliaFiles/neural_network_model_simulation_article.jl}

\subsubsection{Simulation at the double homoclinic orbit}
In the code below, we integrate at parameter values $(Q,E)=(1.459868437376222,0)$, i.e.,
where we located the double homoclinic orbit. We start by locating the three equilibria
points. Note that by the symmetry, we actually only need to solve for one of them.
We calculate the stability of the equilibria and filter out the leading stable and unstable
eigenvectors from the characteristic matrix of the saddle-node point. 
The rest of the code should be pretty straight forward,
since it is very similar as in the simulation in the previous demonstration.
After running this code, we should obtain a similar plot as in
the left plot of \cref{sm:fig:NeuralNetworkSimulationHomoclinic}.
\inputminted[firstline=105, lastline=166]{julia}{\pathToJuliaFiles/neural_network_model_simulation_article.jl}

\subsubsection{Simulation near an unstable periodic orbit}
To show by integration the existence of an unstable periodic orbit, we first
located a periodic orbit in \DDEBIFTOOL. This can be done by continuing a
branch of periodic orbits emanating from a point on the continued Hopf curve.
Then we load the profiles of the periodic orbits into Julia and start
integration near the periodic orbits.  After running the code below, we should
obtain a similar plot as in
\cref{sm:fig:NeuralNetworkPeriodicSimulation}.
\inputminted[firstline=169, lastline=208]{julia}{\pathToJuliaFiles/neural_network_model_simulation_article.jl}

\section[the Van der Pol oscillator with delay feedback]
        {Transcritical Bogdanov--Takens bifurcation in the Van der Pol oscillator with delay feedback}
We consider the Van der Pol oscillator with delay feedback \cite{jiang2007bogdanov}
given by 
\begin{equation}
\ddot{x}(t) + \epsilon(x^2(t)-1)\dot{x}(t) + x(t) = \epsilon g(x(t-\tau))\label{sm:eq:dde_vanderPol}
\end{equation}
where $\epsilon>0$ is a parameter, $\tau>0$ is a delay and $g:\mathbb{R}\rightarrow\mathbb{R}$
is a smooth function with $g(0) = 0$ and $g'(0)\neq0$. We rewrite
the Van der Pol equation \cref{sm:eq:dde_vanderPol} as
\begin{equation}
\label{sm:eq:vanderPolOscillator}
\begin{cases}
    \dot{x}_1 = x_2,\\
    \dot{x}_2 = \epsilon g(x_1(t-\tau))-\epsilon(x_1^2-1)x_2-x_1.
\end{cases}
\end{equation}
Rescaling time with $t\rightarrow\dfrac{t}{\tau}$ to normalize the
delay yields
\begin{equation}
\label{sm:eq:vanderPolOscillatorRescaled}
\begin{cases}
\dot{x}_1 = \tau x_2,\\
\dot{x}_2 = \tau\left(\epsilon g(x_1(t-1))-\epsilon(x_1^2-1)x_2-x_1\right).
\end{cases}
\end{equation}
This allows to treat $\tau$ as a bifurcation parameter.

Following \cite{jiang2007bogdanov}, we consider \cref{sm:eq:dde_vanderPol} with
\[
g(x) = \frac{e^x-1}{c_1e^x + c_2},
\]
where $c_1 = \dfrac{1}{4}$ and $c_2 = \dfrac{1}{2}$. Then the trivial
equilibrium undergoes a transcritical Bogdanov--Takens bifurcation at parameter
values $(\epsilon,\tau) = (0.75,0.75)$, \cite{jiang2007bogdanov} and the
supplement. 

\begin{remark}
    The \MATLAB files for this demonstration can be found in the directory
    \mintinline[breaklines,breakafter=/]{MATLAB}{demos/tutorial/VII/vdpo_bt_transcritical} relative to the main
    directory of the \DDEBIFTOOL package. Here, we omit the code to generate a
    system file. The system file \mintinline{MATLAB}{sym_vdpo_mf.m} has been generated
    with the script \mintinline{MATLAB}{sym_vdpo_mf.m}. Also, we assume that the
    \DDEBIFTOOL package has been loaded as in \cref{sm:lst:searchpath}. The
    code in
    \crefrange{sm:sec:vpdo:pars_and_funcs}{sm:sec:vdpo:convergence_plot}
    highlights the important parts of the file
    \mintinline{MATLAB}{vanderPolOscillator.m}. 
\end{remark}

\subsection{Set parameter names and funcs structure}
\label{sm:sec:vpdo:pars_and_funcs}
As in the previous example, we set the parameter names and define the \mintinline{MATLAB}{funcs} structure.
\inputminted[firstline=31, lastline=37]{MATLAB}{\pathToDDEBifToolDemos/vdpo_bt_transcritical/vanderPolOscillator.m}

\subsection{Set parameter range}
Since we are only interested here in the local unfolding, we restrict the
allowed parameter range for the unfolding parameters. In practice, one may have
physical restrictions which must be satisfied. Additionally, we also limit the
maximum allowed step size during continuation. By doing so, we obtain more refined
data to compare against our predictors.
\inputminted[firstline=39, lastline=42]{MATLAB}{\pathToDDEBifToolDemos/vdpo_bt_transcritical/vanderPolOscillator.m}

\subsection{Stability and coefficients of the transcritical Bogdanov--Takens point}
We manually construct a steady-state at the transcritical Bogdanov--Takens
point and calculate its stability.
\inputminted[firstline=44, lastline=55]{MATLAB}{\pathToDDEBifToolDemos/vdpo_bt_transcritical/vanderPolOscillator.m}

The \MATLAB console shows the following output.
\begin{minted}{shell-session}
ans =

   1.0e-07 *

       0.6223
      -0.6223

\end{minted}
The eigenvalues confirm that the point under consideration is indeed (an
approximation to) a Bogdanov--Takens point. Furthermore, the remaining eigenvalues have
negative real parts. Next, we calculate the normal form coefficients, the
time-reparametrization, and the transformation to the center manifold with the
function \mintinline{MATLAB}{nmfm_bt_orbital}, which implements the coefficients as derived in
\cref{btdde:sec:transcritical-Bogdanov-Takens}. For this, we need to set the argument
\mintinline{MATLAB}{free_pars} to the unfolding parameter $(Q,E)$. These
coefficients will be used to start the continuation of the codimension one branches
emanating from the Bogdanov--Takens point. Also, since we are in the transcritical case,
we set the argument \mintinline{matlab}{generic_unfolding} to \mintinline{matlab}{false}.
\inputminted[firstline=57, lastline=60]{MATLAB}{\pathToDDEBifToolDemos/vdpo_bt_transcritical/vanderPolOscillator.m}

The \MATLAB console shows the following output.
\begin{minted}{shell-session}
ans =

  struct with fields:

          a: 0.1304
          b: -0.2949
  theta1000: 0.0780
  theta0010: -21.1293
  theta0001: -0.3811
       phi0: [1x1 struct]
       phi1: [1x1 struct]
      h2000: [1x1 struct]
      h1100: [1x1 struct]
      h0200: [1x1 struct]
      h3000: [1x1 struct]
      h2100: [1x1 struct]
        K10: [2x1 double]
        K01: [2x1 double]
        K02: [2x1 double]
        K11: [2x1 double]
        K20: [2x1 double]
      h1010: [1x1 struct]
      h1001: [1x1 struct]
      h0110: [1x1 struct]
      h0101: [1x1 struct]
      h2010: [1x1 struct]
      h1110: [1x1 struct]
      h2001: [1x1 struct]
      h1101: [1x1 struct]
      h1002: [1x1 struct]
      h0102: [1x1 struct]
      h1020: [1x1 struct]
      h0120: [1x1 struct]
      h1011: [1x1 struct]
      h0111: [1x1 struct]
          K: @(beta1,beta2)K10*beta1+K01*beta2+1/2*K20*beta1^2
                    +K11*beta1*beta2+1/2*K02*beta2^2
          H: [function_handle]
\end{minted}
Since the sign of $ab$ is negative, we expect to find stable periodic orbits nearby the 
Bogdanov--Takens point.

\subsection{Comparing profiles of computed and predicted homoclinic orbits}
To test the homoclinic asymptotics from
\cref{btdde:sec:generic_bt_homoclinic_asymptotics} we compare the first and third
order asymptotics to the Newton corrected solution. For this, we use the 
function \mintinline[breaklines,breakafter=_]{MATLAB}{C1branch_from_C2point}. This function returns a branch, which
by default returns two initial corrected approximations in order to start continuation of the
codimension one curve under consideration. By setting the argument
\mintinline{MATLAB}{'predictor'} to \mintinline{MATLAB}{true} the approximations are left uncorrected.
To make the comparison visually clear, we set the perturbation parameter 
$\epsilon=0.1$ (\mintinline{MATLAB}{step = 0.1} in the code below).
The code below produces \cref{sm:fig:VDPOCompareProfiles}.
The difference between the two approximations is clearly noticeable. While
the first order asymptotics is close to the Newton corrected solution, the third
order asymptotics is indistinguishable at this scale from the Newton corrected
solution.
\inputminted[firstline=63, lastline=87]{MATLAB}{\pathToDDEBifToolDemos/vdpo_bt_transcritical/vanderPolOscillator.m}

\begin{figure}[ht]
    \centering
    \includegraphics{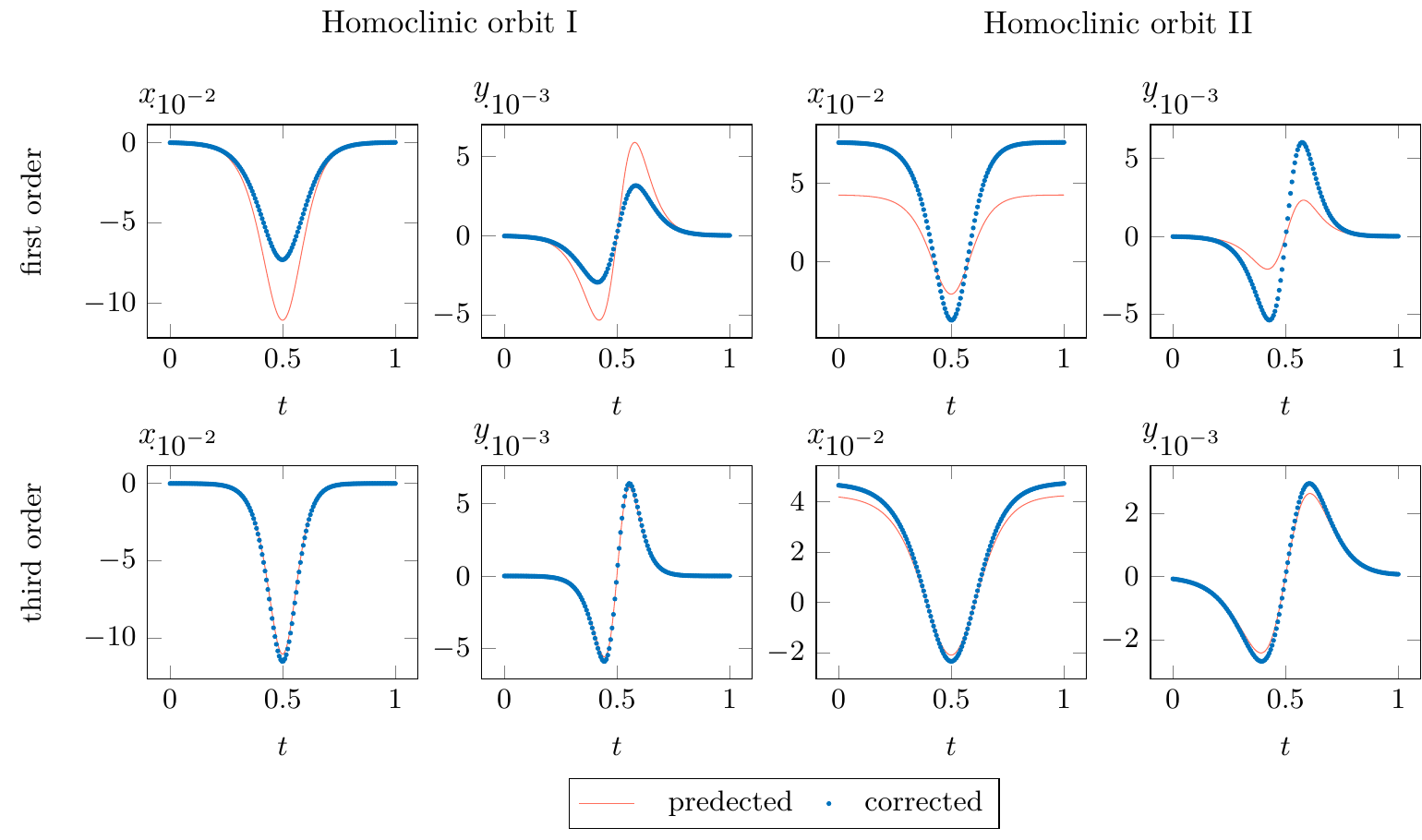}
    \caption{Comparison between the first and third-order homoclinic asymptotics from
    \cref{btdde:sec:transcritical_bt_homoclinic_asymptotics} near the transcritical
        Bogdanov--Takens bifurcation in \cref{sm:eq:vanderPolOscillatorRescaled} with the
        perturbation parameter set to $\epsilon=0.1$.}
    \label{sm:fig:VDPOCompareProfiles}
\end{figure}

\subsection{Continuation of the codimension one curves emanating}
To continue the three codimension one curves emanating from the generic
Bogdanov--Takens point, we can simply use the function
\mintinline[breaklines,breakafter=_]{MATLAB}{C1branch_from_C2point}, as shown in the code below. To monitor the
continuation process, the argument \mintinline{MATLAB}{plot} must be set to \mintinline{MATLAB}{1}.
The most important setting is the perturbation parameter (or multiple),
\mintinline{MATLAB}{step} in the code below. If left out, default step sizes are defined.
However, depending on the problem, no convergence may then be obtained.
\inputminted[firstline=89, lastline=113]{MATLAB}{\pathToDDEBifToolDemos/vdpo_bt_transcritical/vanderPolOscillator.m}

\subsection{Predictors of the codimension one curves emanating from the Bogdanov--Takens point}
Before we provide the bifurcation diagram in the next section, we first obtain the predictors
for the codimension one curves. For this, we again use the function
\mintinline[breaklines,breakafter=_]{MATLAB}{C1branch_from_C2point}. We set the argument \mintinline{MATLAB}{predictor} to \mintinline{MATLAB}{1}
and provide a range of perturbation parameters.
\inputminted[firstline=135, lastline=154]{MATLAB}{\pathToDDEBifToolDemos/vdpo_bt_transcritical/vanderPolOscillator.m}

\subsection{Bifurcation diagram}
The code below produces (a figure similar to) \cref{sm:fig:DoubleAlleeEffectCompareParameters}.
\inputminted[firstline=156, lastline=182]{MATLAB}{\pathToDDEBifToolDemos/vdpo_bt_transcritical/vanderPolOscillator.m}
\begin{figure}[ht]
    \centering
    \includetikzscaled{vanderPolOscillatorCompareParameters}
    \caption{Bifurcation diagram near the derived transcritical Bogdanov-Takens point in
        \cref{sm:eq:vanderPolOscillatorRescaled} comparing computed codimension one curves using
        \DDEBIFTOOL with the third-order homoclinic parameter asymptotics obtained
        in \cref{btdde:sec:transcritical_bt_homoclinic_asymptotics}.}
    \label{sm:fig:VDPOCompareParameters}
\end{figure}

\subsection{Compare homoclinic solutions in phase-space}
To obtain an impression of the third-order homoclinic asymptotics in
phase-space, we compare the corrected and uncorrected homoclinic solutions
with the perturbation parameter ranging from $0.01$ to $0.03$.
The code below produces (a figure similar to) \cref{sm:fig:VDPOCompareOrbitsPhaseSpace}.
We see that the corrected and predicted homoclinic orbits are nearly identical.
\inputminted[firstline=203, lastline=230]{MATLAB}{\pathToDDEBifToolDemos/vdpo_bt_transcritical/vanderPolOscillator.m}
\begin{figure}[ht!]
    \centering
    \includegraphics{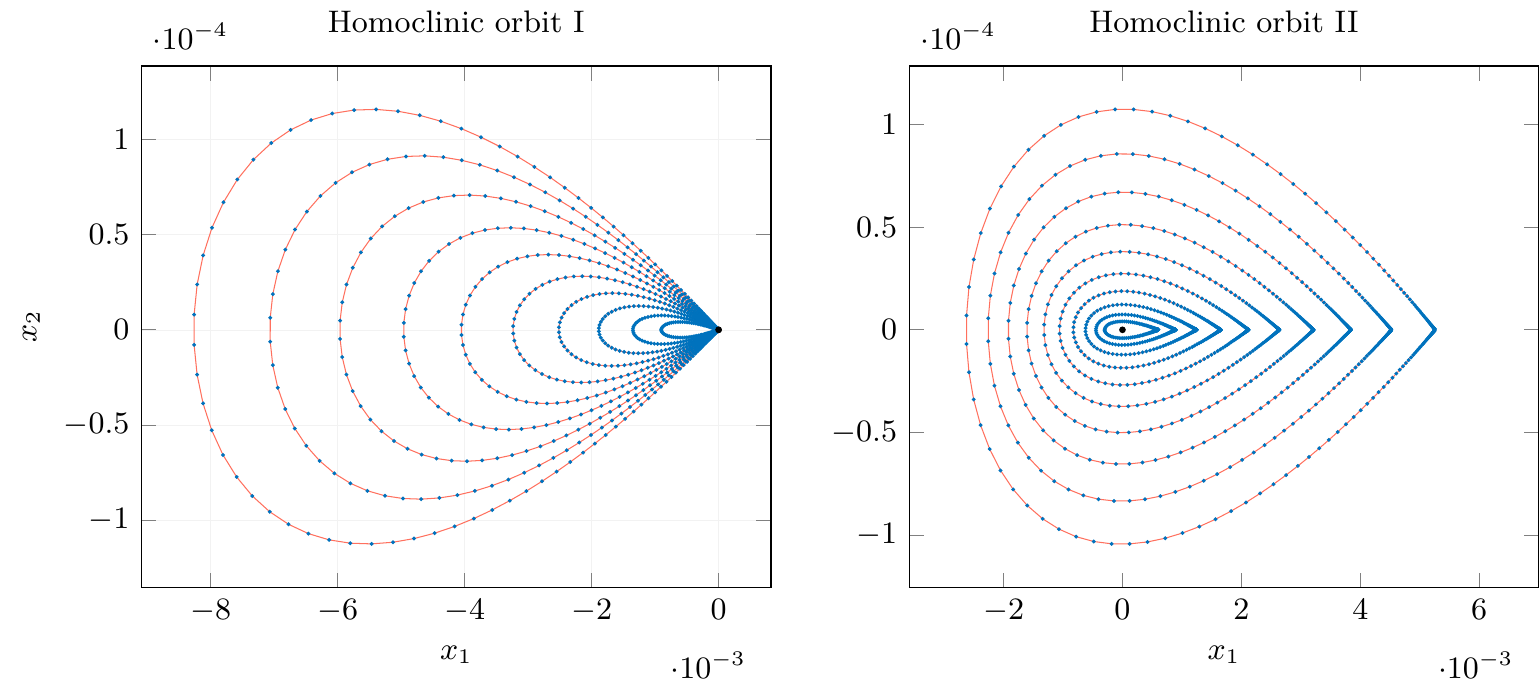} \\
    \vspace*{20pt}
    \includegraphics{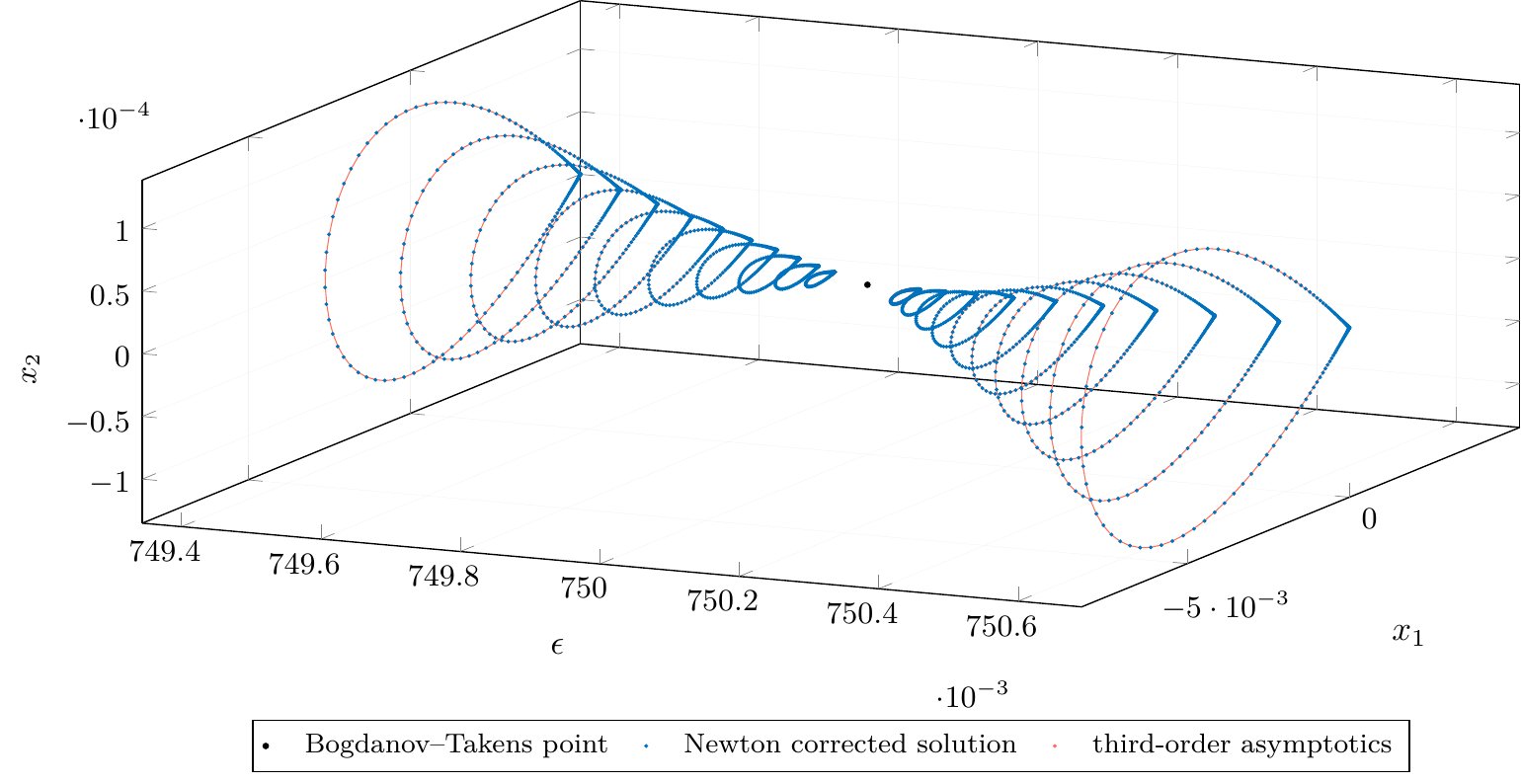}
    \caption{Plot comparing the third-order homoclinic asymptotics from
        \cref{btdde:sec:transcritical_bt_homoclinic_asymptotics} near the
        transcritical Bogdanov--Taken in
        \cref{sm:eq:vanderPolOscillatorRescaled} with the Newton correct
        homoclinic solutions phase-space with the perturbation parameter
        $\epsilon$ ranging from $0.01$ to $0.03$.}
    \label{sm:fig:VDPOCompareOrbitsPhaseSpace}
\end{figure}

\subsection{Convergence plot}
\label{sm:sec:vdpo:convergence_plot}
Using the function from \cref{sm:lst:convergence_plot}, we create a log-log
convergence plot comparing the convergence order of the first and third order
homoclinic asymptotics from \cref{btdde:sec:transcritical_bt_homoclinic_asymptotics}.
The code below yields \cref{sm:fig:VDPOConvergencePlot}.
\inputminted[firstline=232, lastline=243]{MATLAB}{\pathToDDEBifToolDemos/vdpo_bt_transcritical/vanderPolOscillator.m}
\begin{figure}[ht]
    \centering
    \ifcompileimages%
        \tikzsetnextfilename{VDPOConvergencePlot}%
        \input{tikz/VDPOConvergencePlot}%
    \else
        \includegraphics{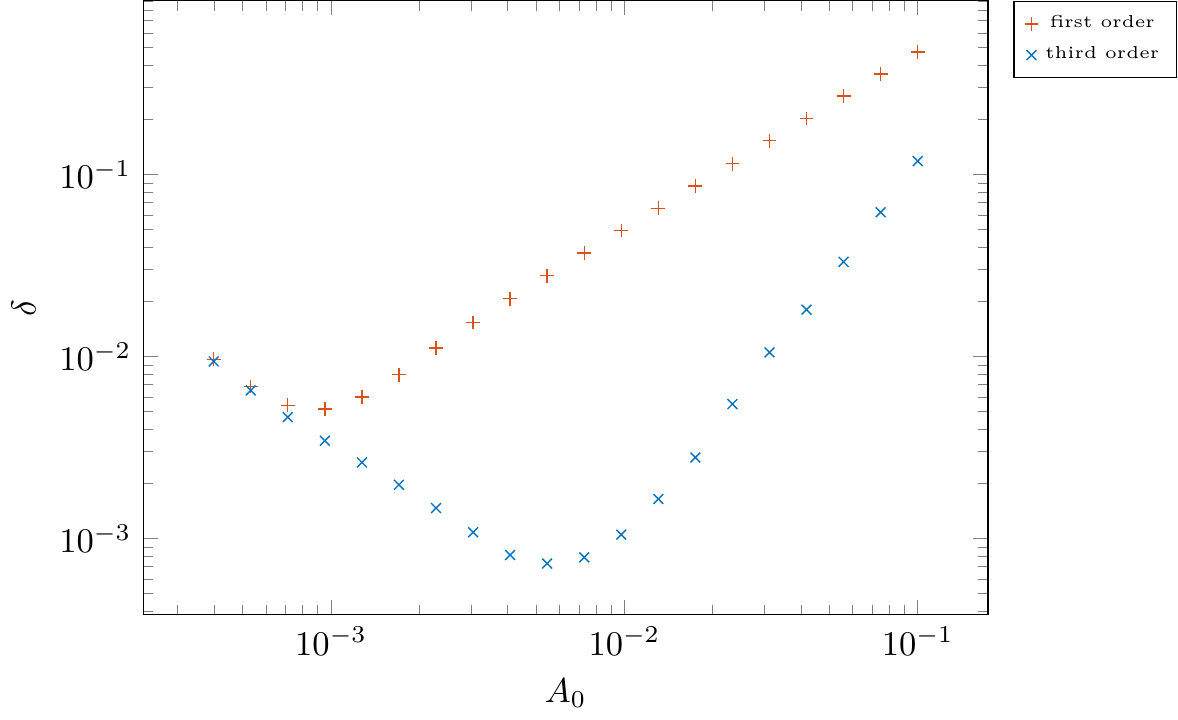}
    \fi

        \caption{On the abscissa is the approximation to the amplitude $A_0$ and on
        the ordinate the relative error $\delta$ between the constructed solution
        \mintinline{MATLAB}{hcli_pred} to the defining system for the homoclinic orbit
        and the Newton corrected solution \mintinline{MATLAB}{hcli_corrected}.}
    \label{sm:fig:VDPOConvergencePlot}
\end{figure}

\subsection{Simulation with {\tt DifferentialEquations.jl}}
Here we will perform four simulations. The first two simulations will be at two
homoclinic orbits located on the two homoclinic curves emanating from the
transcritical Bogdanov--Takens point continued with \DDEBIFTOOL, see
\cref{sm:fig:VDPOSimulationHomoclinic}. The second two simulations will be in
the regions where there should be stable periodic orbits, see
\cref{sm:fig:VDPOPeriodicSimulation}.

\begin{figure}[ht]
    \centering
    \includegraphics{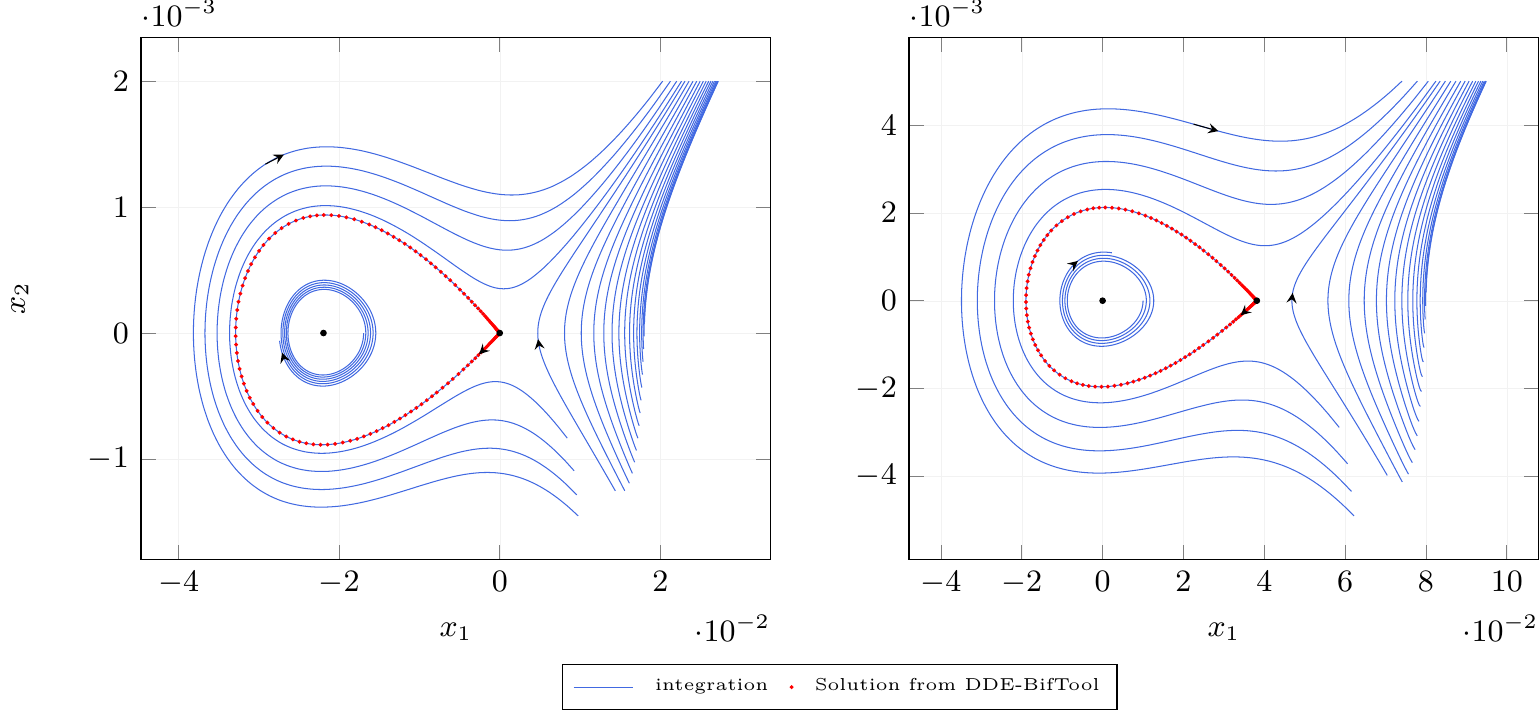}
    \caption{Comparing the computed homoclinic orbits in \cref{sm:eq:vanderPolOscillatorRescaled}
    with \DDEBIFTOOL with the solutions obtained from numerical simulation with Julia.
    We see the numerical integrated solution
    going through all the red points from the solution from \DDEBIFTOOL.}
    \label{sm:fig:VDPOSimulationHomoclinic}
\end{figure}

\begin{figure}[ht]
    \centering
    \includegraphics{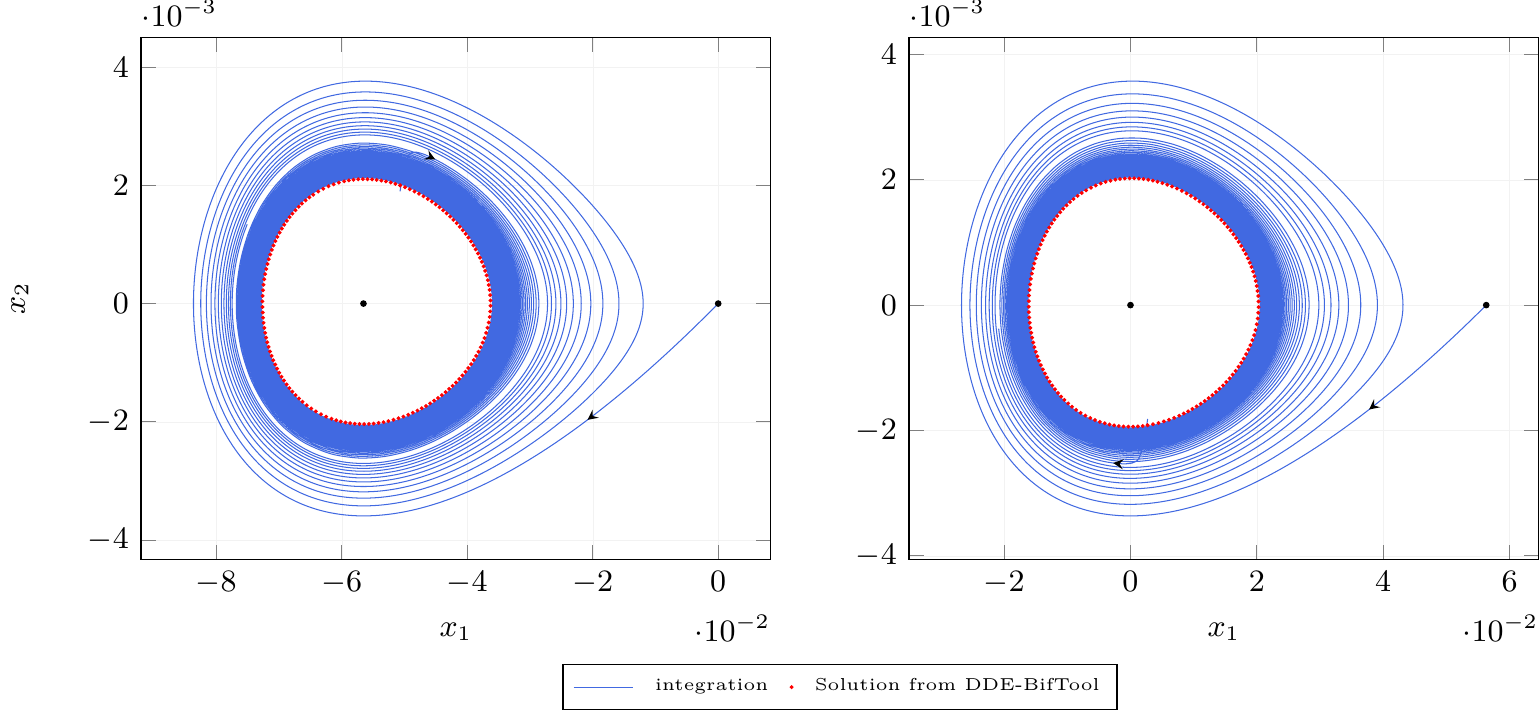}
    \caption{Comparing the computed periodic orbits in \cref{sm:eq:vanderPolOscillatorRescaled}
    with \DDEBIFTOOL with the solutions obtained from numerical simulation with Julia.
    We see the numerical integrated solution
    going through all the red points from the solution from \DDEBIFTOOL.}
    \label{sm:fig:VDPOPeriodicSimulation}
\end{figure}

\subsubsection{Loading necessary Julia packages}
Since we do not have  analytical expressions for the equilibria, we load the
same Julia packages as in the previous demonstration, see
\cref{sm:lst:neuralNetworkLoadingPacakges}.

\subsubsection{Define system}
We define the system to be integrated and also an allocating version used for
stability calculations.
\inputminted[firstline=8, lastline=30]{julia}{\pathToJuliaFiles/vdpo_simulation_article.jl}

\subsubsection{Functions for plotting arrows}
We define a function to show in which direction the orbits flow, which is
useful when plotting in phase-space. We also define functions to show the
direction of the leading eigenvectors of the characteristic matrix.
The code is shown in \cref{sm:lst:arrow_fucntions}.

\subsubsection{Function for creating streamlines plot}
To obtain an impression of the flow near transcritical Bogdanov--Takens point,
we create a streamlines function. This is particularly useful for seeing the
flow around the stable manifold of the saddle-note.
\inputminted[firstline=65, lastline=77]{julia}{\pathToJuliaFiles/vdpo_simulation_article.jl}

\subsubsection{Create figure with several axes}
We create a figure containing multiple axis in which we will plot 
the two homoclinic and two periodic orbits.
\inputminted[firstline=79, lastline=84]{julia}{\pathToJuliaFiles/vdpo_simulation_article.jl}

\subsubsection{Define parameters, equilibria}
We define parameters located on the continued homoclinic branch with
\DDEBIFTOOL. Then calculate the equilibria points in
\cref{sm:eq:vanderPolOscillatorRescaled} near the transcritical
Bogdanov--Takens point.
\inputminted[firstline=86, lastline=93]{julia}{\pathToJuliaFiles/vdpo_simulation_article.jl}

\subsubsection{Plot equilibria and homoclinic orbit}
By plotting the homoclinic orbit obtained with \DDEBIFTOOL located at parameter
values 
\[
    (\epsilon_0, \tau_0) = (0.752774810893411, 0.754736729675371),
\]
we can compare with the numerical simulations.
\inputminted[firstline=95, lastline=101]{julia}{\pathToJuliaFiles/vdpo_simulation_article.jl}

\subsubsection{Leading eigenvectors}
Next, we calculate and plot the leading eigenvectors of the characteristic matrix at the saddle-node bifurcation point.
\inputminted[firstline=103, lastline=117]{julia}{\pathToJuliaFiles/vdpo_simulation_article.jl}

\subsubsection{Define callback}
Since we are only interested in the flow near the equilibria points, we create a
discrete callback to ensure the orbits do not become too large.
\inputminted[firstline=122, lastline=125]{julia}{\pathToJuliaFiles/vdpo_simulation_article.jl}

\subsubsection{Integrate the system at homoclinic orbits I}
Now we define the problem to be integrated and choose the algorithm to be used.
Then we integrate the system for a range of initial history functions using the
function \mintinline{julia}{streamlines}. Next, we integrate the system near
the inner equilibrium, i.e., the equilibrium inside the homoclinic orbit. This
equilibria should be an unstable spiral. By using the unstable eigenvector of the
characteristic matrix, we obtain a solution going through the homoclinic solution
obtained with \DDEBIFTOOL.
\inputminted[firstline=127, lastline=152]{julia}{\pathToJuliaFiles/vdpo_simulation_article.jl}

\subsubsection{Add arrows on solutions}
Lastly, we add arrows to the obtained solutions using the function
\mintinline{julia}{draw_arrow_on_solution} defined above.
\inputminted[firstline=154, lastline=158]{julia}{\pathToJuliaFiles/vdpo_simulation_article.jl}
We should now obtain an interactive figure similar to the left figure in \cref{sm:fig:VDPOSimulationHomoclinic}.

\subsubsection{Simulation near stable periodic orbit I}
The code for numerical simulation near the second homoclinic orbit, see the
right plot in \cref{sm:fig:VDPOSimulationHomoclinic}, is almost identical to
the code above for the first homoclinic orbit and is therefore not included
here.

To show by integration the existence of an stable periodic orbit, we first
located a periodic orbit in \DDEBIFTOOL. This can be done by continuing a
branch of periodic orbits emanating from a point on the continued Hopf curves.
Then we load the profiles of the periodic orbits into Julia. We perform two
simulations. For the first simulation, we integrate with a constant history
function equal to a point inside the periodic orbit. The second starts from the
unstable eigenvector of the characteristic matrix calculated above.
\inputminted[firstline=230, lastline=279]{julia}{\pathToJuliaFiles/vdpo_simulation_article.jl}
After running the above code, we should obtain a similar plot as in \cref{sm:fig:VDPOPeriodicSimulation}.
\begin{remark}
By the intersection of the orbits in the first simulation, we see that, although
the system on the center manifold is equivalent to an ODE, the system we
integrate is still a DDE.
\end{remark}

\section[Tri-neuron BAM neural network model]
        {Transcritical Bogdanov--Takens bifurcation in a tri-neuron BAM neural network model}
We consider a three-component system of a tri-neuron bidirectional
associative memory (BAM) neural network model with multiple delays
\cite{dong2013bogdanov}. The architecture of this BAM model is illustrated in
\cref{sm:fig:BAM_architecture_graph}. 

\begin{figure}
\centering
\includetikzscaled[0.75]{BAM_architecture_graph}
\caption{The graph of architecture for model \cref{sm:eq:tri_neuron_BAM}}
\label{sm:fig:BAM_architecture_graph}
\end{figure}

In this model, there is only one neuron with the activation function
$f_{1}$ on the $I$-layer and there are two neurons with respective
activation functions $f_{2}$ and $f_{3}$ on the $J$-layer. We assume
that the time delay from the $I$-layer to the $J$-layer is $\tau_{1}$,
while the time delay from the $J$-layer to the $I$-layer is $\tau_{2}$.
Then the network can be described by the following delay differential equation:
\begin{equation}
\label{sm:eq:tri_neuron_BAM}
\begin{aligned}
\begin{cases}
\dot{x}_{1}(t) = -\mu_{1}x_{1}(t)+c_{21}f_{1}(x_{2}(t-\tau_{2}))+c_{31}f_{1}(x_{3}(t-\tau_{2})),\\
\dot{x}_{2}(t) = -\mu_{2}x_{2}(t)+c_{12}f_{2}(x_{1}(t-\tau_{1})),\\
\dot{x}_{3}(t) = -\mu_{3}x_{3}(t)+c_{13}f_{3}(x_{1}(t-\tau_{1})),
\end{cases}
\end{aligned}
\end{equation}
where:
\begin{itemize}
\item $x_{i}(t)\,(i=1,2,3)$ denote the state of the neuron at time $t$;
\item $\mu_{i}(i=1,2,3)$ describe the attenuation rate of internal neurons
processing on the $I$-layer and the $J$-layer and $\mu_{i}>0$;
\item the real constants $c_{i1}$and $c_{1i}\,(2,3)$ denote the neurons
in two layers: the $I$-layer and the $J$-layer.
\end{itemize}
Letting $u_{1}(t)=x_{1}(t-\tau_{1}),u_{2}(t)=x_{2}(t),u_{3}(t)=x_{3}(t)$
and $\tau=\tau_{1}+\tau_{2}$, then system \cref{sm:eq:tri_neuron_BAM}
is equivalent to the following system:

\begin{equation}
\label{sm:eq:tri_neuron_BAM-u}
\begin{cases}
\dot{u}_{1}(t) = -\mu_{1}u_{1}(t)+c_{21}f_{1}(u_{2}(t-\tau))+c_{31}f_{1}(u_{3}(t-\tau)),\\
\dot{u}_{2}(t) = -\mu_{2}u_{2}(t)+c_{12}f_{2}(u_{1}(t)),\\
\dot{u}_{3}(t) = -\mu_{3}u_{3}(t)+c_{13}f_{3}(u_{1}(t)).
\end{cases}
\end{equation}

\begin{lemma}
\label{sm:lem:BAM_double_eigenvalue}
Assume that $f_{i}(0)=0\,(i=1,2,3)$, $f_{i}'(0)\neq0\,(i=1,2,3)$ and
$\mu_{2}\neq\mu_{3}$, then the steady-state $(u_{1},u_{2},u_{3})=(0,0,0)$ has a
double zero eigenvalue at 
\begin{align*}
c_{21} & =c_{21}^{0}=\frac{\mu_{2}^{2}\left(\mu_{1}\left(\mu_{3}\tau+1\right)+\mu_{3}\right)}{c_{12}\left(\mu_{2}-\mu_{3}\right)f_{1}'(0)f_{2}'(0)},\\
c_{31} & =c_{31}^{0}=\frac{\mu_{3}^{2}\left(\mu_{1}\left(\mu_{2}\tau+1\right)+\mu_{2}\right)}{c_{13}\left(\mu_{3}-\mu_{2}\right)f_{1}'(0)f_{3}'(0)}.
\end{align*}
\end{lemma}
\begin{proof}
The characteristic matrix of \cref{sm:eq:tri_neuron_BAM-u} is given
by
\[
\Delta(\lambda)=\left(\begin{array}{ccc}
\lambda+\mu_{1} & -e^{-\lambda\tau}c_{21}f_{1}'(0) & -e^{-\lambda\tau}c_{31}f_{1}'(0)\\
-c_{12}f_{2}'(0) & \lambda+\mu_{2} & 0\\
-c_{13}f_{3}'(0) & 0 & \lambda+\mu_{3}

\end{array}\right).
\]
Thus, the characteristic equation becomes 
\begin{align}
\det\Delta(\lambda) & =\lambda^{3}+\left(\mu_{1}+\mu_{2}+\mu_{3}\right)\lambda^{2}+\big(-c_{12}c_{21}f_{1}'(0)f_{2}'(0)e^{-\lambda\tau}\nonumber \\
 & \qquad-c_{13}c_{31}f_{1}'(0)f_{3}'(0)e^{-\lambda\tau}+\mu_{1}\mu_{2}+\mu_{3}\mu_{2}+\mu_{1}\mu_{3}\big)\lambda\nonumber \\
 & \qquad+\mu_{1}\mu_{2}\mu_{3}-e^{-\lambda\tau}\left(c_{12}c_{21}\mu_{3}f_{2}'(0)+c_{13}c_{31}\mu_{2}f_{3}'(0)\right)f_{1}'(0)=0.\label{sm:eq:BAM_characteristic_eq}
\end{align}

Clearly, $\lambda=0$ is a root if and only if
\[
\mu_{1}\mu_{2}\mu_{3}=\left(c_{12}c_{21}\mu_{3}f_{2}'(0)+c_{13}c_{31}\mu_{2}f_{3}'(0)\right)f_{1}'(0).
\]
Taking the derivative of \cref{sm:eq:BAM_characteristic_eq} with respect
to $\lambda$ gives
\begin{align}
\dfrac{d}{d\lambda}\det\Delta(\lambda) & =3\lambda^{2}+2\left(\mu_{1}+\mu_{2}+\mu_{3}\right)\lambda+\big(-c_{12}c_{21}f_{1}'(0)f_{2}'(0)e^{-\lambda\tau}\nonumber \\
 & \qquad-c_{13}c_{31}f_{1}'(0)f_{3}'(0)e^{-\lambda\tau}+\mu_{1}\mu_{2}+\mu_{3}\mu_{2}+\mu_{1}\mu_{3}\big)\nonumber \\
 & \qquad+\tau\left(c_{12}c_{21}f_{2}'(0)e^{-\lambda\tau}+c_{13}c_{31}f_{3}'(0)e^{-\lambda\tau}\right)f_{1}'(0)\lambda\nonumber \\
 & \qquad+\tau e^{-\lambda\tau}\left(c_{12}c_{21}\mu_{3}f_{2}'(0)+c_{13}c_{31}\mu_{2}f_{3}'(0)\right)f_{1}'(0)=0.\label{sm:eq:BAM_characteristic_eq-derivative}
\end{align}
Therefore, we have
\begin{align*}
    \det\Delta'(0) &= \big(-c_{12}c_{21}f_{1}'(0)f_{2}'(0)-c_{13}c_{31}f_{1}'(0)f_{3}'(0)+\mu_{1}\mu_{2}+\mu_{3}\mu_{2}+\mu_{1}\mu_{3}\big)=0.
\end{align*}

For any $\tau>0$, it is easy to see that $\det\Delta(\lambda)=\det\Delta'(\lambda)=0$,
if and only if the following conditions are satisfied
\begin{equation}
\begin{cases}
\left((1-\tau\mu_{3})c_{12}c_{21}f_{2}'(0)+(1-\tau\mu_{2})c_{13}c_{31}f_{3}'(0)\right)f_{1}'(0)=\mu_{1}\mu_{2}+\mu_{3}\mu_{2}+\mu_{1}\mu_{3},\\
\\
\left(c_{12}c_{21}\mu_{3}f_{2}'(0)+c_{13}c_{31}\mu_{2}f_{3}'(0)\right)f_{1}'(0)=\mu_{1}\mu_{2}\mu_{3}.
\end{cases}\label{sm:eq:BAM_double_eigvalue_zero_condition}
\end{equation}
By solving \cref{sm:eq:BAM_double_eigvalue_zero_condition} for $(c_{21},c_{31})$
we get $(c_{21},c_{31})=(c_{21}^{0},c_{31}^{0})$.
Taking the derivative of \cref{sm:eq:BAM_characteristic_eq-derivative}
yields
\begin{align}
\dfrac{d^{2}}{d\lambda^{2}}\det\Delta(\lambda) & =6\lambda+2\left(\mu_{1}+\mu_{2}+\mu_{3}\right)+\tau f_{1}'(0)\big(c_{12}c_{21}f_{2}'(0)e^{-\lambda\tau}+c_{13}c_{31}f_{3}'(0)e^{-\lambda\tau}\big)\nonumber \\
 & \qquad+\tau\left(c_{12}c_{21}f_{2}'(0)e^{-\lambda\tau}+c_{13}c_{31}f_{3}'(0)e^{-\lambda\tau}\right)f_{1}'(0)\nonumber \\
 & \qquad-\tau^{2}\left(c_{12}c_{21}f_{2}'(0)e^{-\lambda\tau}+c_{13}c_{31}f_{3}'(0)e^{-\lambda\tau}\right)f_{1}'(0)\lambda\nonumber \\
 & \qquad-\tau^{2}e^{-\lambda\tau}\left(c_{12}c_{21}\mu_{3}f_{2}'(0)+c_{13}c_{31}\mu_{2}f_{3}'(0)\right)f_{1}'(0)=0.\label{sm:eq:BAM_characteristic_eq-derivative-1}
\end{align}
Then we can obtain
\begin{align*}
 & \dfrac{d^{2}}{d\lambda^{2}}\det\Delta(0)\vert_{(c_{21},c_{31})=(c_{21}^{0},c_{31}^{0})}\\
 & \quad=2\left(\mu_{1}+\mu_{2}+\mu_{3}\right)+2\tau f_{1}'(0)\big(c_{12}c_{21}^{0}f_{2}'(0)+c_{13}c_{31}^{0}f_{3}'(0)\big)\\
 & \qquad-\tau^{2}f_{1}'(0)\left(c_{12}c_{21}^{0}\mu_{3}f_{2}'(0)+c_{13}c_{31}^{0}\mu_{2}f_{3}'(0)\right)\\
 & \quad=2\left(\mu_{1}+\mu_{2}+\mu_{3}\right)+\tau\left(\frac{\mu_{2}^{2}\left(\mu_{1}\left(\mu_{3}\tau+1\right)+\mu_{3}\right)}{\left(\mu_{2}-\mu_{3}\right)}+\frac{\mu_{3}^{2}\left(\mu_{1}\left(\mu_{2}\tau+1\right)+\mu_{2}\right)}{\left(\mu_{3}-\mu_{2}\right)}\right)\\
 & \qquad-\tau^{2}\left(\frac{\mu_{2}^{2}\left(\mu_{1}\left(\mu_{3}\tau+1\right)+\mu_{3}\right)}{\left(\mu_{2}-\mu_{3}\right)}\mu_{3}+\frac{\mu_{3}^{2}\left(\mu_{1}\left(\mu_{2}\tau+1\right)+\mu_{2}\right)}{\left(\mu_{3}-\mu_{2}\right)}\mu_{2}\right)\\
 & \quad=2\left(\mu_{1}+\mu_{2}+\mu_{3}\right)+2\tau\left(\mu_{1}\mu_{2}+\mu_{1}\mu_{3}+\mu_{2}\mu_{3}\right)+\tau^{2}\mu_{1}\mu_{2}\mu_{3}.
\end{align*}

Since $\tau>0$ and $\mu_{i}>0 (i=1,2,3)$ the second derivative of
the characteristic equations at $(\lambda,c_{21},c_{31})=(0,c_{21}^{0},c_{31}^{0})$
does not vanish, and we obtain a double zero eigenvalue.
\end{proof}

\begin{lemma}
\label{sm:lemma:triNeuralBAMNetworkModelEigenvalues}
\textup{Correction to \cite[Lemma 3]{dong2013bogdanov}}
Let $(c_{21},c_{31})=(c_{21}^{0},c_{31}^{0})$,
\begin{equation}
    \label{sm:eq:triNeuralBAMNetworkModel:omega_0} 
    \omega_0 = \frac{\sqrt{-\mu_1^2 - \mu_2^2 - \mu_3^2 + \sqrt{\zeta_0}}}{\sqrt{2}}
\end{equation}
and $0<\tau<\tau_{0}$, where $\tau_0$ is the minimum positive solution to the nonlinear equation
\begin{equation}
    \label{sm:eq:triNeuralBAMNetworkModel:tan} 
    \tan (\tau \omega_0) = \frac{b_0\zeta_1 - a_0\zeta_2}{a_0\zeta_1 + b_0\zeta_2},
\end{equation}
with
\begin{align*}
a_0 &= -\mu_1\mu_2\mu_3, \\ 
b_0 &= -\omega_0(\mu_2\mu_3 + \mu_1(\mu_2 + \mu_3 + \mu_2\mu_3\tau)), \\
\zeta_0 &= \mu_1^4 + (\mu_2^2 + \mu_3^2)^2 + 8\mu_1\mu_2\mu_3(\mu_2 + \mu_3 + \mu_2\mu_3\tau) + \\
        &\qquad 2\mu_1^2(\mu_3^2 + 4\mu_2\mu_3(1 + \mu_3\tau) + \mu_2^2(1 + 2\mu_3\tau(2 + \mu_3\tau))), \\
\zeta_1 &= \mu_1\mu_2\mu_3 - (\mu_1 + \mu_2 + \mu_3)\omega_0^2, \\
\zeta_2 &= \mu_2\mu_3\omega_0 + \mu_1(\mu_2 + \mu_3)\omega_0 - \omega_0^3.
\end{align*}
Then all roots of the characteristic equation \cref{sm:eq:BAM_characteristic_eq},
except the double zero roots, have negative real parts.
\end{lemma}

\begin{proof}
Consider that there are eigenvalues $\lambda\neq0$ on the imaginary
axis for $(c_{21}^{0},c_{31}^{0})$. Substituting $\lambda=i\omega,(\omega>0)$
and $(c_{21}^{0},c_{31}^{0})$ into \cref{sm:eq:BAM_characteristic_eq},
and rearranging terms according to its real and imaginary part yields
\begin{equation}
\label{sm:eq:BAM_real_imag_parts_char_eq}
-\begin{pmatrix}
    \zeta_1 \\
    \zeta_2 
\end{pmatrix}
=
\begin{pmatrix}
    a_0 & \phantom{-}b_0 \\
    b_0 & -a_0
\end{pmatrix}
\begin{pmatrix}
\cos\tau\omega \\
\sin\tau\omega
\end{pmatrix}.
\end{equation}
By squaring and adding the above equations, it follows that
\begin{equation}
    \label{sm:eq:BAM_omega}
    \zeta_1^2 + \zeta_2^2 = a_0^2 + b_0^2.
\end{equation}
Solving the equation for positive $\omega>0$ yields \cref{sm:eq:triNeuralBAMNetworkModel:omega_0}.

Next, from \cref{sm:eq:BAM_real_imag_parts_char_eq}, we obtain \cref{sm:eq:triNeuralBAMNetworkModel:tan}.
First, notice that $\tau \mapsto \omega_0 \tau$ is a strictly increasing
function since it is the product of two strictly increasing positive functions.
It follows that $\tau \mapsto \tan \omega_0 \tau$ is periodic in $\tau$ with
range $\mathbb R$. The next observation is that the denominator in the right-hand
side of \cref{sm:eq:triNeuralBAMNetworkModel:tan} has a unique positive root $\tau=\tau_1$, at which
that numerator does not vanish. In fact, it can be checked that the numerator
does never vanish. Lastly, taking the limit of the right-hand side of
\cref{sm:eq:triNeuralBAMNetworkModel:tan} of $\tau$ to infinity is $0$, i.e.
\[
    \lim_{\tau \rightarrow \infty} \frac{-b_0\zeta_1 + a_0\zeta_2}{\phantom{-}a_0\zeta_1 + b_0\zeta_2} = 0.
\]
By the continuity of the right-hand side, it follows that \cref{sm:eq:triNeuralBAMNetworkModel:tan} has
countable many solutions in $\tau$. Let $\tau_0$ be the minimum positive 
solution. Since for $\tau=0$ all solutions to the characteristic equation, except
of the double zero eigenvalue at the origin, have negative real parts.
We conclude by \cite[Corollary 2.3]{Ruan@2001} that
all eigenvalues, except for the double zero eigenvalues, are located
in the left half plane for $0 < \tau < \tau_0$.
\end{proof}

\begin{remark}
    Note that, although $(\omega,\tau) = (\omega_0,\tau_0)$, solves \cref{sm:eq:triNeuralBAMNetworkModel:tan,sm:eq:BAM_omega},
    this not necessary means it solves the original equations \cref{sm:eq:BAM_real_imag_parts_char_eq}.
    Thus, the center manifold may still be stable for values $\tau>\tau_0$. We will demonstrate this
    in the example below.
\end{remark}

For the numerical verification we consider, as in the simulations in
\cite[Example 1]{dong2013bogdanov}, the system \cref{sm:eq:tri_neuron_BAM-u} with
the activation functions
\begin{equation}
    \label{sm:eq:triNeuralBAMNetworkModelFunctions}
    f_{1}(x)=\tanh(x)+0.1x^{2},\quad f_{2}(x)=f_{3}(x)=\tanh(x),
\end{equation}
and parameter values
\begin{equation}
    \label{sm:eq:triNeuralBAMNetworkModelFixedParameters}
    \mu_{1}=0.1,\mu_{2}=0.3,\mu_{3}=0.2,c_{12}=c_{13}=1,\tau=5.
\end{equation}
Then, from \cref{sm:lem:BAM_double_eigenvalue}, we obtain two critical
values 
\[
(c_{21}^{0},c_{31}^{0})=(0.36,-0.22),
\]
at which there is a transcritical Bogdanov--Takens point. Furthermore, since
$\tau < \tau_0 \approx 5.4320$ the center manifold is locally attractive. In
fact, we will show below that the center manifold is locally attractive for
$0<\tau<13.2309348879375$.

\begin{remark}
    The \MATLAB files for this demonstration can be found in the directory
    \mintinline[breaklines,breakafter=/]{MATLAB}{demos/tutorial/VII/BAM_neural_network_model} relative to the main
    directory of the package \DDEBIFTOOL. Here, we omit the code to generate a
    system file. The system file\\
    \mintinline{MATLAB}{sym_BAMnn_mf.m} has been
    generated with the script \mintinline{MATLAB}{sym_BAMnn.m}. Also, we assume
    that the \DDEBIFTOOL package has been loaded as in
    \cref{sm:lst:searchpath}. The code in
    \crefrange{sm:sec:tri_neuron_BAM:pars_and_funcs}
              {sm:sec:tri_neuron_BAM:convergence_plot}
    highlights the important parts of the file
    \mintinline{MATLAB}{BAMnn.m}. 
\end{remark}

\subsection{Set parameter names and funcs structure} 
\label{sm:sec:tri_neuron_BAM:pars_and_funcs}
As in the previous example, we set the parameter names and define the \mintinline{MATLAB}{funcs} structure.
\inputminted[firstline=27, lastline=33]{MATLAB}{\pathToDDEBifToolDemos/BAM_neural_network_model/BAMnn.m}

\subsection{Set parameter range}
Since we are only interested here in the local unfolding, we restrict the
allowed parameter range for the unfolding parameters. In practice, one may have
physical restrictions which must be satisfied. Additionally, we also limit the
maximum allowed step size during continuation. By doing so, we obtain more refined
data to compare against our predictors.
\inputminted[firstline=35, lastline=38]{MATLAB}{\pathToDDEBifToolDemos/BAM_neural_network_model/BAMnn.m}

\subsection{Stability and coefficients of the transcritical Bogdanov--Takens point}
We manually construct a steady-state at the transcritical Bogdanov--Takens
point and calculate its stability.
\inputminted[firstline=40, lastline=49]{MATLAB}{\pathToDDEBifToolDemos/BAM_neural_network_model/BAMnn.m}

The \MATLAB console shows the following output.
\begin{minted}{shell-session}
ans =

  -0.0000 + 0.0000i
  -0.0000 - 0.0000i
  -0.2246 + 0.6600i
  -0.2246 - 0.6600i
  -0.6371 + 1.8063i
  -0.6371 - 1.8063i
  -0.8483 + 3.0681i
  -0.8483 - 3.0681i
  -0.9849 + 4.3336i
  -0.9849 - 4.3336i
\end{minted}
The eigenvalues confirm that the point under consideration is indeed (an
approximation to) a Bogdanov--Takens point. Furthermore, the remaining eigenvalues have
negative real parts. Next, we calculate the normal form coefficients, the
time-reparametrization, and the transformation to the center manifold with the
function \mintinline{MATLAB}{nmfm_bt_orbital}, which implements the coefficients as derived in
\cref{btdde:sec:transcritical-Bogdanov-Takens}. For this, we need to set the argument
\mintinline{MATLAB}{free_pars} to the unfolding parameter $(\alpha_1,\alpha_2)$. These
coefficients will be used to start the continuation of the codimension one branches
emanating from the Bogdanov--Takens point. Also, since we are in the transcritical case,
we set the argument \mintinline{matlab}{generic_unfolding} to \mintinline{matlab}{false}.
\inputminted[firstline=51, lastline=55]{MATLAB}{\pathToDDEBifToolDemos/BAM_neural_network_model/BAMnn.m}

The \MATLAB console shows the following output.
\begin{minted}{shell-session}

ans =

  struct with fields:

          a: 0.0012
          b: -0.0135
  theta1000: -2.5813
  theta0010: 3.1322e+03
  theta0001: -190.0753
       phi0: [1x1 struct]
       phi1: [1x1 struct]
      h2000: [1x1 struct]
      h1100: [1x1 struct]
      h0200: [1x1 struct]
      h3000: [1x1 struct]
      h2100: [1x1 struct]
        K10: [2x1 double]
        K01: [2x1 double]
        K02: [2x1 double]
        K11: [2x1 double]
        K20: [2x1 double]
      h1010: [1x1 struct]
      h1001: [1x1 struct]
      h0110: [1x1 struct]
      h0101: [1x1 struct]
      h2010: [1x1 struct]
      h1110: [1x1 struct]
      h2001: [1x1 struct]
      h1101: [1x1 struct]
      h1002: [1x1 struct]
      h0102: [1x1 struct]
      h1020: [1x1 struct]
      h0120: [1x1 struct]
      h1011: [1x1 struct]
      h0111: [1x1 struct]
          K: @(beta1,beta2)K10*beta1+K01*beta2+1/2*K20*beta1^2
                    +K11*beta1*beta2+1/2*K02*beta2^2
          H: [function_handle]

\end{minted}
Since the sign of $ab$ is negative, we expect to find stable periodic orbits nearby the 
Bogdanov--Takens point.

\subsection{Comparing profiles of computed and predicted homoclinic orbits}
To test the homoclinic asymptotics from
\cref{btdde:sec:generic_bt_homoclinic_asymptotics} we compare the first and third
order asymptotics to the Newton corrected solution. For this, we use the 
function \mintinline[breaklines,breakafter=_]{MATLAB}{C1branch_from_C2point}. This function returns a branch, which
by default returns two initial corrected approximations in order to start continuation of the
codimension one curve under consideration. By setting the argument
\mintinline{MATLAB}{'predictor'} to \mintinline{MATLAB}{true} the approximations are left uncorrected.
To make the comparison visually clear, we set the perturbation parameter 
$\epsilon=0.02$ (\mintinline{MATLAB}{step = 0.02} in the code below).
The code below produces \cref{sm:fig:TriNeuronBAMCompareProfilesI,sm:fig:TriNeuronBAMCompareProfilesII}.
The difference between the two approximations is clearly noticeable. While
the first order asymptotics is close to the Newton corrected solution, the third
order asymptotics is indistinguishable at this scale from the Newton corrected
solution.
\inputminted[firstline=57, lastline=81]{MATLAB}{\pathToDDEBifToolDemos/BAM_neural_network_model/BAMnn.m}

\begin{figure}[ht]
    \centering
    \includegraphics{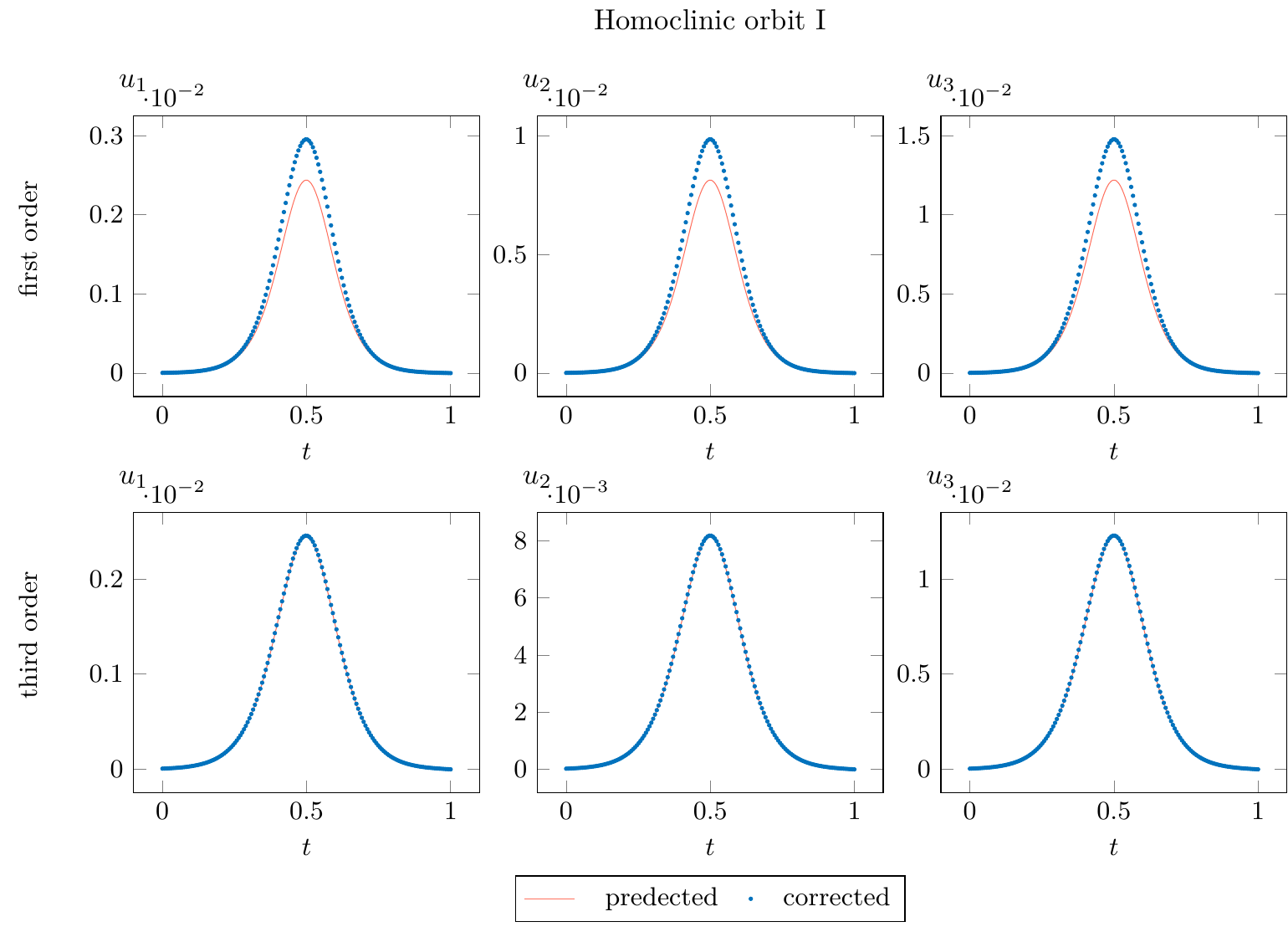}
    \caption{Comparison between the first and third-order homoclinic asymptotics from
    \cref{btdde:sec:transcritical_bt_homoclinic_asymptotics} near the transcritical
        Bogdanov--Takens bifurcation in \cref{sm:eq:tri_neuron_BAM} with the
        perturbation parameter set to $\epsilon=0.02$.}
        \label{sm:fig:TriNeuronBAMCompareProfilesI}
\end{figure}

\begin{figure}[ht]
    \centering
    \includegraphics{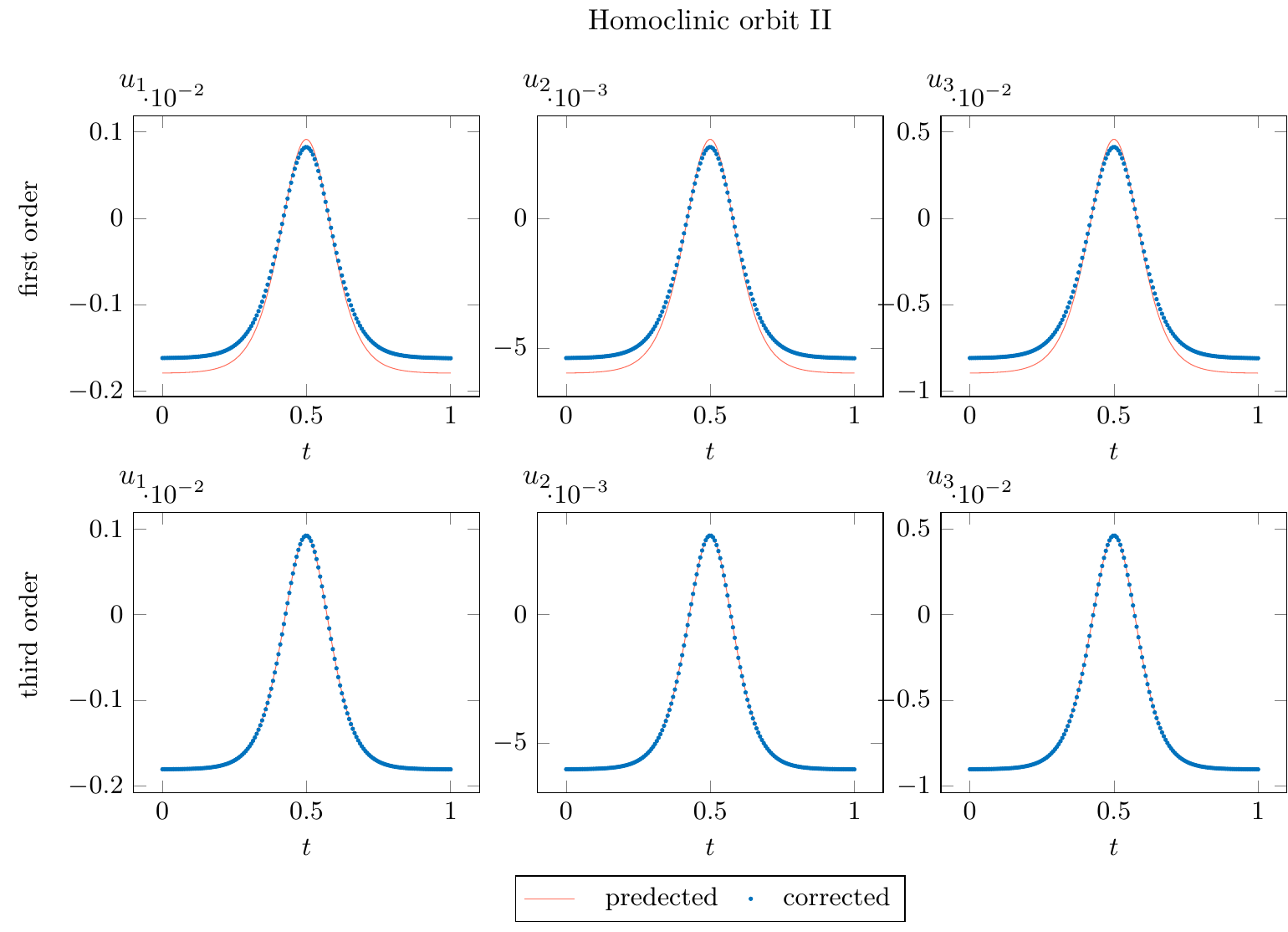}
    \caption{Comparison between the first and third-order homoclinic asymptotics from
    \cref{btdde:sec:transcritical_bt_homoclinic_asymptotics} near the transcritical
        Bogdanov--Takens bifurcation in \cref{sm:eq:tri_neuron_BAM} with the
        perturbation parameter set to $\epsilon=0.02$.}
        \label{sm:fig:TriNeuronBAMCompareProfilesII}
\end{figure}

\subsection{Continuation of the codimension one curves emanating}
To continue the three codimension one curves emanating from the generic
Bogdanov--Takens point, we can simply use the function
\mintinline{MATLAB}{C1branch_from_C2point}, as shown in the code below. To monitor the
continuation process, the argument \mintinline{MATLAB}{plot} must be set to \mintinline{MATLAB}{1}.
The most important setting is the perturbation parameter (or multiple),
\mintinline{MATLAB}{step} in the code below. If left out, default step sizes are defined.
However, depending on the problem, no convergence may then be obtained.
\inputminted[firstline=83, lastline=103]{MATLAB}{\pathToDDEBifToolDemos/BAM_neural_network_model/BAMnn.m}

\subsection{Predictors of the codimension one curves emanating from the Bogdanov--Takens point}
Before we provide the bifurcation diagram in the next section, we first obtain the predictors
for the codimension one curves. For this, we again use the function
\mintinline{MATLAB}{C1branch_from_C2point}. We set the argument \mintinline{MATLAB}{predictor} to \mintinline{MATLAB}{1}
and provide a range of perturbation parameters.
\inputminted[firstline=127, lastline=146]{MATLAB}{\pathToDDEBifToolDemos/BAM_neural_network_model/BAMnn.m}

\subsection{Bifurcation diagram}
The code below produces (a figure similar to) \cref{sm:fig:triNeuronBAMNeuralNetworkModelCompareParametersSupplementI}.
\inputminted[firstline=148, lastline=175]{MATLAB}{\pathToDDEBifToolDemos/BAM_neural_network_model/BAMnn.m}
\begin{figure}[ht]
\includetikzscaled{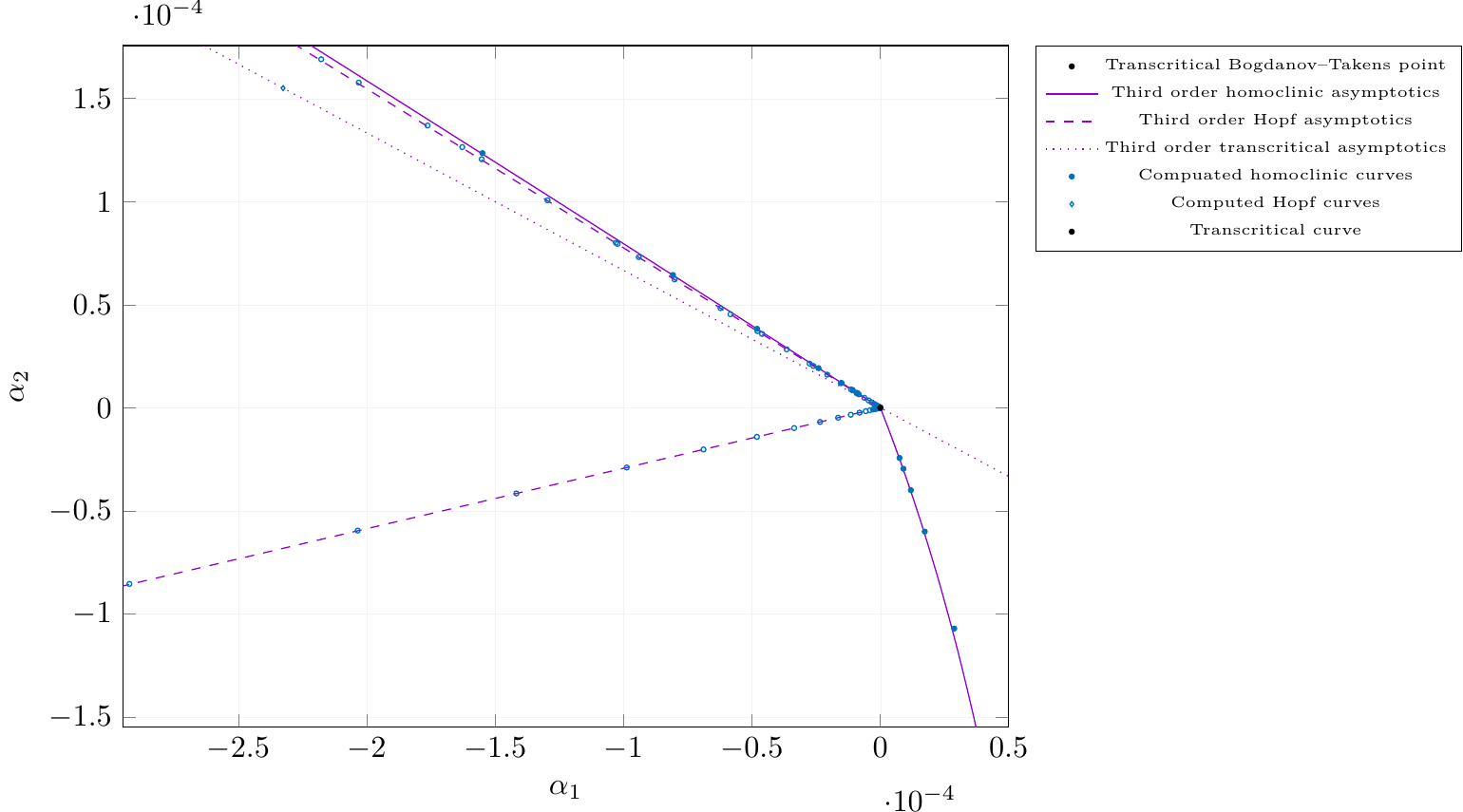}
\caption{Bifurcation diagram near the analytically derived transcritical
    Bogdanov--Takens point in \cref{sm:eq:tri_neuron_BAM} comparing the
    computed codimension one curves emanating form the Bogdanov--Takens point
    using \DDEBIFTOOL with the third-order homoclinic parameter asymptotics
    obtained in \cref{btdde:sec:transcritical_bt_homoclinic_asymptotics}.}
\label{sm:fig:triNeuronBAMNeuralNetworkModelCompareParametersSupplementI}
\end{figure}

\subsection{Detect bifurcations on the second Hopf branch}
We can use the \DDEBIFTOOL function \mintinline{MATLAB}{LocateSpecialPoints} to
locate bifurcation points on the second Hopf branch.
\inputminted[firstline=177, lastline=178]{MATLAB}{\pathToDDEBifToolDemos/BAM_neural_network_model/BAMnn.m}

Inspecting the \MATLAB output gives.
\begin{minted}{shell-session}
HopfCodimension2: calculate stability if not yet present
HopfCodimension2: calculate L1 coefficients
HopfCodimension2: (provisional) 2 gen. Hopf 2 Takens-Bogdanov  detected.
br_insert: detected 1 of 4: genh. Normalform:
    L2: 5.7933e+03
    L1: 1.0106e-09

br_insert: detected 2 of 4: BT. Normalform:
    a2: -5.8499e-04
    b2: -0.0031

br_insert: detected 3 of 4: genh. Normalform:
    L2: -5.7933e+03
    L1: -2.1214e-09

br_insert: detected 4 of 4: BT. Normalform:
    a2: -0.0012
    b2: 0.0135
\end{minted}
Thus, there are two Bogdanov--Takens points detected. By inspection of the normal
form coefficients $a$ and $b$ (\mintinline{MATLAB}{a2} and
\mintinline{MATLAB}{b2} in the output above) with the coefficients of the
transcritical Bogdanov--Takens point, we see that one is (very likely) already
known. Indeed, while continuing the second Hopf curve from the transcritical
Bogdanov--Takens point, we encounter another Bogdanov--Takens point, at which
the Hopf curve turns around, and continues in the reverse direction, back to
the transcritical Bogdanov--Takens point. Similarly, we also deduce that there
is only one generalized Hopf point detected on the Hopf curve. Inspecting 
the parameters indeed confirm our claim.

Since the newly detected Bogdanov--Takens point is on the Hopf curve
on which the equilibrium changes position under variation of the
parameters, we extract the Bogdanov--Takens point and computed the coefficients
of the generic case.
\inputminted[firstline=179, lastline=181]{MATLAB}{\pathToDDEBifToolDemos/BAM_neural_network_model/BAMnn.m}

\subsection{Plot Bogdanov-Takens test function along the Hopf curve}
Before we continue the homoclinic branch emanating from the generic
Bogdanov--Takens point, we first plot the test function for the Bogdanov--Takens
point, i.e. we plot the imaginary part for the critical eigenvalues along the
second Hopf curve. The code below produces (a figure similar to)
\cref{sm:fig:triNeuronBAMNeuralNetworkModelTestfunction}.
\inputminted[firstline=185, lastline=197]{MATLAB}{\pathToDDEBifToolDemos/BAM_neural_network_model/BAMnn.m}
\begin{figure}[ht]
    \centering
    \includetikzscaled{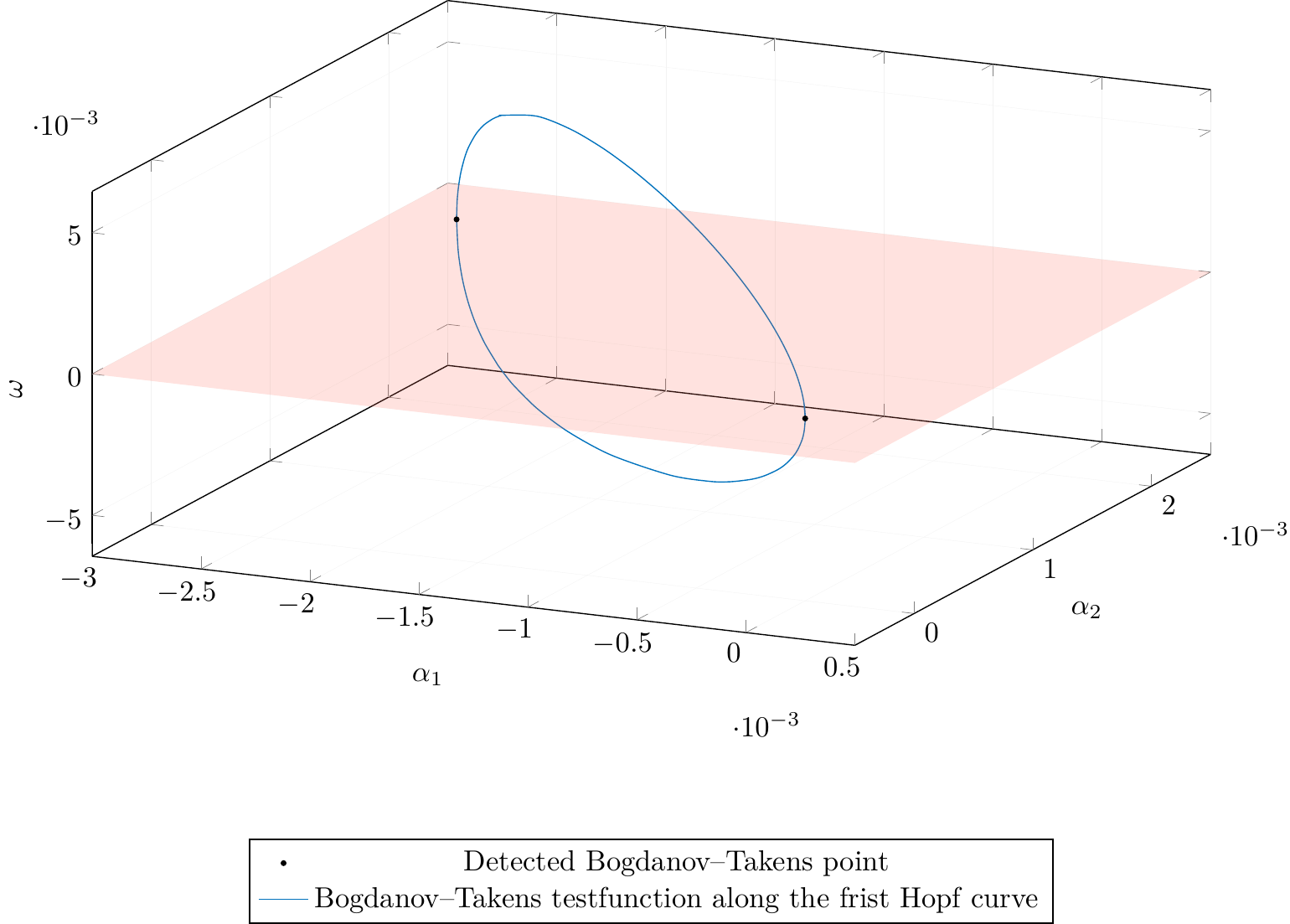}
    \caption{Plot of the Bogdanov--Takens test function along the second Hopf curve. 
    We see that surface $\omega=0$ is intersected transversally two times
    while continuing the Hopf curve.}
    \label{sm:fig:triNeuronBAMNeuralNetworkModelTestfunction}
\end{figure}

\subsection{Continue third homoclinic curve}
The code below continuous the homoclinic solution emanating from the generic Bogdanov--Takens detected
on the second Hopf branch. The last line extracts the parameters used for plotting.
\inputminted[firstline=199, lastline=205]{MATLAB}{\pathToDDEBifToolDemos/BAM_neural_network_model/BAMnn.m}

\subsection{Bifurcation plot}
We are now in the position to recreate the bifurcation plot as shown in the main text. There, we
left out the predictors for the codimension one equilibria bifurcation curves and changed the color
of the computed codimension one equilibria curves to gray. In this way, focus will be on
the homoclinic curves. The code below produces (a figure similar to)
\cref{sm:fig:triNeuronBAMNeuralNetworkModelCompareParameters}.
\inputminted[firstline=213, lastline=243]{MATLAB}{\pathToDDEBifToolDemos/BAM_neural_network_model/BAMnn.m}
\begin{figure}[ht]
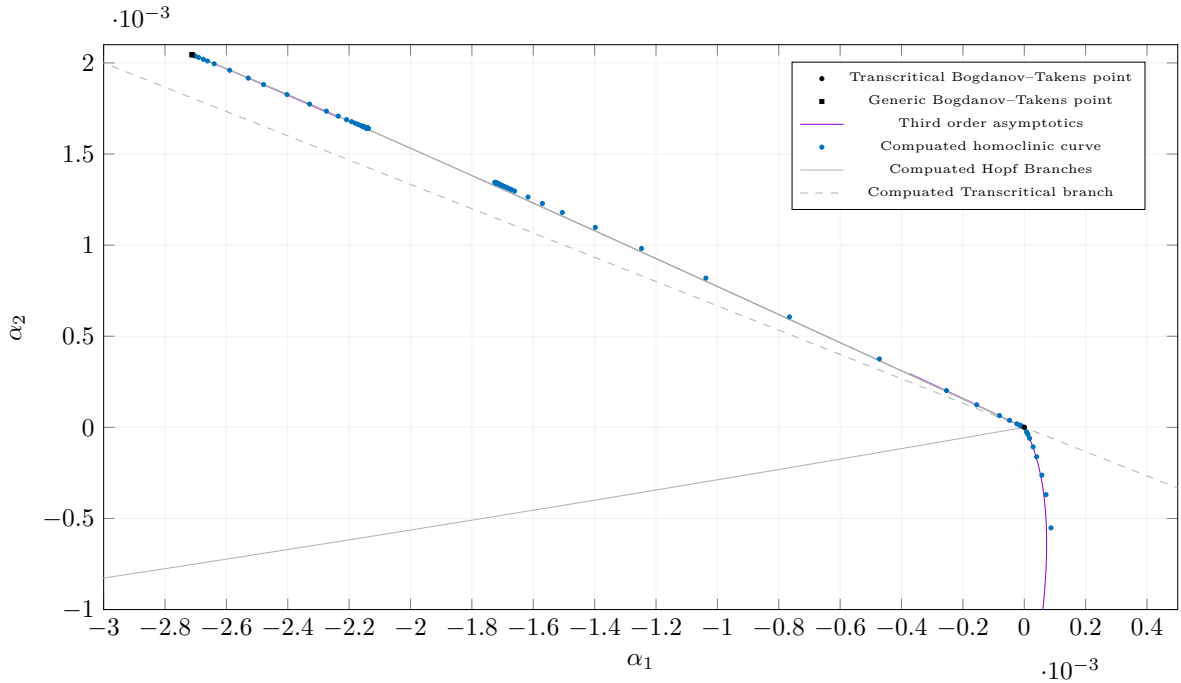

    \includetikzscaled{triNeuronBAMNeuralNetworkModelCompareParameters}
    \caption{
    Bifurcation diagram near the transcritical and generic Bogdanov--Takens
    bifurcation points in \cref{sm:eq:tri_neuron_BAM-u} comparing computed
    codimension one curves using \DDEBIFTOOL with the asymptotics obtained in the 
    main text.}
    \label{sm:fig:triNeuronBAMNeuralNetworkModelCompareParameters}
\end{figure}

\subsection{Large view bifurcation plot without predictors}
Although the first and third homoclinic bifurcation curves exist only in a very small
parameter region, the second is continued in a relatively large parameter range. The
code below produces (a figure similar to)
\cref{sm:fig:triNeuronBAMNeuralNetworkModelLargerBifurctionPlot}.
\inputminted[firstline=245, lastline=267]{MATLAB}{\pathToDDEBifToolDemos/BAM_neural_network_model/BAMnn.m}
\begin{figure}[ht]
    \includetikzscaled{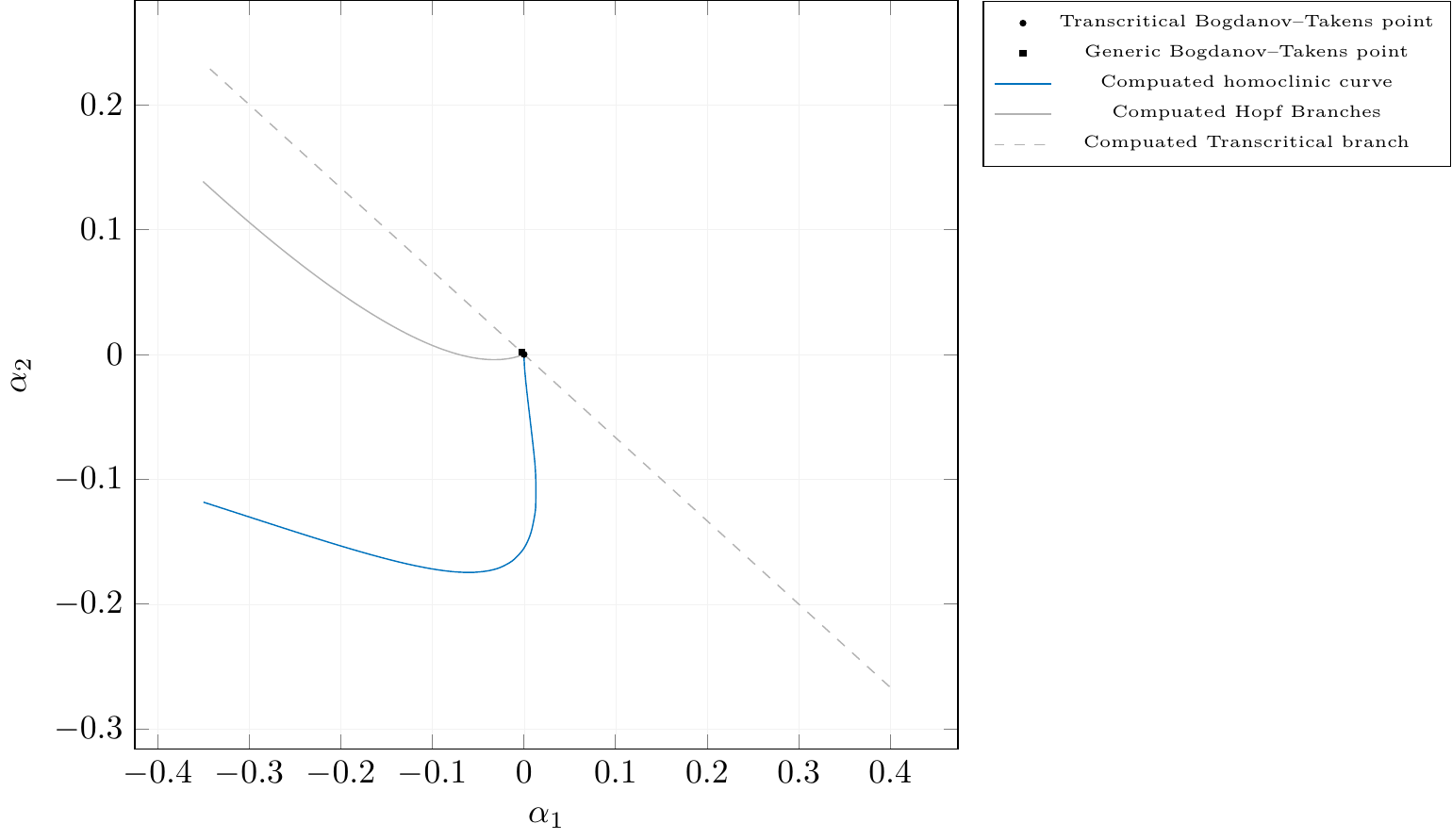}
    \caption{
    Bifurcation diagram near the transcritical and
    a generic Bogdanov--Takens bifurcation points in \cref{sm:eq:tri_neuron_BAM-u} showing the fully
    continued homoclinic branch connected to the non-trivial equilibrium emanating
    from the transcritical Bogdanov--Takens point.}
    \label{sm:fig:triNeuronBAMNeuralNetworkModelLargerBifurctionPlot}
\end{figure}

\subsection{Homoclinic orbits in phase-spase}
To obtain an expression of the continued homoclinic orbits,
we plot the solutions on the various homoclinic branches in phase-space.
The code below produces (a figure similar to)
\cref{sm:fig:triNeuronBAMNeuralNetworkModelOrbitsPhaseSpace}.
In the top left plot, the homoclinic orbits connected to the origin are shown. However,
we see that this does not provide much insight. 
One way to visualize homoclinic solutions is by inspection the
profiles of the solutions. This will be done in the next section.
Another way to reveal the homoclinic solutions connected to the origin is by
rotating the coordinates $(u_1,u_2)$. The result is seen in the top right
plot in \cref{sm:fig:triNeuronBAMNeuralNetworkModelOrbitsPhaseSpace}.
In the bottom left plot the solutions on the second homoclinic branch emanating
from the transcritical Bogdanov--Takens point are shown. Lastly, we plotted the
rotated homoclinic orbits emanating from the generic
Bogdanov--Takens point in the bottom right plot of
\cref{sm:fig:triNeuronBAMNeuralNetworkModelCompareOrbitsPhaseSpace}.
\inputminted[firstline=269, lastline=335]{MATLAB}{\pathToDDEBifToolDemos/BAM_neural_network_model/BAMnn.m}
\begin{figure}[ht!]
    \includegraphics[width=\textwidth]{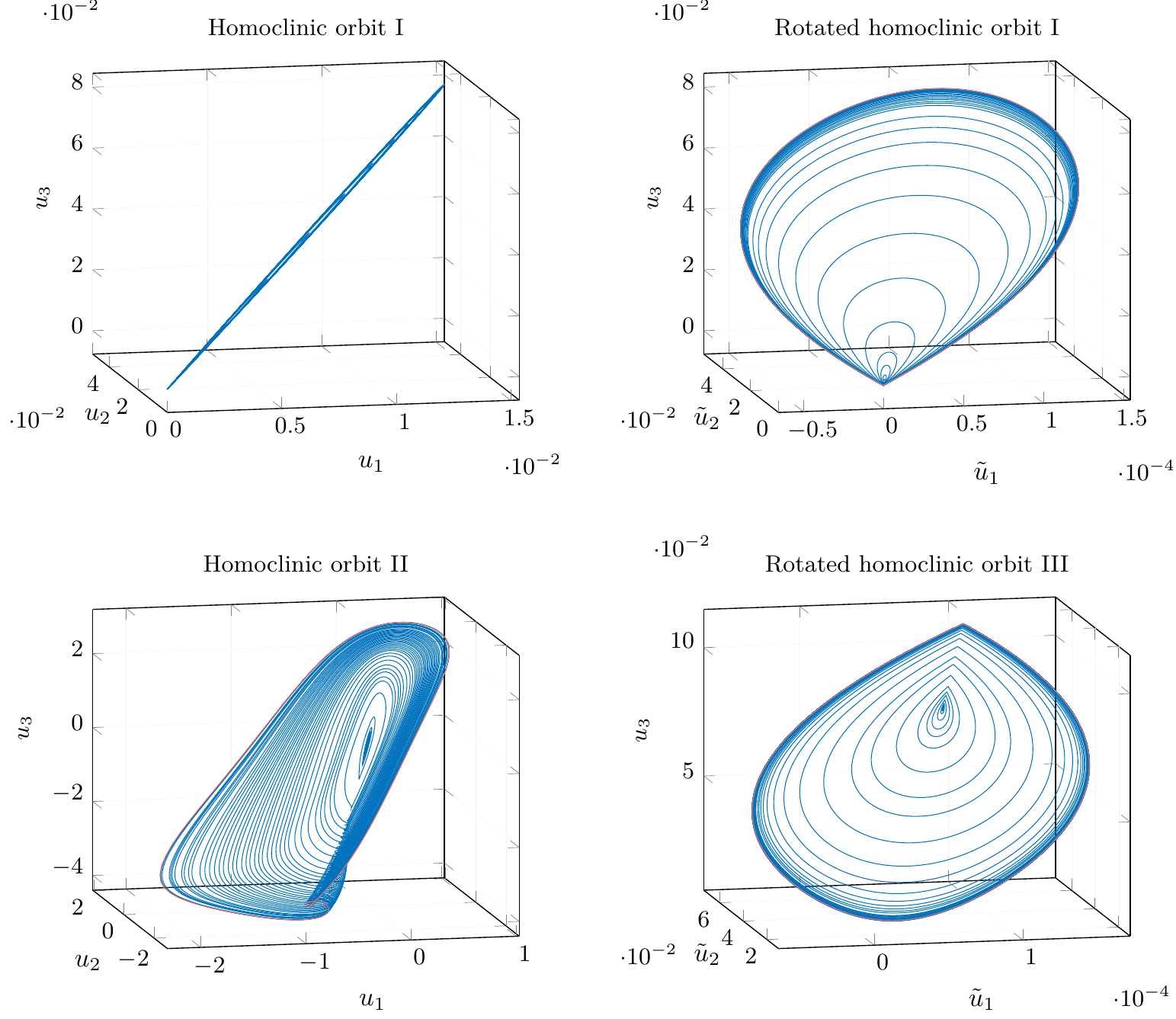}
    \caption{
    Plots of homoclinic solutions emanating from the generic and transcritical
    Bogdanov--Takens bifurcations in \cref{sm:eq:tri_neuron_BAM-u}. In the top
    left plot the homoclinic solutions emanating from the transcritical
    Bogdanov--Takens point with fixed saddle points are shown. In the top right
    plot, we rotated the $(u_1, u_2)$ coordinates in order to make these
    homoclinic solutions visible. The bottom left plot shows the second branch
    of homoclinic solutions emanating from the transcritical Bogdanov--Takens
    point. Lastly, in the bottom right plot, the homoclinic solutions emanating
    from the generic Bogdanov--Takens point are shown. Here also the $(u_1,
    u_2)$ coordinates are rotated.
    }
    \label{sm:fig:triNeuronBAMNeuralNetworkModelOrbitsPhaseSpace}
\end{figure}

\subsection{Compare homoclinic solutions in phase-space}
Here, we compare the corrected and uncorrected profiles of the homoclinic
solutions on the two homoclinic curves emanating from the transcritical
Bogdanov--Takens point with the perturbation parameter ranging from $0.003$ to
$0.009$.
The code below produces (a figure similar to)
\cref{sm:fig:triNeuronBAMNeuralNetworkModelCompareOrbitsPhaseSpace}. We see
that the corrected and predicted homoclinic orbits are nearly identical.
\inputminted[firstline=356, lastline=385]{MATLAB}{\pathToDDEBifToolDemos/BAM_neural_network_model/BAMnn.m}
\begin{figure}[ht]
    \includetikzscaled{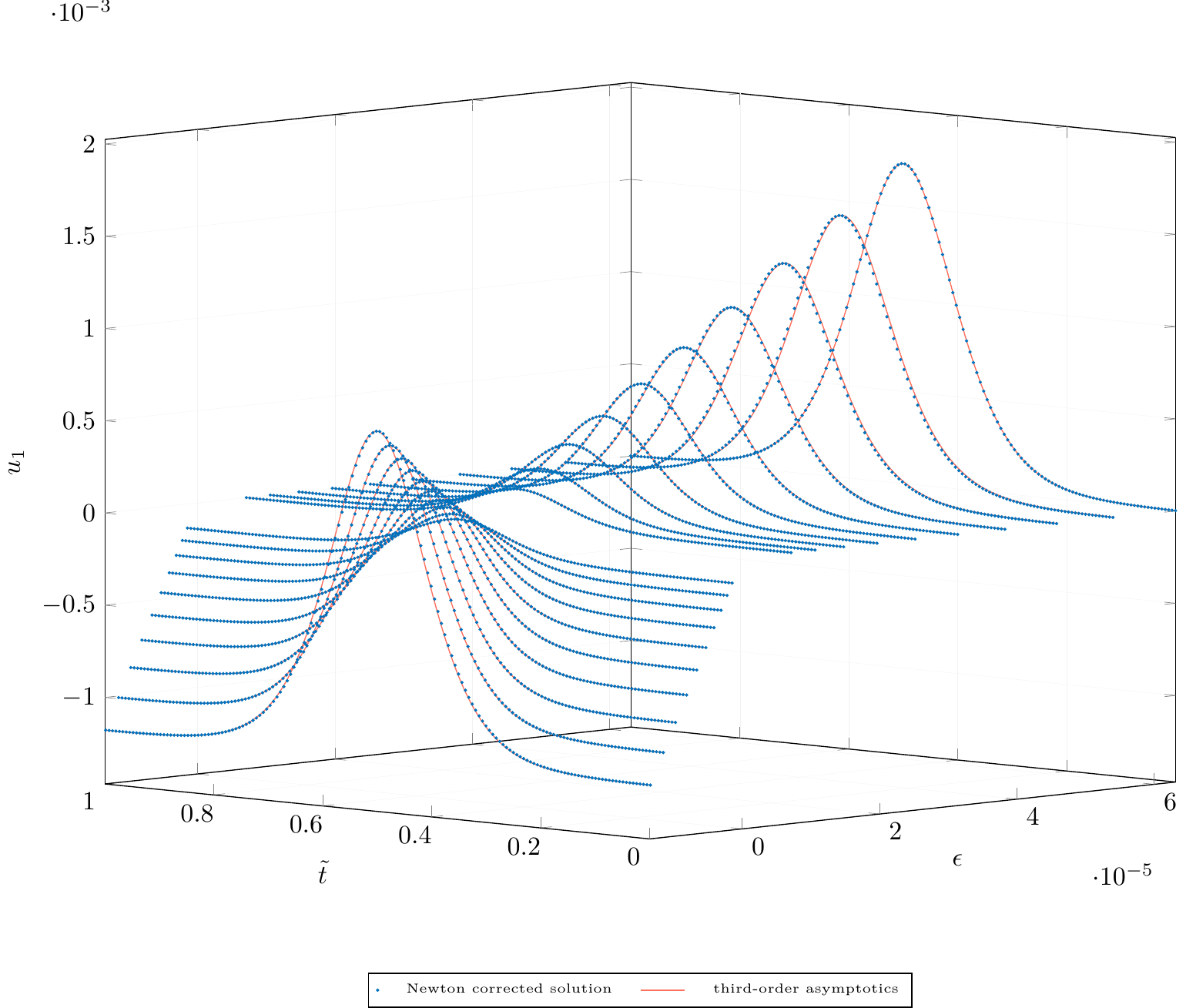}
    \caption{Plot comparing the third-order homoclinic asymptotics with the
        Newton correct homoclinic solutions in $(\epsilon,\tilde t, u_1)$
        phase-space. Here $\tilde t$ is the time $t$ rescaled to the interval
        $[0,1]$.
    }
    \label{sm:fig:triNeuronBAMNeuralNetworkModelCompareOrbitsPhaseSpace}
\end{figure}

\subsection{Continue periodic solutions from Hopf branch to homoclinic branch}
We continue a branch of periodic solutions emanating from point number 29 on the
first Hopf branch emanating from the transcritical Bogdanov--Takens point. The
periodic solution converges to a homoclinic orbit located on the second
homoclinic branch emanating from transcritical Bogdanov--Takens point. We will
use point number 28 on the periodic solution branch to compare against the
simulation in Julia in \cref{sm:sec:triNeuralBAMNetworkModelSimulation}. The
code below produces (a figure similar to)
\cref{sm:fig:triNeuronBAMNeuralNetworkModelPeriodicSolutions}.
\inputminted[firstline=387, lastline=391]{MATLAB}{\pathToDDEBifToolDemos/BAM_neural_network_model/BAMnn.m}
\vspace*{-12pt}
\inputminted[firstline=398, lastline=414]{MATLAB}{\pathToDDEBifToolDemos/BAM_neural_network_model/BAMnn.m}
\begin{figure}[ht]
    \includetikzscaled{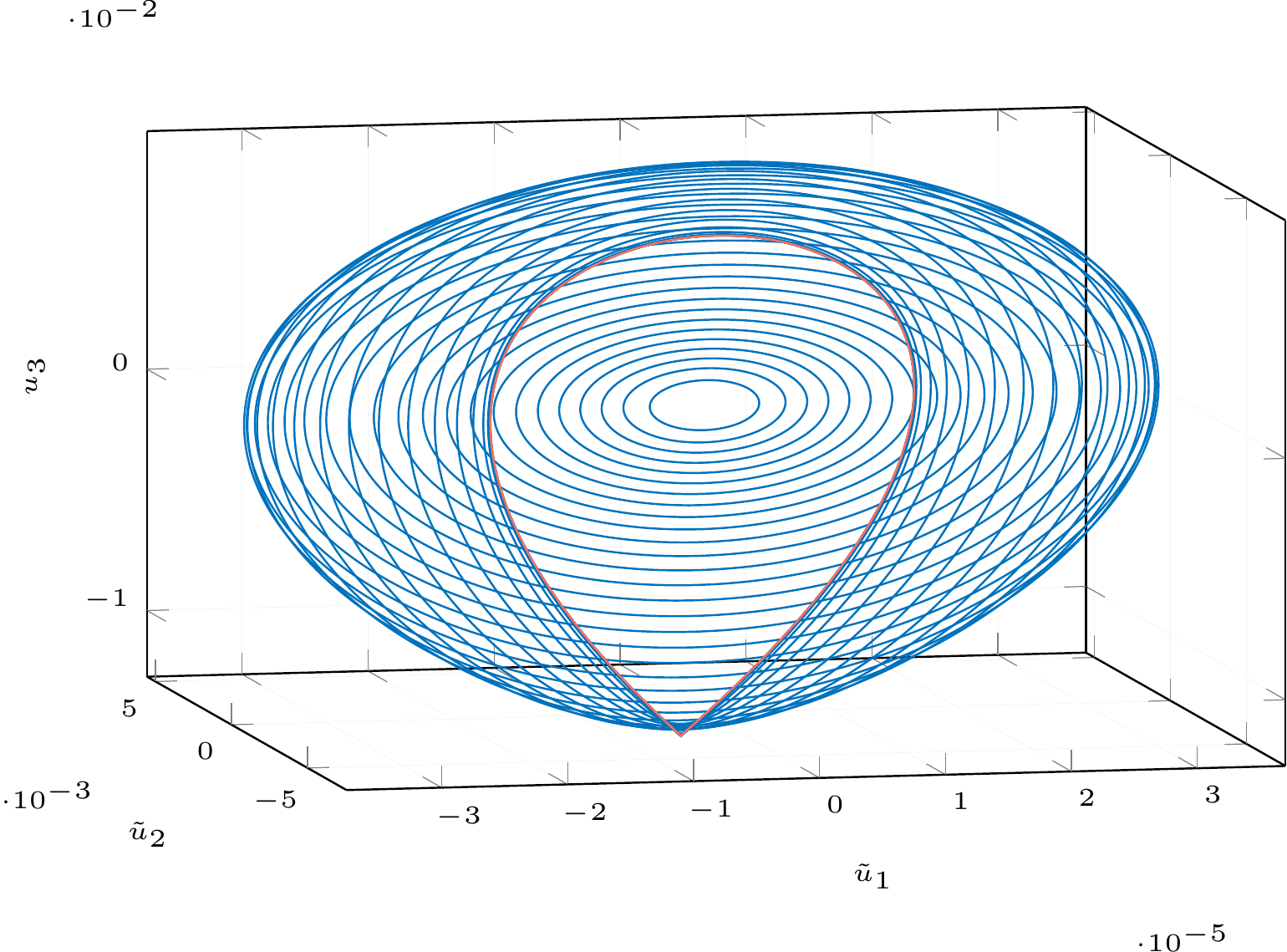}
    \caption{
        Branch of periodic solutions emanating from the last point on the first
        Hopf branch emanating from the transcritical Bogdanov--Takens point in
        \cref{sm:eq:tri_neuron_BAM-u}. The last point on the periodic solution
        branch is colored red, at which the periodic orbits have converged to
        a homoclinic orbit.
    }
    \label{sm:fig:triNeuronBAMNeuralNetworkModelPeriodicSolutions}
\end{figure}

\subsection{{\ifthesis \phantom{ } \fi} Homoclinic branch connecting the two Bogdanov--Takens {\ifthesis \phantom{ } \fi} points}
Here we will show that the transcritical and generic
Bogdanov--Takens points are connected, not only by a Hopf curve, but also
through a homoclinic curve. For this, we first continue the 
second Hopf branch emanating from the transcritical Bogdanov--Takens point again,
but with a smaller step size.
\inputminted[firstline=416, lastline=429]{MATLAB}{\pathToDDEBifToolDemos/BAM_neural_network_model/BAMnn.m}
Next, we continue from (almost) each point on the new Hopf branch the emerging periodic solutions in the parameter $\alpha_2$.
\inputminted[firstline=436, lastline=446]{MATLAB}{\pathToDDEBifToolDemos/BAM_neural_network_model/BAMnn.m}
By plotting the last point on each of the periodic solutions branches in
$(\alpha_1, \tilde u_1, \tilde u_2)$-space,
together with the homoclinic solutions on the first and third homoclinic branches continued
above, it is indeed clear that the two homoclinic branches are connected through a
single homoclinic curve, see
\cref{sm:fig:triNeuronBAMNeuralNetworkModelConnectionHomoclinicSolutions}.
We also see how the transition is made from the homoclinic orbits with a fixed saddle point
at the origin to the homoclinic orbits with a moving saddle. Clearly, \DDEBIFTOOL has
difficulties continuing the homoclinic orbits near the global homoclinic bifurcation point.
\inputminted[firstline=461, lastline=490]{MATLAB}{\pathToDDEBifToolDemos/BAM_neural_network_model/BAMnn.m}
\begin{figure}[ht]
    \includetikzscaled{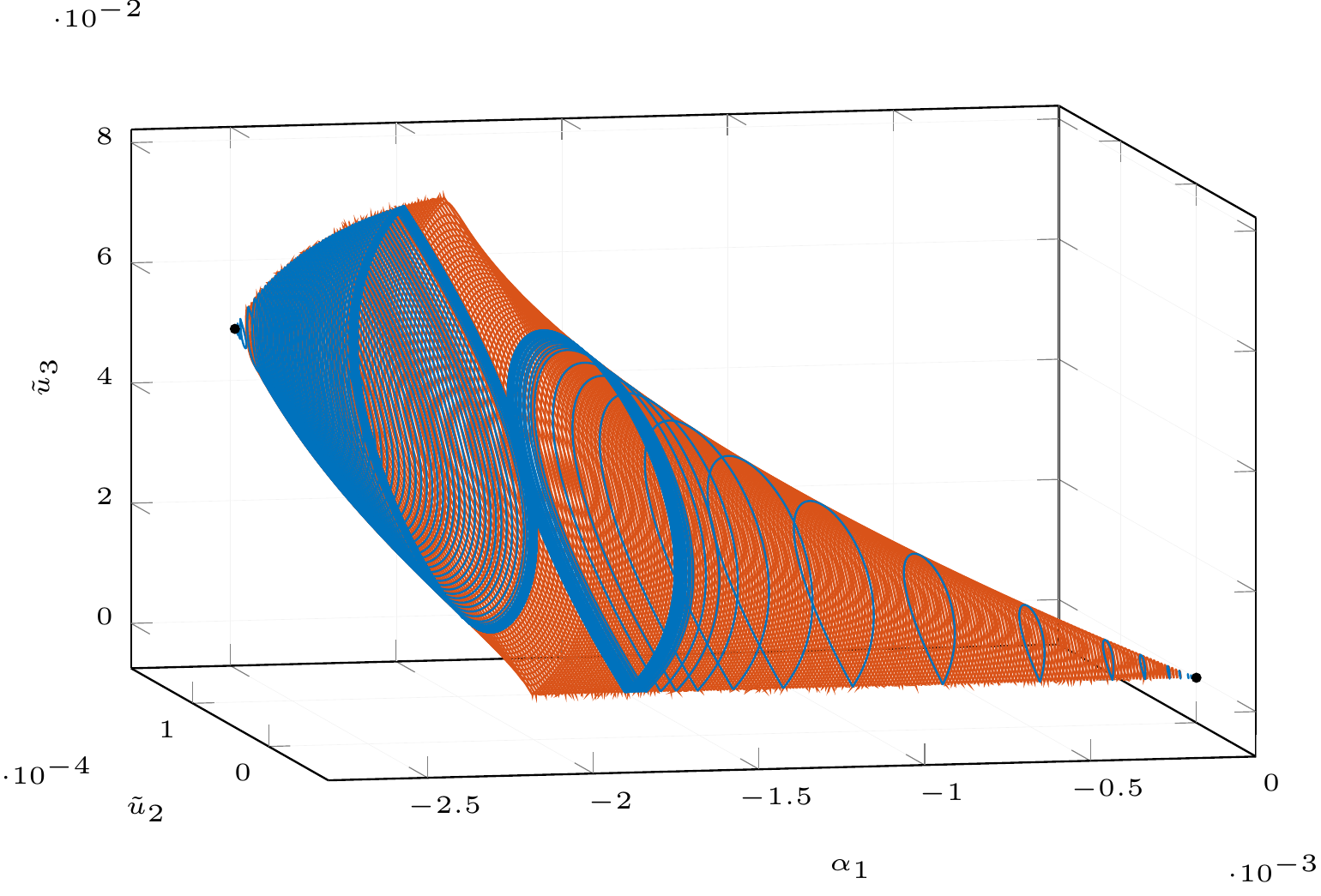}
    \caption{
        Branch of homoclinic solutions (orange) connecting the transcritical Bogdanov--Takens point (the right black dot) with
        the generic Bogdanov--Takens point (the left black dot) in
        \cref{sm:eq:tri_neuron_BAM-u}. The blue homoclinic curve emerging from 
        the transcritical and generic Bogdanov--Takens points are the previously continued
        homoclinic branches.
     }
    \label{sm:fig:triNeuronBAMNeuralNetworkModelConnectionHomoclinicSolutions}
\end{figure}
To obtain an impression of the parameter curves connecting the transcritical and
generic Bogdanov--Takens points, we rotated the  curves through an angle of $\theta_2=0.646045233034992$.
By doing so, the two Bogdanov--Takens points are aligned on the abscissa. In
\cref{sm:fig:triNeuronBAMNeuralNetworkModelConnectionHomoclinicParameters}, we plotted the first and third homoclinic branch, emanating from the
transcritical Bogdanov--Takens and generic Bogdanov--Takens points
respectively, the Hopf curve connecting the two Bogdanov--Takens points, and
the newly obtained homoclinic bifurcation curve.
\inputminted[firstline=501, lastline=529]{MATLAB}{\pathToDDEBifToolDemos/BAM_neural_network_model/BAMnn.m}
\begin{figure}[ht]
    \centering
    \includetikzscaled{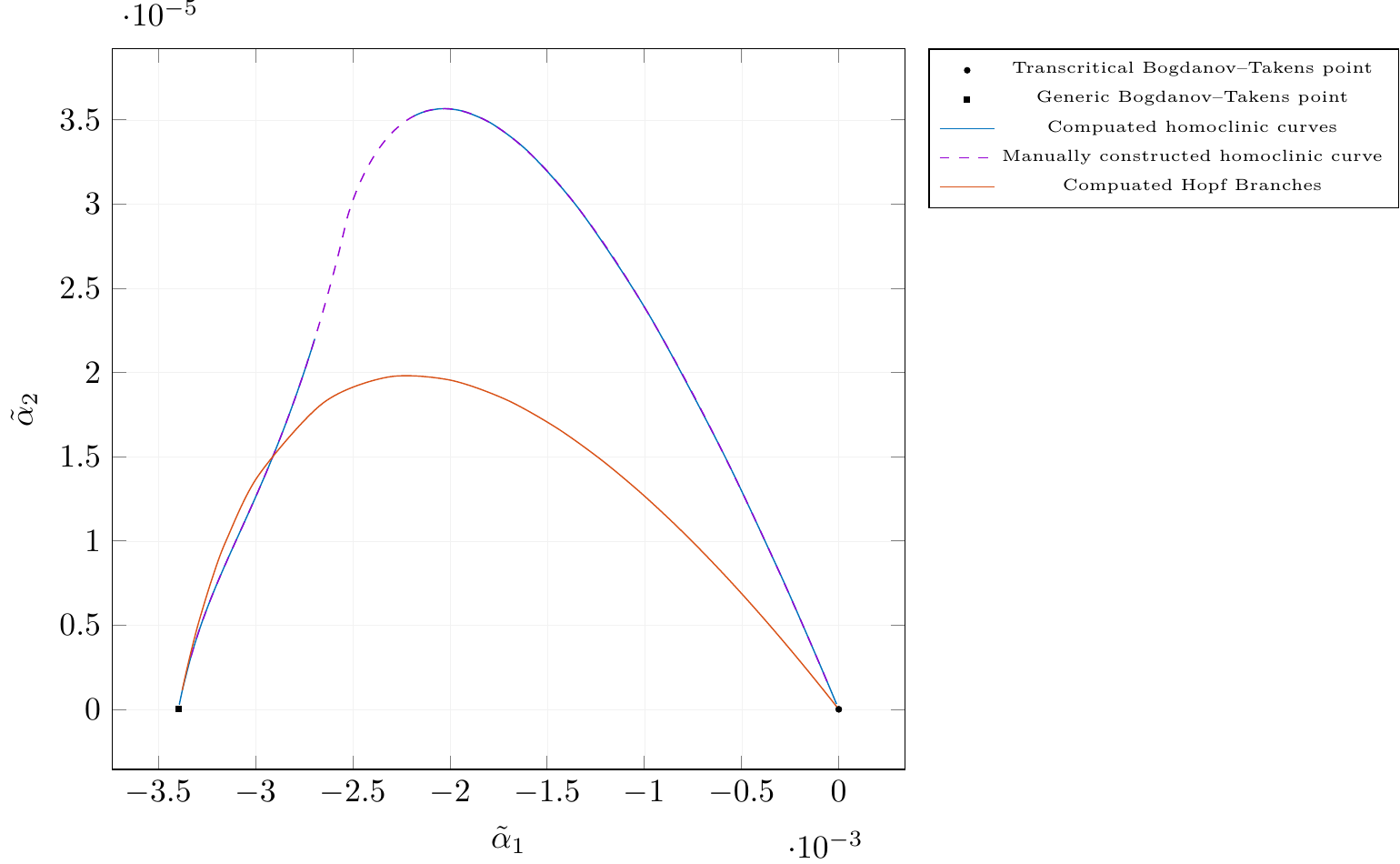}
    \caption{
    Bifurcation diagram near the transcritical and generic Bogdanov--Takens
    bifurcation points in \cref{sm:eq:tri_neuron_BAM-u} showing the fully continued
    homoclinic branch connected to the non-trivial equilibrium emanating from the
    transcritical Bogdanov--Takens point.}
    \label{sm:fig:triNeuronBAMNeuralNetworkModelConnectionHomoclinicParameters}
\end{figure}

\subsection{Convergence plot}
\label{sm:sec:tri_neuron_BAM:convergence_plot}
Using the function from \cref{sm:lst:convergence_plot}, we create a log-log
convergence plot comparing the convergence order of the first and third order
homoclinic asymptotics from \cref{btdde:sec:transcritical_bt_homoclinic_asymptotics}.
The code below yields \cref{sm:fig:triNeuralBAMNetworkModelConvergencePlot}.
\inputminted[firstline=531, lastline=542]{MATLAB}{\pathToDDEBifToolDemos/BAM_neural_network_model/BAMnn.m}
\begin{figure}[ht]
    \centering
    \ifcompileimages%
        \tikzsetnextfilename{triNeuralBAMNetworkModelConvergencePlot}%
        \input{tikz/triNeuralBAMNetworkModelConvergencePlot}%
    \else
        \includegraphics{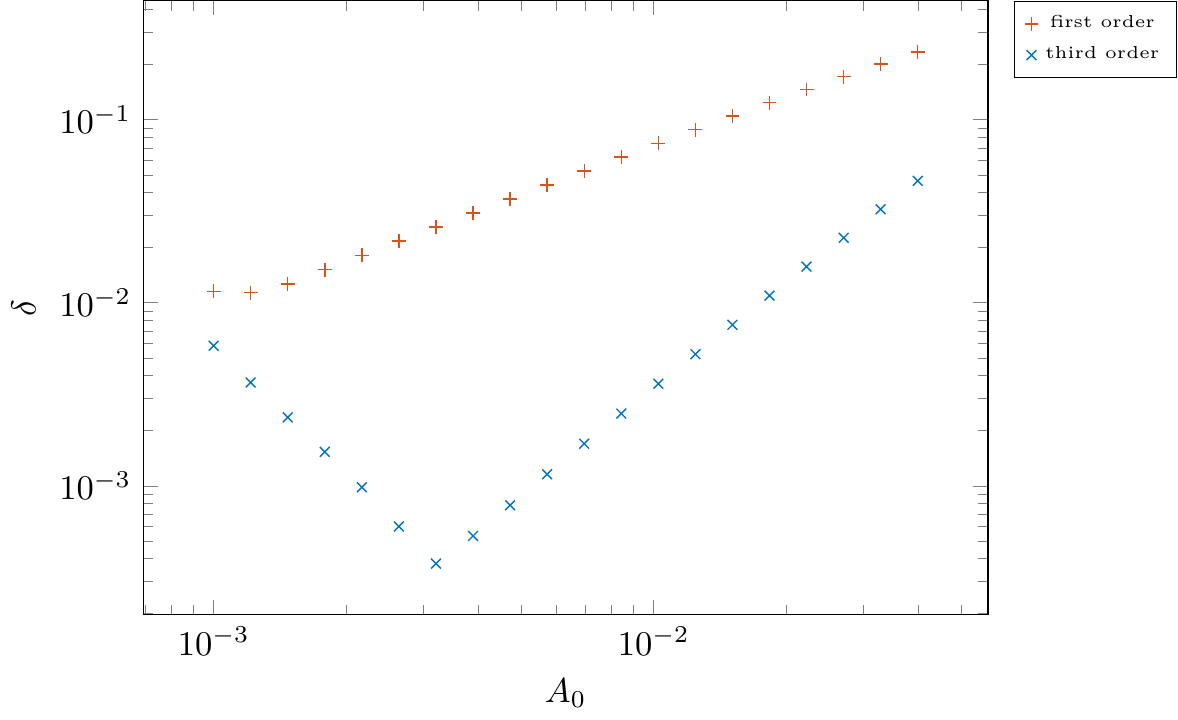}
    \fi

     \caption{On the abscissa is the approximation to the amplitude $A_0$ and on
        the ordinate the relative error $\delta$ between the constructed solution
        \mintinline{MATLAB}{hcli_pred} to the defining system for the homoclinic orbit
        and the Newton corrected solution \mintinline{MATLAB}{hcli_corrected}.}
    \label{sm:fig:triNeuralBAMNetworkModelConvergencePlot}
\end{figure}

\subsection{Simulation with {\tt DifferentialEquations.jl}}
\label{sm:sec:triNeuralBAMNetworkModelSimulation}
Here we will perform four simulations. The first two simulations will be at two
homoclinic orbits located on the two homoclinic curves emanating from the
transcritical Bogdanov--Takens point continued with \DDEBIFTOOL, see
\cref{sm:fig:triNeuralBAMNetworkSimulationHomoclinic}. The second two simulations will be in
the regions where there should be stable periodic orbits, see
\cref{sm:fig:triNeuralBAMNetworkSimulationPeriodic}.

\begin{figure}[ht]
    \includegraphics{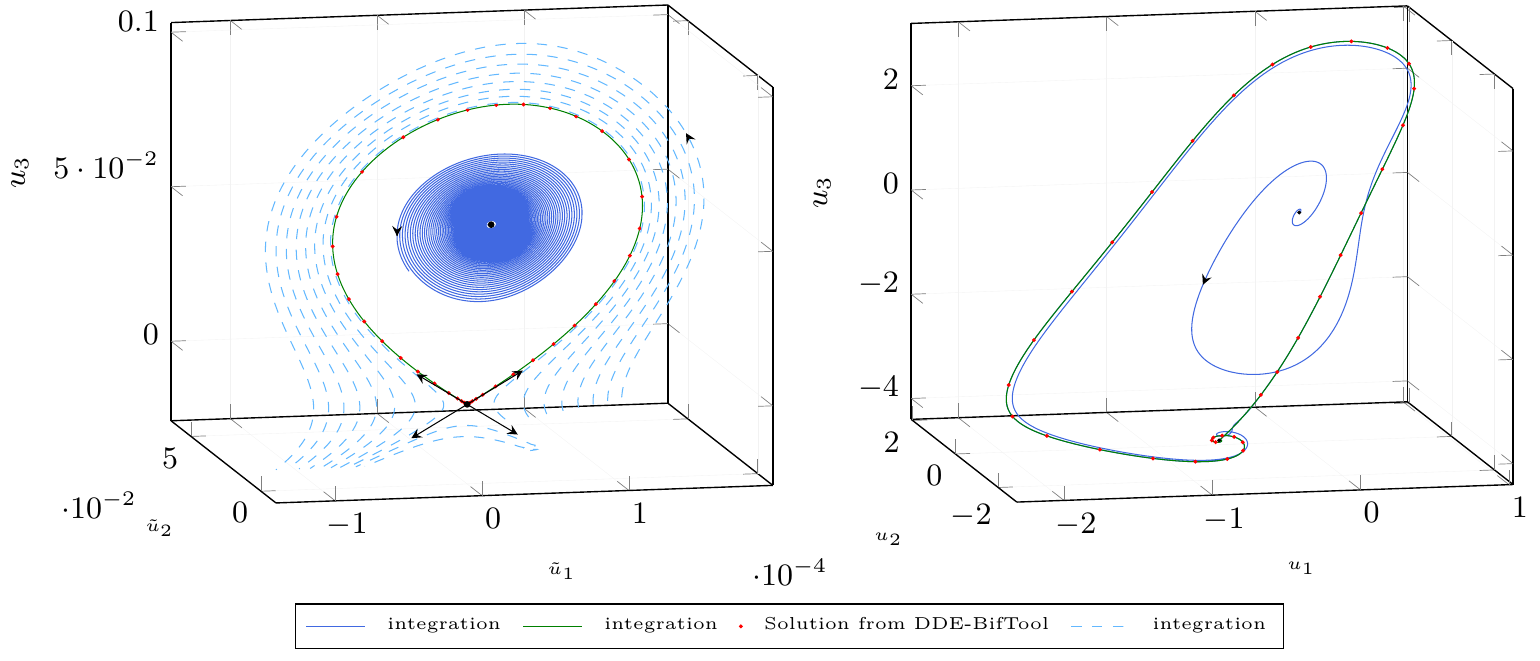}
    \caption{Comparing the computed homoclinic orbits in \cref{sm:eq:tri_neuron_BAM-u}
    with \DDEBIFTOOL with the solutions obtained from numerical simulation with Julia.
    We see the numerical integrated solution
    going through all the red points from the solution from \DDEBIFTOOL.}
    \label{sm:fig:triNeuralBAMNetworkSimulationHomoclinic}
\end{figure}

\begin{figure}[ht]
    \includegraphics{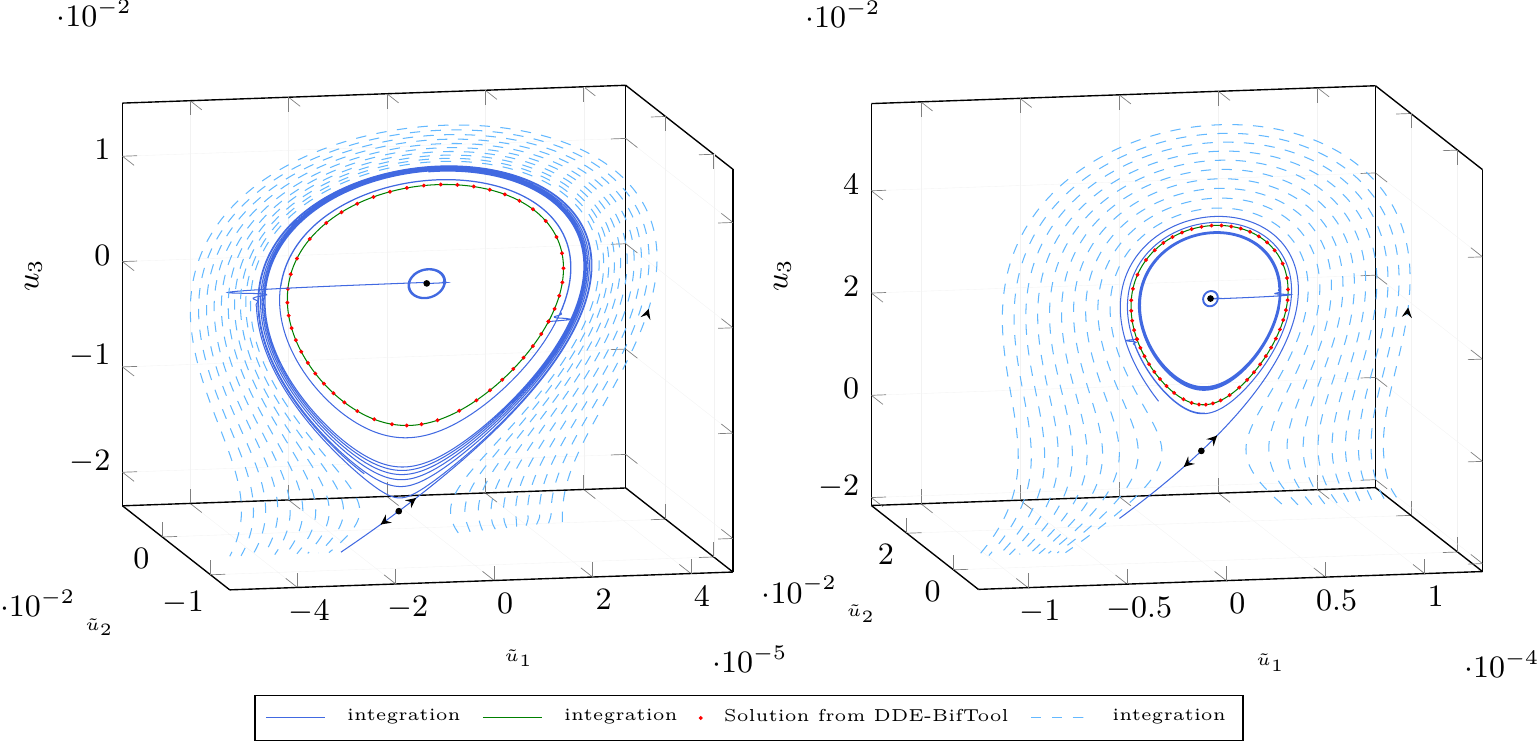}
    \caption{Comparing the computed periodic orbits in \cref{sm:eq:tri_neuron_BAM-u}
    with \DDEBIFTOOL with the solutions obtained from numerical simulation with Julia.
    We see the numerical integrated solution
    going through all the red points from the solution from \DDEBIFTOOL.}
    \label{sm:fig:triNeuralBAMNetworkSimulationPeriodic}
\end{figure}

\subsubsection{Loading necessary Julia packages}
We start by loading the necessary Julia packages. Compared with the previous
demonstrations, we also need to load the Julia package {\tt Symbolics.jl} \cite{Symbolics.jl} to
differentiate the activation functions $f_1,f_2,f_3$.
\inputminted[firstline=1, lastline=9]{julia}{\pathToJuliaFiles/triNeuralBAMNetworkModel_simulation_article.jl}

\subsubsection{Define system}
Next we define the system to be integrated, a system to approximate the reverse
flow, and also an allocating version used for stability calculations. Note that
we need the Julia {\tt Symbolics.jl} to differentiate the activation functions.
\inputminted[firstline=11, lastline=63]{julia}{\pathToJuliaFiles/triNeuralBAMNetworkModel_simulation_article.jl}

\subsubsection{Function for creating streamlines plot}
To obtain an impression of the flow near transcritical Bogdanov--Takens point,
we create a streamlines function. This is particularly useful for seeing the
flow around the stable manifold of the saddle-note.
\inputminted[firstline=65, lastline=77]{julia}{\pathToJuliaFiles/triNeuralBAMNetworkModel_simulation_article.jl}

\subsubsection{Create figure with several axes}
We create a figure containing multiple axis in which we will plot 
the homoclinic, periodic orbits, and the left and right-hand sides
of \cref{sm:eq:triNeuralBAMNetworkModel:tan}.
\inputminted[firstline=79, lastline=89]{julia}{\pathToJuliaFiles/triNeuralBAMNetworkModel_simulation_article.jl}

\subsubsection{Define parameters, equilibria}
We define parameters located on the continued homoclinic branch with
\DDEBIFTOOL. Then calculate the equilibria points in \cref{sm:eq:tri_neuron_BAM}
near the transcritical Bogdanov--Takens point.
\inputminted[firstline=91, lastline=98]{julia}{\pathToJuliaFiles/triNeuralBAMNetworkModel_simulation_article.jl}

\subsubsection{Plot equilibria and homoclinic orbit}
By plotting the homoclinic orbit obtained with \DDEBIFTOOL located at parameter
values 
\[
    (\alpha_1, \alpha_2) = (-0.001724521613831, 0.001344362436730),
\]
we can compare with the numerical simulations. As in the analysis with
\DDEBIFTOOL, we rotate the coordinates of the solutions to visualize the
solutions in phase-space.
\inputminted[firstline=100, lastline=110]{julia}{\pathToJuliaFiles/triNeuralBAMNetworkModel_simulation_article.jl}

\subsubsection{Leading eigenvectors}
Next, we calculate and plot the leading eigenvectors of the characteristic matrix at the saddle-node bifurcation point.
\inputminted[firstline=112, lastline=135]{julia}{\pathToJuliaFiles/triNeuralBAMNetworkModel_simulation_article.jl}

\subsubsection{Define callback}
Since we are only interested in the flow near the equilibria points, we create a
continuous callback to ensure the orbits do not become too large.
\inputminted[firstline=137, lastline=140]{julia}{\pathToJuliaFiles/triNeuralBAMNetworkModel_simulation_article.jl}

\subsubsection{Integrate the system at homoclinic orbits I}
Now we define the problem to be integrated and choose the algorithm to be used.
Then we integrate the system for a range of initial history functions using the
function \mintinline{julia}{streamlines}. Next, we integrate the system near
the inner equilibrium, i.e., the equilibrium inside the homoclinic orbit. This
equilibria should be an unstable spiral. We show the first and last part
of the obtained solutions. We see that the last part of the integrated solutions
completely overlap the homoclinic solution obtained with \DDEBIFTOOL.
\inputminted[firstline=142, lastline=165]{julia}{\pathToJuliaFiles/triNeuralBAMNetworkModel_simulation_article.jl}

\subsubsection{Add arrows on solutions}
Lastly, we add arrows to the obtained solutions and redraw the equilibria.
\inputminted[firstline=167, lastline=178]{julia}{\pathToJuliaFiles/triNeuralBAMNetworkModel_simulation_article.jl}
We should now obtain an interactive figure similar to the left figure in
\cref{sm:fig:triNeuralBAMNetworkSimulationHomoclinic}.

\subsubsection{Simulation near stable periodic orbit I}
The code for numerical simulation near the second homoclinic orbit, see the
right plot in \cref{sm:fig:triNeuralBAMNetworkSimulationHomoclinic}, is not
included here.

To show by integration the existence of a stable periodic orbit, we first
located a periodic orbit in \DDEBIFTOOL. This can be done by continuing a
branch of periodic orbits emanating from a point on the continued Hopf curves.
Then we load the profiles of the periodic orbits into Julia. We perform two
simulations. For the first we integrate with a constant history function equal
to a point inside the periodic orbit. The second simulation starts from the
unstable eigenvector of the characteristic matrix calculated above.
\inputminted[firstline=237, lastline=328]{julia}{\pathToJuliaFiles/triNeuralBAMNetworkModel_simulation_article.jl}
After running the above code, we should obtain a similar plot as in
\cref{sm:fig:triNeuralBAMNetworkSimulationPeriodic}.

\subsubsection{Stability of the center manifold}
To confirm \cref{sm:lemma:triNeuralBAMNetworkModelEigenvalues} numerically, we
consider again
\cref{sm:eq:triNeuralBAMNetworkModelFunctions,sm:eq:triNeuralBAMNetworkModelFixedParameters}.
The code below plots the left and right-hand sides of
\cref{sm:eq:triNeuralBAMNetworkModel:tan}. We see in
\cref{sm:fig:triNeuralBAMNetworkStabilityDeterminingFunction} that there are
two points of intersection.
\inputminted[firstline=421, lastline=439]{julia}{\pathToJuliaFiles/triNeuralBAMNetworkModel_simulation_article.jl}
\begin{figure}[ht]
    \centering
    \includetikzscaled{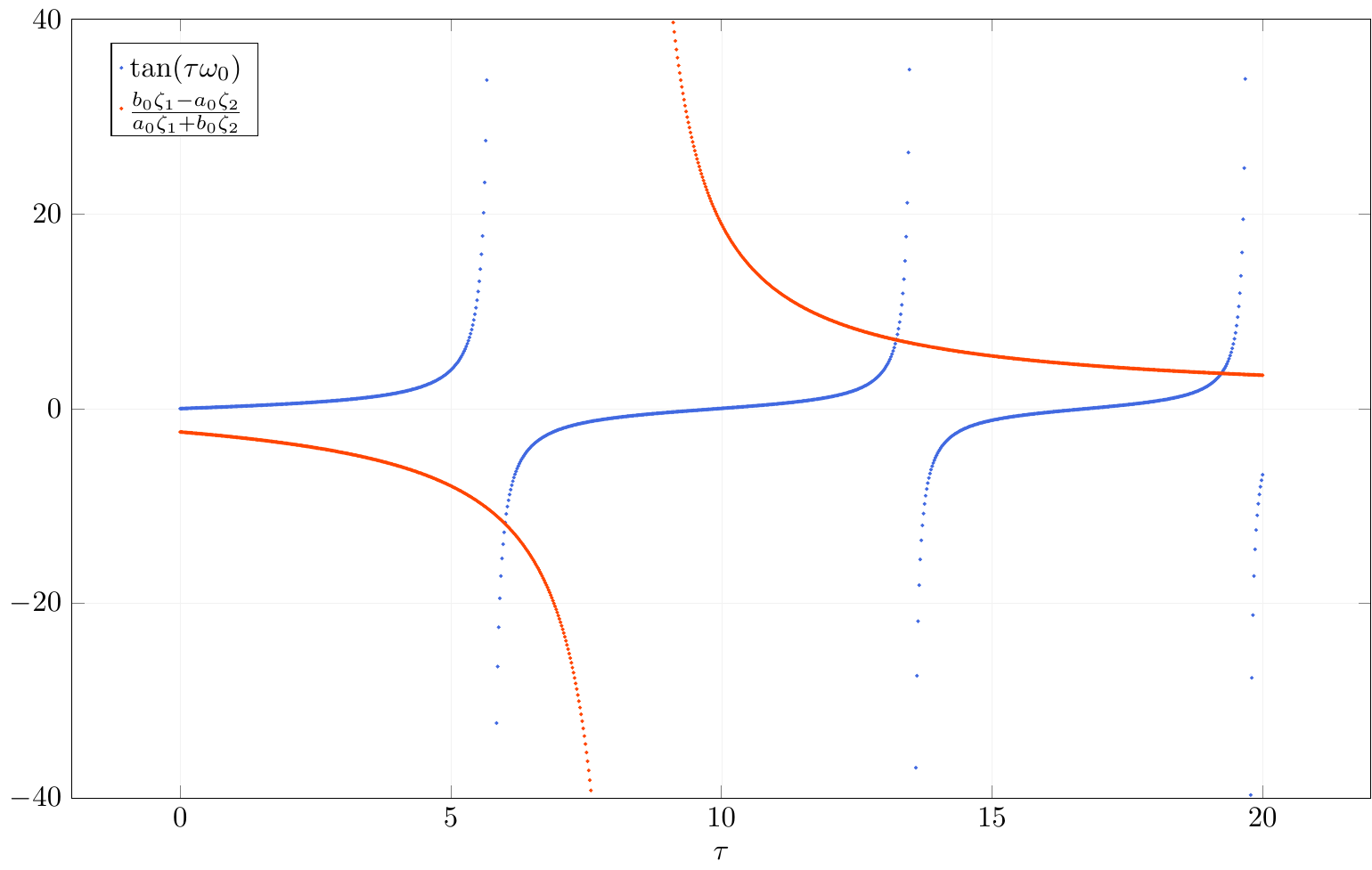}
    \caption{Plot of the left and right-hand sides of
    \cref{sm:eq:triNeuralBAMNetworkModel:tan}. Points of intersection are
    candidates for the center manifold to lose its stability.}
    \label{sm:fig:triNeuralBAMNetworkStabilityDeterminingFunction}
\end{figure}
Using the function \mintinline{julia}{roots} from the package {\tt IntervalRootFinding.jl}
we search for points of intersection. 
\inputminted[firstline=441, lastline=444]{julia}{\pathToJuliaFiles/triNeuralBAMNetworkModel_simulation_article.jl}
In the Julia output we obtain
\begin{minted}{shell-session}
7-element Vector{Root{Interval{Float64}}}:
 Root([13.2309, 13.231], :unique)
 Root([19.2374, 19.2375], :unique)
 Root([19.7337, 19.7338], :unknown)
 Root([6.00301, 6.00302], :unique)
 Root([8.33333, 8.33334], :unknown)
 Root([13.5353, 13.5354], :unknown)
 Root([5.75402, 5.75403], :unknown)
 Root([8.33333, 8.33334], :unknown)
\end{minted}
Note that, since there are multiple discontinuities, there are many unknown
solutions returned. We extract the unique solutions from the list of solutions
and test if these provide solutions to the characteristic equation.
\inputminted[firstline=446, lastline=457]{julia}{\pathToJuliaFiles/triNeuralBAMNetworkModel_simulation_article.jl}
In the Julia output we obtain
\begin{minted}{shell-session}
The centermanifold is locally attractive for 0 < τ < 13.230934887939895.
\end{minted}

\begin{figure}[ht]
    \centering
    \includetikzscaled{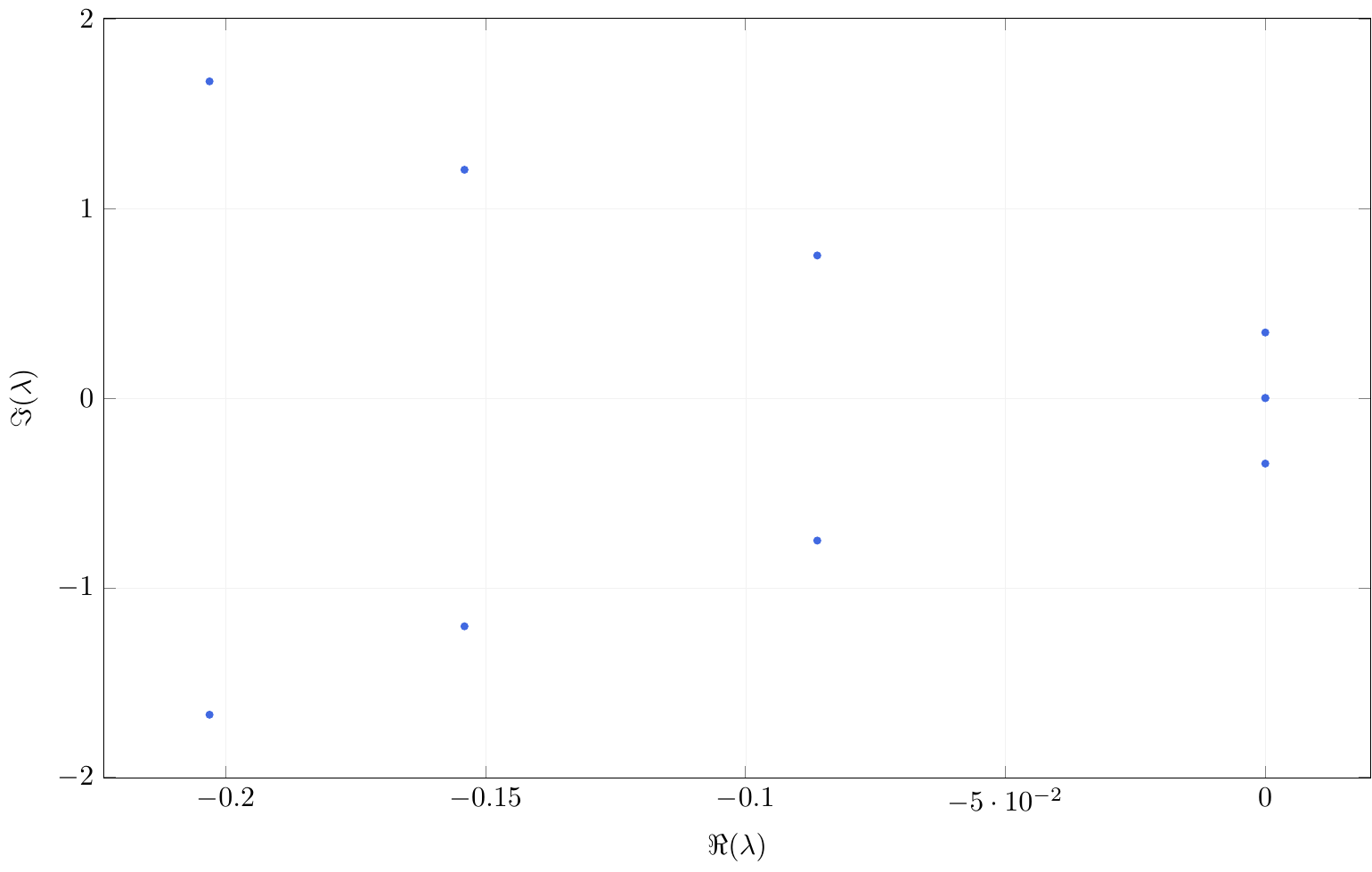}
    \caption{Plot of the leading eigenvalues at the analytically derived
    transcritical Bogdanov--Takens point with $\tau = 13.230934887939895$
    and $\omega = \omega(\tau)$ from \cref{sm:eq:BAM_omega}. At this
    point there are four eigenvalues on the imaginary axis. For
    $\tau > 13.230934887939895$ the center manifold is locally unstable.}
    \label{sm:fig:triNeuralBAMNetworkModelEigenvalues}
\end{figure}
\subsubsection{Calculate and plot stability at Bogdanov--Takens point}
We finish this demonstration by confirming the stability
of the transcritical Bogdanov--Takens point at $\tau = 13.230934887939895$
obtained in the previous section. In \cref{sm:fig:triNeuralBAMNetworkModelEigenvalues}
we have plotted the eigenvalues. We indeed see that at $\tau = 13.230934887939895$
the center manifold loses stability. Lastly, we also verified in the code
below that the eigenvalues with positive imaginary part on the imaginary axis
is approximately equal to the expression for $\omega$ obtained in \cref{sm:eq:BAM_omega}. 
\inputminted[firstline=459, lastline=467]{julia}{\pathToJuliaFiles/triNeuralBAMNetworkModel_simulation_article.jl}

\clearpage
\setcounter{page}{1}
\renewcommand{\thepage}{\roman{page}}
\pdfbookmark[0]{References}{references}
\renewcommand\refname{References for main text and supplement}

\fi

\bibliographystyle{siamplain}
\bibliography{references}

\end{document}